\documentclass[11pt]{amsart}
\usepackage[utf8]{inputenc}

\usepackage{overpic, amssymb, times, graphicx, tikz-cd, yfonts, thmtools, cancel, wrapfig}
\usepackage[all]{xy}
\usepackage{verbatim}
\usepackage{enumitem}
\usepackage{mathabx}
\usepackage{tikz}
\usepackage[backref=page, hypertexnames=false]{hyperref}

\hypersetup{
	colorlinks=true,
	linkcolor=blue,
	filecolor=blue,
	urlcolor=blue,
	citecolor=blue,
}
\makeatletter
\newcommand{\colim@}[2]{%
  \vtop{\m@th\ialign{##\cr
    \hfil$#1\operator@font colim$\hfil\cr
    \noalign{\nointerlineskip\kern1.5\ex@}#2\cr
    \noalign{\nointerlineskip\kern-\ex@}\cr}}%
}
\newcommand{\colim}{%
  \mathop{\mathpalette\colim@{\rightarrowfill@\textstyle}}\nmlimits@
}
\makeatother

\DeclareMathOperator{\Bimod}{Bimod}
\DeclareMathOperator{\Det}{Det} 

\DeclareMathOperator{\Tor}{Tor}

\DeclareMathOperator{\Div}{div}
\DeclareMathOperator{\Collar}{Col}

\DeclareMathOperator{\codim}{codim}

\DeclareMathOperator{\stab}{st} 
\DeclareMathOperator{\Stable}{Stab}
\DeclareMathOperator{\Unstable}{Unst}
\DeclareMathOperator{\ord}{ord}
\DeclareMathOperator{\Support}{Supp}

\DeclareMathOperator{\Sec}{Sec}

\DeclareMathOperator{\Morse}{Mo}
\DeclareMathOperator{\Crit}{Crit}
\DeclareMathOperator{\cm}{cm} 
\DeclareMathOperator{\hot}{hot} 
\DeclareMathOperator{\Map}{Map} 

\DeclareMathOperator{\wind}{wind} 
\DeclareMathOperator{\im}{im} 
\DeclareMathOperator{\ind}{ind} 
\DeclareMathOperator{\Aut}{Aut}
\DeclareMathOperator{\coker}{coker}
\DeclareMathOperator{\rank}{rank}

\DeclareMathOperator{\dom}{dom}

\DeclareMathOperator{\Hom}{Hom}

\DeclareMathOperator{\Symp}{Symp}

\DeclareMathOperator{\Id}{Id}
\DeclareMathOperator{\Maslov}{M}
\DeclareMathOperator{\CZ}{CZ}

\DeclareMathOperator{\Flow}{Flow}

\DeclareMathOperator{\sgn}{sgn}
\DeclareMathOperator{\Diag}{Diag}

\DeclareMathOperator{\Ret}{Ret}

\DeclareMathOperator{\SFT}{SFT}
\DeclareMathOperator{\Sobolev}{W}
\DeclareMathOperator{\Ltwo}{L^{2}}

\newcommand{\Cinfty}{\mathcal{C}^{\infty}}
\newcommand{\CZero}{\mathcal{C}^{0}}
\newcommand{\COne}{\mathcal{C}^{1}}
\newcommand{\norm}[1]{\left\lVert#1\right\rVert}
\newcommand{\lieSymp}{\mathfrak{symp}}

\newcommand{\B}{\mathbb{B}}
\newcommand{\C}{\mathbb{C}}
\newcommand{\R}{\mathbb{R}}
\newcommand{\Z}{\mathbb{Z}}

\newcommand{\Q}{\mathbb{Q}}
\newcommand{\disk}{\mathbb{D}}
\newcommand{\grad}{\nabla}
\newcommand{\bigO}{\mathcal{O}}
\newcommand{\Circle}{\mathbb{S}^{1}}

\newcommand{\half}{\frac{1}{2}}
\newcommand{\Banach}{\mathcal{B}}
\newcommand{\BanachBundle}{\mathcal{E}}
\newcommand{\Field}{\mathbb{F}}
\newcommand{\Ring}{\mathcal{R}}
\newcommand{\HomologyCoeffs}{\mathcal{H}}
\newcommand{\Algebra}{\mathfrak{A}}
\newcommand{\AlgebraB}{\mathfrak{B}}
\newcommand{\AlgebraC}{\mathfrak{C}}
\newcommand{\AlgebraD}{\mathfrak{D}}
\newcommand{\tensorAlg}{\mathcal{T}}
\newcommand{\tensorAlgGraded}{\mathcal{S}}

\newcommand{\Module}{\mathfrak{M}}
\newcommand{\ModuleN}{\mathfrak{N}}
\newcommand{\BarR}{\mathfrak{B}}
\newcommand{\otimesDerived}{\otimes^{\mathbb{L}}}
\DeclareMathOperator{\aug}{\thicc{\epsilon}}
\DeclareMathOperator{\augDelta}{\thicc{\delta}}

\DeclareMathOperator{\ChernVar}{\mathfrak{c}}

\newcommand{\complete}[1]{\widehat{#1}} 
\newcommand{\Liouville}{V}
\newcommand{\sphere}{\mathbb{S}}

\newcommand{\Mxi}{(M,\xi)}

\newcommand{\divSet}{\Gamma}
\newcommand{\xiDivSet}{\xi_{\divSet}}
\newcommand{\MxiDivSet}{(\divSet, \xiDivSet)}

\newcommand{\betaLiouville}{\beta_{\Liouville}}
\newcommand{\hypersurface}{W}
\newcommand{\subsetC}{\mathcal{U}}
\newcommand{\Reeb}{R}
\newcommand{\alphaDivSet}{\widecheck{\alpha}}
\newcommand{\ReebDivSet}{\widecheck{\Reeb}}
\newcommand{\JDivSet}{\widecheck{J}}
\newcommand{\orbit}{\gamma}
\newcommand{\orbitDivSet}{\widecheck{\orbit}}
\newcommand{\mapHomotopyClass}{\thicc{h}}

\newcommand{\twistedFunc}{\mathfrak{F}}
\newcommand{\Norbit}{\mathcal{N}}
\newcommand{\NorbitDivSet}{\Norbit_{\divSet}}
\newcommand{\aCircle}{(\Z/a\Z)}
\newcommand{\action}{\mathcal{A}}
\newcommand{\actionBound}{\thicc{\action}}
\newcommand{\actionBoundDivSet}{\widecheck{\actionBound}}
\newcommand{\framing}{\mathfrak{f}}
\newcommand{\delbar}{\overline{\partial}}

\newcommand{\eigenBound}{\thicc{\lambda}}

\newcommand{\domainJ}{\jmath}
\newcommand{\leaf}{\mathcal{L}}
\newcommand{\pert}{\nu}
\newcommand{\pertSpace}{\mathcal{P}}
\newcommand{\multisec}{\thicc{\mathfrak{s}}}

\newcommand{\Multisec}{\thicc{\mathfrak{S}}}

\newcommand{\Dlinearized}{\thicc{D}}
\newcommand{\AsymptoticOp}{\thicc{A}}
\newcommand{\ModSpace}{\mathcal{M}}
\newcommand{\ZeroSet}{\mathcal{Z}}

\newcommand{\Teich}{\mathcal{T}}

\newcommand{\halfcyl}{P}
\newcommand{\annulus}{A}

\newcommand{\NnegativePunctures}{\underline{N}}
\newcommand{\NpositivePunctures}{\overline{N}}
\newcommand{\Npunctures}{N}

\newcommand{\tree}{\thicc{T}}
\newcommand{\vertex}{\thicc{v}}

\newcommand{\edge}{\thicc{e}}

\newcommand{\orbitVS}{V}

\newcommand{\foliation}{\mathcal{F}}
\newcommand{\foliationNormal}{\mathcal{N}}

\newcommand{\alphaCDivSet}{\alpha^{\C \times \divSet}}
\newcommand{\ReebCDivSet}{\Reeb^{\C \times \divSet}}
\newcommand{\lambdaCDivSet}{\lambda^{\C \times \divSet}}
\newcommand{\JCDivSet}{J^{\C \times \divSet}}
\newcommand{\foliationCDivSet}{\foliation^{\C \times \divSet}}
\newcommand{\alphaRLiouville}{\alpha^{\R \times \Liouville}}
\newcommand{\ReebRLiouville}{\Reeb^{\R \times \Liouville}}
\newcommand{\lambdaRLiouville}{\lambda^{\R \times \Liouville}}
\newcommand{\JRLiouville}{J^{\R \times \Liouville}}
\newcommand{\foliationRLiouville}{\foliation^{\R \times \Liouville}}
\newcommand{\alphaRIntervalDivSet}{\alpha^{\R \times [-1,1] \times \divSet}}

\newcommand{\lambdaRIntervalDivSet}{\lambda^{\R \times [-1,1] \times \divSet}}
\newcommand{\JRIntervalDivSet}{J^{\R \times [-1,1] \times \divSet}}
\newcommand{\foliationRIntervalDivSet}{\foliation^{\R \times [-1,1]\times \divSet}}
\newcommand{\alphaRHypersurface}{\alpha^{\R \times \hypersurface}}
\newcommand{\ReebRHypersurface}{\Reeb^{\R \times \hypersurface}}
\newcommand{\lambdaRHypersurface}{\lambda^{\R \times \hypersurface}}
\newcommand{\JRHypersurface}{J^{\R \times \hypersurface}}
\newcommand{\foliationRHypersurface}{\foliation^{\R \times \hypersurface}}

\newcommand{\thicc}[1]{\pmb{#1}} 
\newcommand{\be}{\begin{enumerate}}
\newcommand{\ee}{\end{enumerate}}

\newtheorem{thm}{Theorem}[subsection]

\newtheorem{ex}[thm]{Example}
\newtheorem{assump}[thm]{Assumptions}

\newtheorem{defn}[thm]{Definition}

\newtheorem{lemma}[thm]{Lemma}
\newtheorem{cor}[thm]{Corollary}

\newtheorem{rmk}[thm]{Remark}

\newtheorem{model}[thm]{Model}

\title{An algebraic generalization of Giroux's criterion}
\author{Russell Avdek}
\date{\today}

\begin{document}

\maketitle
\numberwithin{equation}{subsection}
\setcounter{tocdepth}{1}

\begin{abstract}
We compute the contact homology algebra of a neighborhood $(\R_{\tau} \times \hypersurface, \xi)$ of a convex hypersurface $\hypersurface$ of any even dimension. This algebra is the derived tensor product of the augmentations of the chain-level contact homology of the dividing set, determined by the positive and negative regions of $\hypersurface$. We describe methods for determining when the contact homology of $(\R_{\tau} \times \hypersurface, \xi)$ is non-zero in certain cases, which can be used to infer tightness of $\xi$. Full computations of contact homology are carried out in the case $\dim \hypersurface = 2$.
\end{abstract}

\tableofcontents

\pagebreak

\section{Introduction}

This article concerns closed convex hypersurfaces, $\hypersurface$, in contact manifolds $\Mxi$. We assume familiarity with such objects throughout this introduction and will provide a review in \S \ref{Sec:Prelim}. By definition, $\hypersurface$ has a neighborhood $(-\epsilon, \epsilon)_{\tau} \times \hypersurface$ in $\Mxi$ along which the contact distribution, $\xi$, is $\tau$-invariant. Any $\hypersurface$ can also be described constructively by gluing together two Liouville fillings $(\Liouville^{\pm}, \beta^{\pm})$ of a $\dim=2n-1$ contact manifold $\MxiDivSet$ along their boundaries, so that these fillings constitute the positive and negative regions of $\hypersurface$. Throughout this article, we will think of convex hypersurfaces as being described by such a gluing construction.

The purpose of this article is to compute the (sutured) contact homology algebra $CH(\R_{\tau} \times \hypersurface, \xi)$ \cite{BH:ContactDefinition, CGHH:Sutures, EGH:SFTIntro, Pardon:Contact} with coefficients in a ring $\Ring$. This contact homology algebra must vanish if $(\R_{\tau} \times \hypersurface, \xi)$ is overtwisted in the sense of \cite{BEM:OT} (by combining the results of \cite{BvK:Stabilization, CMP:OT} or appealing to \cite{Yau:VanishingCH} in the $\dim = 3$ case). Therefore the non-vanishing of $CH$ can be used to detect tightness of the contact structure, $\xi$. Recall that Giroux's criterion \cite{Giroux:Criterion} (stated in Theorem \ref{Thm:GirouxCriterion}) completely determines the tightness of $(\R_{\tau} \times \hypersurface, \xi)$ in the case $\dim \hypersurface=2$. Our intention is to develop a tool which can detect tightness of neighborhoods of convex hypersurfaces having any even dimension.

\subsection{Setup for main results}

To describe our computations, we need to first establish some vocabulary regarding chain-level contact homology algebras with $\Ring$ coefficients. Throughout we assume that either $\Ring=\Field$, $\Ring=\Field[\ChernVar^{\pm}]$ (where $\ChernVar$ records relative first Chern numbers of maps contributing to $CH$ differentials), or $\Ring=\Field[H_{2}(\hypersurface)]$. Here $\Field$ is a field containing $\Q$. The use of $\Ring=\Field[H_{2}(\hypersurface)]$ coefficients yields the most interesting of our computations (\S \ref{Sec:Computations}), but requires some explanation (\S \ref{Sec:TwistedCoeffsOverview}, \ref{Sec:FinalTwistedCoeffs}). The reader is encouraged to focus on the case $\Ring=\Field$ for simplicity.

Suppose that the Liouville fillings $(\Liouville^{\pm}, \beta^{\pm})$ induce the same contact form $\alphaDivSet$ on $\MxiDivSet$, and $\alphaDivSet$ is $\actionBoundDivSet$-non-degenerate (meaning all closed Reeb orbits of action $\leq \actionBoundDivSet$ are non-degenerate). Provided appropriate geometric and perturbative data \cite{BH:ContactDefinition, Pardon:Contact} we may define the chain-level contact homology algebra
\begin{equation*}
\Algebra = \Algebra(\alphaDivSet, \actionBoundDivSet),
\end{equation*}
which is a semi-free commutative differential graded algebra (CDGA). The algebra is generated by good Reeb orbits of action $\leq \actionBoundDivSet$ and the differential $\partial$ counts perturbed holomorphic curves in $\R \times \divSet$ having genus zero, one positive puncture, and any number of negative punctures. We say that $\Algebra$ is \emph{sublinear} if its differential counts only holomorphic planes and cylinders (\S \ref{Sec:Sublinear}). Such $\Algebra$ appear in contact homology computations of certain Boothby-Wang bundles, cf. \cite{Bourgeois:Thesis, DiogoLisiI, DiogoLisiII, NelsonWeiler}.

Provided appropriate geometric and perturbative data, the Liouville fillings induce augmentations
\begin{equation*}
\aug^{\pm}: \Algebra \to (\Ring, 0), \quad \Algebra = \Algebra(\alphaDivSet, \actionBoundDivSet).
\end{equation*}
Provided the $\aug^{\pm}$, we write
\begin{equation*}
\Ring_{\aug^{\pm}} \in \Bimod_{\Algebra}
\end{equation*}
as $\Ring$ equipped with the structure of an $\Algebra$-bimodule as determined by the augmentation.

Suppose we are given a sequence $(\alphaDivSet_{i}, \actionBoundDivSet_{i})_{i\in \Z_{\geq 0}}$ of such pairs so that $\alphaDivSet_{i} = h_{i}\alphaDivSet$ for some $h_{i} \in \Cinfty(\divSet, \R_{<1})$, $h_{i+1} < h_{i}$ point-wise, the $\alphaDivSet_{i}$ converge in $\CZero$ to a nowhere-zero $1$-form, and $\lim \actionBoundDivSet_{i} \to \infty$. We say that such sequences $(\alphaDivSet_{i}, \actionBoundDivSet_{i})$ are \emph{$CH$-admissible}. Using such sequences, we may define the contact homology algebra
\begin{equation*}
CH\MxiDivSet = \colim H\Algebra_{i}, \quad \Algebra_{i} = \Algebra(\alphaDivSet_{i}, \actionBoundDivSet_{i})
\end{equation*}
as a colimit determined by an appropriate system of cobordism maps, $H\Algebra_{i} \to H\Algebra_{i+1}$. We may also define
\begin{equation*}
H\aug^{\pm}: CH\MxiDivSet \to \Ring, \quad H\aug^{\pm} = \colim H\aug^{\pm}_{i}, \quad  \aug^{\pm}_{i}: \Algebra_{i} \to (\Ring, 0)
\end{equation*}
as a colimit as described in \cite[\S 9]{BH:ContactDefinition}.

\subsection{Statement of our main results}

\begin{thm}\label{Thm:MainChainLevel}
For any $\actionBoundDivSet < \infty, \alphaDivSet$ as above, there is a constant $\actionBound > \actionBoundDivSet$, a contact form $\alpha$ for $(\R_{\tau} \times \hypersurface, \xi)$, and geometric and perturbative data so that
\begin{equation*}
H\Algebra(\alpha, \actionBound) = \Ring_{\aug^{+}} \otimesDerived_{\Algebra} \Ring_{\aug^{-}}
\end{equation*}
is the derived tensor product of the modules $\Ring_{\aug^{\pm}}$. Consequently, when using $\Ring=\C$ coefficients and $\Z/2\Z$ grading, the following are equivalent:
\be
\item $H\Algebra(\alpha, \actionBound)\neq 0$.
\item The $\aug^{\pm}$ are DG-homotopic (Definition \ref{Def:DGHomotopy}).
\ee
Moreover, a $CH$-admissible sequence $(\alphaDivSet_{i}, \actionBoundDivSet_{i})$ for $\MxiDivSet$ determines a $CH$-admissible sequence for $(\R_{\tau} \times \hypersurface, \xi)$ whose contact homology algebras satisfy the above conditions.
\end{thm}

In fact we compute the chain-level algebra $\Algebra(\alpha, \actionBound)$. It is the homotopy colimit of the diagram
\begin{equation*}
\Ring \xleftarrow{\aug^{+}} \Algebra \xrightarrow{\aug^{-}} \Ring
\end{equation*}
which we explicitly describe as
\begin{equation}\label{Eq:ChainLevelDerivedTensor}
\Algebra(\alpha, \actionBound) = \Ring_{\aug^{+}} \otimes_{\Algebra} \Algebra^{c} \otimes_{\Algebra} \Ring_{\aug^{-}}
\end{equation}
where $\Algebra^{c}$ is a cylinder for $\Algebra$ (\S \ref{Sec:CylinderDGAs}). Thus the homology of this algebra computes the derived tensor product. The second statement in Theorem \ref{Thm:MainChainLevel} follows from (the purely algebraic) Lemma \ref{Lemma:HomotopicIffNonzeroDTP}.

We remark that the derived tensor product carries rich information and is extremely difficult to compute in general (even when $\aug^{+}=\aug^{-}$). For example, for certain CW complexes $X$, derived tensor products of $C^{\ast}(X)$-modules can compute cohomology rings of the based loop space, $\Omega X$: See, for example, discussions of Eilenberg-Moore spectral sequences in \cite{FHT, Holstein:RHT, McCleary:UsersGuide}. Hence we cannot expect to have a concise, generally applicable expression for the $CH$ algebra of Theorem \ref{Thm:MainChainLevel}.

As discussed in \S \ref{Sec:CDGABar}, it is even quite difficult to determine the vanishing or non-vanishing of $\Ring_{\aug^{+}} \otimesDerived_{\Algebra} \Ring_{\aug^{-}}$ when $\aug^{+} \neq \aug^{-}$. On the other hand, such partial computations are feasible in the sublinear case.

\begin{thm}\label{Thm:MainHomology}
Suppose that $\Algebra$ is sublinear and we are working with $\Ring=\Field$ coefficients. Then 
\begin{equation*}
H\Algebra(\alpha, \actionBound)=0 \iff H\aug^{+} \neq H\aug^{-}.
\end{equation*}
Moreover if $\MxiDivSet$ admits a $CH$-admissible sequence $(\alphaDivSet_{i}, \actionBoundDivSet_{i})$ for which the $\Algebra_{i}$ are sublinear, then
\begin{equation*}
CH(\R_{\tau} \times \hypersurface, \xi)=0 \iff H\aug^{+} \neq H\aug^{-}.
\end{equation*}
\end{thm}

The following corollary (which justifies the title of this article) then follows easily from a combination of Theorem \ref{Thm:MainHomology} and Giroux's criterion (for $\Ring=\Field$ coefficients) and the direct computations of \S \ref{Sec:dim2Alpha1Computations} (for any $\Ring$).

\begin{cor}\label{Cor:dim2Computation}
Assume $\dim \hypersurface = 2$. For any $\Ring$, $CH(\R_{\tau} \times \hypersurface, \xi)=0$ iff $(\R_{\tau} \times \hypersurface, \xi)$ is overtwisted.
\end{cor}

The techniques developed in the proof of Theorem \ref{Thm:MainChainLevel} are applied to fully compute $CH(\R_{\tau} \times \hypersurface, \xi)$ in the case $\dim \hypersurface =2$ in \S \ref{Sec:Computations} from which Corollary \ref{Cor:dim2Computation} follows. Computations using $\Ring = \Field[H_{2}(\hypersurface)]$ coefficients (Theorem \ref{Thm:ComputeG0}) demonstrate that $CH$ ``sees'' some information about the topology of symplectic fillings by its ability to recover a known result of Eliashberg-Gromov (Corollary \ref{Cor:ContractibleSphere}). The same computation shows that $H\aug^{+} \neq H\aug^{-}$ does not imply the vanishing of contact homology when using twisted coefficient systems, $\Ring = \Field[H_{2}(\hypersurface)]$. Further computations, in particular for convex hypersurfaces determined by ``pencil pairs'' \cite{Avdek:PencilPairs} (whose dividing sets are of Boothby-Wang type), will appear in a follow-up article \cite{Avdek:ChernNumbers}.

In \cite{ContactHomotopy} we will expand upon the techniques and proof strategies of this article to establish the well-definedness of homotopy types of chain-level $CH$ algebras in collaboration with Erkao Bao, Georgios Dimitroglou Rizell, and Zhengyi Zhou. We note that our proof of Theorem \ref{Thm:MainChainLevel} uses finite action bounds $\actionBound$ in an essential way (as indicated by Lemmas \ref{Lemma:DynamicsOnDivSetC}, \ref{Lemma:DynamicsOnRLiouville}, and \ref{Lemma:DynamicsRHypersurface}). In order to formulate a version of the theorem without finite action bounds, one may speak of augmentations of the homotopy colimits of the $\Algebra(\alphaDivSet_{i}, \actionBoundDivSet_{i})$, which serve as chain-level models for $CH\MxiDivSet$. The well-definedness of such homotopy colimits as well as augmentations determined by colimits of the $\aug^{\pm}_{i}$ will be established in \cite{ContactHomotopy}.

\subsection{Proof strategies for our main theorems}\label{Sec:ProofStrategies}

Our proof strategy for Theorem \ref{Thm:MainChainLevel} starts by breaking up the manifold $\R_{\tau} \times \hypersurface$ into a neighborhood $\C \times \divSet$ of the dividing set and thickenings $\R_{\tau} \times \Liouville^{\pm}$ of the positive and negative regions. We then glue the pieces back together in such a way that the chain-level contact homology algebras are ``glued together'' via tensor products of CDGAs.

Specifically, we establish the following results:
\be
\item There is a contact form $\alphaCDivSet_{3}$ on $\C \times \divSet$ whose chain-level $CH$ algebra $\Algebra(\alphaCDivSet_{3}, \actionBound)$ is a cylinder over $\Algebra$ (Lemma \ref{Lemma:CHCylinder}).
\item There are contact forms $\alpha_{2}^{\R_{\tau} \times \Liouville^{\pm}}$ on $\R_{\tau} \times \Liouville^{\pm}$ whose chain-level $CH$ algebras $\Algebra(\alpha_{2}^{\R \times \Liouville^{\pm}}, \actionBound)$ are semi-free $\Algebra$-modules which are quasi-isomorphic to $\Ring_{\aug^{\pm}}$ as $\Algebra$-modules (Lemma \ref{Lemma:CHAugmentationModule}).
\item The contact form $\alphaRHypersurface_{5}$ on $\R_{\tau} \times \hypersurface$ has a chain-level $CH$ algebra expressible as
\begin{equation}\label{Eq:BigTensorIdentification}
\Algebra(\alphaRHypersurface_{5}, \actionBound) = \Algebra(\alpha_{2}^{\R \times \Liouville^{+}}, \actionBound) \otimes_{\Algebra} \Algebra(\alphaCDivSet_{3}, \actionBound) \otimes_{\Algebra} \Algebra(\alpha_{2}^{\R \times \Liouville^{-}}, \actionBound).
\end{equation}
\ee
By the results of \S \ref{Sec:Algebra}, $\Algebra(\alphaRHypersurface_{5}, \actionBound)$ recovers the quasi-isomorphism type of the homotopy colimit of Equation \eqref{Eq:ChainLevelDerivedTensor}. Hence $H\Algebra(\alphaRHypersurface_{5}, \actionBound)$ is the derived tensor product of the $\Ring_{\aug^{\pm}}$.

The contact forms $\alpha^{\ast}_{k}$ (with $\ast = \C \times \divSet$ or $\R_{\tau} \times \Liouville^{\pm}$) described above have $k$ Reeb orbits for each Reeb orbit in $\divSet$ (when action bounds are accounted for). Thus each such contact form carries more orbits than is strictly necessary to define contact homology. For example, we may construct contact forms $\alphaRHypersurface_{1}$ having exactly one closed Reeb orbit for each Reeb orbit in $\divSet$. See \S \ref{Sec:VaugonForms} which generalizes a construction of \cite{Vaugon:Bypass}. Our replacing the strict contact manifold $(\divSet, \alphaDivSet)$ with $(\C \times \divSet, \alphaCDivSet_{3})$ parallels the replacement of Legendrians $L \subset \C^{n} \times \R$ with Legendrians $\Phi^{\delta}_{f} \simeq \R \times L$ in $\C \times \C^{n} \times \R$ in \cite[\S 4.3]{EES:Orientations}.

The extra Reeb orbits associated to the $\alpha^{\ast}_{k}$ facilitate diagrammatic -- rather than direct -- computation of contact homologies. Using the $\alpha^{\ast}_{k}$ (and special almost complex structures), we establish Lemma \ref{Lemma:CHCylinder}, Lemma \ref{Lemma:CHAugmentationModule}, and Equation \eqref{Eq:BigTensorIdentification} by showing that the $\Algebra(\alpha^{\ast}_{k}, \actionBound)$ fit into certain commutative diagrams.

This diagrammatic approach minimizes analysis of transversality issues. Specifically we avoid deep analysis which would require the use of either
\be
\item obstruction bundle gluing techniques \cite{HT:GluingII} applied to the framework of \cite{BH:ContactDefinition} or
\item computations of homotopy sheaves in the framework of \cite{Pardon:Contact}.
\ee
Even if such analysis were carried out successfully to obtain explicit descriptions of the aforementioned chain-level $CH$ algebras, the differentials obtained would be too complicated to be practically useful for either direct computation or even identification of quasi-isomorphism type.

In particular, for any choices of perturbative data, $\Algebra(\alphaCDivSet_{3}, \actionBound)$ must be a ``small resolution'' cylinder for $\Algebra$. As described in \S \ref{Sec:SmallResolutions}, hands-on computations of the derived tensor products such as $\Algebra(\alphaRHypersurface_{5}, \actionBound)$ are extremely difficult using such small resolutions. Moreover, due to the necessary complexity of the differentials involved, it would be extremely difficult to view an explicit presentation of $\Algebra(\alphaCDivSet_{3}, \actionBound)$ and recognize it as a cylinder for $\Algebra$.

\subsection{Outline of the article}

Here we provide an outline of this article. Along the way, we explain some more details of our proofs, in particular around holomorphic curves and transversality issues.

\S \ref{Sec:Algebra} reviews the algebraic tools required to establish Theorem \ref{Thm:MainChainLevel}. We start with a treatment of the necessary algebra (rather than the geometry) since the algebraic tools described are necessary to formulate Theorem \ref{Thm:MainChainLevel}. For readers familiar with rational homotopy theory or derived algebraic geometry, many of the results of \S \ref{Sec:Algebra} are likely familiar or easy to work out by hand. Anticipating that our target audience members are likely less familiar with the objects mentioned, we give a rather detailed treatment. Following the same reasoning, the section closes with the supplementary \S \ref{Sec:AlgRemarks}, which explains how some known constructions in Legendrian contact homology can be phrased in the language of cylinder objects and homotopy colimits. In particular, we point out that the main computational result (Lemma 2.3) of \cite{Ekholm:Nonloose} is the Legendrian version of our Theorem \ref{Thm:MainChainLevel}.

\S \ref{Sec:Prelim} gives an overview of the basic geometric notions appearing in this article and establishes some notational conventions.

\S \ref{Sec:DynamicsAndFoliations} describes contact structures on manifolds $M$ of the form $\C \times \divSet$, $\R_{\tau} \times \Liouville$ (associated to a Liouville domain $(\Liouville, \beta)$), and $\R_{\tau} \times \hypersurface$ as well as Liouville forms and almost complex structures, $J$, on the associated manifolds $\R_{s} \times M$. In \S \ref{Sec:DynamicsAndFoliations}, we work with some quite general geometric structures. Special examples of these geometric structures used to establish our main theorems do not appear until \S \ref{Sec:GeometryExamples}. The $\R_{s}\times M$ are equipped with foliations $\foliation$ whose leaves are $\codim=2$, $J$-holomorphic submanifolds.

\S \ref{Sec:LeafNotationAsymptotics} starts by establishing notation which is common to the various manifolds $\R_{s} \times M$ of interest and their geometric structures as described in \S \ref{Sec:DynamicsAndFoliations}. We then impose certain restrictions on these geometric structures. With these assumptions in place, we describe the asymptotics of (perturbed) $J$-holomorphic curves in the $\R_{s} \times M$ as they approach closed Reeb orbits at $s \to \pm \infty$.

\S \ref{Sec:JCurvesInFoliations} opens with a general discussion of perturbed holomorphic curves in almost complex manifolds $(\complete{X}, J)$ admitting $J$-holomorphic foliations, $\foliation$. In particular, we show that if such curves $u:\Sigma \to \complete{X}$ satisfy perturbed $\delbar$ equations of the form $\delbar_{J, \domainJ}u = \pert^{0, 1}$ for which the perturbation term $\pert$ is ``adapted to $\foliation$'', then integers $\leaf \bullet u$ counting intersection numbers between (certain) leaves $\leaf$ of $\foliation$ and the image of $u$ are both well-defined and non-negative. Unperturbed versions of such results are by now well-known. See for example \cite{Moreno:Thesis, MorenoSiefring} for results pertaining to foliations and \cite{CFC:RelativeContactHomology} for results pertaining to holomorphic submanifolds of almost complex cobordisms. Our intersection theoretic results are then combined with the assumptions of \S \ref{Sec:LeafNotationAsymptotics} to establish that when using $\foliation$-adapted perturbations, perturbed holomorphic curves in the $\R_{s} \times M$ must lie in the leaves of the foliation $\foliation$. This ``curves are stuck in leaves'' result is Lemma \ref{Lemma:CurvesStuckInLeaves}.

\S \ref{Sec:FredholmTheory} works out the Fredholm theory of perturbed holomorphic curves in the manifolds $\R_{s} \times M$ of \S \ref{Sec:DynamicsAndFoliations} after describing a general ``Banach manifold of maps'' construction. Such a construction is required to define and compute linearizations of perturbed $\delbar$ equations. The main result of the section is Lemma \ref{Lemma:AppliedAutoTransverse} which applies an automatic-transversality result from \cite{Wendl:Automatic} to the perturbed holomorphic curves in $\R_{s} \times M$.

\S \ref{Sec:ThickenedModuli} establishes that for certain leaves, $\leaf$ of $\foliation$, and appropriate choices of perturbation data, one may identify virtual fundamental cycles (VFCs) of perturbed holomorphic curves in $\leaf$ with those in $\R_{s} \times M$. This means that counts of perturbed holomorphic curves into such $\leaf$ match the corresponding counts of curves in $\R_{s} \times M$. Intuitively, this is clear from the result that certain perturbed holomorphic curves must map to $\leaf$ (by Lemma \ref{Lemma:CurvesStuckInLeaves}) and that if such curves are ``transversely cut out'' as maps into $\leaf$, then such curves are transversely cut out as maps into $\R_{s} \times M$ (according to Lemma \ref{Lemma:AppliedAutoTransverse}).

To formalize this intuition we must discuss thickened moduli spaces and detail some VFC constructions. By ``thickening a moduli space'', we mean replacing a moduli space of (unperturbed) holomorphic curves with a larger space of perturbed holomorphic curves, the latter carrying a manifold or orbifold structure. These are the smooth, finite dimensional orbifolds extracted from Kuranishi charts. In the frameworks of both \cite{BH:ContactDefinition} and \cite{Pardon:Contact}, VFCs are extracted from thickened moduli space data and then used to define chain-level contact homology differentials and cobordism maps. After reviewing the construction of VFCs from \cite{BH:ContactDefinition} in some detail, we show that thickened moduli spaces and VFCs of curves in certain leaves $\leaf$ of the foliation $\foliation$ may be identified with those in $\R_{s} \times M$. We do not provide a detailed proof of the corresponding result within the VFC framework of \cite{Pardon:Contact}, but provide some explanatory remarks aimed at readers familiar with that framework. In terms of exposition, our strategy is to delay specialization to curves of contact homology type (meaning zero genus and a single positive puncture) for as long as possible. This may assist an interested reader in studying ``full'' symplectic field theory (SFT) or rational symplectic field theory (RSFT) invariants of neighborhoods of convex hypersurfaces, for example using the algebraic invariants of \cite{LW:Torsion, MorenoZhou:RSFTI, MorenoZhou:RSFTII}. At the time of writing we do not know if such algebraic tools contain more information about the tightness of neighborhoods of convex hypersurfaces (or the symplectic fillability of closed contact manifolds which contain them) than may be inferred from Theorem \ref{Thm:MainHomology}.

\S \ref{Sec:GeometryExamples} describes specific geometric structures on the manifolds $\R_{s} \times \C \times \divSet$, $\R_{s} \times \R_{\tau} \times \Liouville$, and $\R_{s} \times \R_{\tau} \times \hypersurface$ which we use to establish our main theorems per the strategy described in \S \ref{Sec:ProofStrategies}. This section closes with a proof that all convex hypersurface neighborhoods can be described using the contact forms we use to compute contact homology (Lemma \ref{Lemma:ConstructionGenerality}).

\S \ref{Sec:Computations} computes the contact homologies of neighborhoods of $\dim=2$ convex hypersurfaces. We describe three approaches to such computations: First we obtain partial results using known techniques. Second, we carry out a geometric computation using the contact forms $\alphaRHypersurface_{1}$ (which we recall are essentially identical to the contact forms of \cite{Vaugon:Bypass}). Third, we compute using derived tensor products, thereby providing a sanity check on Theorems \ref{Thm:MainHomology} and \ref{Thm:MainChainLevel}.

\subsection{Acknowledgments}

This article is a completely rewritten correction of a first version which had appeared on the arXiv in 2023. The first version of the article relied on the techniques of \cite{HT:GluingII} and was inspired by Jo Nelson and Jacob Rooney's ``Obstruction Bundle Gluing Zoominar'' (2020). We thank Erkao Bao, Fr\'{e}d\'{e}ric Bourgeois, Ko Honda, Sam Lisi, Patrick Massot, Klaus Niederkr\"{u}ger-Eid, and Takahiro Oba for early conversations around the project, many of which started at the CNRS conferences ``Advances in symplectic topology'' (2021) and ``Convexity in contact and symplectic topology'' (2022) held at the Institut Henri Poincar\'{e} in Paris.

At the 2026 European Kylerec conference some critical errors in the algebra of the first version of the article were discovered by Mingyuan Hu. The nature of these algebraic errors indicated that our obstruction bundle computations must have also contained errors. We thank Mingyuan, Lukas Nakamura, and Matthias Scharitzer for discussions at that conference and their feedback on the first version of this article.

The algebraic strategies used in this version of the article are heavily influenced by discussions with Georgios Dimitroglou Rizell which started during the author's postdoctoral fellowship at Uppsala University (2021-2023). Further algebraic discussions with Georgios, Efe \.{I}zbudak, Alex Takeda, and Zhengyi Zhou at the 2026 G\"{o}kova Geometry and Topology conference were extremely helpful.

Many of the ideas in this paper were borne from discussions with Erkao Bao, Georgios Dimitroglou Rizell, and Zhengyi Zhou. Without their generosity and guidance over many years, this project would not have been possible.

Much of the work for the first version of this project was completed at Uppsala University where we were partly supported by the Knut and Alice Wallenberg Foundation's grant KAW 2016.0198. This project has received funding from the European Union’s Horizon 2020 research and innovation programme under the Marie Skłodowska-Curie grant agreement No 101034255. We thank Cofund MathInGreaterParis and the Fondation Mathématique Jacques Hadamard for supporting us during our time at the Laboratoire de Math\'{e}matiques d'Orsay at Universit\'{e} Paris-Saclay (2023-2024) when the first version and some updates were completed. Finally, we thank the Institut de Mathématiques de Jussieu-Paris Rive Gauche, Sorbonne Université, and the CNRS for their support and for providing a wonderful working environment.

Finally, we address AI usage: Example \ref{Ex:NonSublinear} was provided to us by Mingyuan Hu, who found it using AI tools. Such tools were also used to search for errors in various calculations (all of which were originally worked out by the author) and to find grammatical and type-setting issues.

\section{Algebra}\label{Sec:Algebra}

In this section we describe the algebra needed for applications to the contact homologies of convex hypersurfaces. In particular, we work out the following facts:
\be
\item Cylinders $\Algebra^{c}$ of commutative differential graded algebras (CDGAs) $\Algebra$ over a ring $\Ring$ can be used to compute derived tensor products $\Ring_{\aug} \otimesDerived_{\Algebra} \Ring_{\augDelta} = \Tor^{\Algebra}(\Ring_{\aug}, \Ring_{\augDelta})$ associated to augmentations $\aug, \augDelta: \Algebra \to (\Ring, 0)$ (Equation \eqref{Eq:DerivedTensorFromCylinder}).
\item When $\Algebra$ is a finitely filtered CDGA with $\C$ coefficients and $\Z/2\Z$ grading, then two augmentations $\aug$ and $\augDelta$ of $\Algebra$ are DG-homotopic iff $\C_{\aug} \otimesDerived_{\Algebra} \C_{\augDelta} \neq 0$ (Lemma \ref{Lemma:HomotopicIffNonzeroDTP}).
\item When the CDGAs $\Algebra$ appearing in our applications are sublinear and defined over a field, such derived tensor products $\Ring_{\aug} \otimesDerived_{\Algebra} \Ring_{\augDelta}$ are non-vanishing iff the induced morphisms $H\aug,H\augDelta: H\Algebra \to \Ring$ are equal (Lemma \ref{Lemma:SublinearOtimes}).
\ee
The definitions of the terms involved in the above statements appear below.

As described in the introduction, we are interested in being able to determine when derived tensor products are non-zero. Some discussion of the difficulty of such computations is given in the supplementary \S \ref{Sec:CDGABar}. Likewise \S \ref{Sec:AlgRemarks} includes some remarks on how some known constructions in the Legendrian contact homology literature (e.g. \cite{BC:Bilinearized, BG:Geography, Ekholm:Nonloose}) can be interpreted in the language of cylinders and homotopy colimits.

\subsection{Differential graded algebras}

Let $\Ring$ be a commutative ring with unit which is $G$-graded for $G=\Z/\Div\Z$ or $G=\Q/\Div\Z$ where
\begin{equation*}
\Div \in 2\Z_{\geq 0}.
\end{equation*}
Let $A$ be a $G$-graded $\Ring$-algebra. That is, $A$ is a graded $\Ring$-module with a product $\mu: A^{\otimes 2} \to A$ that is associative and compatible with gradings,
\begin{equation*}
A = \oplus_{i \in G} A_{i}, \quad \mu(A_{i}\otimes A_{j}) \subset A_{i+j}, \quad \mu(\Id_{A} \otimes \mu)= \mu(\mu \otimes \Id_{A}).
\end{equation*}
Here and throughout tensor products are over $\Ring$ unless otherwise indicated. We assume throughout that $A$ is unital. Henceforth, $\mu$ will be suppressed from our notation.

Let $\partial: A \to A$ be a $\deg=-1$ differential satisfying $\partial 1 = 0$ and the Leibniz rule. That is,
\begin{equation*}
\partial^{2}=0, \quad \partial(xy) = (\partial x)y + (-1)^{|x|}x(\partial y)
\end{equation*}
for $x, y \in A$ of homogeneous degree. We say that
\begin{equation*}
\Algebra = (A, \partial)
\end{equation*}
is a \emph{differential graded algebra} (DGA) over $\Ring$. The homology, $H\Algebra$, of a DGA is then a $G$-graded algebra over $\Ring$. It is easily verified that the cycle $1 \in A$ is a boundary iff $H\Algebra  =0$.

\subsubsection{Semi-free CDGAs}

If for all $x, y \in A$ of homogeneous degree, we have
\begin{equation*}
	xy = (-1)^{|x|\cdot |y|}yx
\end{equation*}
then $\Algebra$ is a \emph{commutative differential graded algebra} (CDGA). A CDGA $\Algebra$ is \emph{semi-free} when $A$ has the form
\begin{equation*}
A = \tensorAlgGraded(V)
\end{equation*}
where $V$ is a graded free $\Ring$-module and $\tensorAlgGraded(V)$ is the free graded-commutative tensor algebra. In this case, we assume that $V$ is spanned by generators $v$ of homogeneous degree $|v|$. We say that a semi-free CDGA $\Algebra$ is \emph{filtered} if we can express $A = \tensorAlgGraded(V)$ so that $V$ admits an exhaustive filtration
\begin{equation*}
0 = V^{0} \subset V^{1} \subset \cdots \subset V, \quad \cup V^{k}=V, \quad \partial v^{k} \in \tensorAlgGraded(V^{k-1}) \ \forall v^{k} \in V^{k}
\end{equation*}
by free $\Ring$-modules. We say that $\Algebra$ is \emph{finitely filtered} if each $V^{k}$ is finitely generated. The actions associated to Reeb orbits ensure that chain-level contact homology algebras are always finitely filtered, and we will typically use actions of generators to designate filtrations in examples. A proof of the following variant of the standard ``lifting lemma'' is included since it is typically established under alternative hypotheses.\footnote{In the rational homotopy theory literature, our filtration condition is typically replaced with restrictions on gradings.}

\begin{lemma}\label{Lemma:Lifting}
Suppose that we have a commutative diagram of CDGA morphisms
\begin{equation*}
\begin{tikzcd}
\Algebra \arrow[r]\arrow[d, "I"] & \AlgebraC \arrow[d, "\pi"]\\
\AlgebraB \arrow[r, "\phi"] & \AlgebraD,
\end{tikzcd}\quad \Algebra = (\tensorAlgGraded(V), \partial_{A}), \quad \AlgebraB = (\tensorAlgGraded(V \oplus W), \partial_{B})
\end{equation*}
for graded, free $\Ring$-modules $V$ and $W$, with $I$ induced by the inclusion of $V$ into $V \oplus W$. Suppose $\pi$ is a surjective quasi-isomorphism, that $\Algebra$ and $\AlgebraB$ are filtered and that the filtration on $\AlgebraB$ has the form
\begin{equation}\label{Eq:TechnicalFiltrationCondition}
(V \oplus W)^{k} \subset V \oplus W^{k}, \quad W^{k} = (V\oplus W)^{k} \cap W.
\end{equation}
Then there exists a $\psi: \AlgebraB \to \AlgebraC$ yielding a commutative diagram
\begin{equation*}
\begin{tikzcd}
\Algebra \arrow[r]\arrow[d, "I"] & \AlgebraC \arrow[d, "\pi"]\\
\AlgebraB \arrow[r, "\phi"] \arrow[ur, dashed, "\psi"]& \AlgebraD.
\end{tikzcd}
\end{equation*}
\end{lemma}

We remark that when the CDGAs $\Algebra, \AlgebraB$ are contact homology algebras, Equation \eqref{Eq:TechnicalFiltrationCondition} is satisfied by action considerations.

\begin{proof}
We must define $\psi(w)$ for each generator $w$ of $W$. Inductively, we may assume that $\psi$ is defined on the subspace $W^{k}$ and seek to extend the definition over $W^{k+1}$. Thus we assume that $w \in W^{k+1}$ with $\psi(\partial w)$ already defined.

Since $\pi$ is surjective, there exists some $c_{0} \in \AlgebraC$ for which $\pi(c_{0}) = \phi(w)$. If $\partial c_{0} = \psi(\partial w)$, then we set $\psi(w) = c_{0}$. Otherwise, it's easily checked that $\partial c_{0} - \psi(\partial w)$ is closed and in the kernel of $\pi$. Since $\pi$ induces an injection $\pi_{\ast}$ on homology, it then must be the case that $\partial c_{0} - \psi(\partial w)$ is exact. We assume that it is the boundary of some $-c_{1} \in \AlgebraC$, implying $\psi(\partial w) = \partial(c_{0} + c_{1})$. We find that $\pi c_{1}$ is closed in $\AlgebraD$. Since $\pi$ and $\pi_{\ast}$ are surjections, there must exist some $c_{2} \in \AlgebraC$ which is closed and for which $\pi c_{2} = -\pi c_{1}$. It follows that
\begin{equation*}
\partial(c_{0} + c_{1} + c_{2}) = \psi(\partial w), \quad \pi(c_{0} + c_{1} + c_{2}) = \phi(w).
\end{equation*}
Thus we may define $\psi(w) = c_{0} + c_{1} + c_{2}$ and continue the induction procedure.
\end{proof}

\begin{lemma}\label{Lemma:QIInversion}
Suppose that $\Algebra = (\tensorAlgGraded(V), \partial_{A})$ and $\AlgebraB = (\tensorAlgGraded(W), \partial_{B})$ are finitely filtered, semi-free CDGAs, defined over $\Ring\supset \Q$. Provided a quasi-isomorphism $\phi: \Algebra \to \AlgebraB$, there exists a quasi-isomorphism $\psi: \AlgebraB \to \Algebra$ for which $H\phi\circ H\psi = \Id_{H\AlgebraB}$.
\end{lemma}

\begin{proof}
Consider the CDGA, $\AlgebraC$, and morphisms
\begin{equation*}
\begin{gathered}
\AlgebraC = (\tensorAlgGraded(V \oplus W \oplus \widecheck{W}), \partial), \quad \partial v = \partial_{A}v, \quad \partial w = \widecheck{w}, \quad |\widecheck{w}| = |w| - 1,\\
I_{\Algebra}: \Algebra \to \AlgebraC, \quad I_{\Algebra}(v) = v,\\
\pi_{\Algebra}: \AlgebraC \to \Algebra, \quad \pi_{\Algebra}(v) = v, \quad \pi_{\Algebra}(w) = \pi_{\Algebra}(\widecheck{w}) = 0,\\
\Phi: \AlgebraC \to \AlgebraB, \quad \Phi (v) = \phi(v), \quad \Phi(w) = w, \quad \Phi(\widecheck{w}) = \partial_{B} w.
\end{gathered}
\end{equation*}
Then $\pi_{\Algebra}\circ I_{\Algebra} = \Id_{\Algebra}$. Both $\pi_{\Algebra}$ and $\Phi$ are surjective and are easily seen to be quasi-isomorphisms. Hence Lemma \ref{Lemma:Lifting} provides us with a quasi-isomorphism $\Psi$ fitting into a commutative diagram
\begin{equation*}
\begin{tikzcd}
& \AlgebraC \arrow[d, "\Phi"]\\
\AlgebraB \arrow[ur, "\Psi"] \arrow[r, "\Id_{\AlgebraB}"] & \AlgebraB.
\end{tikzcd}
\end{equation*}
Then $\psi = \pi_{\Algebra} \circ \Psi$ is as desired.
\end{proof}

The algebra $\tensorAlgGraded(V)$ underlying a semi-free CDGA, $\Algebra$, decomposes into a direct sum of free $\Ring$-modules,
\begin{equation*}
\tensorAlgGraded(V) = \oplus_{k=0}^{\infty} \tensorAlgGraded^{k}(V), \quad \tensorAlgGraded^{k}(V) = V^{\otimes k}/\sim, \quad xy \sim (-1)^{|x|\ |y|}yx, \quad V^{\otimes 0} = \Ring
\end{equation*}
with the unit being $1 \in V^{\otimes 0}$. The differential, $\partial$, is then completely determined by $\partial|_{V}$. We will often find it useful to write
\begin{equation*}
\partial = \sum_{k=0}^{\infty} \partial_{k}, \quad \partial_{k}: \tensorAlgGraded^{i}(V) \to \tensorAlgGraded^{i-1+k}(V)
\end{equation*}
expressing $\partial$ as a sum of $\Ring$-linear, $\deg=-1$ maps.

We note that a tensor product of semi-free CDGAs is naturally a semi-free CDGA,
\begin{equation*}
(\tensorAlgGraded(V), \partial)\otimes (\tensorAlgGraded(V'), \partial') \simeq (\tensorAlgGraded(V \oplus V'), \partial''), \quad \partial''v = \partial v, \quad \partial''v' = \partial'v'
\end{equation*}
for $v \in V$ and $v' \in V'$. The identification is determined by reordering monomials in $\tensorAlgGraded(V \oplus V')$, pushing all of the $v \in V$ to the left and the $v' \in V'$ to the right.

In this article we are most concerned with semi-free CDGAs since they naturally appear as the chain-level DGAs underlying contact homology algebras. If the algebra, $A$, underlying a DGA $\Algebra$ is the tensor algebra $\tensorAlg(V)$ of a graded, free $\Ring$-module, $V$, then we say that $\Algebra$ is a \emph{semi-free non-commutative DGA}. Such algebras naturally appear in contact topology as the Chekanov-Eliashberg algebras of Legendrian submanifolds in contact manifolds (admitting contact forms without closed Reeb orbits). See for example \cite{Chekanov:LCH, EtnyreNg:LCHSurvey, EES:PtimesR}. Actions of chords give such semi-free DGAs filtration structures with respect to which analogues of Lemma \ref{Lemma:Lifting} may be established.

\subsubsection{Augmentations of DGAs}

An \emph{augmentation} is a DGA morphism
\begin{equation*}
\aug: \Algebra \to (\Ring, 0).
\end{equation*}
Any such augmentation induces a morphism of graded unital algebras
\begin{equation*}
H\aug: H\Algebra \to \Ring
\end{equation*}
so that the natural composition
\begin{equation*}
\Ring \xrightarrow{1 \mapsto [1]} H\Algebra \xrightarrow{H\aug} \Ring
\end{equation*}
is the identity. Therefore the existence of an augmentation guarantees that $H\Algebra \neq 0$.

\begin{lemma}\label{Lemma:AugEqualsNonZeroHomology}
Let $\Algebra = (\tensorAlgGraded(\orbitVS), \partial)$ be a finitely filtered, semi-free CDGA with $\Ring=\C$ coefficients and $\Z/2\Z$ grading. Then the following are equivalent:
\be 
\item $\Algebra$ admits an augmentation.
\item $H\Algebra \neq 0$.
\ee
\end{lemma}

The full proof of this lemma, which we attribute to Georgios Dimitroglou Rizell, appears in \cite{ALZ:AlgOvertwisted}. For the sake of completeness, we include a proof of the case $\dim \orbitVS < \infty$, which is relevant to chain-level contact homology algebras with finite action bounds appearing in our main theorems.

\begin{proof}[Proof of Lemma \ref{Lemma:AugEqualsNonZeroHomology} when $\dim \orbitVS < \infty$]
Clearly (1) implies (2) so we seek to establish that (2) implies (1). Let $\orbitVS_{ev}$ and $\orbitVS_{odd}$ be the subspaces of $\orbitVS$ of even and odd grading. We have an associated augmentation variety
\begin{equation*}
Z \subset \C^{N}, \quad N = \dim \orbitVS_{ev}
\end{equation*}
with coordinate variables indexed by generators $v_{ev, i} \in \orbitVS_{ev}$ defined as follows. Let
\begin{equation*}
P: \Algebra \rightarrow \C[\orbitVS_{ev}]
\end{equation*}
be the map which sets each generator $v_{odd} \in \orbitVS_{odd}$ to zero. Then let
\begin{equation*}
\mathcal{I} \subset \C[\orbitVS_{ev}]
\end{equation*}
be the ideal generated by the $P(\partial v_{odd})$. Then $Z$ is the variety associated to $\mathcal{I}$. The points
\begin{equation*}
\aug = (\aug_{1}, \cdots, \aug_{N}) \in Z
\end{equation*}
exactly correspond to the augmentations of $\Algebra$ by $\aug v_{ev, i} = \aug_{i}$. The following conditions are then equivalent:
\be
\item $1 \in \mathcal{I}$
\item $Z = \emptyset$,
\item $\Algebra$ does not admit an augmentation.
\ee
The first equivalence follows from the weak Nullstellensatz and the second follows from the definition of an augmentation of $\Algebra$.

Suppose that $\Algebra$ does not admit an augmentation, so that $1 \in \mathcal{I}$. It follows that there are some $p_{i} \in \C[\orbitVS_{ev}]$ such that
\begin{equation*}
1 = \sum p_{i}P(\partial v_{odd, i})
\end{equation*}
where the sum is taken over the $v_{odd, i} \in \orbitVS_{odd}$. Since the $p_{i}$ are polynomials in the additive generators $v_{ev}$ of $\orbitVS_{ev}$, each non-zero summand of $\partial p_{i}$ contains at least one generator $v_{odd} \in \orbitVS_{odd}$ as a factor. Hence using the definition of $P$ and the Leibniz rule we obtain
\begin{equation*}
\partial(p_{i}v_{odd, i}) = (\partial p_{i})v_{odd, i} + p_{i}\partial v_{odd, i} \implies 1 = P(\partial x), \quad x = \sum p_{i} v_{odd,i}.
\end{equation*}

By the definition of $P$, we can write
\begin{equation*}
\partial x = 1 + z, \quad P(z) = 0
\end{equation*}
with $z$ being an even degree element of $\Algebra$. Each summand of $z$ must have at least two $\orbitVS_{odd}$ factors by degree considerations. Since squares of odd-degree generators vanish in $\Algebra$, there is some $m > 0$ for which $z^{m} = 0$. Moreover for all $j > 0$
\begin{equation*}
\partial z^{j} =  j z^{j-1}\partial z = jz^{j-1}\partial^{2}x = 0.
\end{equation*}
It follows that
\begin{equation*}
\partial \sum_{j=1}^{m} (-1)^{j+1}z^{j-1} x = (1 + z) - z(1 + z) + \cdots + (-1)^{m-1}z^{m-1}(1 + z) = 1.
\end{equation*}
So we've established that if $\Algebra$ does not admit an augmentation, then the unit is exact in $\Algebra$, implying $H\Algebra = 0$.
\end{proof}

\subsubsection{Sublinear and superlinear CDGAs}\label{Sec:Sublinear}

We will be interested in calculating the homologies of semi-free CDGAs and so we identify subclasses of such CDGAs for which computations are sometimes practically feasible.

We say that such a CDGA $\Algebra = (\tensorAlgGraded(V), \partial)$ is \emph{sublinear} if its differential takes the form
\begin{equation*}
\partial = \partial_{0} + \partial_{1}.
\end{equation*}
As mentioned in the introduction, chain-level contact homology CDGAs are sublinear iff their differentials count only holomorphic planes and cylinders. In the sublinear case, $\partial$ is non-increasing with respect to the tensor length filtration
\begin{equation*}
\Ring = \tensorAlgGraded^{\leq 0}(V) \subset \tensorAlgGraded^{\leq 1}(V) \subset \cdots \subset \tensorAlgGraded(V).
\end{equation*}
Moreover, we find that
\begin{equation*}
\partial_{1}^{2} = \partial_{0}\partial_{1} = 0.
\end{equation*}
Thus in the sublinear case we find that
\be
\item $(V, \partial_{1})$ is a chain complex of free $\Ring$-modules and
\item $\partial_{0}: (V, \partial_{1}) \to \Ring$ is a chain map inducing a map $H\partial_{0}$ on homology.
\ee

\begin{lemma}\label{Lemma:SublinearVanishing}
Suppose that $\Ring = \Field$ and that $\Algebra$ is sublinear. Then the following are equivalent:
\be
\item $H\partial_{0}: H(V, \partial_{1}) \to \Ring$ is zero.
\item $\Algebra$ admits an augmentation.
\item $H\Algebra\neq 0$.
\ee
\end{lemma}

\begin{proof}
Over a field, $H\partial_{0}=0$ is equivalent to $\partial_{0}$ being null-homotopic as a map $(\orbitVS, \partial_{1}) \to (\Ring, 0)$. Let's say $\partial_{0} = \psi \partial_{1}$ where $\psi$ has degree $0$. Then $v \mapsto -\psi v$ provides an augmentation of $\Algebra$ by extension from $\orbitVS$ to $\tensorAlgGraded(\orbitVS)$. So we have established (1) implies (2) which clearly implies (3).

We now seek to show that (3) implies (1). Let $v \in \orbitVS$ be such that $\partial_{1}v = 0$ so that
\begin{equation*}
\partial v = \partial_{0}v = H\partial_{0}[v]
\end{equation*}
where $[v] \in H(\orbitVS, \partial_{1})$ is the associated class. Assuming $H\Algebra\neq 0$, we must have $H\partial_{0}[v] = 0$ as otherwise the unit would be exact in $\Algebra$.
\end{proof}

We say that a semi-free CDGA $\Algebra$ is \emph{superlinear} if its differential takes the form
\begin{equation*}
\partial = \sum_{k=1}^{\infty}\partial_{k} \iff \partial_{0} = 0.
\end{equation*}
Such a CDGA admits a \emph{canonical augmentation}, given by $\aug_{0}v = 0$ for all $v\in V$. We find that $\partial_{1}^{2} = 0$ so that $(V, \partial_{1})$ is a chain complex whose homology is the \emph{linearized homology of $\aug_{0}$}. Any semi-free CDGA $(\tensorAlgGraded(V), \partial)$ admitting an augmentation $\aug$ is isomorphic to a superlinear CDGA
\begin{equation}\label{Eq:SuperLinearTransform}
(\tensorAlgGraded(V), \partial^{\aug}), \quad \partial^{\aug} =\phi_{\aug}\partial\phi_{\aug}^{-1}, \quad \phi_{\aug}v = v + \aug(v).
\end{equation}
This transformation takes $\aug$ to the canonical augmentation and demonstrates that a semi-free CDGA is isomorphic to a superlinear CDGA iff it admits an augmentation. See, for example, \cite{CL:SFTStringTop} regarding the role of such transformations in defining linearized contact homology theories. Since the differential $\partial$ preserves the tensor length filtration
\begin{equation*}
\cdots \subset \tensorAlgGraded^{\geq k}(V) \subset \cdots \subset \tensorAlgGraded^{\geq 1}(V) \subset \tensorAlgGraded^{\geq 0}(V) = \tensorAlgGraded(V)
\end{equation*}
it determines an associated spectral sequence which may assist in computation.

\subsection{Modules over DGAs}

Let $\Algebra= (A, \partial)$ be a DGA over $\Ring$. Let
\begin{equation*}
\Module = (M, \partial_{M})
\end{equation*}
be a chain complex of $\Ring$-modules. Then $\Module$ is a \emph{left $\Algebra$-module} if $M$ is a left $A$-module and
\begin{equation*}
\partial_{M}(a\cdot m) = (\partial a)\cdot m + (-1)^{|a|}a\cdot (\partial_{M}m)
\end{equation*}
for all $a \in A$ and $m \in M$. Likewise, $\Module$ is a \emph{right $\Algebra$-module} if it is a right $A$-module satisfying
\begin{equation*}
\partial_{M}(m\cdot a) =(\partial_{M}m)\cdot a + (-1)^{|m|}m \cdot (\partial a).
\end{equation*}
We say that $\Module$ is a \emph{$\Algebra$-bimodule} or, alternatively, a \emph{left-right $\Algebra$-module} if it is both a left- and right $\Algebra$-module. The homology $H\Module$ of a left (resp. right, or left-right) $\Algebra$-module, $\Module$, naturally carries the structure of a left (resp. right, or left-right) module over the homology $H\Algebra$.

\begin{ex}
If $\phi: \Algebra \to \AlgebraB$ is a morphism of DGAs for some $\AlgebraB = (B, \partial)$, then we may consider $\AlgebraB$ as a $\Algebra$-bimodule, where the left and right  multiplications are given by
\begin{equation*}
\begin{gathered}
A \otimes B \to B, \quad a \otimes b \mapsto \phi(a)\cdot b,\\
B \otimes A \to B, \quad b \otimes a \mapsto b\cdot \phi(a).
\end{gathered}
\end{equation*}
We write $\AlgebraB_{\phi}$ when we consider $\AlgebraB$ as a bimodule over $\Algebra$ in this way. In particular, $\Algebra_{\Id}$ denotes $\Algebra$ as a bimodule over itself.
\end{ex}

Provided a right $\Algebra$-module, $\Module = (M, \partial_{M})$, and a left $\Algebra$-module, $\ModuleN = (N, \partial_{N})$, we can define a tensor product over $\Algebra$ as
\begin{equation*}
\Module \otimes_{\Algebra} \ModuleN = (M \otimes_{A} N, \partial), \quad \partial(m\otimes_{A} n) = \partial_{M}m \otimes_{A} n + (-1)^{|m|}m\otimes_{A} \partial_{N}n.
\end{equation*}
When $\Module$ and $\ModuleN$ are $\Algebra$-bimodules, then so is $\Module \otimes_{\Algebra} \ModuleN$.

\subsubsection{Morphisms of $\Algebra$-modules}

Let $\Module = (M, \partial_{M})$ and $\ModuleN = (N, \partial_{N})$ be right (resp. left, or left-right) $\Algebra$-modules. A $\Ring$-linear chain map
\begin{equation*}
\phi: M \to N
\end{equation*}
compatible with the $A$-actions by multiplication is a \emph{morphism of left (resp. right, or left-right) $\Algebra$-modules}. In this case we write
\begin{equation*}
\phi: \Module \to \ModuleN.
\end{equation*}
We say that $\phi$ is a \emph{quasi-isomorphism} if the induced morphism of $H\Algebra$ modules
\begin{equation*}
H\phi: H\Module \to H\ModuleN
\end{equation*}
is an isomorphism.

\subsubsection{Semi-free resolutions}

Suppose that $\Module$ is a left $\Algebra$-module and that $M$ is free as a left-$A$ module. This means we can write
\begin{equation*}
M = A \otimes X
\end{equation*}
for a graded, free $\Ring$-module $X$. In this case, we can write
\begin{equation*}
\partial_{M}(a\otimes x) = \partial a \otimes x + (-1)^{|a|}a\cdot \partial_{in}x
\end{equation*}
for a $\Ring$-linear map
\begin{equation*}
\partial_{in}: X \to M
\end{equation*}
which we call an \emph{internal differential}. In this case we say that $\Module$ is \emph{semi-free} if $M$ and $X$ admit a filtration
\begin{equation*}
0 = M^{0} \subset M^{1} \subset \cdots \subset M, \quad X^{k} = M^{k}\cap X \quad \partial_{in}X^{k} \subset M^{k-1},
\end{equation*}
with each $X^{k}$ a free $\Ring$-module. Likewise, a right $\Algebra$-module is semi-free if we can write
\begin{equation*}
\begin{gathered}
M = X\otimes A,\\
\partial_{M}(x\otimes a) = (\partial_{in}x)\cdot a + (-1)^{|x|}x\otimes(\partial a),\\
\partial_{in}X^{k} \subset M^{k-1}
\end{gathered}
\end{equation*}
for a similarly defined internal differential and filtration on $X$ by free $\Ring$ submodules. Finally, a $\Algebra$-bimodule is semi-free if it is semi-free as an $A \otimes A^{op}$-module. This means for $a, b \in A$ we can write
\begin{equation*}
\begin{gathered}
M = A\otimes X\otimes A,\\
\partial_{M}(a\otimes x \otimes b) = (\partial a)\otimes x \otimes b + (-1)^{|a|}a\cdot (\partial_{in}x)\cdot b + (-1)^{|a\cdot x|}a\otimes x\otimes (\partial b),\\
\partial_{in}X^{k} \subset M^{k-1}.
\end{gathered}
\end{equation*}

\begin{defn}
Let $\Module$ be a left (resp. right, or left-right) $\Algebra$-module. A \emph{semi-free resolution of $\Module$} is the data of a semi-free left (resp. right, or left-right) $\Algebra$-module
\begin{equation*}
\widetilde{\Module} = (\widetilde{M}, \partial_{\widetilde{M}})
\end{equation*}
and a quasi-isomorphism of $\Algebra$-modules
\begin{equation*}
\pi: \widetilde{\Module} \to \Module.
\end{equation*}
\end{defn}

For a proof of the following general lemma, see, for example, \cite[\S 6]{FHT}.

\begin{lemma}
Every $\Algebra$-module admits a semi-free resolution. Suppose that we have right $\Algebra$-modules $\Module$ and $\Module'$, left $\Algebra$-modules $\ModuleN$ and $\ModuleN'$, and a pair of quasi-isomorphisms
\begin{equation*}
\phi: \Module \to \Module', \quad \psi: \ModuleN \to \ModuleN'.
\end{equation*}
If both $\Module, \Module'$ are semi-free or both $\ModuleN, \ModuleN'$ are semi-free, then
\begin{equation*}
\phi \otimes_{\Algebra} \psi: \Module \otimes_{\Algebra} \ModuleN \to \Module' \otimes_{\Algebra} \ModuleN'
\end{equation*}
is a quasi-isomorphism.
\end{lemma}

Now suppose that $\Module$ is a right $\Algebra$-module and $\ModuleN$ is a left $\Algebra$-module. Take semi-free resolutions
\begin{equation*}
\pi_{\Module}: \widetilde{\Module} \to \Module, \quad \pi_{\ModuleN}: \widetilde{\ModuleN} \to \ModuleN.
\end{equation*}
The above lemma (combined with a moment's thought) tells us that $\widetilde{\Module} \otimes_{\Algebra} \ModuleN$ and $\Module \otimes_{\Algebra} \widetilde{\ModuleN}$ are quasi-isomorphic. Thus
\begin{equation*}
\Module \otimesDerived_{\Algebra} \ModuleN = H(\widetilde{\Module} \otimes_{\Algebra} \ModuleN)
\end{equation*}
is well-defined as a graded $\Ring$-module.

\begin{defn}
We say that $\Module \otimesDerived_{\Algebra} \ModuleN$ is the \emph{derived tensor product} of $\Module$ and $\ModuleN$.
\end{defn}

We can likewise interpret the derived tensor product as the (ungraded) $\Tor(\Module, \ModuleN)$. We will discuss graded $\Tor_{\ast}$ groups in \S \ref{Sec:AlgRemarks}. We will be most interested in studying derived tensor products of modules determined by augmentations of $\Algebra$.

\subsection{Cylinder DGAs}\label{Sec:CylinderDGAs}

Let $\Algebra$ be a DGA. A \emph{cylinder} for $\Algebra$ is a DGA $\Algebra^{c}$ such that there exists a commutative diagram
\begin{equation*}
\begin{tikzcd}
\Algebra \sqcup \Algebra \arrow[r, "I"] \arrow[rr, bend right=30, "\mu"]& \Algebra^{c} \arrow[r, "\pi"] & \Algebra
\end{tikzcd}
\end{equation*}
where $I$ is an injective DGA morphism, $\mu$ is the codiagonal, and $\pi$ is a surjective quasi-isomorphism. This may alternatively be expressed with a commutative diagram
\begin{equation}\label{Eq:CylinderStructure}
\begin{tikzcd}
\Algebra \arrow[r, "I^{+}"] \arrow[dr, swap, "\Id_{\Algebra}"] & \Algebra^{c} \arrow[d, "\pi"] & \arrow[l,swap,"I^{-}"] \arrow[dl, "\Id_{\Algebra}"] \Algebra\\
& \Algebra
\end{tikzcd}, \quad I = I^{+} \sqcup I^{-}.
\end{equation}

Note that for CDGAs the coproducts $\sqcup$ above correspond to tensor products over $\Ring$ and $\mu$ is multiplication. Suppose that a CDGA $\Algebra = (\tensorAlgGraded(V), \partial)$ admits a cylinder of the form
\begin{equation}\label{Eq:NiceCylinder}
\Algebra^{c} = (\tensorAlgGraded(V^{+} \oplus W \oplus V^{-}), \partial^{c})
\end{equation}
so that the structure maps $I^{+}, I^{-}$ are induced by isomorphisms of free $\Ring$-modules,
\begin{equation*}
V^{+} \xleftarrow{I^{+}} V \xrightarrow{I^{-}} V^{-}.
\end{equation*}
Here $W$ is some graded, free $\Ring$-module. We say that such a cylinder is a \emph{small resolution of $\Algebra$}. The following lemma states that all small resolution cylinders are quasi-isomorphic.

\begin{lemma}\label{Lemma:CylindersEquivalent}
Suppose that the CDGA $\Algebra = (\tensorAlgGraded(V), \partial)$ is finitely filtered and let
\begin{equation*}
\Algebra^{c_{i}}=(\tensorAlgGraded(V^{+} \oplus W_{i} \oplus V^{-}), \partial^{c_{i}}), \quad i=0,1
\end{equation*}
be a pair of cylinders of $\Algebra$ as in Equation \eqref{Eq:NiceCylinder} for graded free $\Ring$-modules $W_{i}$, with structure maps $\pi_{i}: \Algebra^{c_{i}} \to \Algebra$. If $\Algebra^{c_{0}}$ is finitely filtered as well, then there is a commutative diagram of CDGA morphisms
\begin{equation*}
\begin{tikzcd}
\Algebra^{\otimes 2} \arrow[r, "I"]\arrow[d, swap ,"I"] & \Algebra^{c_{1}} \arrow[d, "\pi_{1}"]\\
\Algebra^{c_{0}} \arrow[r, "\pi_{0}"] \arrow[ur, dashed, "\psi"] & \Algebra.
\end{tikzcd}
\end{equation*}
for which $\psi$ is a quasi-isomorphism.
\end{lemma}

\begin{proof}
The existence of $\psi$ follows from Lemma \ref{Lemma:Lifting}. Since the $\pi_{i}$ are quasi-isomorphisms, so is $\psi$.
\end{proof}

\begin{defn}\label{Def:DGHomotopy}
Let $\AlgebraB$ be a CDGA. We say that morphisms $\phi, \psi: \Algebra \to \AlgebraB$ are \emph{DG-homotopic} if there exists an extension $H$
\begin{equation*}
\begin{tikzcd}
\Algebra \arrow[r, "I^{+}"] \arrow[dr, swap, "\phi"] & \Algebra^{c} \arrow[d, "H"] & \arrow[l,swap,"I^{-}"] \arrow[dl, "\psi"] \Algebra\\
& \AlgebraB
\end{tikzcd}
\end{equation*}
over a cylinder of $\Algebra$ as in Equation \eqref{Eq:NiceCylinder}.
\end{defn}

Using Lemma \ref{Lemma:CylindersEquivalent} it can be seen that the above definition is independent of the choice of cylinder.

\subsubsection{CDGA structures on homotopy colimits}

Let $\Algebra^{c}$ be a cylinder as in Equation \eqref{Eq:NiceCylinder}. If the $\Algebra$-bimodule $\Algebra^{c}_{I}$ is semi-free then for each right $\Algebra$-module $\Module$ and left $\Algebra$-module $\ModuleN$, $\Module\otimes_{\Algebra}\Algebra^{c}_{I}$ and $\Algebra^{c}_{I}\otimes_{\Algebra}\ModuleN$ are both semi-free $\Algebra$-modules. Thus we can compute the derived tensor product as
\begin{equation}\label{Eq:DerivedTensorFromCylinder}
\Module \otimesDerived_{\Algebra} \ModuleN = H(\Module \otimes_{\Algebra} \Algebra^{c}_{I} \otimes_{\Algebra} \ModuleN).
\end{equation}

We have seen that $\Module \otimesDerived_{\Algebra} \ModuleN$ is independent of the semi-free resolution (of either $\Module$ or $\ModuleN$) used to compute it as a graded $\Ring$-module. In certain situations (relevant to contact homology), we can use cylinders to equip $\Module \otimes_{\Algebra} \Algebra^{c}_{I} \otimes_{\Algebra} \ModuleN$ with the structure of a CDGA whose quasi-isomorphism type is independent of the choice of cylinder, $\Algebra^{c}$, when $\Algebra^{c}$ meets certain criteria. Thus, in such situations, $\Module \otimesDerived_{\Algebra} \ModuleN$ will be well-defined as a graded algebra.

Suppose that we have semi-free CDGAs
\begin{equation*}
\Module = (\tensorAlgGraded(V_{M}), \partial_{M}), \quad \ModuleN=(\tensorAlgGraded(V_{N}), \partial_{N})
\end{equation*}
and that we are provided CDGA morphisms
\begin{equation*}
\Module \xleftarrow{\phi_{M}} \Algebra \xrightarrow{\phi_{N}} \ModuleN
\end{equation*}
with respect to which we view $\Module$ and $\ModuleN$ as $\Algebra$-bimodules. Then we can identify
\begin{equation*}
\Module \otimes_{\Algebra} \Algebra^{c}_{I} \otimes_{\Algebra} \ModuleN = (\tensorAlgGraded(V_{M} \oplus W \oplus V_{N}), \partial)
\end{equation*}
as a semi-free CDGA with $\partial$ induced by the differentials $\partial_{A}, \partial_{M}, \partial_{N}$ and the morphisms $\phi_{M}, \phi_{N}$. Explicitly, $\partial$ is determined by replacing each $v^{+} \in V^{+}$ (resp. $v^{-}$) appearing in the differential of a $w \in W \subset \Algebra^{c}$ with $\phi_{M}(v)$ (resp. $\phi_{N}(v)$). We say that $\Module \otimes_{\Algebra} \Algebra^{c}_{I} \otimes_{\Algebra} \ModuleN$ is the \emph{homotopy colimit}. A slight variation of the proof of Lemma \ref{Lemma:CylindersEquivalent} shows that the quasi-isomorphism class of this CDGA is independent of the choice of cylinder $\Algebra^{c}_{I}$.

\begin{lemma}
Suppose that $\Algebra, \Module, \ModuleN$ are semi-free, filtered CDGAs as above. Let $\Algebra^{c_{i}}, i=0,1$ be cylinders of $\Algebra$ as in Lemma \ref{Lemma:CylindersEquivalent}. Then a quasi-isomorphism, $\psi$, as described in that lemma induces a quasi-isomorphism
\begin{equation*}
\Psi = \Id_{\Module} \otimes_{\Algebra} \psi \otimes_{\Algebra} \Id_{\ModuleN}
\end{equation*}
of CDGAs for which the diagram
\begin{equation*}
\begin{tikzcd}[sep=large]
\Module \otimes_{\Algebra} \ModuleN \arrow[r, "\Id_{\Module}\otimes 1 \otimes \Id_{\ModuleN}"]\arrow[d, swap ,"\Id_{\Module}\otimes 1 \otimes \Id_{\ModuleN}"] & \Module \otimes_{\Algebra} \Algebra^{c_{1}} \otimes_{\Algebra} \ModuleN \arrow[d]\\
\Module \otimes_{\Algebra} \Algebra^{c_{0}} \otimes_{\Algebra} \ModuleN \arrow[ur, dashed, "\Psi"] \arrow[r] & \Module \otimes_{\Algebra} \ModuleN
\end{tikzcd}
\end{equation*}
is commutative. Here the unlabeled arrows are induced by the structure maps, $\pi_{i}$, for the cylinders.

In particular, the quasi-isomorphism type of the semi-free CDGA $\Module \otimes_{\Algebra} \Algebra^{c} \otimes_{\Algebra} \ModuleN$ (whose homology computes $\Module \otimesDerived_{\Algebra} \ModuleN$) is independent of the choice of cylinder $\Algebra^{c}$, so long as $\Algebra$ and $\Algebra^{c}$ are filtered with $\Algebra^{c}$ as in Equation \eqref{Eq:NiceCylinder}.
\end{lemma}

\begin{proof}
Lemma \ref{Lemma:CylindersEquivalent} provides a diagram
\begin{equation*}
\begin{tikzcd}[sep=large]
\Module \otimes_{\Algebra} \Algebra \arrow[r, "\Id_{\Module}\otimes I^{-}"]\arrow[d, swap ,"\Id_{\Module}\otimes I^{-}"] & \Module \otimes_{\Algebra} \Algebra^{c_{1}} \arrow[d]\\
\Module \otimes_{\Algebra} \Algebra^{c_{0}} \arrow[ur, dashed, "\Id_{\Module} \otimes_{\Algebra} \psi"] \arrow[r] & \Module
\end{tikzcd}
\end{equation*}
by taking tensor products with $\Module$ over $\Algebra$ on the left. The arrows coming out of the $\Module \otimes \Algebra^{c_{i}}$ are given by $\Id_{\Module}$ for elements of $\Module$ and by $\phi_{M}\circ \pi_{i}$ for elements of $\Algebra^{c_{i}}$. These arrows together with their sources provide semi-free resolutions of $\Module$. Since these arrows are quasi-isomorphisms, so is the diagonal arrow, $\Id_{\Module} \otimes_{\Algebra} \psi$. The diagram appearing in the statement of the lemma is then obtained by taking tensor products with $\ModuleN$ on the right. Since the $\Module \otimes_{\Algebra} \Algebra^{c_{i}}$ are semi-free as right $\Algebra$-modules and $\Id_{\Module} \otimes_{\Algebra} \psi$ is a quasi-isomorphism, so is $\Psi$.
\end{proof}

The proof of the following lemma relies on the existence of small resolutions of semi-free finitely filtered CDGAs, which will be established in \S \ref{Sec:SmallResolutions}.

\begin{lemma}\label{Lemma:CylinderAugEqualsHomotopy}
Suppose that $\aug$ and $\augDelta$ are augmentations of a semi-free, finitely filtered CDGA $\Algebra$ over a ring $\Ring$ containing $\Q$. Then $\aug$ and $\augDelta$ are DG-homotopic iff $\Ring_{\aug} \otimes_{\Algebra} \Algebra^{c} \otimes_{\Algebra} \Ring_{\augDelta}$ admits an augmentation, for $\Algebra^{c}$ a cylinder for $\Algebra$. 
\end{lemma}

\begin{proof}
Using a cylinder as in Equation \eqref{Eq:NiceCylinder}, we can write
\begin{equation*}
\Ring_{\aug} \otimes_{\Algebra} \Algebra^{c} \otimes_{\Algebra} \Ring_{\augDelta} = (\tensorAlgGraded(W), \partial^{\aug, \augDelta})
\end{equation*}
where $\partial^{\aug, \augDelta}$ is determined by $\partial^{c}$ by setting each $I^{+}v$ to $\aug(v)$ and each $I^{-}v $ to $\augDelta(v)$. We then have a diagram
\begin{equation}\label{Eq:ChainDerivedProjection}
\begin{tikzcd}
\Algebra \arrow[r, "I^{+}"] \arrow[dr, swap, "\aug"] & \Algebra^{c} \arrow[d] & \arrow[l,swap,"I^{-}"] \arrow[dl, "\augDelta"] \Algebra\\
& (\tensorAlgGraded(W), \partial^{\aug, \augDelta})
\end{tikzcd}
\end{equation}
where the vertical arrow sends $w \in W$ to $w$.

Any augmentation of the homotopy colimit then induces an augmentation $H$ of $\Algebra^{c}$ (as described in Equation \eqref{Eq:NiceCylinder}) fitting into the diagram of Definition \ref{Def:DGHomotopy}, via composition with the vertical arrow of Equation \eqref{Eq:ChainDerivedProjection}. Thus the existence of an augmentation of $(\tensorAlgGraded(W), \partial^{\aug, \augDelta})$ implies that $\aug$ and $\augDelta$ are DG-homotopic. Conversely, it is easy to check that any DG-homotopy $H$ between $\aug$ and $\augDelta$ must factor through $(\tensorAlgGraded(W), \partial^{\aug, \augDelta})$. Quasi-isomorphisms of finitely filtered CDGAs can be quasi-isomorphically inverted by Lemma \ref{Lemma:QIInversion}. Consequently the existence of an augmentation of $\Ring_{\aug} \otimes_{\Algebra} \Algebra^{c} \otimes_{\Algebra} \Ring_{\augDelta}$ depends only on its quasi-isomorphism type and our proof is independent of our choice of chain-level model, $(\tensorAlgGraded(W), \partial^{\aug, \augDelta})$.
\end{proof}

\begin{lemma}\label{Lemma:HomotopicIffNonzeroDTP}
Suppose that $\aug$ and $\augDelta$ are augmentations of a semi-free finitely filtered CDGA, $\Algebra$. If $\Algebra$ has $\Ring=\C$ coefficients and is $\Z/2\Z$ graded, then the following are equivalent:
\be
\item $\aug$ and $\augDelta$ are DG-homotopic.
\item $\C_{\aug} \otimesDerived_{\Algebra} \C_{\augDelta} \neq 0$.
\ee
\end{lemma}

\begin{proof}
According to Lemma \ref{Lemma:AugEqualsNonZeroHomology}, $\C_{\aug} \otimesDerived_{\Algebra} \C_{\augDelta} \neq 0$ iff its chain-level model $\Ring_{\aug} \otimes_{\Algebra} \Algebra^{c} \otimes_{\Algebra} \Ring_{\augDelta}$ admits an augmentation. By Lemma \ref{Lemma:CylinderAugEqualsHomotopy} this is equivalent to $\aug$ and $\augDelta$ being DG-homotopic.
\end{proof}

\subsection{An example of a small resolution}\label{Sec:SmallResolutions}

We describe cylinder objects for finitely filtered, semi-free CDGAs, $\Algebra$, following \cite[\S 2.2]{FOT} under the assumption that $\Q \subset \Ring$. This construction demonstrates the existence of cylinders for such CDGAs which include chain-level contact homology CDGAs. We remark that \cite{BL:MinimalModel} provides a construction of cylinders of semi-free non-commutative DGAs which will be discussed further in \S \ref{Sec:AlgRemarks}.

Let $V^{r}$ be a copy of $V$ with the same grading and $\widehat{V}$ be a copy with the grading shifted up by one. That is, for each additive generator $v$ of $V$, we have additive generators $v^{r} \in V^{r}$ and $\widehat{v} \in \widehat{V}$ with
\begin{equation*}
|v^{r}| = |v|, \quad |\widehat{v}| = |v| + 1.
\end{equation*}
Define a finitely filtered semi-free CDGA, $\Algebra^{c}$, by
\begin{equation*}
\begin{gathered}
\Algebra^{c} = (\tensorAlgGraded(V^{c}), \partial^{c}), \quad V^{c} = V \oplus \widehat{V} \oplus V^{r},\\
\action(v^{r}) = \action(v) < \action(\widehat{v}),\\
\partial^{c} v = \partial v, \quad \partial^{c} \widehat{v} = v^{r}, \quad \partial^{c} v^{r} = 0.
\end{gathered}
\end{equation*}
Now we define maps $I^{-}$, $I^{+}$, and $\pi$ giving $\Algebra^{c}$ the structure of a cylinder on $\Algebra$.

Define a $\deg = 1$ derivation $s$ on $\tensorAlgGraded(V^{c})$ by
\begin{equation*}
s(v) = \widehat{v}, \quad s(\widehat{v}) = s(v^{r}) = 0 \implies s^{2} = 0.
\end{equation*}
We also define
\begin{equation*}
\theta = s\circ \partial^{c} + \partial^{c} \circ s.
\end{equation*}
Since $\theta$ is a graded bracket of $\deg = 1$ and $\deg=-1$ derivations, it is a $\deg  = 0$ derivation. Moreover we find
\begin{equation*}
\begin{gathered}
\theta \widehat{v} = \theta v^{r} = 0,\\
s\theta = s\partial^{c} s = \theta s, \quad \partial^{c}\theta = \partial^{c} s \partial^{c} = \theta \partial^{c},\\
\theta^{k} = (s\partial^{c})^{k} + (\partial^{c}s)^{k}\ \forall k\geq 1,\\
\theta v = v^{r} + s\partial v,\\
(\partial^{c}s)^{2}v = (\partial^{c}s)v^{r} = 0 \implies (\partial^{c}s)^{k}|_{V} = 0\ \forall k \geq 2.
\end{gathered}
\end{equation*}
These observations inform us that
\begin{equation*}
e^{\theta}v := \sum_{k=0}^{\infty} (k!)^{-1}\theta^{k}v = v + v^{r} + \sum_{k=1}^{\infty}(k!)^{-1}(s\circ \partial)^{k}v = v^{r} + e^{s \circ \partial}v
\end{equation*}
defines a CDGA morphism
\begin{equation*}
e^{\theta}: \Algebra \to \Algebra^{c}.
\end{equation*}
To complete the verification of this claim, we must first check that $e^{\theta}$ is well-defined as a $\Ring$-linear map. That is, for each $v$, $e^{\theta}v$ must consist of finitely many non-zero summands. Indeed, each application of $s\partial^{c}$ either lowers actions, replaces a $\widehat{v}$ with a $v^{r}$, or adds a hat to a $v$. Since $s$ kills monomials without $V$ factors, $(s\partial)^{k}v=0$ for each $v \in V$ for $k$ sufficiently large.\footnote{In \cite{FOT} a grading condition on $V$, rather than an action filtration, is used to establish that $e^{\theta}$ is well-defined.} Second, we must verify the identity
\begin{equation*}
e^{\theta}(xy) = e^{\theta}x\cdot e^{\theta}y
\end{equation*}
for $x,y \in \tensorAlgGraded(V)$. Details here are left to the reader.

We also define a surjective CDGA morphism
\begin{equation*}
\pi: \Algebra^{c} \to \Algebra, \quad \pi v = v, \quad \pi \widehat{v} = \pi v^{r} = 0.
\end{equation*}
Writing $I_{\Algebra}: \Algebra \to \Algebra^{c}$ for the obvious inclusion $v \mapsto v$, we have
\begin{equation*}
\pi I_{\Algebra} = \pi e^{\theta} = \Id_{\Algebra}.
\end{equation*}
To see that $\pi$ is a quasi-isomorphism, consider
\begin{equation*}
\Algebra^{c} = \Algebra \otimes (\tensorAlgGraded(\widehat{V} \oplus V^{r}), \partial\widehat{v}=v^{r})
\end{equation*}
whose second factor is quasi-isomorphically mapped to $\Ring$ by $\pi$. It follows that $\Algebra^{c}$ has the structure of a CDGA cylinder on $\Algebra$,
\begin{equation*}
\begin{tikzcd}
\Algebra \arrow[r, "I_{\Algebra}"] \arrow[dr, swap, "\Id_{\Algebra}"] & \Algebra^{c} \arrow[d, "\pi"] & \arrow[l,swap,"e^{\theta}"] \arrow[dl, "\Id_{\Algebra}"] \Algebra\\
& \Algebra
\end{tikzcd}
\end{equation*}
so that $I^{+} = I_{\Algebra}$ and $I^{-} = e^{\theta}$. It can easily be seen that $\Algebra^{c}$ is a small resolution of $\Algebra$ via a change of variables
\begin{equation*}
\tensorAlgGraded(V^{c}) = \tensorAlgGraded(V \oplus \widehat{V} \oplus e^{\theta}V), \quad e^{\theta}V = \oplus \Ring e^{\theta}v.
\end{equation*}

\subsubsection{Computing derived tensor products using small resolutions}

Provided a pair of augmentations $\aug, \augDelta$, we can use $\Algebra^{c}$ to compute the derived tensor product $\Ring_{\aug} \otimesDerived_{\Algebra} \Ring_{\augDelta}$ as the homology of the algebra $(\tensorAlgGraded(\widehat{V}), \partial^{\aug, \augDelta})$. Here $\partial^{\aug, \augDelta}$ is the differential induced by $\partial^{c}$ after replacing
\begin{equation*}
v \mapsto \aug v, \quad e^{\theta}v \mapsto \augDelta v.
\end{equation*}
The differentials $\partial^{\aug, \augDelta}\widehat{v}$ can be explicitly solved by recursive application of the above rules.

In general, $(\tensorAlgGraded(\widehat{V}), \partial^{\aug, \augDelta})$ is not guaranteed to be either sublinear or superlinear, even when $\Algebra$ is superlinear. This makes computation of $\Ring_{\aug} \otimesDerived_{\Algebra} \Ring_{\augDelta}$ -- even determination of its vanishing or non-vanishing -- quite difficult.

\begin{lemma}\label{Lemma:SublinearOtimes}
Suppose that $\Algebra = (\tensorAlgGraded(V), \partial = \partial_{0} + \partial_{1})$ is finitely filtered, sublinear CDGA with augmentations $\aug, \augDelta$, defined over a ring $\Ring \supset \Q$. Then the homotopy colimit of $\Ring_{\aug}$ and $\Ring_{\augDelta}$ is given by
\begin{equation}\label{Eq:SublinearDerivedTensor}
(\tensorAlgGraded(\widehat{V}), \partial^{\aug, \augDelta}), \quad \partial^{\aug, \augDelta}\widehat{v} = \augDelta (v) - \aug(v) - \widehat{\partial_{1}v},
\end{equation}
and hence is sublinear. Therefore, when working over a field $\Ring = \Field \supset \Q$,
\begin{equation*}
\Ring_{\aug} \otimesDerived_{\Algebra} \Ring_{\augDelta} = H(\tensorAlgGraded(\widehat{V}), \partial^{\aug, \augDelta}) = 0 \iff H\aug \neq H\augDelta.
\end{equation*}
\end{lemma}

\begin{proof}
Equation \eqref{Eq:SublinearDerivedTensor} is a straightforward computation. Thus according to Lemma \ref{Lemma:SublinearVanishing}, $\Ring_{\aug} \otimesDerived_{\Algebra} \Ring_{\augDelta}$ is non-zero iff $\aug$ and $\augDelta$ agree on $\partial_{1}$-cycles in $V$. Note that any $\partial_{1}$-cycle is also a $\partial$ cycle. After conjugating by $\phi_{\aug}$ as described in Equation \eqref{Eq:SuperLinearTransform}, $\Algebra \simeq (\tensorAlgGraded(V), \partial_{1})$, and the K\"{u}nneth formula yields $H\Algebra \simeq \tensorAlgGraded(H(V, \partial_{1}))$. Therefore $H\aug = H\augDelta$ as morphisms of rings $H\Algebra \to \Ring$ is the same as $H\aug = H\augDelta$ as morphisms of additive groups $H(V, \partial_{1}) \to \Ring$.
\end{proof}

\subsection{Examples and bar construction computations}\label{Sec:CDGABar}

The following example demonstrates that the conclusions of Lemma \ref{Lemma:SublinearOtimes} do not hold in general for CDGAs which are not sublinear.

\begin{ex}\label{Ex:NonSublinear}
Consider the CDGA $\Algebra = (\tensorAlgGraded(V), \partial)$ over $\Ring=\Field \supset \Q$ generated by variables $x, y,$ and $z$ with
\begin{equation*}
|x| = |y| = 0, \quad |z| = -1, \quad \partial x = yz, \quad \partial y = \partial z = 0.
\end{equation*}
This algebra admits an action filtration with $\action(x)=3$, $\action(y)=\action(z)=1$. Let $\aug$ and $\augDelta$ be the augmentations
\begin{equation*}
\aug x = 1, \quad \aug y = \aug z = 0, \quad \augDelta x = \augDelta y = \augDelta z = 0.
\end{equation*}
Then
\begin{equation*}
H\aug = H\augDelta, \quad \Field_{\aug} \otimesDerived_{\Algebra} \Field_{\augDelta} = 0.
\end{equation*}
\end{ex}

\begin{proof}
Since $z$ has odd degree,
\begin{equation*}
\begin{gathered}
A_{0} = \Field[x, y], \quad A_{-1} = z\cdot \Field[x,y],\\
\partial p = (\partial_{x}p)yz, \quad \partial(pz) = (\partial_{x}p)\, yz^{2} = 0
\end{gathered}
\end{equation*}
for $p \in A_{0}$. Hence $\ker \partial_{x} = \Field[y]$ forms the $\deg=0$ cycles on which the augmentations $\aug$ and $\augDelta$ agree, implying $H\aug = H\augDelta$.

To compute the derived tensor product we use the cylinder $\Algebra^{c}$ of \S \ref{Sec:SmallResolutions} which is described by
\begin{equation*}
\begin{gathered}
|\widehat{x}| = |\widehat{y}| = 1, \quad |\widehat{z}| = 0, \quad |x^{r}| = |y^{r}| = 0, \quad |z^{r}| = -1,\\
s\partial x = s(yz) = \widehat{y}z + y\widehat{z}, \quad \partial^{c}(s\partial x) = y^{r}z + yz^{r},\\
(s\partial)^{2}x = y^{r}\widehat{z} + \widehat{y}z^{r}, \quad \partial^{c}\left((s\partial)^{2}x\right) = 2y^{r}z^{r}, \quad (s\partial)^{3}x = s(2y^{r}z^{r}) = 0,\\
\implies e^{\theta}y = y + y^{r}, \quad e^{\theta}z = z + z^{r}, \quad e^{\theta}x = x + x^{r} + \widehat{y}z + y\widehat{z} + \tfrac{1}{2}\left(y^{r}\widehat{z} + \widehat{y}z^{r}\right).
\end{gathered}
\end{equation*}
Thus for any pair of augmentations $\aug, \augDelta$ of $\Algebra$, the homotopy colimit is $(\tensorAlgGraded(\widehat{x}, \widehat{y}, \widehat{z}), \partial^{\aug, \augDelta})$ with
\begin{equation*}
\partial^{\aug, \augDelta}\widehat{x} = (\augDelta x - \aug x) - \tfrac{1}{2}(\aug y + \augDelta y)\widehat{z}, \quad \partial^{\aug, \augDelta}\widehat{y} = \augDelta y - \aug y, \quad \partial^{\aug, \augDelta}\widehat{z} = 0.
\end{equation*}
Compared with Equation \eqref{Eq:SublinearDerivedTensor}, the new feature is the $\widehat{z}$ term in $\partial^{\aug, \augDelta}\widehat{x}$, which arises from the quadratic term in $\partial x$. For our particular choices of $\aug$ and $\augDelta$, we obtain
\begin{equation*}
\partial^{\aug, \augDelta}\widehat{x} = -1, \quad \partial^{\aug, \augDelta}\widehat{y} = \partial^{\aug, \augDelta}\widehat{z} = 0.
\end{equation*}
Clearly $1 \in \tensorAlgGraded(\widehat{V})$ is a boundary so that $\Field_{\aug} \otimesDerived_{\Algebra} \Field_{\augDelta} = 0$.
\end{proof}

In light of the above example, we cannot in general determine the vanishing or non-vanishing of $\Ring_{\aug} \otimesDerived_{\Algebra} \Ring_{\augDelta}$ from the $H\aug$ and $H\augDelta$. An infinite series of obstructions to the non-vanishing of such derived tensor products can be described using bar resolutions as follows.

Observe that $\ker \aug \subset \Algebra$ is an ideal preserved by $\partial$. We write $\widehat{\ker\aug}$ for this ideal having grading shifted up by one. That is, its elements have the form $\widehat{a}$ with $|\widehat{a}| = |a| + 1$ for $a \in \ker \aug$. Define $\Ring$-modules
\begin{equation*}
\BarR_{k}^{\aug} = \Algebra \otimes (\widehat{\ker\aug})^{\otimes k} \otimes \Algebra, \quad \BarR^{\aug} = \oplus_{k \geq 0} \BarR_{k}^{\aug}.
\end{equation*}
We equip each $\BarR_{k}^{\aug}$ with $\Ring$-linear operators
\begin{equation}\label{Eq:BarInternal}
\begin{gathered}
\partial_{out}, \partial_{mid}: \BarR_{k}^{\aug} \to \BarR_{k}^{\aug},\\
\begin{aligned}
\partial_{out}(a \widehat{b}_{1}\cdots \widehat{b}_{k} c) =& (\partial a)\widehat{b}_{1}\cdots \widehat{b}_{k} c + (-1)^{|a \widehat{b}_{1}\cdots \widehat{b}_{k}|}a \widehat{b}_{1}\cdots \widehat{b}_{k}(\partial c),\\
\partial_{mid}(a \widehat{b}_{1}\cdots \widehat{b}_{k} c)=& -\sum_{i=1}^{k} (-1)^{|a \widehat{b}_{1}\cdots \widehat{b}_{i-1}|}a \widehat{b_{1}}\cdots \widehat{b}_{i-1}\widehat{\partial b_{i}}\widehat{b}_{i+1}\cdots \widehat{b_{k}}c.
\end{aligned}
\end{gathered}
\end{equation}
Here we suppress the appearances of tensor products when there is no risk of ambiguity. We also equip $\BarR^{\aug}$ with an operator which decreases tensor length
\begin{equation}\label{Eq:BarExternal}
\begin{gathered}
\partial_{dec}: \BarR_{k}^{\aug} \to \BarR_{k-1}^{\aug},\\
\begin{aligned}
\partial_{dec}(a \widehat{b}_{1}\cdots \widehat{b}_{k} c) =& (-1)^{|a|}(ab_{1})\widehat{b}_{2}\cdots \widehat{b}_{k}c\\
&+ \sum_{i=1}^{k-1}(-1)^{|a \widehat{b}_{1}\cdots \widehat{b}_{i}|}a \widehat{b_{1}} \cdots (\widehat{b_{i}b_{i+1}}) \cdots \widehat{b}_{k}c\\
&- (-1)^{|a \widehat{b}_{1}\cdots \widehat{b}_{k-1}|}a \widehat{b}_{1}\cdots \widehat{b}_{k-1}(b_{k}c).
\end{aligned}
\end{gathered}
\end{equation}
Then we may compute
\begin{equation*}
(\partial_{out} + \partial_{mid})^{2} = \partial_{dec}^{2} = (\partial_{out} + \partial_{mid})\partial_{dec} + \partial_{dec}(\partial_{out} + \partial_{mid}) = 0.
\end{equation*}
Thus the $\deg = -1$ operator
\begin{equation*}
\partial_{\BarR^{\aug}} = \partial_{out} + \partial_{mid} + \partial_{dec}: \BarR^{\aug} \to \BarR^{\aug}
\end{equation*}
squares to zero. Since we may write
\begin{equation*}
\BarR^{\aug} = \Algebra \otimes X \otimes \Algebra, \quad X = \oplus_{k=0}^{\infty} (\widehat{\ker \aug})^{\otimes k},
\end{equation*}
$\BarR^{\aug}$ is semi-free as a bimodule over $\Algebra$. This is the \emph{two-sided bar resolution associated to the pair $(\Algebra, \aug)$}.

When equipped with the shuffle product of \cite[\S 12]{MacLane:Homology}, $\BarR^{\aug}$ becomes a CDGA whose unit is $1 \otimes 1 \in \BarR^{\aug}_{0}$. It can moreover be seen that $\BarR^{\aug}$ is a cylinder for $\Algebra$ with structure maps
\begin{equation*}
I^{+}v = v\otimes 1, \quad I^{-}v = 1 \otimes v, \quad \pi (a \otimes c) = ac, \quad \pi (1\otimes \widehat{b} \otimes 1) = 0.
\end{equation*}
Thus we may compute the quasi-isomorphism class of the homotopy colimit as
\begin{equation*}
\BarR^{\aug, \augDelta}:= \Ring_{\aug} \otimes_{\Algebra} \BarR^{\aug} \otimes_{\Algebra} \Ring_{\augDelta} = (\oplus_{k \geq 0} \widehat{\ker\aug}^{\otimes k}, \partial_{\BarR^{\aug, \augDelta}}), \quad H\left(\Ring_{\aug} \otimes_{\Algebra} \BarR^{\aug} \otimes_{\Algebra} \Ring_{\augDelta}\right) = \Ring_{\aug} \otimesDerived_{\Algebra} \Ring_{\augDelta}
\end{equation*}
where $\partial_{\BarR^{\aug, \augDelta}}$ is the induced differential.

This differential preserves the (increasing) tensor length filtration
\begin{equation*}
\BarR^{\aug, \augDelta}_{\leq k} = \oplus_{i \leq k} \BarR^{\aug, \augDelta}_{i} \subset \BarR^{\aug, \augDelta}, \quad \BarR^{\aug, \augDelta}_{k} = \widehat{\ker\aug}^{\otimes k}.
\end{equation*}
The restriction of the differential to tensor-length $1$ elements $\widehat{\ker \aug}$ is given by
\begin{equation*}
\partial_{\BarR^{\aug, \augDelta}_{1}}(\widehat{a - \aug(a)}) = -\widehat{\partial a} + \aug(a) - \augDelta(a), \quad a \in \Algebra,\ |a|=0.
\end{equation*}
Thus $H\aug = H\augDelta$ for $\Ring = \Field$ coefficients (or $1 \notin \im(H\aug - H\augDelta)$ for general $\Ring$) is required in order that the unit not vanish in the derived tensor product.

We can also use the tensor length filtration to reorganize $\BarR^{\aug, \augDelta}$ as a sequence of chain complexes
\begin{equation*}
\begin{gathered}
\cdots \xrightarrow{\partial_{dec}^{\aug, \augDelta}} (\BarR^{\aug, \augDelta}_{2}, \partial_{in}^{\aug, \augDelta}) \xrightarrow{\partial_{dec}^{\aug,\augDelta}} (\BarR^{\aug, \augDelta}_{1}, \partial_{in}^{\aug, \augDelta}) \xrightarrow{\partial_{dec}^{\aug, \augDelta}} (\Ring, 0) \to 0,
\end{gathered}
\end{equation*}
where $\partial_{dec}^{\aug, \augDelta}$ and $\partial_{in}^{\aug, \augDelta} = \partial_{mid}^{\aug, \augDelta} + \partial_{out}^{\aug, \augDelta}$ are the induced operators determined by $\partial_{dec}, \partial_{mid}, \partial_{out}$ and the augmentations. Taking homology with respect to $\partial_{in}^{\aug, \augDelta}$ then yields the $E^{1}$ page of the tensor length spectral sequence. When using field coefficients,
\begin{equation*}
E^{2} = \Ring_{H\aug} \otimesDerived_{H\Algebra} \Ring_{H\augDelta} = \Tor^{H\Algebra}(\Ring_{H\aug}, \Ring_{H\augDelta}).
\end{equation*}
See, for example, \cite[Theorem 7.6]{McCleary:UsersGuide}. In general, we can view the vanishing of the $E^{k+1}$ page of this spectral sequence as the $k$th obstruction to the non-vanishing of $\Ring_{\aug} \otimesDerived_{\Algebra} \Ring_{\augDelta}$.

It can easily be seen that the second obstruction is non-trivial in Example \ref{Ex:NonSublinear}: Indeed, $1$ is the boundary of an element of tensor length two, while $1 \neq 0$ in $E^{2}$ since $H\aug = H\augDelta$.

For more on bar resolutions of CDGAs and associated spectral sequences, see \cite{FHT, Holstein:RHT, McCleary:UsersGuide}.

\subsection{Remarks on non-commutative DGAs and bilinearization}\label{Sec:AlgRemarks}

Cylinder objects and associated spectral sequences can be used to rephrase some existing results and constructions from \cite{BC:Bilinearized, BG:Geography, Ekholm:Nonloose} pertaining to non-commutative semi-free DGAs. The material here is supplementary and not required for the establishment of our main theorems. Here we do not assume that our rings $\Ring$ contain $\Q$.

Let $\Algebra^{nc} = (\tensorAlg(V), \partial)$ be such a DGA defined over a unital commutative ring $\Ring$. In the context of Chekanov-Eliashberg algebras associated to a Legendrian submanifold, $\Lambda$, inside the contactization $\Mxi = (\R_{t} \times W, \ker(dt + \beta))$ of a Liouville manifold $(W, \beta)$ one typically takes $\Ring$ to be one of $\Z$, $\Z/2\Z$, or an associated group ring over one of $H_{1}(\Lambda)$ or $H_{2}(M, \Lambda)$. See for example \cite{EES:PtimesR, EtnyreNg:LCHSurvey, Ng:KCHIntro}.

Cylinders, $\Algebra^{BL}$, for such DGAs are constructed in \cite{BL:MinimalModel} as follows. The algebra underlying $\Algebra^{BL}$ is
\begin{equation*}
A^{BL} = \tensorAlg(V \oplus \widehat{V} \oplus V^{r})
\end{equation*}
with multiplicative generators $\widehat{v} \in \widehat{V}$ and $v^{r} \in V^{r}$ associated to each $v \in V$ so that $|v^{r}| = |v|$ and $|\widehat{v}| = |v| + 1$. Define a $\Ring$-linear operator
\begin{equation*}
\stab: \tensorAlg(V) \to A^{BL}, \quad \stab(1)=0, \quad \stab(v_{1}\cdots v_{k}) = \sum_{i=1}^{k} (-1)^{|v_{1}\cdots v_{i-1}|}v_{1}\cdots v_{i-1}\widehat{v}_{i}v_{i+1}^{r}\cdots v^{r}_{k}.
\end{equation*}
We then define
\begin{equation*}
\partial^{BL}v = \partial v, \quad \partial^{BL}v^{r} = (\partial v)^{r}, \quad \partial^{BL}\widehat{v} = v - v^{r} - \stab(\partial v).
\end{equation*}
Then $\Algebra^{BL}$ is a cylinder for $\Algebra^{nc}$ with structure maps
\begin{equation*}
I^{+}v = v, \quad I^{-}v = v^{r}, \quad \pi v = \pi v^{r} = v, \quad \pi \widehat{v} = 0.
\end{equation*}

If $\aug$ and $\augDelta$ are augmentations of $\Algebra^{nc}$, then we can compute the homotopy colimit of
\begin{equation*}
\Ring \xleftarrow{\aug} \Algebra^{nc} \xrightarrow{\augDelta} \Ring
\end{equation*}
as both a semi-free non-commutative DGA and a quotient
\begin{equation*}
\begin{gathered}
\Algebra^{\aug, \augDelta} = (\tensorAlg(\widehat{V}), \partial^{\aug, \augDelta}) = \Algebra^{BL}/\sim, \\
\partial^{\aug, \augDelta}|_{\widehat{V}} = \partial_{0}^{\aug, \augDelta} + \partial_{1}^{\aug, \augDelta}, \quad v \sim \aug(v), \quad v^{r} \sim \augDelta(v).
\end{gathered}
\end{equation*}
Here $\partial_{0}^{\aug, \augDelta}\widehat{v} = \aug(v) - \augDelta(v)$ and $\partial_{1}^{\aug, \augDelta}\widehat{v}$ is determined by applying $\sim$ to $\stab(\partial v)$. For CDGAs such homotopy colimits agree with $\Ring_{\aug} \otimes_{\Algebra} \Algebra^{c} \otimes_{\Algebra} \Ring_{\augDelta}$ whose homologies compute derived tensor products.

In the case of augmentations associated to exact Lagrangian fillings $L_{0}, L_{1}$ of a Legendrian $\Lambda$ in the $\dim=2n-1$ standard contact Euclidean space, $\Algebra^{\aug, \augDelta}$ is exactly the algebra described in \cite[Lemma 2.3]{Ekholm:Nonloose}, which is the Chekanov-Eliashberg algebra of a Legendrian in $\dim=2n+1$ contact Euclidean space built from the $L_{i}$.

Note that the homotopy colimit $\Algebra^{\aug, \augDelta}$ is sublinear, even when $\Algebra^{nc}$ is not, contrasting with the case of CDGAs. Hence we may apply a tensor-length spectral sequence to $\Algebra^{\aug, \augDelta}$, for which
\begin{equation*}
H(\widehat{V}, \partial_{1}^{\aug, \augDelta}) \subset E^{1}.
\end{equation*}
This additive group is, up to a grading shift, exactly the bilinearized homology of \cite{BC:Bilinearized}. Moreover, the $\Ring$-linear map
\begin{equation}\label{Eq:DGAFundamentalClass}
\begin{gathered}
H\partial_{0}^{\aug, \augDelta}: H(\widehat{V}, \partial_{1}^{\aug, \augDelta}) \to \Ring, \quad H\partial_{0}^{\aug, \augDelta}[\widehat{v}] = \aug(v) - \augDelta(v),
\end{gathered}
\end{equation}
which is the restriction of the $E^{1}$ differential to $H(\widehat{V}, \partial_{1}^{\aug, \augDelta})$, is exactly the \emph{fundamental class} of \cite{BG:Geography}. For more interpretations of geometric constructions in Legendrian contact homology as derived functors see, for example, \cite{DR:DG, Legout:CY}.

\section{Preliminary geometric notions}\label{Sec:Prelim}

In this section we provide background on the geometric objects in this article.

\subsection{General notation}

Throughout we write
\begin{equation*}
\omega_{0} = \sum dx_{i} \wedge dy_{i}
\end{equation*}
for the standard symplectic form on $\R^{2n}$. The Lie group of linear symplectomorphisms of $(\R^{2n}, \omega_{0})$ will be denoted by $\Symp_{n}$. The associated Lie algebra will be denoted $\lieSymp_{n}$. We write
\begin{equation*}
\beta_{0} = \half \sum (x_{i}dy_{i} - y_{i}dx_{i}) \in \Omega^{1}(\R^{2n})
\end{equation*}
so that $d\beta_{0} = \omega_{0}$. The standard complex structure on $\R^{2n}$ will be denoted by $J_{0}$, and is defined by
\begin{equation*}
J_{0}\partial_{x_{i}} = \partial_{y_{i}}.
\end{equation*}

\subsection{Contact and symplectic manifolds}\label{Sec:ContactSymplecticBasics}

Let $M$ be an oriented, smooth manifold of $\dim=2n+1$. A \emph{contact form} on $M$ is a $1$-form $\alpha \in \Omega^{1}(M)$ for which
\begin{equation}\label{Eq:ContactCondition}
\alpha \wedge (d\alpha)^{n} > 0
\end{equation}
with respect to the orientation on $M$. We say that the pair $(M, \alpha)$ is a \emph{strict contact manifold}. A \emph{contact manifold} is a pair $\Mxi$ for which $\xi = \ker \alpha$ where $\alpha$ is a contact form.

A \emph{symplectic manifold} is a pair $(X, \omega)$ consisting of a smooth manifold $X$, and a $2$-form $\omega \in \Omega^{2}(X)$ for which $\omega$ is closed and non-degenerate as a fiber-wise skew-symmetric bilinear form. Provided a function $H \in \COne(X)$, the Hamiltonian vector field $X_{H}$ is determined by the equation
\begin{equation*}
dH = \omega(\ast, X_{H}).
\end{equation*}
An \emph{exact symplectic manifold} is a pair $(X, \lambda)$ for which $\lambda \in \Omega^{1}(X)$ is such that $(X, d\lambda)$ is a symplectic manifold. The \emph{Liouville vector field} of an exact symplectic manifold is the vector field $X_{\lambda}$ determined by the equation
\begin{equation*}
d\lambda(X_{\lambda}, \ast) = \lambda.
\end{equation*}
Thus the Liouville vector field and Hamiltonian vector field associated to a function, $H$, are related by
\begin{equation*}
\lambda(X_{H}) = dH(X_{\lambda}).
\end{equation*}

Let $(X, \lambda)$ be an exact symplectic manifold for which $X_{\lambda}$ is transverse to $\partial X$. Then we may write
\begin{equation*}
\partial X = \partial_{-}X \cup \partial_{+}X
\end{equation*}
where the Liouville vector field points into $X$ along $\partial_{-}X$ and out of $X$ along $\partial_{+}X$. Then the
\begin{equation*}
(\partial_{\pm}X, \alpha_{\pm}), \quad \alpha_{\pm} = \lambda|_{T\partial_{\pm}X}
\end{equation*}
are strict contact manifolds. Suppose that there is some $\epsilon \in \R_{>0}$ for which the time $s \in [0, \epsilon)$ flow of $\mp X_{\lambda}$ is defined for all trajectories starting in $\partial_{\pm}X$, ensuring that the $\partial_{\pm}X$ have disjoint collar neighborhoods in $X$. Then we say that $(X, \lambda)$ is a \emph{Liouville cobordism between the strict contact manifolds $(\partial_{\pm}X, \alpha_{\pm})$} and write
\begin{equation*}
(X, \lambda): (\partial_{+}X, \alpha_{+}) \to (\partial_{-}X, \alpha_{-}).
\end{equation*}
In this setting the $\partial_{\pm}X$ have collar neighborhoods
\begin{equation*}
([0, \epsilon) \times \partial_{-}X, e^{s}\alpha_{-}), ((-\epsilon,0] \times \partial_{+}X, e^{s}\alpha_{+}) \subset (X, \lambda)
\end{equation*}
for which the $\{ 0\} \times \partial_{\pm}X$ give the $\partial_{\pm}X$. Note that along the collars,
\begin{equation*}
X_{\lambda} = \partial_{s}.
\end{equation*}
The \emph{completion of $(X, \lambda)$}, denoted by $(\complete{X}, \lambda)$, is then defined by extending $s$-coordinates of the collars and extending the definition of $\lambda$ over the cylindrical ends in the obvious fashion
\begin{equation*}
\complete{X} = X \cup (\R_{s < \epsilon} \times \partial_{-}X) \cup (\R_{s>-\epsilon} \times \partial_{+}X), \quad \lambda|_{\R_{\pm s > -\epsilon} \times \partial_{\pm}X} = e^{s}\alpha_{\pm}.
\end{equation*}
A \emph{Liouville domain} is an exact symplectic manifold $(X, \lambda)$ for which $X$ is compact and $\partial_{-}X=\emptyset$.

\subsection{Convex hypersurfaces}

Let $\hypersurface$ be a compact, oriented, $\codim=1$ submanifold of a contact manifold $\Mxi$. We say that $\hypersurface$ is \emph{convex} if there is a vector field, $Z$, defined in a neighborhood of $\hypersurface$ in $M$ which is transverse to $\hypersurface$ and which preserves $\xi$. Given such a $\hypersurface$, we can identify a neighborhood of $\hypersurface$ of the form
\begin{equation*}
(-\epsilon, \epsilon)_{\tau} \times \hypersurface, \quad Z=\partial_{\tau}.
\end{equation*}
By taking a given contact form $\alpha$ for $\xi$ along this neighborhood and normalizing it to be $\tau$-invariant, we can assume that along this neighborhood $\xi$ takes the form
\begin{equation*}
\xi = \ker \alpha, \quad \alpha = f d\tau + \beta, \quad f \in \Cinfty(\hypersurface), \quad \beta \in \Omega^{1}(\hypersurface).
\end{equation*}

The \emph{dividing set} of $\hypersurface$ is the subset $\divSet \subset \hypersurface$ along which $Z \in \xi$. In the above neighborhood model,
\begin{equation*}
\divSet = f^{-1}(0).
\end{equation*}
As a consequence of the contact condition of Equation \eqref{Eq:ContactCondition}, $\divSet$ is transversely cut out by the function $f$ and is a $\codim=2$ contact submanifold of $\Mxi$ with induced contact structure
\begin{equation*}
\xiDivSet = T\divSet \cap \xi = \ker \alphaDivSet, \quad \alphaDivSet = \beta|_{T\divSet}.
\end{equation*}
It can likewise be shown that the
\begin{equation*}
(\Liouville^{\pm}, \beta_{\pm}), \quad \Liouville^{\pm} = f^{-1}(\R_{\pm t \geq \epsilon}), \quad \beta^{\pm} = \pm \beta/f
\end{equation*}
are Liouville domains for sufficiently small $\epsilon>$, whose contact boundaries can be identified with $\MxiDivSet$. Their completions can be identified with the
\begin{equation*}
(\complete{\Liouville}^{\pm}, \beta^{\pm}), \quad \complete{\Liouville}^{\pm} = f^{-1}(\R_{\pm t >0}), \quad \beta^{\pm} = \pm \beta/f.
\end{equation*}
See also \cite{Giroux:IdealLiouville} for a slightly different point of view. We say that $\complete{\Liouville}^{+} \subset \hypersurface$ is the \emph{positive region} and that $\complete{\Liouville}^{-} \subset \hypersurface$ is the \emph{negative region}. Alternatively, the positive (resp. negative) region is the subset of $\hypersurface$ along which the $\xi$-preserving vector field $Z$ is positively (resp. negatively) transverse to $\hypersurface$.

Note that the orientation of $\complete{\Liouville}^{+}$ (resp. $\complete{\Liouville}^{-}$) determined by $d\beta^{+}$ (resp. $d\beta^{-}$) agrees (resp. disagrees) with the orientation inherited from the orientation on $\hypersurface$. Therefore if we reverse the orientation on $\hypersurface$ and swap $Z$ for $-Z$, the positive and negative regions will be interchanged by definition.

The following fundamental theorem of Giroux \cite{Giroux:Criterion} relates the tightness of neighborhoods of convex hypersurfaces to the topology of their decompositions into positive and negative regions.

\begin{thm}[Giroux's criterion]\label{Thm:GirouxCriterion}
Suppose that $\hypersurface$ is connected and of $\dim=2$.
\be
\item If $\hypersurface \simeq \sphere^{2}$, then a neighborhood of $\hypersurface$ is tight iff its dividing set $\divSet$ is connected.
\item Otherwise, a neighborhood of $\hypersurface$ is tight iff none of the connected components of the $\Liouville^{\pm}$ are simply connected.
\ee
\end{thm}

According to \cite{Giroux:Convexity} any closed, embedded, oriented surface in a $\dim=3$ contact manifold may be $\Cinfty$ approximated by a convex surface. In \cite{Breen:Convex, HH:Convex} it is shown that any closed, oriented, embedded hypersurface in contact manifold of any dimension may be $\CZero$ approximated by a convex hypersurface. See also \cite{EliashPanch:Convex}. By \cite{Chaidez:Convex} (see also \cite{ChaidezHuang}), this last result may not be upgraded from $\CZero$ to $\mathcal{C}^{2}$ (let alone $\Cinfty$).

\subsection{Reeb dynamics}\label{Sec:ReebBasics}

Let $(M, \alpha)$ be a strict contact manifold. The \emph{Reeb vector field}, $\Reeb$, is the vector field on $M$ determined by the equations
\begin{equation*}
\alpha(\Reeb) = 1, \quad d\alpha(\Reeb, \ast) = 0.
\end{equation*}
For each closed Reeb orbit $\orbit$ we pick a point $m_{\orbit} \in \orbit$ which we call a \emph{marker}. Such a marker uniquely determines a parameterization, $\orbit(t)$, of $\orbit$ defined by
\begin{equation*}
\orbit(t): \aCircle \rightarrow M, \quad \partial_{t}\orbit(t) = \Reeb|_{\orbit(t)}, \quad \orbit(0) = m_{\orbit}, \quad a \in \R_{>0}.
\end{equation*}
The \emph{action} of such a closed orbit $\orbit$ of $\Reeb$ is defined to be its period,
\begin{equation*}
\action_{\alpha}(\orbit) = a = \int_{\orbit}\alpha.
\end{equation*}

For each $\orbit$ with $\action_{\alpha}(\orbit) = a$, the time $a$ linearized flow $T\Flow^{a}_{\Reeb}$ sends $\xi|_{m_{\orbit}}$ to itself, preserving $d\alpha$, and therefore determining a linear symplectic transformation
\begin{equation*}
\Ret_{\orbit} = T\Flow^{a}_{\Reeb}|_{\xi_{m_{\orbit}} } \in \Symp\left(\xi|_{m_{\orbit}}, d\alpha\right)
\end{equation*}
which we call the \emph{Poincaré return map}. We say that $\orbit$ is \emph{non-degenerate} if $1$ is not an eigenvalue of the Poincaré return map. For a constant $\actionBound \in (0, \infty]$, we say that a contact form $\alpha$ is \emph{$\actionBound$-non-degenerate} if all Reeb orbits of action less than $\actionBound$ are non-degenerate. When $\actionBound=\infty$, we simply say that $\alpha$ is \emph{non-degenerate}.

A \emph{framing}, $\framing$, of a closed Reeb orbit, $\orbit$, is a trivialization of $\orbit^{\ast}\xi$, identifying it with the trivial bundle so that $d\alpha$ agrees with $\omega_{0}$. Markers and framings for multiply-covered orbits will be assumed to be determined  by markers and framings on their underlying embedded orbits. A framing allows us to write the linearized flow $T\Flow^{t}_{\Reeb}$ restricted to $\xi|_{m_{\orbit}}$ for $t \in [0, a]$ as a path of $2n \times 2n$ symplectic matrices,
\begin{equation*}
A_{\framing} = A_{\framing}(t): [0, a]\rightarrow \Symp_{n},
\end{equation*}
for which $\Ret = A_{\framing}(a)$. Assuming $\orbit$ is non-degenerate, we write 
\begin{equation*}
\CZ_{\framing}(\orbit) = \CZ(A_{\framing})
\end{equation*}
for the Conley-Zehnder index of the path $A_{\framing}$. For an in-depth treatment of Conley-Zehnder indices of paths of symplectic matrices see \cite{Gutt:CZ}. For the purposes of this article we'll only need to have the following facts on hand:

\begin{lemma}\label{Lemma:CZSummary}
The Conley-Zehnder index satisfies the following properties:
\be
\item (Conjugacy invariance) Provided a pair of paths $A, B: [0, a] \to \Symp_{n}$,
\begin{equation*}
\CZ(BAB^{-1}) = \CZ(A).
\end{equation*}
\item (Loop) If $B: \aCircle_{t} \to \Symp_{n}$ is a loop and $A(t): [0,a] \to \Symp_{n}$ is a path, then
\begin{equation*}
\CZ(BA) = 2\Maslov(B) + \CZ(A)
\end{equation*}
where $\Maslov(B) \in \Z$ is the Maslov index.
\item (Direct sum) For paths $A$ in $\Symp_{n_{0}}$ and $B$ in $\Symp_{n_{1}}$, the path $A \oplus B$ in $\Symp_{n_{0} + n_{1}}$ has
\begin{equation*}
\CZ(A \oplus B) = \CZ(A) + \CZ(B).
\end{equation*}
\item ($\dim=2$ elliptic model) For $\epsilon \in \R \setminus \frac{1}{2\pi}\Z$,
\begin{equation*}
A(t) = e^{J_{0}t\epsilon/a}: [0, a] \rightarrow \Symp_{2} \implies \CZ(A) = 2\left\lfloor \frac{\epsilon}{2\pi}\right\rfloor + 1
\end{equation*}
where $\left\lfloor \ast \right\rfloor: \R \to \Z$ is the floor function.
\item ($\dim=2$ hyperbolic model) For $b \in \R_{\neq 0}$,
\begin{equation*}
A(t) = \Diag(e^{bt}, e^{-bt}): [0, a] \rightarrow \Symp_{2} \implies \CZ(A) = 0.
\end{equation*}
\ee
\end{lemma}

From the loop property above, it follows that the $\bmod 2$ Conley-Zehnder index of a Reeb orbit,
\begin{equation*}
\CZ_{2}(\orbit) = \CZ_{\framing}(\orbit) \in \Z/2\Z,
\end{equation*}
is independent of the choice of framing, $\framing$.

\begin{ex}
Let $M = \Circle \times \disk(\rho)$ where $\disk(\rho) \subset \C$ is a disk of radius $\rho$. Consider $1$-forms
\begin{equation*}
\alpha_{\pm} = f_{\pm} dt + \half (xdy - ydx), \quad f_{\pm} = 1 - \frac{\epsilon}{2}(x^{2}  \pm y^{2})
\end{equation*}
for $\epsilon \in \R_{\neq 0}$ of small absolute value. The $\alpha_{\pm}$ are contact for $\epsilon, \rho$ sufficiently small. Moreover $\{ x=y = 0\}$ is a Reeb orbit, $\orbit$, with framing $\framing$ given by $\partial_{x}, \partial_{y}$. With respect to this framing the path of symplectic matrices $A(t)$ obtained is exactly as described in
\be
\item the $\dim=2$ elliptic model of Lemma \ref{Lemma:CZSummary} for $\alpha_{+}$ whence $\CZ_{\framing}(\orbit) = \sgn(\epsilon)$ and
\item the $\dim=2$ hyperbolic model of Lemma \ref{Lemma:CZSummary} for $\alpha_{-}$ whence $\CZ_{\framing}(\orbit) = 0$.
\ee
This example will be generalized in \S \ref{Sec:SimpleModels}, Lemma \ref{Lemma:DynamicsOnDivSetC}, and in Assumptions \ref{Assump:MorseLocalModel} below.
\end{ex}

The following definition is justified by the orientation scheme for SFT \cite{BM:Orientations, EGH:SFTIntro}.

\begin{defn}
A closed Reeb orbit $\orbit$ is \emph{bad} if it is a multiple cover of an embedded Reeb orbit $\orbit'$ so that
\begin{equation*}
\CZ_{2}(\orbit) - \CZ_{2}(\orbit') = 1 \in \Z/2\Z.
\end{equation*}
Such an orbit $\orbit$ is \emph{good} if it is not bad.
\end{defn}

\subsection{Almost complex structures}\label{Sec:TameJOverview}

An almost complex structure $J$ on a symplectic manifold $(X, \omega)$ is \emph{$\omega$-tame} if $\omega(v, Jv) > 0$ for all non-zero $v \in TX$ and is \emph{$\omega$-compatible} if
\begin{equation*}
g_{\omega, J} = \omega(\ast, J\ast)
\end{equation*}
is a $J$-invariant Riemannian metric. It is easily computed that
\begin{equation*}
X_{H} = J \grad H
\end{equation*}
for an $\omega$-compatible $J$ and $H \in \Cinfty(W)$, where the gradient $\grad H$ is computed with respect to $g_{\omega, J}$.

\begin{defn}\label{Def:TameJ}
Let $\alpha$ be a contact form for some $\Mxi$ with Reeb field $\Reeb$ and let $I \subset \R$ be a $\codim=0$ submanifold. An almost complex structure $J$ on an $I$-symplectization $I_{s} \times M$ is \emph{$\alpha$-tame} if
\be
\item $J$ is invariant under translation in the $s$-coordinate,
\item $J(\partial_{s}) = F \Reeb$ for some $F \in \Cinfty(M, (0, \infty))$, and
\item $d\alpha(v, Jv)>0$ for all non-zero $v$ in the subspace of complex tangencies,
\begin{equation*}
v\in \xi_{J}  = TM \cap JTM \subset TM \subset T(I \times M).
\end{equation*}
\ee
Using the notation of \S \ref{Sec:ContactSymplecticBasics}, an almost complex structure $J$ on a completed Liouville cobordism $(\complete{X}, \lambda)$ is \emph{$\lambda$-tame} if it is $d\lambda$-compatible and there exists $C > 0$ such that the restriction of $J$ to the subsets $\R_{\pm s \geq C} \times \partial_{\pm}X$ of the extended collars of $\partial_{\pm}X$ are $\alpha_{\pm}$-tame.
\end{defn}

\begin{rmk}
Item (3) of Definition \ref{Def:TameJ} is equivalent to $d\alpha(v, Jv)>0$ for all non-zero $v \in \xi$. This follows from the observation that the fiber-wise linear projection $\xi_{J} \to \xi$ defined $v \mapsto v - \alpha(v)\Reeb$ is a bundle isomorphism preserving the fiber-wise quadratic form $d\alpha(\ast, J\ast)$.
\end{rmk}

\begin{rmk}
In the ``classical'' setup of symplectic field theory, one assumes $F=1$ and $\xi_{J}=\xi$ on a symplectization. The benefit of working with the more general $\alpha$-tame $J$ is that these almost complex structures sometimes facilitate explicit description of certain holomorphic curves and holomorphic submanifolds. We will encounter various examples throughout this article.
\end{rmk}

According to \cite[\S 3.4]{BH:Cylindrical}, an $\alpha$-tame $J$ satisfies the $\SFT$-compactness results of \cite{SFTCompactness} and may be used to compute contact homology as defined in \cite{BH:ContactDefinition}. Such $J$ can also be applied to the definition of contact homology in \cite{Pardon:Contact} (which uses a ``classical'' SFT setup), without modification.

\subsection{Simple geometric models}\label{Sec:SimpleModels}

We describe simple models for contact forms and Reeb dynamics near closed Reeb orbits following \cite{BH:ContactDefinition}. We employ the notation ($\orbit$, $a = \action_{\alpha}(\orbit)$, $\Ret$) laid out in \S \ref{Sec:ReebBasics}.

\begin{model}\label{Model:SimpleNbhd}
Let $n_{0}, n_{-} \in \Z_{\geq 0}$ and choose matrices $X_{-}\in \lieSymp_{n_{-}}$, $X_{0} \in \lieSymp_{n_{0}}$. These matrices determine quadratic forms
\begin{equation*}
Q_{0}(z_{0}) = \half\omega_{0}(X_{0}z_{0}, z_{0}), \quad Q_{-}(z_{-}) = \half\omega_{0}(X_{-}z_{-}, z_{-})
\end{equation*}
on $\R^{2n_{0}}_{z_{0}}$ and $\R^{2n_{-}}_{z_{-}}$, respectively. This data determines a $1$-form
\begin{equation*}
\alpha = (1 + Q_{0} + Q_{-})dt + \beta_{0} \in \Omega^{1}(\R_{t} \times \R^{2n_{0}}_{z_{0}} \times \R^{2n_{-}}_{z_{-}})
\end{equation*}
which is contact on a neighborhood
\begin{equation*}
\widetilde{\Norbit} = \{ \norm{z_{0}}^{2} +\norm{z_{-}}^{2} \leq \rho\}
\end{equation*}
of the line $\{ z_{0} = z_{-} = 0\}$ for $\rho$ sufficiently small. Along $\widetilde{\Norbit}$, the Reeb vector field is given by
\begin{equation*}
\Reeb|_{(t, z_{-}, z_{0})} = \partial_{t} + X_{0}z_{0} + X_{-}z_{-}.
\end{equation*}
Then $\alpha$ descends to the quotient
\begin{equation*}
\Norbit = \widetilde{\Norbit}/ \sim, \quad (t + a, z_{0}, z_{-}) \sim (t, z_{0}, -z_{-}).
\end{equation*}
so that
\begin{equation*}
\orbit = \{ z_{-} = z_{0} = 0 \} \subset \Norbit
\end{equation*}
is a closed, embedded Reeb orbit of action $a$. Using the standard symplectic bases for $\xi|_{(0, 0, 0)}$, the linear return map for $\orbit$ is then expressed as a block matrix
\begin{equation*}
\Ret = \begin{pmatrix}
e^{aX_{0}} & 0 \\ 0 & -e^{aX_{-}}
\end{pmatrix} \in \Symp_{n_{0} + n_{-}}.
\end{equation*}
\end{model}

\begin{assump}
When using the above model, we always assume that $Q_{-}$ is positive definite. The constant $\rho$ governing the size of the neighborhood $\Norbit$ will be assumed to be arbitrarily small and will be suppressed from our notation. Moving forward, we'll write $\Reeb = \partial_{t} + X_{0} + X_{-}$ with the understanding that $X_{0}$ (resp. $X_{-}$) only depends on $z_{0}$ (resp. $z_{-}$).
\end{assump}

The stated properties of the model follow from direct computations, the details of which we leave to the reader. It will sometimes be convenient to view $\Norbit$ as a disk bundle over the circle (of radius $a$),
\begin{equation}\label{Eq:NorbitFibration}
\disk_{\rho}^{2n} \rightarrow \Norbit \xrightarrow{\pi} \aCircle
\end{equation}
When $n_{-} > 0$, there are two obvious choices of trivializations of this bundle determining framings $\framing_{\pm}$ of the orbit determined by the paths
\begin{equation*}
\Id_{\R^{2n_{0}}} \oplus e^{\pm J_{0}\frac{\pi}{a}t}: \R^{2n_{0}} \oplus \R^{2n_{-}} \rightarrow \pi^{-1}(t) \subset \Norbit.
\end{equation*}
Assuming non-degeneracy of $\orbit$ it is not difficult to check that
\begin{equation*}
\CZ_{\framing^{\pm}}(\orbit) =  \CZ\left(t\mapsto e^{atX_{0}}\right) \mp n_{-}.
\end{equation*}

The following lemma is a restatement of \cite[Lemma 3.1.2]{BH:ContactDefinition} which is proved using Darboux-Moser-Weinstein techniques (cf. \cite[\S 3.2]{MS:SymplecticIntro}).

\begin{lemma}\label{Lemma:ApproxSimpleNbhd}
Let $\orbit$ be a closed, embedded Reeb orbit of action $a$ associated to a contact form $\alpha$ on some manifold $M$ of $\dim = 2n+1$. Then there exists $\rho > 0, n_{0}, n_{-}, X_{0}, X_{-}$ and a local embedding
\begin{equation*}
\phi: \Norbit \rightarrow M
\end{equation*}
for $\Norbit$ as in the model with $n = n_{0} + n_{-}$ and such that using the coordinates on $t, z_{0}, z_{-}$ on $\Norbit$,
\begin{equation*}
\begin{gathered}
\alpha = (1 + Q_{0} + Q_{-})dt + \beta_{0} + \alpha_{\hot}, \quad \alpha_{\hot} \in \bigO\left(\norm{z_{0}}^{3} + \norm{z_{-}}^{3}\right)\\
\Reeb = \partial_{t} + X_{0} + X_{-} + \Reeb_{\hot}, \quad R_{\hot} \in \bigO\left(\norm{z_{0}}^{2} + \norm{z_{-}}^{2}\right)
\end{gathered}
\end{equation*}
Here the $Q_{0}, Q_{-}$ are determined by the $X_{0}, X_{-}$ as in Model \ref{Model:SimpleNbhd}.
\end{lemma}

\begin{defn}\label{Def:SimpleEnough}
We say that a parameterized neighborhood $\Norbit$ about a closed, embedded Reeb orbit $\orbit$ as described above is \emph{simple enough} if in that neighborhood, the Reeb vector field may be expressed
\begin{equation*}
\Reeb = H(\partial_{t} + X_{0} + X_{-})
\end{equation*}
for some $H \in \Cinfty(\Norbit)$ with $H=1$ along $\orbit$. Such a neighborhood is \emph{simple} if $H=1$ along all of $\Norbit$. For $\actionBound \in \R_{\geq 0}$, we say that a contact form $\alpha$ is \emph{$\actionBound$-simple enough} (resp. $\actionBound$-simple) if it is $\actionBound$-non degenerate and each embedded orbit of action $\leq \actionBound$ has a simple enough (resp. simple) neighborhood.
\end{defn}

In practice $\actionBound$-simple contact forms can be constructed by choosing neighborhoods $\Norbit$ of embedded orbits associated to a given contact form $\alpha$ as in Lemma \ref{Lemma:ApproxSimpleNbhd} and eliminating the $\alpha_{\hot}$ terms through $1$-parameter deformations. Such deformations retain the isotopy class of contact structure (via application of the Darboux-Moser-Weinstein argument).

\begin{defn}\label{Def:SimpleJ}
Let $\orbit$ be a closed orbit of the Reeb vector field $\Reeb$ associated to a contact form $\alpha$. An almost complex structure $J$ on $\R_{\pm s \geq C} \times M$ is \emph{simple near $\orbit$} if $\orbit$ has a simple enough neighborhood along which 
\begin{equation*}
J \partial_{s} = \partial_{t} + X_{0} + X_{-}, \quad J|_{\pi^{-1}(t)} = J_{0}.
\end{equation*}
For $\actionBound \in \R_{>0}$ we say that $J$ is \emph{$\actionBound$-simple} if it is simple near all embedded Reeb orbits of action $\leq \actionBound$. An almost complex structure $J$ on a completed Liouville cobordism $(\complete{X}, \lambda)$ is \emph{$\actionBound$-simple} if it is $d\lambda$-compatible and there is some $C \geq 0$ so that the restriction of $J$ to the subsets $\R_{\pm s > C} \times \partial_{\pm}X$ of the collars of $\complete{X}$ are $\actionBound$-simple.
\end{defn}

In this article, we will encounter simple enough neighborhoods of Reeb orbits determined by contact forms of the form
\begin{equation}\label{Eq:SimpleEnoughModel}
\alpha = (1 + Q)\alphaDivSet + \beta \in \Omega^{1}(\disk^{2}_{x, y} \times \NorbitDivSet), \quad \alphaDivSet = (1 + Q_{0} + Q_{-})dt + \beta_{0} \in \NorbitDivSet.
\end{equation}
Here $\beta$ is some Liouville form on $\disk^{2}_{x,y}$ and the contact form on $\NorbitDivSet \subset \R_{t} \times \R^{2(n_{0} + n_{-})}$ is as described in Model \ref{Model:SimpleNbhd} with $n = 1 + n_{0} + n_{-}$. While the contact forms in the above equation are not exactly of the form described in Model \ref{Model:SimpleNbhd}, their Reeb vector fields have the desired form, and so these Reeb vector fields are simple enough along $\Norbit = \disk^{2} \times \NorbitDivSet$. Such neighborhoods can be seen to be simple when $\beta = \half(xdy-ydx)$.

Our $\actionBound$-simple condition on almost complex structures is a slight generalization of the $L$-simple condition of \cite{BH:ContactDefinition} designed to be compatible with the contact forms of Equation \eqref{Eq:SimpleEnoughModel}. Because the $\actionBound$-simple $J$ of this article and the $L$-simple $J$ of \cite{BH:ContactDefinition} take exactly the same form, the analysis of asymptotics of holomorphic curves as they converge to Reeb orbits and construction of perturbations from \cite{BH:ContactDefinition} are applicable to $\actionBound$-simple almost complex structures.

\section{Reeb dynamics and holomorphic foliations}\label{Sec:DynamicsAndFoliations}

Let $\MxiDivSet$ be a $\dim=2n-1$ contact manifold bounding a Liouville domain $(\Liouville, \betaLiouville)$ or a pair of Liouville domains $(\Liouville^{\pm}, \beta_{\Liouville^{\pm}})$. In this section we define contact forms on $\C \times \divSet$, $\R_{\tau} \times \Liouville$, and $\R_{\tau} \times \hypersurface$ where $\hypersurface$ is obtained by gluing the $\Liouville^{\pm}$ together along their boundaries. We also define Liouville forms and adapted almost complex structures on the $\R_{s} \times \C \times \divSet$, $\R_{s} \times \R_{\tau} \times \Liouville$, and $\R_{s} \times \R_{\tau} \times \hypersurface$. In each case our almost complex structures give rise to foliations by holomorphic submanifolds of $\codim=2$. Specific choices of these geometric structures which allow us to identify chain-level contact homology algebras with modules over the chain-level contact homology of $\MxiDivSet$ will be described in \S \ref{Sec:GeometryExamples}.

\subsection{Geometry of $\C \times \divSet$ and $\R \times \C \times \divSet$}\label{Sec:GeometryCDivSet}

Here we describe geometries on products of contact manifolds with the complex line and some associated symplectic cobordisms.

\subsubsection{Contact forms on $\C \times \divSet$}\label{Sec:ContactOnDivSetC}

Suppose that $\alphaDivSet$ is a contact form on a closed contact manifold $\MxiDivSet$ of $\dim \divSet = 2n-1$.  Write $\ReebDivSet$ for the Reeb vector field on $\divSet$ associated to $\alphaDivSet$. Let us fix
\begin{equation*}
\beta \in \Omega^{1}(\C), \quad f \in \Cinfty(\C, \R_{>0})
\end{equation*}
such that $d\beta$ is symplectic.

\begin{lemma}\label{Lemma:ContactOnDivSetC}
The $1$-form
\begin{equation*}
\alphaCDivSet_{f}= \beta + f\alphaDivSet \in \Omega^{1}(\C \times \divSet)
\end{equation*}
is contact iff
\begin{equation*}
f - \beta(X_{f}) > 0
\end{equation*}
where $X_{f}$ is the Hamiltonian vector field of $f$ with respect to $d\beta$ on $\C$. Assuming this holds, the Reeb vector field of $\alphaCDivSet_{f}$ is
\begin{equation*}
\ReebCDivSet_{f} = (f - \beta(X_{f}))^{-1}\left( \ReebDivSet - X_{f}\right).
\end{equation*}
\end{lemma}

\begin{proof}
Writing $\alpha = \alphaCDivSet_{f}$, we compute
\begin{equation*}
\begin{gathered}
d\alpha= f d\alphaDivSet + df \wedge \alphaDivSet + d\beta,\\
(d\alpha)^{n} = n(fd\alphaDivSet)^{n-1}\wedge (df \wedge \alphaDivSet + d\beta),\\
\alpha \wedge (d\alpha)^{n} = nf^{n-1}\alphaDivSet \wedge (d\alphaDivSet)^{n-1} \wedge (fd\beta + \beta \wedge df).
\end{gathered}
\end{equation*}
Therefore we find that $\alphaCDivSet_{f}$ is a contact form iff $fd\beta + \beta \wedge df$ is symplectic on $\C$. 

At points in $\C \times \divSet$ where $df = 0$, we have
\begin{equation*}
fd\beta + \beta \wedge df = fd\beta
\end{equation*}
is point-wise symplectic and
\begin{equation*}
f - \beta(X_{f}) = f > 0.
\end{equation*}
As a consequence of the fact that
\begin{equation*}
JX_{f} = -\grad_{f}
\end{equation*}
(where the gradient is determined by the metric $d\beta(\ast, J\ast)$), we see that
\begin{equation}\label{Eq:HamiltonianNormTrick}
(fd\beta + \beta \wedge df)(X_{f}, JX_{f}) = \norm{df}^{2}(f - \beta(X_{f})).
\end{equation}
Applying this computation to points where $df \neq 0$, we see that
\begin{equation*}
fd\beta + \beta \wedge df \ \text{symplectic} \iff f - \beta(X_{f}) > 0.
\end{equation*}
Assuming this criterion is met, the Reeb vector field equation is easily checked.
\end{proof}

\begin{assump}\label{Assump:AlphaCDivSet}
We henceforth assume that $\alphaCDivSet_{f}$ is contact on $\C \times \divSet$ and that $f$ is Morse with $f \geq 1$ point-wise and $\min f = 1$. Writing $\disk(\delta) \subset \C$ for the unit disk of radius $\delta>0$, we further assume that there are $\epsilon, \delta \in \R_{>0}$ such that
\begin{equation*}
\beta = \half r^{2}d\theta, \quad f = 1 + \frac{\epsilon}{2}r^{2}\ \text{along}\ \C \setminus \disk(\delta)
\end{equation*}
(in polar coordinates). Thus on $(\C \setminus \disk(\delta)) \times \divSet$ we have
\begin{equation*}
f - \beta(X_{f}) = 1, \quad \ReebCDivSet_{f} = \ReebDivSet - \epsilon\partial_{\theta}.
\end{equation*}
\end{assump}

\begin{figure}[h]
\begin{overpic}[scale=.8]{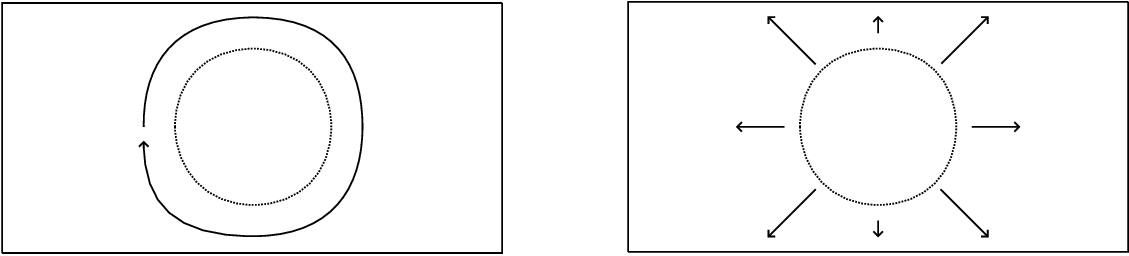}
\end{overpic}
\caption{The left subfigure shows the projection of $\ReebCDivSet_{f}$ to $\C$ outside of $\disk(\delta)$, which is bounded by a dashed circle. On the right, $\grad_{f}$ flow lines are shown in $\C \setminus \disk(\delta)$.}
\label{Fig:AlphaCDivSetAssumptions}
\end{figure} 

Figure \ref{Fig:AlphaCDivSetAssumptions} graphically describes the above assumptions. We will impose further assumptions on the function $f$ in \S \ref{Sec:fAssumptions}.

Let $\actionBoundDivSet \in \R_{>0}$ and assume that all closed $\ReebDivSet$ orbits $\orbitDivSet$ of action $\leq \actionBoundDivSet$ are non-degenerate. Fix a framing $\framing_{\orbitDivSet}$ on each such orbit in $\divSet$, with respect to which an integral Conley-Zehnder index may be defined
\begin{equation*}
\CZ(\orbitDivSet) = \CZ_{\framing_{\orbitDivSet}}(\orbitDivSet) \in \Z.
\end{equation*}
For a critical point
\begin{equation*}
z \in \Crit(f) \subset \C,
\end{equation*}
we write
\begin{equation*}
\ind_{\Morse}(f, z) \in \{ 0, 1, 2\}
\end{equation*}
for the Morse index. According to the above lemma, for each pair $(z, \orbitDivSet)$ we have a closed Reeb orbit
\begin{equation}\label{Eq:MorseOrbit}
\orbit_{\zeta} = \{\zeta\} \times \orbitDivSet.
\end{equation}
Each such $\orbit_{\zeta}$ is then framed by an extension of $\framing_{\orbitDivSet}$, defined
\begin{equation*}
\framing_{\orbit_{\zeta}} = \left(\framing_{\orbitDivSet}, \partial_{x} - f^{-1}\beta(\partial_{x})\ReebDivSet, \partial_{y} - f^{-1}\beta(\partial_{y})\ReebDivSet\right).
\end{equation*}

\begin{lemma}\label{Lemma:DynamicsOnDivSetC}
Let $\alphaCDivSet_{f}$ be a contact form satisfying Assumptions \ref{Assump:AlphaCDivSet} and let $\actionBoundDivSet \in \R_{>0}$. Then if $f$ is sufficiently $\mathcal{C}^{1}$-close to the constant function $1$ on $\disk$ and the constant $\epsilon$ of Assumptions \ref{Assump:AlphaCDivSet} is sufficiently small, all closed $\ReebCDivSet_{f}$ orbits of action at most
\begin{equation*}
\actionBound = \max(f|_{\disk(\delta)}, 1+\pi \epsilon) \actionBoundDivSet
\end{equation*}
are of the form $\orbit_{\zeta}$ as described in Equation \eqref{Eq:MorseOrbit}, associated to closed $\ReebDivSet$ orbits $\orbitDivSet$ of action $\leq \actionBoundDivSet$ and Morse critical points $\zeta \in \Crit(f)$. The actions and Conley-Zehnder indices of these orbits (defined with respect to the framings $\framing_{\orbit_{\zeta}}$) are then computed as
\begin{equation*}
\action(\orbit_{\zeta}) = f(\zeta)\action(\orbitDivSet), \quad \CZ(\orbit_{\zeta}) = \CZ(\orbitDivSet) + \ind_{\Morse}(f, \zeta) -1.
\end{equation*}
Therefore each such $\orbit_{\zeta}$ is good iff $\orbitDivSet$ is good and the $\Z$-valued contact homology gradings of these orbits (defined by the provided framings) are related by
\begin{equation*}
|\orbit_{\zeta}| = |\orbitDivSet| + \ind_{\Morse}(f, \zeta).
\end{equation*}
\end{lemma}

\begin{proof}
Given the expression of $\ReebCDivSet_{f}$ in Lemma \ref{Lemma:ContactOnDivSetC} and momentarily ignoring the $(f - \beta(X_{f}))^{-1}$ rescaling factor, we see that each closed orbit of $\ReebCDivSet_{f}$ may be parameterized by a function $t \mapsto (\orbitDivSet(t), z(t))$ where $\orbitDivSet(t)$ is a parameterized $\ReebDivSet$ orbit and $z(t)$ is a $-X_{f}$ orbit. In order that the orbit close up, the period of $z(t)$ must agree with that of $\orbitDivSet$. The action of such an orbit is then computed
\begin{equation*}
\action_{\alphaCDivSet_{f}}(\orbitDivSet(t), z(t)) = \int_{0}^{\action(\orbitDivSet)} (f(z(t)) - \beta(X_{f}(z(t))))dt.
\end{equation*}
To make $f$ closer to $1$ with respect to the $\COne$-norm, we may replace $f$ with
\begin{equation*}
f_{c} = 1 + c(f - 1)
\end{equation*}
for some $c \in \R_{>0}$. This transforms $f$ on the complement of $\disk(\delta) \subset \C$ as $1 + \frac{\epsilon}{2}r^{2} \mapsto 1 + \frac{c \epsilon}{2}r^{2}$ so that Assumptions \ref{Assump:AlphaCDivSet} continue to hold. We find that each $(\orbitDivSet(t), z(t))$ is still a periodic orbit, but its action is now
\begin{equation*}
\action_{\alphaCDivSet_{f_{c}}}(\orbitDivSet(t), z(t)) = (1-c)\action(\orbitDivSet) + c\action_{\alphaCDivSet_{f}}(\orbitDivSet(t), z(t)).
\end{equation*}
Since $f - \beta(X_{f}) > 0$, both summands in this formula are positive for $c$ close to $0$. Assuming that there are no closed $\ReebDivSet$ orbits of action between $\actionBoundDivSet$ and $(1+\epsilon')\actionBoundDivSet$ with $\epsilon', c > 0$ small, we find that if $(\orbitDivSet(t), z(t))$ has action at most $\actionBound = \max(f)\actionBoundDivSet$ and $z(t) \in \disk(\delta)$ for all $t$ then $\orbitDivSet$ must have action at most $\actionBoundDivSet$, implying that $z(t)$ must be constant. Orbits for which $z(t) \in \C \setminus \disk(\delta)$ are studied similarly. Therefore for $c$ small each closed $\ReebCDivSet_{f}$ orbit of action $\leq \actionBound$ must be one of the $\orbit_{\zeta}$.

At such an $\orbit_{\zeta}$ we find that the linearized flow preserves the $T\C$ summand of $T(\C \times \divSet)$. In the case $\ind_{\Morse}(f, \zeta) = 0$ the gradient $\grad_{f}$ points outward in $\C$, and hence $-X_{f} = -J\grad_{f}$ applies a negative rotation. By shrinking $\lambda$ we may assume that the total rotation angle along the $T\C$ summand is arbitrarily close to zero. Since there are only finitely many orbits under consideration, we may assume that this rotation angle is close to zero for all orbits of action $\leq \actionBound$. Using the additivity of $\CZ$ with respect to direct sums, we see that this rotation contributes $-1$ to the Conley-Zehnder index of $\orbit_{\zeta}$. Likewise in the case $\ind_{\Morse}(f, \zeta) = 1$, the normal direction contributes $0$ to the Conley-Zehnder index. When $\ind_{\Morse}(f, \zeta) = 2$, this contribution is $1$. This yields the desired formula.

If a closed Reeb orbit $\orbit_{\zeta}$ in $\C \times \divSet$ multiply covers some embedded Reeb orbit, then the embedded Reeb orbit must have the form $\orbit'_{\zeta}$ where $\orbitDivSet'$ is an embedded Reeb orbit in $\divSet$ multiply covered by $\orbitDivSet$. Provided our computation of Conley-Zehnder indices,
\begin{equation*}
\CZ(\orbit_{\zeta}) - \CZ(\orbit_{\zeta'}) = \CZ(\orbitDivSet) - \CZ(\orbitDivSet')
\end{equation*}
and the statement concerning good and bad Reeb orbits clearly follows.

As for contact homology gradings, we compute
\begin{equation*}
|\orbit_{\zeta}| = (n+1 - 3) + \CZ(\orbit_{\zeta}) = n - 3 + \CZ(\orbitDivSet) + \ind_{\Morse}(f, \zeta) = |\orbitDivSet| + \ind_{\Morse}(f, \zeta).
\end{equation*}
\end{proof}

\subsubsection{Liouville forms on $\R \times \C \times \divSet$}\label{Sec:LiouvilleRCDivSet}

Let $f \in \Cinfty(\R_{s} \times \C \times \divSet, \R_{>0})$ be such that
\begin{equation*}
\exists s_{0} > 0\ \text{such that}\ |s| > s_{0} \implies \partial_{s}f = 0.
\end{equation*}

\begin{lemma}\label{Lemma:LiouvilleRDivSetC}
Provided a function $f$ as above, the $1$-form
\begin{equation}\label{Eq:LambdaFDefinition}
\lambdaCDivSet_{f} = e^{s}(\beta + f\alphaDivSet) \in \Omega^{1}(\R_{s} \times \C \times \divSet)
\end{equation}
is symplectic iff
\begin{equation*}
f + \partial_{s}f - \beta(X_{f}) > 0.
\end{equation*}
Assuming $d\lambdaCDivSet_{f}$ is symplectic, its Liouville vector field is given by
\begin{equation*}
X_{\lambdaCDivSet_{f}} = \frac{(f - \beta(X_{f})\partial_{s}+(\partial_{s}f)X_{\beta}}{f + \partial_{s}f  - \beta(X_{f})}.
\end{equation*}
\end{lemma}

\begin{proof}
Writing $\lambda = \lambdaCDivSet_{f}$ we compute
\begin{equation*}
\begin{gathered}
d_{\C}f := \partial_{x}f dx + \partial_{y}f dy,\\
d\lambda = e^{s}(ds \wedge ((f + \partial_{s}f)\alphaDivSet + \beta) + d_{\divSet}(f\alphaDivSet) + d_{\C}f \wedge \alphaDivSet + d\beta),\\
(e^{-s}d\lambda)^{n+1} = n(n+1)f^{n-1}ds \wedge \alphaDivSet \wedge (d\alphaDivSet)^{n-1} \wedge ((f + \partial_{s}f)d\beta + \beta \wedge d_{\C}f).
\end{gathered}
\end{equation*}
Therefore $\lambdaCDivSet_{f}$ is symplectic on $\R_{s} \times \C \times \divSet$ iff $(f + \partial_{s}f)d\beta + \beta \wedge d_{\C}f$ is symplectic on $\C$. Following Equation \eqref{Eq:HamiltonianNormTrick} from the proof of Lemma \ref{Lemma:ContactOnDivSetC}, we find that
\begin{equation*}
(f + \partial_{s}f)d\beta + \beta \wedge d_{\C}f > 0 \iff f + \partial_{s}f - \beta(X_{f}) > 0.
\end{equation*}
The computation of the Liouville vector field is straightforward.
\end{proof}

Assume that the conclusion of the above lemma holds. If $f$ is independent of $\R_{s}$, then $(\R_{s} \times \C \times \divSet, \lambdaCDivSet_{f})$ is the symplectization of $(\C \times \divSet, \alphaCDivSet_{f})$. In general, we can view $(\R_{s} \times \C \times \divSet, \lambdaCDivSet_{f})$ as the completion of a Liouville cobordism between strict contact manifolds
\begin{equation*}
\begin{gathered}
([-s_{0}, s_{0}]_{s} \times \C \times \divSet, \lambdaCDivSet_{f}): (\C \times \divSet, e^{s_{0}}\alphaCDivSet_{f(s_{0}, \ast)}) \to (\C \times \divSet, e^{-s_{0}}\alphaCDivSet_{f(-s_{0}, \ast)}).
\end{gathered}
\end{equation*}

\subsubsection{Almost complex structures on $\R \times \C \times \divSet$}

Let $f$ be as in Lemma \ref{Lemma:LiouvilleRDivSetC}, let $F \in \Cinfty(\R_{s} \times \C \times \divSet, \R_{>0})$, and write $J_{0}$ for the standard complex structure on $\C$. Let
\begin{equation*}
\xiDivSet^{J} \subset T\divSet
\end{equation*}
be an oriented subbundle which is transverse to $\ReebDivSet$. Let $J_{\divSet}$ be a complex structure on $\xiDivSet^{J}$ such that
\begin{equation*}
d\alphaDivSet(v, J_{\divSet}v) > 0\ \forall v \in \xiDivSet^{J}, \ v \neq 0.
\end{equation*}
This data uniquely determines a $1$-form $\alphaDivSet^{J}$ on $\divSet$ such that
\begin{equation}\label{Eq:AlphaDivSetJDefinition}
\ker \alphaDivSet^{J} = \xiDivSet^{J}, \quad \alphaDivSet^{J}(\ReebDivSet) = 1.
\end{equation}
Provided this data, define an almost complex structure
\begin{equation}\label{Eq:JFDefinition}
\begin{gathered}
\JCDivSet_{f} \in \Aut(T(\R_{s} \times \C \times \divSet)),\\
\JCDivSet_{f}\partial_{s} = F(\ReebDivSet - X_{f}), \quad \JCDivSet_{f}|_{T\C} = J_{0}, \quad \JCDivSet_{f}|_{\xiDivSet^{J}} = J_{\divSet},\\
\implies \JCDivSet_{f}\ReebDivSet = -F^{-1}\partial_{s} + J_{0}X_{f}.
\end{gathered}
\end{equation}

\begin{lemma}\label{Lemma:JCompatibleOnRDivSetC}
Assume that $\lambdaCDivSet_{f}$ satisfies the conclusions of Lemma \ref{Lemma:LiouvilleRDivSetC}. Then
\begin{equation*}
d\lambdaCDivSet_{f}(v, \JCDivSet_{f}v) > 0
\end{equation*}
for all non-zero $v \in T(\R_{s} \times \C \times \divSet)$ iff
\begin{equation*}
4F(f + \partial_{s}f - \beta(X_{f})) > \norm{X_{\beta}}^{2}.
\end{equation*}
\end{lemma}

\begin{proof}
Throughout the proof, we simply write $J = \JCDivSet_{f}$. Suppose that we have a tangent vector of the form
\begin{equation*}
\begin{gathered}
v = a\partial_{s} + bJ\partial_{s} + X + Y, \quad X \in T\C, \quad Y \in \xiDivSet^{J},\\
\implies Jv = -b\partial_{s} + aJ\partial_{s} + J_{0}X + J_{\divSet}Y.
\end{gathered}
\end{equation*}
Applying the expression for $d\lambdaCDivSet_{f}$ appearing in the proof of Lemma \ref{Lemma:LiouvilleRDivSetC} we compute
\begin{equation*}
\begin{aligned}
e^{-s}d\lambdaCDivSet_{f}(v, Jv) &= ds(a\partial_{s})\cdot ((f + \partial_{s}f)\alphaDivSet + \beta)(aJ\partial_{s} + J_{0}X)\\
&- ((f + \partial_{s}f)\alphaDivSet + \beta)(bJ\partial_{s} + X)\cdot ds(-b\partial_{s})\\
&+ fd\alphaDivSet(Y, JY)\\
&+ df(bJ\partial_{s} + X)\cdot \alphaDivSet(aJ\partial_{s}) - \alphaDivSet(bJ\partial_{s})\cdot df(aJ\partial_{s} + J_{0}X)\\
&+ d\beta(bJ\partial_{s} + X, aJ\partial_{s} + J_{0}X)\\
=& a(aF(f + \partial_{s}f - \beta(X_{f})) +\beta(J_{0}X))\\
&+ b(bF(f + \partial_{s}f - \beta(X_{f})) + \beta(X))\\
&+ fd\alphaDivSet(Y, JY)\\
&+ df(X)aF - bF df(J_{0}X)\\
&+ d\beta(-bFX_{f} + X, -aFX_{f} + J_{0}X)\\
=& F(f + \partial_{s}f - \beta(X_{f}))(a^{2} + b^{2})\\
&+ a\beta(J_{0}X) + b\beta(X)\\
&+ fd\alphaDivSet(Y, JY) + d\beta(X, J_{0}X).
\end{aligned}
\end{equation*}

We view $e^{-s}d\lambdaCDivSet_{f}(v, Jv)$ as a quadratic function $Q(a, b)$ of $a, b \in \R$ with $X$ and $Y$ fixed. Since the non-linear terms of $Q$ are
\begin{equation*}
H(a^{2} + b^{2}), \quad H = F(f + \partial_{s}f - \beta(X_{f}))
\end{equation*}
with $H > 0$ by assumption, we see that $Q$ must have a unique global minimum. Since the coefficients of the $a^{2}$ and $b^{2}$ coefficients in $Q$ are positive, it follows that $Q$ is everywhere positive iff $Q > 0$ at its global minimum.

To search for this minimum of $Q$, we compute
\begin{equation*}
\begin{aligned}
\partial_{a}Q(a, b) &= 2a H + \beta(J_{0}X),\\
\partial_{b}Q(a, b) &= 2b H + \beta(X).
\end{aligned}
\end{equation*}
Therefore the unique critical point of $Q$ is located at
\begin{equation*}
\begin{gathered}
a_{0} = -\beta(J_{0}X)/(2H), \quad b_{0} = -\beta(X)/(2H).
\end{gathered}
\end{equation*}
At this critical point, $Q$ assumes the value
\begin{equation*}
Q(a_{0}, b_{0}) = -\frac{1}{4H}(\beta(J_{0}X)^{2} + \beta(X)^{2}) + fd\alphaDivSet(Y, JY) + d\beta(X, J_{0}X).
\end{equation*}
This quantity is clearly positive whenever $\beta = 0$. At points in $\C$ where $\beta$ is non-zero, the Liouville vector field $X_{\beta}$ for $\beta$ is non-zero and we can uniquely write
\begin{equation*}
X = e^{J_{0}\theta}cX_{\beta}
\end{equation*}
for some $c \in \R$ and $\theta \in [0, 2\pi)$ since then $X_{\beta}$ and $J_{0}X_{\beta}$ span $T\C$. Therefore
\begin{equation*}
\begin{aligned}
Q(a_{0}, b_{0}) =& -\frac{1}{4H}c^{2}\norm{X_{\beta}}^{4} + fd\alphaDivSet(Y, JY) + c^{2}\norm{X_{\beta}}^{2}\\
=& \frac{4H - \norm{X_{\beta}}^{2}}{4H}c^{2}\norm{X_{\beta}}^{2} + fd\alphaDivSet(Y, JY).
\end{aligned}
\end{equation*}
Since $H$ is positive point-wise, we conclude that
\begin{equation*}
d\lambdaCDivSet_{f}(v, Jv) > 0 \ \forall v \neq 0 \iff 4H > \norm{X_{\beta}}^{2}.
\end{equation*}
\end{proof}

\subsubsection{Foliations on $\R_{s} \times \C \times \divSet$}\label{Sec:FoliationCDivSet}

Now we describe a foliation of $\R_{s} \times \C \times \divSet$ whose leaves are $\JCDivSet_{f}$-holomorphic submanifolds. For each $s \in \R$, write
\begin{equation*}
\grad_{f(s, \ast)} = -\JCDivSet_{f}X_{f(s, \ast)} = -J_{0}X_{f(s, \ast)}
\end{equation*}
the gradient of $f(s, \ast) \in \Cinfty(\C)$ defined using the Riemannian metric $d\beta(\ast, J_{0}\ast)$. Consider solutions
\begin{equation*}
z(s): \R \to \C
\end{equation*}
to the non-autonomous ordinary differential equation
\begin{equation}\label{Eq:HoloEmbeddingODE}
\partial_{s}z(s) = F(s, z(s))\grad_{f(s, \ast)}|_{z(s)}.
\end{equation}
For each such $z(s)$ define
\begin{equation*}
\Phi: \R_{s} \times \divSet \to \R_{s} \times\C \times  \divSet, \quad \Phi(s, p) = (s, z(s), p).
\end{equation*}
Associated to this data, we also define an almost complex structure $J_{\Phi}$ on $\R_{s} \times \divSet$ by
\begin{equation}\label{Eq:JOnRCDivSet}
J_{\Phi}\partial_{s} = F\ReebDivSet, \quad F = F(s, z(s)), \quad J_{\Phi}|_{\xiDivSet^{J}} = J_{\divSet} \implies J_{\Phi}\ReebDivSet = -F^{-1}\partial_{s}.
\end{equation}

\begin{lemma}\label{Lemma:FoliationRCDivSet}
The maps $\Phi$ associated to solutions $z(s)$ of Equation \eqref{Eq:HoloEmbeddingODE} are $\codim=2$, $(J_{\Phi}, \JCDivSet_{f})$-holomorphic embeddings. The images of the $\Phi$ are the leaves of a foliation $\foliationCDivSet_{f}$ tangent to the distribution
\begin{equation*}
T\foliationCDivSet_{f} = \R(\partial_{s} - FJ_{0}X_{f}) \oplus T\divSet \subset T(\R_{s} \times \C \times \divSet).
\end{equation*}
For each such leaf $\leaf$, there is a map of the form
\begin{equation*}
\phi: \R_{t} \times \divSet \to \R_{s} \times \divSet, \quad \phi(t, x) \to (s(t, x), x)
\end{equation*}
so that complex structure on $\R_{t} \times \divSet$ giving by pulling back $J_{\Phi}$ via $\phi$ is $\R_{t}$-invariant. Thus $(\leaf, \JCDivSet_{f})$ is diffeomorphic to a symplectization.
\end{lemma}

\begin{proof}
Recalling the definition of $\alphaDivSet^{J}$ from Equation \eqref{Eq:AlphaDivSetJDefinition}, we compute
\begin{equation*}
\begin{gathered}
T\Phi = (1, F\grad_{f}, 0)\otimes ds + (0, 0, \ReebDivSet)\otimes \alphaDivSet^{J} + (0, 0, \Id_{\xiDivSet^{J}})\\
\JCDivSet_{f}T\Phi = (0, 0, F\ReebDivSet)\otimes ds + (-F^{-1}\partial_{s}, J_{0}X_{f}, 0)\otimes \alphaDivSet^{J} + (0, 0, J_{\divSet}\Id_{\xiDivSet^{J}})
\end{gathered}
\end{equation*}
We leave it to the reader to verify that
\begin{equation*}
\JCDivSet_{f}T\Phi = T\Phi J_{\Phi},
\end{equation*}
which informs us that each $\Phi$ is a holomorphic embedding.

To see that the images of the $\Phi$ foliate $\R_{s} \times \C \times \divSet$, we note that there is exactly one leaf, $\im \Phi$, passing through each point in the target. Indeed the leaf passing through some $(s_{0}, p, z)$ is the image of the $\Phi$ associated to the initial value problem $z = z(s_{0})$ associated to Equation \eqref{Eq:HoloEmbeddingODE}.

For the last statements of the lemma, we apply Equation \eqref{Eq:JOnRCDivSet}. The map $s(t, x)$ must solve
\begin{equation*}
\partial_{t}F(s(t, x), x) = 0, \quad t \leq s_{0} \implies s(t) = t.
\end{equation*}
This is a $1$-dimensional initial value problem ODE and can be uniquely solved by integration.
\end{proof}

\subsection{Geometry of $\R_{\tau} \times \Liouville_{\sigma_{0}}$ and $\R_{s} \times \R_{\tau} \times \Liouville_{\sigma_{0}}$}\label{Sec:GeometryRLiouville}

Let $(\Liouville, \betaLiouville)$ be a compact Liouville manifold of $\dim \Liouville = 2n$ whose strict contact boundary is $(\divSet, \alphaDivSet)$. That is,
\begin{equation*}
\partial \Liouville = \divSet, \quad \betaLiouville|_{\divSet} = \alphaDivSet.
\end{equation*}
The trace of the time $\sigma \in (-\infty, 0]$ flow of the Liouville vector field $X_{\betaLiouville}$ provides an embedding
\begin{equation*}
\Collar: (-\infty, 0] \times \divSet \to \Liouville, \quad (\sigma, p) \mapsto \Flow^{\sigma}_{X_{\betaLiouville}}p, \quad \Collar^{\ast}\betaLiouville = e^{\sigma}\alphaDivSet
\end{equation*}
of a half-infinite symplectization of $(\divSet, \alphaDivSet)$ into $\Liouville$. By an abuse of notation, we refer to both the map $\Collar$ and its image as the \emph{collar} of $(\Liouville, \betaLiouville)$. For $\sigma_{0} \in (-\half, 0)$, we write
\begin{equation*}
\Liouville_{\sigma_{0}} = \Liouville \setminus \Collar((\sigma_{0}, 0]\times \divSet).
\end{equation*}
for the subspace obtained by removing a portion of the collar. The \emph{completion} of $(\Liouville, \betaLiouville)$ is the Liouville manifold
\begin{equation*}
(\complete{\Liouville}, \complete{\beta}_{\Liouville}), \quad \complete{\Liouville} = \Liouville \cup [0, \infty)_{\sigma} \times \divSet, \quad \complete{\beta}_{\Liouville}|_{\Liouville} = \betaLiouville, \quad \complete{\beta}_{[0, \infty)_{\sigma} \times \divSet} = e^{\sigma}\alphaDivSet.
\end{equation*}

The \emph{contactization} of $(\Liouville_{\sigma_{0}}, \betaLiouville)$ is the strict contact manifold
\begin{equation*}
(\R_{\tau} \times \Liouville_{\sigma_{0}}, \alpha_{\Liouville}), \quad \alpha_{\Liouville} = 3d\tau + \betaLiouville.
\end{equation*}
The factor of $3$ is perhaps unusual but will be useful in our constructions of contact and Liouville forms. Along $\R_{\tau}$ times the collar, we have
\begin{equation*}
\alpha_{\Liouville} = 3d\tau + e^{\sigma}\alphaDivSet.
\end{equation*}

\subsubsection{Contact forms on $\R_{\tau} \times \Liouville_{\sigma_{0}}$}\label{Sec:AlphaContactization}

We define model contact forms on the $\R_{\tau} \times \Liouville_{\sigma_{0}}$ for $\sigma_{0} \in (-\half, 0)$ in two steps. First we modify $\alpha_{\Liouville}$ along $\R_{\tau} \times \Liouville_{-1}$ without creating any closed Reeb orbits. Then we extend the definition of modified contact forms to $\R_{\tau} \times \Liouville_{\sigma_{0}}$ so that the already-established results of \S \ref{Sec:ContactOnDivSetC} are applicable along $\{ \sigma \in [-1, \sigma_{0}]\}$.

\begin{lemma}\label{Lemma:AlphaContactization}
For functions
\begin{equation*}
f, g \in \Cinfty(\R_{\sigma \leq -1}, \R_{>0}),
\end{equation*}
define a $1$-form
\begin{equation*}
\alphaRLiouville_{f} = gd\tau + f\alphaDivSet \in \Omega^{1}(\R_{\tau} \times \R_{\sigma \leq -1} \times \divSet).
\end{equation*}
Then $\alphaRLiouville_{f}$ is contact and determines the same orientation as $\alpha_{\Liouville}$ iff
\begin{equation*}
g \partial_{\sigma}f - f\partial_{\sigma}g > 0
\end{equation*}
in which case the Reeb vector field, $\ReebRLiouville_{f}$ for $\alphaRLiouville_{f}$ is given by
\begin{equation*}
\ReebRLiouville_{f} = (g\partial_{\sigma}f -f\partial_{\sigma}g)^{-1}\left(-\partial_{\sigma}g \ReebDivSet + \partial_{\sigma}f \partial_{\tau}\right).
\end{equation*}
\end{lemma}

\begin{proof}
Writing $\alpha = \alphaRLiouville_{f}$ we compute
\begin{equation*}
\begin{gathered}
\alpha = gd\tau + f\alphaDivSet, \quad f,g \in \Cinfty(\R_{\sigma \leq -1}, \R_{>0}),\\
d\alpha = \partial_{\sigma}gd\sigma \wedge d\tau + \partial_{\sigma}f d\sigma\wedge \alphaDivSet + fd\alphaDivSet,\\
(d\alpha)^{n} = n (\partial_{\sigma}gd\sigma \wedge d\tau + \partial_{\sigma}f d\sigma\wedge \alphaDivSet)\wedge (fd\alphaDivSet)^{n-1},\\
\alpha \wedge (d\alpha)^{n} = n f^{n-1}\alphaDivSet \wedge (d\alphaDivSet)^{n-1} \wedge (g \partial_{\sigma}f - f\partial_{\sigma}g)d\tau \wedge d\sigma.
\end{gathered}
\end{equation*}
The contact condition is clear from the last equality. The calculation of $\ReebRLiouville_{f}$ is easily verified.
\end{proof}

\begin{assump}\label{Assump:AlphaRLiouville1}
We henceforth assume that $f, g \in \Cinfty(\R_{\sigma \leq -1}, \R_{>0})$ have the following properties:
\be
\item $g$ is a positive, non-increasing function with $\partial_{\sigma}g < 0$ for $\sigma > -3$.
\item $g = 3$ for $\sigma < -3$ and $g=-\sigma$ for $\sigma > -2$.
\item $f$ is positive and strictly increasing.
\item $f = e^{\sigma}$ for $\sigma < -2$ and $f$ is $\COne$-close to the constant $1$ near $\{ \sigma = -1\}$.
\ee
\end{assump}

With these assumptions in place, the contact condition of Lemma \ref{Lemma:AlphaContactization} is satisfied. We view $\R_{\tau} \times \R_{\sigma \leq -1} \times \divSet$ as being contained in $\R_{\tau}$ times the collar of $\Liouville$. The above conditions then imply that $\alphaRLiouville_{f}$ and $\alpha_{\Liouville}$ agree along $\R_{\tau} \times \R_{\sigma \leq -3} \times \divSet \subset \R_{\tau} \times \Liouville_{-1}$. Hence we may extend the definition of $\alphaRLiouville_{f}$ over all of $\R_{\tau} \times \Liouville_{-1}$ by setting
\begin{equation*}
\alphaRLiouville_{f}|_{\R_{\tau} \times \Liouville_{-3}} = \alpha_{\Liouville}.
\end{equation*}
According to the above assumptions, for $\sigma \in [-2, -1]$, we have
\begin{equation*}
\alphaRLiouville_{f} = -\sigma d\tau + f\alphaDivSet, \quad f = f(\sigma).
\end{equation*}

\begin{assump}\label{Assump:AlphaRLiouville2}
Over the subset $\R_{\tau} \times [-1, \sigma_{0}]_{\sigma} \times \divSet$ with $\sigma_{0} \in (-\half, 0)$ we allow $f$ to depend on both $\sigma$ and $\tau$, with $f$ again being $\COne$-close to $1$. Along this set, we extend the definition of $\alphaRLiouville_{f}$ as
\begin{equation*}
\alphaRLiouville_{f} = -\sigma d\tau + f\alphaDivSet, \quad f = f(\tau, \sigma).
\end{equation*}
We require that there is some small $\epsilon, \delta \in \R_{>0}$ such that
\begin{equation*}
f = 1 + \epsilon\sigma \ \text{along}\ \{ |\sigma| \in [-1 + \delta, \sigma_{0} - \delta ]\} \cup \{ |\tau| > \delta\}.
\end{equation*}
\end{assump}

\begin{figure}[h]
\begin{overpic}[scale=.8]{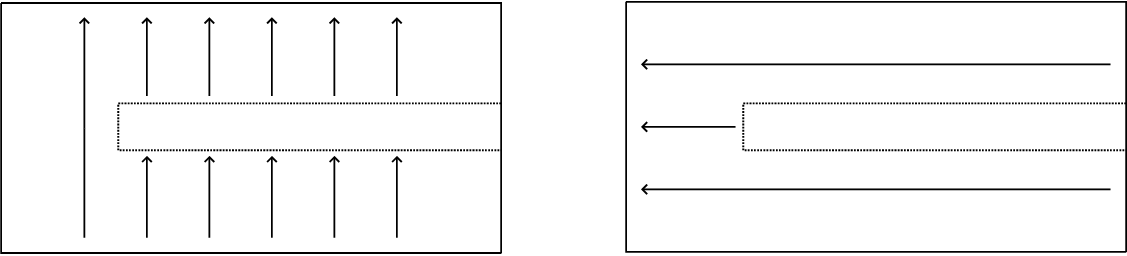}
\end{overpic}
\caption{The left subfigure shows the projection of $\ReebRLiouville_{f}$ to $\R_{\tau} \times [-1, \sigma_{0}]$ outside of $\{ \sigma > \sigma_{0} - \delta\} \cup \{ |\tau| > \delta, \sigma > -1\}$, which is bounded by a dashed curve. On the right, $\grad_{f}$ flow lines are shown in the same subset of $\R_{\tau} \times [-1, \sigma_{0}]$. In both subfigures, $\partial_{\tau}$ points upwards and $\partial_{\sigma}$ points to the left.}
\label{Fig:AlphaRLiouvilleAssumptions}
\end{figure} 

Figure \ref{Fig:AlphaRLiouvilleAssumptions} graphically describes the above assumptions. We will impose further assumptions on the function $f$ in \S \ref{Sec:fAssumptions}. The following lemma summarizes dynamical properties of the contact forms $\alphaRLiouville_{f}$.

\begin{lemma}\label{Lemma:DynamicsOnRLiouville}
Let $\alphaRLiouville_{f} \in \Omega^{1}(\R_{\tau} \times \Liouville)$ be as described above. The Reeb vector field is 
\begin{equation*}
\ReebRLiouville_{f} = \begin{cases}
\frac{1}{3}\partial_{\tau} & \R_{\tau} \times \Liouville_{-3},\\
(f - \sigma\partial_{\sigma}f)^{-1}(\ReebDivSet + \partial_{\sigma}f \partial_{\tau} - \partial_{\tau}f\partial_{\sigma}) & \R_{\tau} \times [-3, \sigma_{0}] \times \divSet,\\
\ReebDivSet + \epsilon\partial_{\tau} & \{ \sigma > \sigma_{0} - \delta\} \cup \{ |\tau| > \delta, \sigma > -1\}.
\end{cases}
\end{equation*}
If $f|_{\{ \sigma > -1\}}$ is sufficiently $\mathcal{C}^{1}$-close to the constant $1$, all closed $\ReebRLiouville_{f}$ orbits of action at most
\begin{equation*}
\actionBound = \max(f|_{\sigma > -1}) \actionBoundDivSet
\end{equation*}
are of the form
\begin{equation*}
\orbit_{\zeta} = \{\zeta\} \times \orbitDivSet
\end{equation*}
associated to pairs of closed $\ReebDivSet$ orbits $\orbitDivSet$ of action $\leq \actionBoundDivSet$ and 
\begin{equation*}
z \in \Crit(f) \subset \R_{\tau} \times [-1, \sigma_{0}]_{\sigma}.
\end{equation*}
The actions and Conley-Zehnder indices of these orbits are computed as in Lemma \ref{Lemma:DynamicsOnDivSetC}. Hence each $\orbit_{\zeta}$ is a good orbit iff $\orbitDivSet$ is good.
\end{lemma}

\begin{proof}
We view $\R_{\tau} \times [-1, \sigma_{0}]_{\sigma} \times \divSet$ as being contained in $\C \times \divSet_{\tau, \sigma}$. Then Lemma \ref{Lemma:ContactOnDivSetC} is applicable with $\beta = -\sigma d\tau$. In particular, along this subset of $\R_{\tau} \times \Liouville_{\sigma_{0}}$ with $\sigma_{0} \in [-1, 0]$, we have
\begin{equation*}
\ReebRLiouville_{f} = (f - \sigma\partial_{\sigma}f)^{-1}(\ReebDivSet + \partial_{\sigma}f \partial_{\tau} - \partial_{\tau}f\partial_{\sigma}).
\end{equation*}
This agrees with the expression for $\ReebRLiouville_{f}$ along $\{ \sigma \in [-3, -1]\}$ provided in the statement of the lemma, where our assumptions guarantee that $\partial_{\sigma}f > 0$ and $\partial_{\tau}f = 0$. Consequently any Reeb orbit which touches the region $\{ \sigma \in [-3, -1]\}$ must remain there as it moves forward in time, and the orbit cannot be closed. We conclude that any closed Reeb orbit must be entirely contained in $\{ \sigma > -1\}$, where the results of Lemma \ref{Lemma:DynamicsOnDivSetC} may be applied.
\end{proof}

\subsubsection{Liouville forms on $\R_{s} \times \R_{\tau} \times \Liouville_{\sigma_{0}}$}\label{Sec:LiouvilleRLiouville}

We consider Liouville forms on $\R_{s} \times \R_{\tau} \times \Liouville_{\sigma_{0}}$ which interpolate between contact forms $\alphaRLiouville_{f}$. For $\sigma_{0} \in (-\half, 0)$ consider functions
\begin{equation*}
f, g \in \Cinfty(\R_{s} \times \R_{\tau} \times (-\infty, \sigma_{0}]_{\sigma})
\end{equation*}
satisfying the following conditions:
\be
\item $g$ is independent of $s$ and satisfies Assumptions \ref{Assump:AlphaRLiouville1}.
\item There is an $s_{0} \in \R_{>0}$ such that $\partial_{s} f = 0$ for $|s| > s_{0}$.
\item For $\sigma < -1$, $f$ is independent of $s$ and satisfies Assumptions \ref{Assump:AlphaRLiouville1}.
\item For $\sigma \in [-1, \sigma_{0}]$, $f + \partial_{s}f - \sigma \partial_{\tau}f > 0$.
\item $f = 1 + \epsilon\sigma$ is independent of $s$ along $\{ \sigma > \sigma_{0} - \delta\} \cup \{ |\tau| > \delta, \sigma > -1\}$.
\ee
Then we define
\begin{equation}\label{Eq:LambdaGHDefinition}
\begin{gathered}
\lambdaRLiouville_{f} \in \Omega^{1}(\R_{s} \times \R_{\tau} \times \Liouville_{\sigma_{0}}),\\
\lambdaRLiouville_{f}|_{\R_{s} \times \R_{\tau} \times \Liouville_{-3}} = e^{s}\alpha_{\Liouville} = e^{s}(3d\tau + \betaLiouville),\\
\lambdaRLiouville_{f}|_{\R_{s} \times \R_{\tau} \times [-3, \sigma_{0}] \times \divSet} = e^{s}(gd\tau + f\alphaDivSet).
\end{gathered}
\end{equation}

It follows that $(\R_{s} \times \R_{\tau} \times \Liouville_{-1}, \lambdaRLiouville_{f})$ (along which $f$ and $g$ are independent of $s$) is a symplectization of the strict contact manifold $(\R_{\tau} \times \Liouville_{-1}, \alphaRLiouville_{f})$. Along $\{ \sigma > -1\}$, we may view $\R_{s} \times \R_{\tau} \times [-3, \sigma_{0}] \times \divSet$ as a subset of $\R_{s} \times \C \times \divSet$, whence Lemma \ref{Lemma:LiouvilleRDivSetC} is applicable to see that $d\lambdaRLiouville_{f}$ is symplectic. We view $(\R_{s} \times \R_{\tau} \times \Liouville_{\sigma_{0}}, \lambdaRLiouville_{f})$ as the completion of a Liouville cobordism between strict contact manifolds
\begin{equation*}
\begin{gathered}
([-s_{0}, s_{0}]_{s} \times \R_{\tau} \times \Liouville_{\sigma_{0}}, \lambdaRLiouville_{f}): (\R_{\tau} \times \Liouville_{\sigma_{0}}, e^{s_{0}}\alphaRLiouville_{f(s_{0}, \ast)}) \to (\R_{\tau} \times \Liouville_{\sigma_{0}}, e^{-s_{0}}\alphaRLiouville_{f(-s_{0}, \ast)}).
\end{gathered}
\end{equation*}

\subsubsection{Almost complex structures on $\R_{s} \times \R_{\tau} \times \Liouville_{\sigma_{0}}$}\label{Sec:JGH}

Now we describe an adapted almost complex structure $\JRLiouville_{f}$ on the complete cobordisms described above. Let $J_{\Liouville}$ be an almost complex structure on $\Liouville$ which is compatible with $d\betaLiouville$. We assume that
\begin{equation*}
J_{\Liouville}\partial_{\sigma} = \ReebDivSet, \quad J_{\Liouville}|_{\xiDivSet^{J}} = J_{\divSet} \quad \text{along} \ [-4, 0] \times \divSet.
\end{equation*}
We define $\JRLiouville_{f}$ separately along the subsets
\begin{equation}\label{Eq:JGHSubsets}
\R_{s} \times \R_{\tau} \times \Liouville_{-3}, \ \R_{s} \times \R_{\tau} \times [-3, -1]_{\sigma} \times \divSet,\ \R_{s} \times \R_{\tau} \times [-1, \sigma_{0}]_{\sigma} \times \divSet,
\end{equation}
of $\R_{s} \times \R_{\tau} \times \Liouville_{\sigma_{0}}$, ensuring that the definitions overlap where appropriate.

We define $\JRLiouville_{f}$ over $\R_{s} \times \R_{\tau} \times \Liouville_{-3}$ by
\begin{equation}\label{Eq:JGHmin3}
\JRLiouville_{f}\partial_{s} = \partial_{\tau}, \quad \JRLiouville_{f}|_{T\Liouville_{-3}} = J_{\Liouville}|_{\Liouville_{-3}}.
\end{equation}
To extend our definition over the remainder of $\R_{s} \times \R_{\tau} \times \Liouville_{\sigma_{0}}$, we choose a function $F \in \Cinfty(\R_{s} \times \R_{\tau} \times [-3, \sigma_{0}], \R_{>0})$ such that
\be
\item there is some $s_{0} > 0$ such that $\partial_{s}F=0$ for $|s|>s_{0}$ (as is the case with $f$) and
\item $F = e^{-\sigma}$ in a neighborhood of $\{ \sigma =-3\}$.
\ee
Over the ``middle'' region, $\R_{s} \times \R_{\tau} \times [-3, -1]_{\sigma} \times \divSet$, we define $\JRLiouville_{f}$ by
\begin{equation}\label{Eq:JGHMid}
\begin{gathered}
\JRLiouville_{f}\partial_{s} := F\partial_{\sigma}f\partial_{\tau} -F\partial_{\sigma}g\ReebDivSet, \quad \JRLiouville_{f}\partial_{\sigma} := (1+\partial_{\sigma}g)\ReebDivSet + \partial_{\sigma}g\partial_{\tau}, \quad \JRLiouville_{f}|_{\xiDivSet^{J}} = J_{\divSet}.
\end{gathered}
\end{equation}
Finally, we define $\JRLiouville_{f}$ along $\R_{s} \times \R_{\tau} \times [-1, \sigma_{0}] \times \divSet$ by
\begin{equation}\label{Eq:JGHEnd}
\JRLiouville_{f} \partial_{s} = F\left(-\partial_{\sigma}g \ReebDivSet + \partial_{\sigma}f \partial_{\tau} - \partial_{\tau}f\partial_{\sigma}\right), \quad \JRLiouville_{f} \partial_{\tau} = \partial_{\sigma}, \quad \JRLiouville_{f}|_{\xiDivSet^{J}} = J_{\divSet}.
\end{equation}

Note that our conditions on $f$, $g$, $J_{\Liouville}$, and $F$ ensure that $\JRLiouville_{f}$ is coherently defined along neighborhoods of $\{ \sigma = -3,-1\}$. The following lemma will impose more conditions on $F$ to ensure that $\JRLiouville_{f}$ is compatible with $d\lambdaRLiouville_{f}$.

\begin{lemma}\label{Lemma:JGHCompatible}
If $F$ is appropriately large (point-wise) over $\{ \sigma \in [-3, -1]\}$ and satisfies
\begin{equation}\label{Eq:JGHEndInequality}
4F(f + \partial_{s}f - \sigma \partial_{\sigma}f) > \sigma^{2}
\end{equation}
along the subset $\{ \sigma > -1\}$, then $\JRLiouville_{f}$ defines an almost complex structure which is compatible with $d\lambdaRLiouville_{f}$ over $\R_{s} \times \R_{\tau} \times \Liouville_{\sigma_{0}}$.
\end{lemma}

\begin{proof}
We work over the subsets of Equation \eqref{Eq:JGHSubsets} separately. Along $\R_{s} \times \R_{\tau} \times \Liouville_{-3}$, $\JRLiouville_{f}$ is a standard ``symplectization of a contactization'' almost complex structure. It is easily verified that $\JRLiouville_{f}$ is compatible with $d\lambdaRLiouville_{f}$ in this region.

Now we work over the middle region, where $\sigma \in [-3, -1]$ and $f, g$ depend only on $\sigma$. This will require the most work, so we use the short-hand notation
\begin{equation}\label{Eq:JGHShortSubs}
A=Ff', \quad B=-Fg', \quad C=1+g', \quad D=g', \quad g'=\partial_{\sigma}g, \quad f'=\partial_{\sigma}f.
\end{equation}
Over this region we can then write
\begin{equation*}
\begin{gathered}
\JRLiouville_{f}\partial_{s} := A\partial_{\tau} + B\ReebDivSet, \quad \JRLiouville_{f}\partial_{\sigma} := C\ReebDivSet + D\partial_{\tau}, \quad \JRLiouville_{f}|_{\xiDivSet^{J}} = J_{\divSet}.
\end{gathered}
\end{equation*}
We claim that $\JRLiouville_{f}$ is compatible with $d\lambdaRLiouville_{f}$ along $\{ \sigma \in [-3, -1]\}$ iff
\begin{equation}\label{Eq:JGHCompatibleInequality}
4(AC - BD)(g\partial_{\sigma}f - f\partial_{\sigma}g) > (fC + gD - A\partial_{\sigma}g - B\partial_{\sigma}f)^{2}.
\end{equation}

To establish the claim, suppose that we have a tangent vector of the form
\begin{equation*}
v = w\partial_{s} + x\partial_{\tau} + y\partial_{\sigma} + z\ReebDivSet + X, \quad X \in \xiDivSet^{J}.
\end{equation*}
A long computation then reveals
\begin{equation}\label{Eq:InterpolationCompatibleCheck}
\begin{aligned}
e^{-s}d\lambdaRLiouville_{f}(v, \JRLiouville_{f}v) =&\ (gA + fB)w^{2} + (D\partial_{\sigma}g + C\partial_{\sigma}f)y^{2}\\
&+ (gD + fC + A\partial_{\sigma}g + B\partial_{\sigma}f)wy\\
&+ (AC - BD)^{-1}\Big[(gC - B\partial_{\sigma}g)x^{2} + (A\partial_{\sigma}f- fD)z^{2}\\
&\hphantom{+ (AC - BD)^{-1}\Big[}+ (fC - gD + A\partial_{\sigma}g - B\partial_{\sigma}f)xz\Big]\\
&+ f d\alphaDivSet(X, J_{\divSet}X).
\end{aligned}
\end{equation}

Following the strategy of the proof of Lemma \ref{Lemma:JCompatibleOnRDivSetC}, we view $e^{-s}d\lambdaRLiouville_{f}(v, \JRLiouville_{f}v)$ as a quadratic function $Q$ of the variables $w,x,y,z \in \R$ with $X$ fixed. Explicitly, we may write
\begin{equation*}
\begin{gathered}
Q = Q_{0} + Q_{1} + fd\alphaDivSet(X, J_{\divSet}X),\\
Q_{0} = a_{0}x^{2} + b_{0}z^{2} + c_{0}xz,\quad Q_{1} = a_{1}w^{2} + b_{1}y^{2} + c_{1}wy,\\
\end{gathered}
\end{equation*}
where the $a_{i}, b_{i}, c_{i}$ are determined by Equation \eqref{Eq:InterpolationCompatibleCheck}.

We seek to understand what conditions ensure that $Q_{0}$ is positive definite. At a critical point of $Q_{0}$ we have
\begin{equation*}
\partial_{x}Q_{0}=0 \implies x = -(c_{0}/2a_{0})z \implies Q_{0} =  (b_{0} - c_{0}^{2}/(4a_{0}))z^{2}.
\end{equation*}
Therefore $Q_{0}$ is positive definite iff
\begin{equation*}
4a_{0}b_{0} > c_{0}^{2}.
\end{equation*}
Filling in the definitions of the $a_{0}, b_{0}, c_{0}$ yields
\begin{equation*}
4(gC - B\partial_{\sigma}g)(A\partial_{\sigma}f - fD) > (fC - gD + A\partial_{\sigma}g - B\partial_{\sigma}f)^{2}.
\end{equation*}
After rearranging some terms, we find that this is equivalent to Equation \eqref{Eq:JGHCompatibleInequality}. Applying the same computational technique to find critical points of $Q_{1}$, we see that it is positive definite iff
\begin{equation*}
4a_{1}b_{1} > c_{1}^{2} \iff 4(gA + fB)(D\partial_{\sigma}g + C\partial_{\sigma}f) > (gD + fC + A \partial_{\sigma}g + B\partial_{\sigma}f)^{2}.
\end{equation*}
Some arithmetic applied to this new inequality shows that it is also equivalent to Equation \eqref{Eq:JGHCompatibleInequality}. To conclude the proof of the claim, we observe that $Q$ is positive definite iff both $Q_{0}$ and $Q_{1}$ are positive definite.

Let's write $LHS$ and $RHS$ for the left- and right-hand sides of the inequality provided in Equation \eqref{Eq:JGHCompatibleInequality}, respectively. We will check that we can ensure $LHS > RHS$ by controlling the function $F$ along the subsets $\{ \sigma \in [-3, -2]\}$ and $\{ \sigma \in [-2, -1]\}$ separately.

Along the subset $\{ \sigma \in [-3, -2]\}$ we have $f = f' = e^{\sigma}$ with $g$ interpolating between $3$ near $\sigma=-3$ and $-\sigma$ near $\sigma=-2$. Reapplying the substitutions of Equation \eqref{Eq:JGHShortSubs}, we find that
\begin{equation*}
\begin{aligned}
LHS =& 4(AC - BD)(g\partial_{\sigma}f - f\partial_{\sigma}g) = 4Ff(f(1 + g') + (g')^{2})(g-g')\\
RHS =&  (fC + gD - A\partial_{\sigma}g - B\partial_{\sigma}f)^{2} = (f(1 + g') + gg')^{2}
\end{aligned}
\end{equation*}
Both $LHS$ and $RHS$ are positive, so we can ensure $LHS > RHS$ by making $F$ appropriately large, point-wise as in the statement of the lemma.

Along the subset $\{ \sigma \in [-2, -1]\}$ we have $g = -\sigma$ so that $B=F, C=0, D=-1$. Therefore
\begin{equation*}
LHS = 4F(f-\sigma f'), \quad RHS = \sigma^{2}.
\end{equation*}
We note that this is a special instance of Equation \eqref{Eq:JGHEndInequality}, which is itself a special instance of the inequality of Lemma \ref{Lemma:JCompatibleOnRDivSetC}. That lemma is applicable in the present context by viewing $\R_{\tau}\times [-1, \sigma_{0}]$ as a subset of $\C$ and setting $\beta=-\sigma d\tau$. Thus if Equation \eqref{Eq:JGHEndInequality} is satisfied, $\JRLiouville_{f}$ is compatible with $d\lambdaRLiouville_{f}$ along $\{ \sigma \in [-1, \sigma_{0}]\}$.
\end{proof}

The following convexity result will later allow us to ensure that $\JRLiouville_{f}$-holomorphic curves cannot escape the ``horizontal'' boundary $\{ \sigma=\sigma_{0}\}$ of $\R_{s} \times \R_{\tau} \times \Liouville$. The proof is a straightforward computation.

\begin{lemma}\label{Lemma:BlockingRRLiouville}
With $\JRLiouville_{f}$ as above, we have
\begin{equation*}
-d(df \circ \JRLiouville_{f}) = \epsilon^{2}d\alphaDivSet\ \text{along}\ \{ \sigma > \sigma_{0} - \delta\} \cup \{ |\tau| > \delta, \sigma > -1\}.
\end{equation*}
\end{lemma}

\subsubsection{Foliations on $\R_{s} \times \R_{\tau} \times \Liouville_{\sigma_{0}}$}\label{Sec:FoliationRLiouville}

Now we are ready to describe a foliation of $\R_{s} \times \R_{\tau} \times \Liouville_{\sigma_{0}}$ by $\codim = 2$, $\JRLiouville_{f}$-holomorphic submanifolds. Again we work separately over the subsets of Equation \eqref{Eq:JGHSubsets}.

Within $\R_{s} \times \R_{\tau} \times \Liouville_{-3}$, the
\begin{equation*}
\{ s\} \times \{ \tau\} \times \Liouville_{-3} \subset \R_{s} \times \R_{\tau} \times \Liouville_{\sigma_{0}}
\end{equation*}
are $\JRLiouville_{f}$-holomorphic submanifolds. Hence they constitute a foliation $\foliation_{\Liouville_{-3}}$ whose tangent distribution is
\begin{equation}\label{Eq:TFoliationStart}
T\foliation_{\Liouville_{-3}} = T\Liouville_{-3} \subset T(\R_{s} \times \R_{\tau} \times \Liouville_{-3}).
\end{equation}

Now we work on the ``middle region'', $\{ \sigma \in [-3, -1]\}$. We keep in mind the short-hand notation of Equation \eqref{Eq:JGHShortSubs} which will simplify some computations. Let $(s(t), \sigma(t))$ be a solution to the differential equation
\begin{equation*}
\partial_{t}(s, \sigma) = (-D(\sigma(t)), A(\sigma(t))) = (-\partial_{\sigma}g(\sigma(t)), F\partial_{\sigma}f(\sigma(t))).
\end{equation*}
Since $A > 0$, we can choose the domain $[a, b]$ of $t$ so that $\sigma(t)$ is a diffeomorphism
\begin{equation*}
[a, b] \to [-3, -1].
\end{equation*}
For each such solution and $\tau \in \R$ we define an embedding
\begin{equation*}
\begin{gathered}
\Psi: [a, b]_{t}\times \divSet \to \R_{s} \times \R_{\tau} \times [-3, -1] \times \divSet,\\
\Psi(t, p) = (s(t), \tau, \sigma(t), p).
\end{gathered}
\end{equation*}

\begin{lemma}
If the domain of $\Psi$ is equipped with the almost complex structure
\begin{equation*}
J_{\Psi}\partial_{t} = (AC-BD)\ReebDivSet, \quad J_{\Psi}|_{\xiDivSet^{J}} = J_{\divSet},
\end{equation*}
then $\Psi$ is a $(J_{\Psi}, \JRLiouville_{f})$-holomorphic embedding. The images of the $\Psi$ are the leaves of a foliation $\foliation_{mid}$ whose tangent distribution is
\begin{equation}\label{Eq:TFoliationMid}
T\foliation_{mid} = \R(-D\partial_{s} + A\partial_{\sigma}) \oplus T\divSet = \R(-\partial_{\sigma}g\partial_{s} + F\partial_{\sigma}f\partial_{\sigma}) \oplus T\divSet.
\end{equation}
\end{lemma}

\begin{proof}
Writing $s' = \partial_{t}s, \sigma' = \partial_{t}\sigma$ we compute
\begin{equation*}
\begin{aligned}
T\Psi =& (s', 0, \sigma', 0)\otimes dt + (0, 0, 0, \ReebDivSet)\otimes \alphaDivSet^{J} + (0, 0, 0, \Id_{\xiDivSet^{J}}),\\
\JRLiouville_{f}T\Psi =& (0, s'A + \sigma'D, 0, (s'B + \sigma'C)\ReebDivSet)\otimes dt\\
&+ (AC-BD)^{-1}(D, 0, -A, 0)\otimes \alphaDivSet^{J} + (0,0,0,J_{\divSet}\Id_{\xiDivSet^{J}}).
\end{aligned}
\end{equation*}
It is then straightforward to verify that $\Psi$ is $(J_{\Psi}, \JRLiouville_{f})$-holomorphic.

As in the proof of Lemma \ref{Lemma:FoliationRCDivSet}, we see that there is a unique $\im \Psi$ passing through each point in $\R_{s} \times \R_{\tau} \times [-3, -1] \times \divSet$ since the $\Psi$ are determined by solutions to ordinary differential equations.
\end{proof}

Finally we describe a foliation $\foliation_{end}$ and an associated distribution on $\{ \sigma \in [-1, \sigma_{0}]\}$. By identifying $(\R_{\tau} \times [-1, \sigma_{0}], -\sigma d\tau) \subset (\C, \beta)$ and $(\R_{s} \times \R_{\tau} \times [-1, \sigma_{0}] \times \divSet, \lambdaRLiouville_{f})$ with a subset of
\begin{equation*}
(\R_{s} \times \C \times \divSet, \lambdaCDivSet_{f}), \quad \lambdaCDivSet_{f} = e^{s}(\beta + f\alphaDivSet),
\end{equation*}
we see that along $\{\sigma \in [-1, \sigma_{0}]\}$, $\JRLiouville_{f}$ is exactly as described in Equation \eqref{Eq:JOnRCDivSet}. Following Lemma \ref{Lemma:FoliationRCDivSet}, we have a foliation which we denote by $\foliation_{end}$ whose tangent distribution is given by
\begin{equation}\label{Eq:TFoliationEnd}
T\foliation_{end} = \R(\partial_{s} - FJ_{0}X_{f}) \oplus T\divSet \subset T(\R_{s} \times \R_{\tau} \times [-1, \sigma_{0}]_{\sigma} \times \divSet).
\end{equation}

We note that the $\codim = 2$ distributions
\be
\item $T\foliation_{\Liouville_{-3}}$ on $\R_{s} \times \R_{\tau} \times \Liouville_{-3}$ (Equation \eqref{Eq:TFoliationStart}),
\item $T\foliation_{mid}$ on $\{\sigma \in [-3, -1]\}$ (Equation \eqref{Eq:TFoliationMid}), and
\item $T\foliation_{end}$ on $\{\sigma \in [-1, \sigma_{0}]\}$ (Equation \eqref{Eq:TFoliationEnd})
\ee
agree along their overlaps:
\begin{equation*}
\begin{gathered}
T\foliation_{mid}|_{\{ \sigma < -3\}} = T\Liouville_{-3} = T\foliation_{\Liouville_{-3}},\\
T\foliation_{mid}|_{\{ \sigma > -2\}} = \R(\partial_{s} +F\grad_{f}) \oplus T\divSet = T\foliation_{end}.
\end{gathered}
\end{equation*}
Therefore they patch together to yield a $\JRLiouville_{f}$-invariant distribution
\begin{equation*}
T\foliationRLiouville_{f} \subset T(\R_{s} \times \R_{\tau} \times \Liouville_{\sigma_{0}}).
\end{equation*}
This distribution is integrable since the $T\foliation_{\Liouville_{-3}}$, $T\foliation_{mid}$, and $T\foliation_{end}$ integrate to foliations. The following lemma describes the leaves
\begin{equation*}
\leaf \subset \foliationRLiouville_{f}, \quad T\leaf \subset T\foliationRLiouville_{f}
\end{equation*}
of the resulting foliation, provided a ``vertical boundary'' hypothesis.

\begin{lemma}\label{Lemma:FoliationGHProperties}
Suppose that $\partial_{\sigma}f >0$ along $\{ \sigma = \sigma_{0} \}$. Then each leaf
\begin{equation*}
\leaf \subset \R_{s} \times \R_{\tau} \times \Liouville_{\sigma_{0}}
\end{equation*}
of $\foliationRLiouville_{f}$ is of one of the following types:
\be
\item $\leaf$ is diffeomorphic to $\R \times \divSet$.
\item $\leaf$ is diffeomorphic to $\R_{\leq 0} \times \divSet$ with $\partial \leaf = \leaf \cap \{ \sigma = \sigma_{0} \}$.
\item $\leaf$ is diffeomorphic to a completion $\complete{\Liouville}$ of $\Liouville$.
\item $\leaf$ is diffeomorphic to $\Liouville$ with $\partial \leaf = \leaf \cap \{ \sigma = \sigma_{0} \}$.
\ee
\end{lemma}

\begin{proof}
We need only to track how a leaf intersects the subsets
\begin{equation*}
U_{-1} = \R_{s} \times \R_{\tau} \times \Liouville_{-1}, \quad U_{end} = \{ \sigma \in [-1, \sigma_{0}]\}.
\end{equation*}
The intersection of each leaf $\leaf$ with $U_{-1}$ is either empty or a copy of $\Liouville_{-1}$.

If this intersection is empty, then $\leaf$ is entirely contained in $U_{end}$. As in Lemma \ref{Lemma:FoliationRCDivSet}, $\leaf$ is determined by its projection to $\R_{s} \times \R_{\tau} \times [-1, \sigma_{0}]$, which has the form
\begin{equation*}
t \mapsto (t, \tau(t), \sigma(t)), \quad \partial_{t}(\tau, \sigma) = F\grad_{f(t, \ast)}.
\end{equation*}
Then $\leaf$ is diffeomorphic to the product of the domain of this flow-line with $\divSet$, for the projection of the leaf onto the $\divSet$ factor of $\R_{s} \times \R_{\tau} \times [-1, \sigma_{0}]$ is a submersion. By our hypotheses on $f|_{\{\sigma = \sigma_{0}\}}$ such a flow line must either
\be
\item be infinite, in which case $\leaf$ is diffeomorphic to $\R \times \divSet$ or 
\item end at some finite time and touch $\{\sigma = \sigma_{0}\}$ in which case $\leaf$ is a diffeomorphic to $\R_{\leq 0} \times \divSet$.
\ee

If $\leaf$ intersects $U_{-1}$ in a copy of $\Liouville_{-1}$, then an identical analysis indicates that $\leaf \cap U_{end}$ either
\be
\item touches $\{ \sigma = \sigma_{0}\}$ in which case $\leaf \cap U_{end}$ is diffeomorphic to some $[a, b] \times \divSet$, so that $\leaf \simeq \Liouville$ or
\item is non-compact, in which case $\leaf \cap U_{end} \simeq \R_{\leq 0} \times \divSet$ so that $\leaf$ is diffeomorphic to $\complete{\Liouville}$.
\ee
\end{proof}

\subsection{Geometry of $\R_{\tau} \times \hypersurface$ and $\R_{s} \times \R_{\tau} \times \hypersurface$}\label{Sec:GeometryRHypersurface}

Now we combine our previous constructions to describe geometric structures on neighborhoods of convex hypersurfaces and cobordisms between them. Throughout, let $(\Liouville^{\pm}, \beta_{\Liouville^{\pm}})$ be a pair of compact Liouville domains of the same dimension, $2n$, for which
\begin{equation*}
\partial \Liouville^{\pm} = \divSet, \quad \beta_{\Liouville^{\pm}}|_{\divSet} = \alphaDivSet.
\end{equation*}
Our convex hypersurface will be built by gluing the $\Liouville^{\pm}$ together along their boundaries. Smoothly $\hypersurface$ is given as
\begin{equation*}
\hypersurface = (\R_{\tau} \times \Liouville^{+}_{\sigma_{0}}) \cup (\R_{\tau} \times [-1, 1]_{\sigma} \times \divSet) \cup (\R_{\tau} \times \Liouville^{-}_{\sigma_{0}}).
\end{equation*}
The unions are determined by the following inclusions
\begin{equation}\label{Eq:OverlapInclusions}
\begin{gathered}
\left(\R_{\tau} \times \Liouville_{\sigma_{0}}^{\pm}\right) \xleftarrow{1_{\tau} \times \Collar} \left(\R_{\tau} \times [-1, \sigma_{0}]_{\sigma} \times \divSet\right) \xrightarrow{I^{\pm}} \left(\R_{\tau} \times [-1, 1]_{\sigma} \times \divSet\right),\\
I^{\pm}(\tau, \sigma, p) = (\pm \tau, \pm\sigma, p).
\end{gathered}
\end{equation}

\begin{rmk}
Our construction of $\hypersurface$ implicitly makes use of a contactomorphism between the boundaries of the $\Liouville^{\pm}$. Modifying the identification of the $\partial \Liouville^{\pm}$ may yield an inequivalent $\hypersurface$. See \cite[\S 10]{HH:Bypass}.
\end{rmk}

Throughout, we specify a contact form $\alphaDivSet$ on $\divSet$ with Reeb vector field $\ReebDivSet$ and an action bound $\actionBoundDivSet > 0$ such that all closed $\ReebDivSet$ orbits of action $\leq \actionBoundDivSet$ are non-degenerate.

\subsubsection{Contact forms on $\R_{\tau} \times \hypersurface$}\label{Sec:AlphaRHypersurface}

Consider a $1$-form
\begin{equation*}
\alphaRIntervalDivSet_{f} = \beta + f\alphaDivSet \in \Omega^{1}(\R_{\tau} \times [-1, 1]_{\sigma} \times \divSet), \quad f = f(\tau, \sigma), \quad \beta = -\sigma d\tau.
\end{equation*}
Viewing $\R_{\tau} \times [-1, 1]_{\sigma} \subset \C$, the contact condition and Reeb vector computation for $\alphaRIntervalDivSet_{f}$ are described in Lemma \ref{Lemma:ContactOnDivSetC} and restated below in Lemma \ref{Lemma:DynamicsRHypersurface}. Also consider $1$-forms
\begin{equation*}
\alpha^{\R \times \Liouville^{\pm}}_{f^{\pm}} \in \Omega^{1}(\R_{\tau} \times \Liouville^{\pm}_{\sigma_{0}}), \quad \alpha^{\R \times \Liouville^{\pm}}_{f^{\pm}}|_{\R_{\tau} \times [-1, \sigma_{0}] \times \divSet} = -\sigma d\tau + f^{\pm}\alphaDivSet
\end{equation*}
as described in \S \ref{Sec:AlphaContactization}.

\begin{assump}\label{Assump:AlphaRHypersurface}
We assume that $f$ and the $f^{\pm}$ are such that
\begin{equation}\label{Eq:FHOverlap}
\sigma \in [-1, \sigma_{0}] \implies f^{\pm}(\pm \tau, \pm \sigma)= f(\tau, \sigma).
\end{equation}
We assume that $f$ is $\COne$ close to $1$ for $\sigma \in [-1, 1]$. Finally, assume that there is some $\delta \in (0, 1)$ such that
\be
\item $\grad_{f}$ points into $\R_{\tau} \times [-\delta, \delta]_{\sigma}$ along $\{|\sigma|=\delta\}$ and
\item $\grad_{f}$ points out of $[-\delta', \delta']_{\tau} \times [-\delta, \delta]_{\sigma}$ along each $\{ |\tau| = \delta'\}$ for $\delta' > \delta$.
\ee
\end{assump}

\begin{figure}[h]
\begin{overpic}[scale=.8]{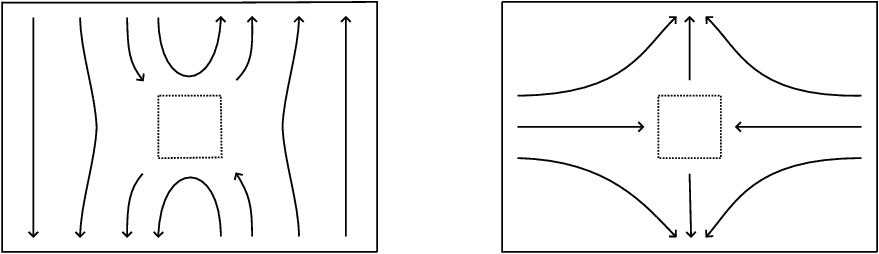}
\end{overpic}
\caption{The left subfigure shows the projection of $\ReebRHypersurface_{f}$ to $\R_{\tau} \times [-1, 1]$ outside of some $\{ |\tau| \leq \delta', |\sigma| \leq \delta\}$, which is bounded by a dashed rectangle. On the right, $\grad_{f}$ flow lines are shown in the same subset of $\R_{\tau} \times [-1, 1]$. In both subfigures, $\partial_{\tau}$ points upwards and $\partial_{\sigma}$ points to the left.}
\label{Fig:AlphaRHypersurfaceAssumptions}
\end{figure} 

Figure \ref{Fig:AlphaRHypersurfaceAssumptions} graphically describes the second collection of assumptions described above. We will impose further assumptions on the function $f$ in \S \ref{Sec:fAssumptions}. According to the first assumption above, the maps of Equation \eqref{Eq:OverlapInclusions} are such that
\begin{equation*}
(1_{\tau} \times \Collar)^{\ast}\alpha^{\R \times \Liouville^{\pm}}_{f^{\pm}} = (I^{\pm})^{\ast}\alphaRIntervalDivSet_{f}.
\end{equation*}
Hence $\alphaRIntervalDivSet_{f}$ and the $\alpha^{\R \times \Liouville^{\pm}}_{f^{\pm}}$ determine a $1$-form
\begin{equation*}
\alphaRHypersurface_{f} \in \Omega^{1}(\R_{\tau} \times \hypersurface), \quad \alphaRHypersurface_{f} = \begin{cases}
\alphaRIntervalDivSet_{f} & \text{along}\ \R_{\tau} \times [-1,1]_{\sigma} \times\divSet,\\
\alpha^{\R \times \Liouville^{\pm}}_{f^{\pm}} & \text{along}\ \R_{\tau} \times \Liouville^{\pm}_{\sigma_{0}}.
\end{cases}
\end{equation*}
The following lemma summarizes properties of the $\alphaRHypersurface_{f}$. It is essentially a restatement of results in \S \ref{Sec:ContactOnDivSetC}, organized for later reference.

\begin{lemma}\label{Lemma:DynamicsRHypersurface}
Suppose that $f$ is Morse with
\begin{equation*}
f -\beta(X_{f}) = f - \sigma\partial_{\sigma}f > 0.
\end{equation*}
Then $\alphaRHypersurface_{f}$ is a contact form whose Reeb vector field $\ReebRHypersurface_{f}$ satisfies
\begin{equation*}
\begin{gathered}
\begin{aligned}
\ReebRHypersurface_{f}|_{\R_{\tau} \times [-1, 1]_{\sigma} \times \divSet} =& (f - \sigma\partial_{\sigma}f)^{-1}(\ReebDivSet -X_{f})\\
=& (f - \sigma\partial_{\sigma}f)^{-1}(\ReebDivSet +\partial_{\sigma}f \partial_{\tau}-\partial_{\tau}f\partial_{\sigma})
\end{aligned}\\
\pm d\tau(\ReebRHypersurface_{f})|_{\R_{\tau} \times \Liouville^{\pm}_{-1}} > 0.
\end{gathered}
\end{equation*}
If $f$ is sufficiently $\COne$-close to $1$, then all $\ReebRHypersurface_{f}$ orbits of action at most
\begin{equation*}
\actionBound = \max(f)\actionBoundDivSet
\end{equation*}
must be of the form
\begin{equation*}
\orbit_{\zeta} = \{ \zeta\} \times \orbitDivSet, \quad z \in \Crit(f) \subset \R_{\tau} \times [-1,1]_{\sigma}, \quad \action(\orbit_{\zeta}) = f(z)\action(\orbitDivSet)
\end{equation*}
where $\orbitDivSet$ is a closed $\ReebDivSet$ orbit of action $\leq \actionBoundDivSet$. As described in \S \ref{Sec:ContactOnDivSetC}, framings of the $\orbitDivSet$ naturally determine framings of the $\orbit_{\zeta}$ so that their $\Z$-valued Conley-Zehnder indices and contact homology gradings satisfy
\begin{equation*}
\CZ(\orbit_{\zeta}) = \CZ(\orbitDivSet) + \ind_{\Morse}(f, \zeta) -1, \quad |\orbit_{\zeta}| = |\orbitDivSet| + \ind_{\Morse}(f, \zeta).
\end{equation*}
Each such $\orbit_{\zeta}$ is a good Reeb orbit iff $\orbitDivSet$ is a good Reeb orbit.
\end{lemma}

\subsubsection{Liouville forms on $\R_{s} \times \R_{\tau} \times \hypersurface$}\label{Sec:LiouvilleRHypersurface}

Now we describe Liouville forms on $\R_{s} \times \R_{\tau} \times \hypersurface$ which yield symplectizations of the $(\R_{\tau} \times \hypersurface, \alphaRHypersurface_{f})$ and completed cobordisms between them.

Consider $1$-forms $\lambdaRIntervalDivSet_{f}, \lambda^{\R \times \Liouville^{\pm}}_{f^{\pm}}$,
\begin{equation*}
\begin{gathered}
\lambdaRIntervalDivSet_{f} = e^{s}(\beta + f\alphaDivSet) \in \Omega^{1}(\R_{s} \times \R_{\tau} \times [-1, 1]_{\sigma} \times \divSet_{x}), \quad \beta = -\sigma d\tau,\\
\lambda^{\R \times \Liouville^{\pm}}_{f^{\pm}} \in \Omega^{1}(\R_{s} \times \R_{\tau} \times \Liouville_{\sigma_{0}}^{\pm}),\\
\lambda^{\R \times \Liouville^{\pm}}_{f^{\pm}}|_{\R_{s} \times \R_{\tau} \times \Liouville_{-3}^{\pm}} = e^{s}\alpha_{\Liouville} = e^{s}(3d\tau + \beta_{\Liouville^{\pm}}),\\
\lambda^{\R \times \Liouville^{\pm}}_{f^{\pm}}|_{\R_{s} \times \R_{\tau} \times [-3, \sigma_{0}] \times \divSet} = e^{s}(g^{\pm}d\tau + f^{\pm}\alphaDivSet).
\end{gathered}
\end{equation*}
as respectively described in Equations \eqref{Eq:LambdaFDefinition} and \eqref{Eq:LambdaGHDefinition}. Suppose that
\begin{equation}\label{Eq:LambdaFGHAssumptions}
\begin{gathered}
\partial_{s} g^{\pm} =0,\\
\exists\ s_{0} > 0\ \text{such that}\ |s|>s_{0} \implies \partial_{s}f  = \partial_{s}f^{\pm} = 0,\\
\sigma \in [-1, \sigma_{0}] \implies f^{\pm}(s, \pm \tau, \pm \sigma, x) = f(s, \tau, \sigma, x),\\
\end{gathered}
\end{equation}
Then $\lambdaRIntervalDivSet_{f}$ and the $\lambda^{\R \times \Liouville^{\pm}}_{f^{\pm}}$ determine a $1$-form
\begin{equation*}
\begin{gathered}
\lambdaRHypersurface_{f} \in \Omega^{1}(\R_{s} \times \R_{\tau} \times \hypersurface), \quad \lambdaRHypersurface_{f} = \begin{cases}
\lambdaRIntervalDivSet_{f} & \text{along}\ \R_{s} \times \R_{\tau} \times [-1,1]_{\sigma} \times\divSet,\\
\lambda^{\R \times \Liouville^{\pm}}_{f^{\pm}} & \text{along}\ \R_{s} \times \R_{\tau} \times \Liouville^{\pm}_{\sigma_{0}}.
\end{cases}
\end{gathered}
\end{equation*}

We assume that $\lambdaRIntervalDivSet_{f}$ satisfies the conclusions of Lemma \ref{Lemma:LiouvilleRDivSetC} and that $\lambda^{\R \times \Liouville^{\pm}}_{f^{\pm}}$ is as described in \S \ref{Sec:LiouvilleRLiouville}. Then $d\lambdaRHypersurface_{f}$ is symplectic. Equation \eqref{Eq:LambdaFGHAssumptions} allows us to view $(\R_{s} \times \R_{\tau} \times \hypersurface, \lambdaRHypersurface_{f})$ as the completion of a Liouville cobordism between strict contact manifolds
\begin{equation*}
\begin{gathered}
([-s_{0}, s_{0}]_{s} \times \R_{\tau} \times \hypersurface, \lambdaRHypersurface_{f}): (\R_{\tau} \times \hypersurface, e^{s_{0}}\alphaRHypersurface_{f(s_{0}, \ast)}) \to (\R_{\tau} \times \hypersurface, e^{-s_{0}}\alphaRHypersurface_{f(-s_{0}, \ast)}).
\end{gathered}
\end{equation*}

\subsubsection{Almost complex structures on $\R_{s} \times \R_{\tau} \times \hypersurface$}\label{Sec:JRHypersurface}

Now we describe almost complex structures on $\R_{s} \times \R_{\tau} \times \hypersurface$ which give rise to holomorphic foliations. Suppose that we are given almost complex structures
\be
\item $\JRIntervalDivSet_{f}$ on $\R_{s} \times \R_{\tau} \times [-1, 1]_{\sigma} \times \divSet$ as described in Equation \eqref{Eq:JFDefinition} and
\item $J^{\R \times \Liouville^{\pm}}_{f^{\pm}}$ on $\R_{s} \times \R_{\tau} \times \Liouville_{\sigma_{0}}^{\pm}$ as described in Equations \eqref{Eq:JGHEnd}, \eqref{Eq:JGHmin3}, and \eqref{Eq:JGHMid}.
\ee
Suppose also that these almost complex structures agree along their overlaps in $\R_{s} \times \R_{\tau} \times \hypersurface$. This means that the functions $F$ in Equations \eqref{Eq:JFDefinition} and \eqref{Eq:JGHEnd} agree along the overlapping subsets
\begin{equation*}
\{ \pm \sigma \in [-1, \sigma_{0}]\} = (\R_{s} \times \R_{\tau} \times \Liouville_{\sigma_{0}}^{\pm}) \cap (\R_{s} \times \R_{\tau} \times [-1, 1]_{\sigma} \times \divSet) \subset \R_{s} \times \R_{\tau} \times \hypersurface.
\end{equation*}

With these assumptions in place, we may define an almost complex structure $\JRHypersurface_{f}$ on $\R_{s} \times \R_{\tau} \times \hypersurface$ by
\begin{equation*}
\JRHypersurface_{f} = \begin{cases}
J^{\R \times \Liouville^{+}}_{f^{+}} & \R_{s} \times \R_{\tau} \times \Liouville_{\sigma_{0}}^{+},\\
\JRIntervalDivSet_{f} & \R_{s} \times \R_{\tau} \times [-1, 1]_{\sigma} \times \divSet,\\
J^{\R \times \Liouville^{-}}_{f^{-}} & \R_{s} \times \R_{\tau} \times \Liouville_{\sigma_{0}}^{-}.
\end{cases}
\end{equation*}
If $\JRIntervalDivSet_{f}$ is compatible with $d\lambdaRIntervalDivSet_{f}$ (as in Lemma \ref{Lemma:JCompatibleOnRDivSetC}) and the $J^{\R \times \Liouville^{\pm}}_{f^{\pm}}$ are compatible with the $d\lambda^{\R \times \Liouville^{\pm}}_{f^{\pm}}$ (as in Lemma \ref{Lemma:JGHCompatible}), then $\JRHypersurface_{f}$ is compatible with $d\lambdaRHypersurface_{f}$. We assume that these compatibility conditions hold.

\subsubsection{Foliations on $\R_{s} \times \R_{\tau} \times \hypersurface$}\label{Sec:FoliationRHypersurface}

Since the almost complex structures agree on the overlapping subsets, we may also patch together the integrable, $\codim=2$ distributions $T\foliationRIntervalDivSet_{f}$ (locally defined as in Lemma \ref{Lemma:FoliationRCDivSet} by viewing $\R_{\tau} \times  [-1,1] \subset \C$) and $T\foliation^{\R \times \Liouville^{\pm}}_{f^{\pm}}$ (as in Lemma \ref{Lemma:FoliationGHProperties}) to define
\begin{equation*}
T\foliationRHypersurface_{f} \subset T(\R_{s} \times \R_{\tau} \times \hypersurface), \quad T\foliationRHypersurface_{f} = \begin{cases}
T\foliation^{\R \times \Liouville^{+}}_{f^{+}}& \R_{s} \times \R_{\tau} \times \Liouville_{\sigma_{0}}^{+},\\
T\foliationRIntervalDivSet_{f} & \R_{s} \times \R_{\tau} \times [-1, 1]_{\sigma} \times \divSet,\\
T\foliation^{\R \times \Liouville^{-}}_{f^{-}} & \R_{s} \times \R_{\tau} \times \Liouville_{\sigma_{0}}^{-}.
\end{cases}
\end{equation*}
The $\codim=2$ submanifolds, $\leaf$, tangent to this distribution are then the leaves of a foliation $\foliationRHypersurface_{f}$ on $\R_{s} \times \R_{\tau} \times \hypersurface$.

\begin{lemma}\label{Lemma:FoliationFGHProperties}
Suppose that $\partial_{\sigma}f^{\pm} > 0$ along each $\{ \sigma = \sigma_{0}\} \subset \R_{s} \times \R_{\tau} \times \Liouville_{\sigma_{0}}^{\pm}$. Then each leaf, $\leaf$, of $\foliationRHypersurface_{f}$ satisfies exactly one of the following three conditions:
\be
\item $\leaf$ is entirely contained in $\R_{s} \times \R_{\tau} \times (-1, 1)\times \divSet$ and is diffeomorphic to $\R \times \divSet$.
\item $\leaf$ intersects both $\R_{s} \times \R_{\tau} \times \Liouville_{\sigma_{0}}^{+}$ and $\R_{s}\times \R_{\tau} \times [-1, 1]_{\sigma} \times \divSet$. The leaf is diffeomorphic to the completion $\complete{\Liouville}^{+}$ of $\Liouville^{+}$ and is disjoint from $\R_{s} \times \R_{\tau} \times \Liouville_{\sigma_{0}}^{-}$.
\item $\leaf$ intersects both $\R_{s} \times \R_{\tau} \times \Liouville_{\sigma_{0}}^{-}$ and $\R_{s}\times \R_{\tau} \times [-1, 1]_{\sigma} \times \divSet$. The leaf is diffeomorphic to the completion $\complete{\Liouville}^{-}$ of $\Liouville^{-}$ and is disjoint from $\R_{s} \times \R_{\tau} \times \Liouville_{\sigma_{0}}^{+}$.
\ee
\end{lemma}

The proof of the lemma follows from inspection of the leaves of the $\foliationRIntervalDivSet_{f}$ and $\foliation^{\R \times \Liouville^{\pm}}_{f^{\pm}}$ described in Lemmas \ref{Lemma:FoliationRCDivSet} and \ref{Lemma:FoliationGHProperties}. The essential point here is that the condition $\partial_{\sigma}f^{\pm} > 0$ allows leaves to enter $\R_{s} \times \R_{\tau} \times [-1, 1]_{\sigma} \times \divSet$ (as their $s$-coordinate increases) but never exit this same region. This is a consequence of
\begin{equation*}
T\foliationRHypersurface_{f} = \R(\partial_{s} + F(\partial_{\tau}f \partial_{\tau} + \partial_{\sigma}f \partial_{\sigma}))\oplus T\divSet, \quad F > 0
\end{equation*}
along the overlapping subset, as described in Lemma \ref{Lemma:FoliationRCDivSet}. Many particular examples will be studied in \S \ref{Sec:GeometryExamples}.

\subsection{Constructing $CH$-admissible sequences}\label{Sec:AdmissibleSequences}

The following lemma describes how a $CH$-admissible sequence $(h_{i}\alphaDivSet, \actionBoundDivSet_{i})$ for $\MxiDivSet$ determines a $CH$-admissible sequence for $(\R_{\tau} \times \hypersurface, \xi)$ whose contact forms are as described in \S \ref{Sec:AlphaRHypersurface}. This will later be important for establishing that the contact forms of \S \ref{Sec:AlphaRHypersurface} can be used to compute contact homology as stated in Theorems \ref{Thm:MainChainLevel} and \ref{Thm:MainHomology}.

\begin{lemma}\label{Lemma:CHSequence}
Let $(h_{i}\alphaDivSet, \actionBoundDivSet_{i})_{i \in \Z_{\geq 0}}$ be a $CH$-admissible sequence for $\MxiDivSet$. Let $\alphaRHypersurface_{f}$ be a contact form on $\R_{\tau} \times \hypersurface$ associated to functions $f, g$ and the contact form $\alphaDivSet$ for $\MxiDivSet$ as described in \S \ref{Sec:AlphaRHypersurface}. Then possibly after passing to a subsequence of $(h_{i}\alphaDivSet, \actionBoundDivSet_{i})$, there are
\be
\item contact forms $\alphaRHypersurface_{f_{i}}$ associated to functions $f_{i},g$ and the contact forms $h_{i}\alphaDivSet$ as described in \S \ref{Sec:AlphaRHypersurface},
\item constants $\actionBound_{i}, s_{i} > 0$, and
\item diffeomorphisms $\Phi_{i}$ of $\R_{\tau} \times \hypersurface$
\ee
such that $(e^{-s_{i}}\Phi_{i}^{\ast}\alphaRHypersurface_{f_{i}}, e^{-s_{i}}\actionBound_{i})$ is a $CH$-admissible sequence for $(\R_{\tau} \times \hypersurface, \xi)$ satisfying the following properties:
\be
\item $s_{i} < s_{i+1}$, and $\lim_{i \to \infty} s_{i} < \infty$.
\item $f_{i} = 1 + c_{i}(f - 1)$ along $\R_{\tau} \times [-1,1]_{\sigma}$ and hence have the same critical points, where $c_{i}$ is a decreasing sequence of constants in $(0,1]$ converging to $0$ with $c_{0}=1$.
\item The restriction of $\alphaRHypersurface_{f_{i}}$ to $\R_{\tau} \times \Liouville^{\pm}_{-3}$ takes the form $3d\tau + \beta_{\Liouville^{\pm}, i}$, where the $\beta_{\Liouville^{\pm}, i}$ restricted to $\partial \Liouville^{\pm}_{-3}$ give $e^{-3}h_{i}\alphaDivSet$.
\ee
\end{lemma}

\begin{proof}
Our strategy is to construct
\be
\item contact forms $\alpha_{i}$ on $\R_{\tau} \times \hypersurface$ as described in \S \ref{Sec:AlphaRHypersurface} associated to functions $f_{i},g$ and the contact forms $h_{i}\alphaDivSet$ and
\item  Liouville forms $\lambda_{i} \in \Omega^{1}([-2\delta_{i}, 2\delta_{i}] \times \R_{\tau} \times \hypersurface)$ interpolating between the $\alpha_{i}$ and $\alpha_{i+1}$
\ee
with $\delta_{i}$ converging to $0$ (sufficiently fast) and the Liouville vector fields $X_{\lambda_{i}}$ having bounded, positive $\partial_{s}$ coefficients. Then following finite-time flows of the $X_{\lambda_{i}}$ starting at the tops of the cobordisms will provide the desired diffeomorphisms $\Phi_{i}$.

After passing to a subsequence of the $(h_{i}\alphaDivSet, \actionBoundDivSet_{i})$ we assume that
\begin{equation*}
\sum \sup_{\divSet}(h_{i} - h_{i+1}) < \infty.
\end{equation*}
We then choose the constants $c_{i}$ to be
\begin{equation}\label{Eq:CIDef}
\begin{gathered}
c_{i} = 1 - \left( \sum_{j \geq 0} \sup_{\divSet}(h_{j} - h_{j+1})\right)^{-1}\sum_{j < i} \sup_{\divSet}(h_{j} - h_{j+1})\\ \implies c_{i} - c_{i+1} = \left( \sum_{j \geq 0} \sup_{\divSet}(h_{j} - h_{j+1})\right)^{-1}\sup_{\divSet}(h_{i} - h_{i+1}).
\end{gathered}
\end{equation}
Clearly $c_{i} \to 0$ as $i \to \infty$. 

We will work separately on the region $\{ \sigma \in (-3, 3)\}$ and the $\R_{\tau} \times \Liouville^{\pm}_{-3}$. Consider a function
\begin{equation*}
a: (-3, 3)_{\sigma} \to [0,1], \quad a|_{(-3,-2)}=0, \quad a|_{[-1,1]} = 1,\quad  a'|_{[-2, -1]} \geq 0, \quad a(-\sigma) = a(\sigma).
\end{equation*}
Using $f$, $a$, and the constants $c_{i}$, define functions
\begin{equation*}
f_{i}: \R_{\tau} \times (-3,3) \to \R_{>0},\quad  f_{i} = (1 - a)e^{-|\sigma|} + a(1 + c_{i}(f - 1)).
\end{equation*}
Using the functions $f_{i}$, define
\begin{equation*}
\alpha_{i} = f_{i}h_{i}\alphaDivSet + gd\tau \in \Omega^{1}(\R_{\tau} \times (-3,3)_{\sigma} \times \divSet).
\end{equation*}
It follows that the $\alpha_{i}$ are contact where they are so far defined.

Let $\delta_{i} > 0$ be some positive constant, to be later specified, and let $B$ be a function
\begin{equation*}
\begin{gathered}
B: [-2\delta_{i}, 2\delta_{i}]_{s} \to [0, 1], \quad B(-2\delta_{i}) = 0, \quad B(2\delta_{i}) = 1,\\
B' \in [0, \delta_{i}^{-1}], \quad |s| \geq \delta_{i} \implies B'=0.
\end{gathered}
\end{equation*}
Define a $1$-form
\begin{equation*}
\begin{gathered}
\lambda_{i} = e^{s}(F\alphaDivSet + gd\tau) \in \Omega^{1}([-2\delta_{i}, 2\delta_{i}] \times \R_{\tau} \times (-3, 3)_{\sigma} \times \divSet),\\
F = Bf_{i}h_{i} + (1-B)f_{i+1}h_{i+1} \implies \partial_{s}F = B'(f_{i}h_{i} - f_{i+1}h_{i+1}).
\end{gathered}
\end{equation*}
Recall that $g$ depends only on $\sigma$ and that for $\sigma \in (-3,3)$, $g' < 0$. We may apply Lemma \ref{Lemma:LiouvilleRDivSetC} to $\lambda_{i}$ with
\begin{equation*}
\beta = gd\tau, \quad X_{\beta} = \frac{g}{g'}\partial_{\sigma}.
\end{equation*}
According to the lemma, $d\lambda_{i}$ is symplectic iff
\begin{equation}\label{Eq:dLambdaiSymp}
\begin{aligned}
0 &< F - \beta(X_{F}) + \partial_{s}F = F - \frac{g}{g'}\partial_{\sigma}F +\partial_{s}F\\
&= B(f_{i} - \frac{g}{g'}\partial_{\sigma}f_{i})h_{i} + (1-B)(f_{i+1} - \frac{g}{g'}\partial_{\sigma}f_{i+1})h_{i+1} + B'(f_{i}h_{i} - f_{i+1}h_{i+1}).
\end{aligned}
\end{equation}
Only the last summand above, $\partial_{s}F$, can be negative and it can only be negative along $\{ \sigma \in (-2, 2)\}$. The symplectic condition would be implied by the existence of a constant $K_{0} > 0$ for which
\begin{equation*}
\begin{aligned}
|\partial_{s}F| &= B'|f_{i}h_{i} - f_{i+1}h_{i+1}| \leq \delta_{i}^{-1}|f_{i}h_{i} - f_{i+1}h_{i+1}| \leq K_{0}\\
&< B(f_{i} - \frac{g}{g'}\partial_{\sigma}f_{i})h_{i} + (1-B)(f_{i+1} - \frac{g}{g'}\partial_{\sigma}f_{i+1})h_{i+1}
\end{aligned}
\end{equation*}
point-wise along the set $\{ \sigma \in (-2,2)\}$. It is easy to check that there is a $K_{0}$ satisfying the last inequality over all of $\{ \sigma \in (-2,2)\}$ independent of the index $i$. Then we compute
\begin{equation*}
\begin{aligned}
f_{i}h_{i} - f_{i+1}h_{i+1} &= (h_{i} - h_{i+1})(a + (1-a)e^{-|\sigma|}) + a(f-1)(c_{i}h_{i} - c_{i+1}h_{i+1}),\\
|f_{i}h_{i} - f_{i+1}h_{i+1}| &\leq (h_{i} - h_{i+1}) + \sup_{\{ \sigma \in [-2,2]\}}(|f-1|)(c_{i}h_{i} - c_{i+1}h_{i+1})\\
&= (h_{i} - h_{i+1}) + \sup_{\{ \sigma \in [-2,2]\}}(|f-1|)(c_{i}(h_{i}-h_{i+1}) + h_{i+1}(c_{i} - c_{i+1}))\\
&< K_{1}\sup_{\divSet}(h_{i} - h_{i+1})
\end{aligned}
\end{equation*}
for some constant $K_{1}$ which is independent of $i$. The existence of $K_{1}$ depends on the choices of the $c_{i}$ in Equation \eqref{Eq:CIDef}. Therefore if we choose
\begin{equation}\label{Eq:DeltaIDef}
\delta_{i} = K\sup_{\divSet}(h_{i} - h_{i+1}), \quad K \geq K_{1}/K_{0}
\end{equation}
then the symplectic condition is met over $\{ |\sigma| < 2\}$ and clearly holds on $\{ |\sigma| \in (2,3)\}$.

The Liouville vector field for $\lambda_{i}$ along $[-2\delta_{i}, 2\delta_{i}] \times \R_{\tau} \times (-3,3)_{\sigma} \times \divSet$ is then given by
\begin{equation}\label{Eq:LiouvilleInterpolation1}
\begin{gathered}
X_{\lambda_{i}} = \frac{(F - \beta(X_{F}))\partial_{s}+(\partial_{s}F)X_{\beta}}{F - \beta(X_{F}) + \partial_{s}F} = \frac{(F - \frac{g}{g'}\partial_{\sigma}F)\partial_{s}+\frac{g}{g'}(\partial_{s}F)\partial_{\sigma}}{F - \frac{g}{g'}\partial_{\sigma}F +\partial_{s}F}.
\end{gathered}
\end{equation}
From the above computations regarding the bounds on $|\partial_{s}F|$ we see that the $\partial_{s}$ component of $X_{\lambda_{i}}$ is always positive, bounded below, and $\norm{X_{\lambda_{i}} - \partial_{s}}$ is $\CZero$-bounded over $\{ |\sigma| < 3\}$ uniformly for all $i$.

Now we extend the construction over $\R_{\tau} \times \Liouville_{-3}^{+}$ and $[-2\delta_{i}, 2\delta_{i}] \times \R_{\tau} \times \Liouville_{-3}^{+}$. We have already specified that our contact forms $\alpha_{i}$ take the form
\begin{equation*}
\alpha_{i} = e^{\sigma}h_{i}\alphaDivSet + gd\tau
\end{equation*}
along $\{ \sigma \in [-3,-2]\}$. Recall that $g = 3$ for $\sigma \leq -3$. Identify $\R_{\sigma \leq -3} \times \divSet \subset \Liouville^{+}_{-3}$ as the collar obtained by flowing along $-X_{\beta_{\Liouville^{+}}}$. Let $C \gg 3$ be a constant and let $b$ be a function
\begin{equation*}
b: [-C, -3] \to [0,1], \quad b(-C) = 0, \quad b(-3) = 1, \quad b' \in [0, 2C^{-1}]
\end{equation*}
with $b$ constant near the endpoints of the interval. We extend the definition of the $\alpha_{i}$ over $\R_{\tau} \times [-C, -3]_{\sigma} \times \divSet \subset \R_{\tau} \times \Liouville^{+}_{-3}$ as
\begin{equation*}
\begin{gathered}
\alpha_{i} = H_{i}e^{\sigma}\alphaDivSet + 3d\tau, \quad H_{i} = 1-b + bh_{i}.
\end{gathered}
\end{equation*}
To ensure that $\alpha_{i}$ is contact over this region, we may assume that
\begin{equation*}
0 < \partial_{\sigma}(H_{i}e^{\sigma}) = e^{\sigma}(H_{i} + b'(h_{i}-1))
\end{equation*}
for all $i$ by making $C$ sufficiently large, using the fact that the $h_{i}$ have a $\CZero$ limit to a strictly positive function $h_{\infty} \in \CZero(\divSet, \R_{>0})$. Specifically, this is possible for
\begin{equation*}
C > 2(\min h_{\infty})^{-1}(1 - \min h_{\infty}).
\end{equation*}
The $1$-form $\alpha_{i}$ naturally extends over $\R_{\tau} \times \Liouville_{-C}^{+}$ as
\begin{equation*}
\alpha_{i} = \beta_{\Liouville^{+}} + 3d\tau.
\end{equation*}
We extend the definition of $\lambda_{i}$ over $[-2\delta_{i}, 2\delta_{i}]_{s} \times \R_{\tau} \times [-C, -3]_{\sigma} \times \divSet$ as
\begin{equation*}
\lambda_{i} = e^{s}(3d\tau + H\alphaDivSet), \quad H = (BH_{i} + (1-B)H_{i+1})e^{\sigma}.
\end{equation*}
Along $[-2\delta_{i}, 2\delta_{i}]_{s} \times \R_{\tau} \times [-C, -3]_{\sigma} \times \divSet$, we compute
\begin{equation*}
\begin{gathered}
d\lambda_{i} = e^{s}(ds \wedge (3d\tau + (H + \partial_{s}H)\alphaDivSet) + (\partial_{\sigma}H)d\sigma \wedge \alphaDivSet + d_{\divSet}(H\alphaDivSet)),\\
(e^{-s}d\lambda_{i})^{n+1} = 3n(n+1)H^{n-1}(\partial_{\sigma}H) d\tau \wedge d\sigma \wedge ds \wedge \alphaDivSet \wedge (d\alphaDivSet)^{n-1}.
\end{gathered}
\end{equation*}
Therefore $d\lambda_{i}$ is symplectic on this region where its Liouville vector field is given by
\begin{equation}\label{Eq:LiouvilleInterpolation2}
\begin{aligned}
X_{\lambda_{i}} = \partial_{s} - \frac{\partial_{s}H}{\partial_{\sigma}H}\partial_{\sigma} = \partial_{s} - B'b\frac{h_{i} - h_{i+1}}{1 + (b + b')(Bh_{i} + (1-B)h_{i+1} - 1)}\partial_{\sigma}.
\end{aligned}
\end{equation}
Again, we see that the $\partial_{s}$ component of $X_{\lambda_{i}}$ is positive and that the $\CZero$ norm of $X_{\lambda_{i}} - \partial_{s}$ is bounded independent of $i$.

Over $\R_{\tau} \times \Liouville^{+}_{-C}$ and $[-2\delta_{i}, 2\delta_{i}] \times \R_{\tau} \times \Liouville^{+}_{-C}$, we respectively define
\begin{equation*}
\alpha_{i} = \beta_{\Liouville^{+}} + 3d\tau, \quad \lambda_{i} = e^{s}\alpha_{i} \implies X_{\lambda_{i}} = \partial_{s}.
\end{equation*}
We may extend the definition of $\alpha_{i}$ and $\lambda_{i}$ over $\R_{\tau} \times \Liouville^{-}_{-3}$ and $[-2\delta_{i}, 2\delta_{i}] \times \R_{\tau} \times \Liouville^{-}_{-3}$, respectively in a symmetric fashion.

In summary, we have constructed Liouville forms $\lambda_{i}$ on $[-2\delta_{i}, 2\delta_{i}] \times \R_{\tau} \times \hypersurface$ which interpolate between the $\alpha_{i}$ and satisfy the following conditions:
\be
\item The $\lambda_{i}$ are $s$-invariant along $\{ |s| > \delta_{i}\}$.
\item The $\delta_{i}$ converge to $0$ as $i\to \infty$ and satisfy $\sum_{i} \delta_{i} = K\sum \sup_{\divSet}(h_{i} - h_{i+1}) < \infty$ as described in Equation \eqref{Eq:DeltaIDef}.
\item The $X_{\lambda_{i}}$ are linear combinations of $\partial_{s}$ and $\partial_{\sigma}$.
\item The $\partial_{s}$ component of $X_{\lambda_{i}}$ is positive and bounded below, and $\norm{X_{\lambda_{i}} - \partial_{s}}_{\CZero}$ is bounded independent of $i$.
\ee
Extend the $\lambda_{i}$ as $e^{s}\alpha_{i+1}$ over $\R_{s < -2\delta_{i}} \times \R_{\tau} \times \hypersurface$. Consider maps
\begin{equation*}
\begin{gathered}
\Psi_{i}: \R_{s \leq 0} \times \R_{\tau} \times \hypersurface \to \R_{s<2\delta_{i}} \times \R_{\tau} \times \hypersurface,\\
\Psi_{i}(s, x) = \Flow^{s}_{X_{\lambda_{i}}}(2\delta_{i},x), \quad \Psi_{i}^{\ast}(e^{-2\delta_{i}}\lambda_{i}) =e^{s}\alpha_{i}.
\end{gathered}
\end{equation*}
We choose $\delta_{i}' > 0$ so that
\begin{equation*}
\Psi_{i}(\{ s = -\delta'_{i}\}) \subset \{ s \leq -\delta_{i}\},
\end{equation*}
where $\lambda_{i} = e^{s}\alpha_{i+1}$. By the last condition above, we may assume that $\delta_{i}' \leq K'\delta_{i}$ for some $K' > 0$. Write $\phi_{i}$ for the map obtained by composing $\Psi_{i}|_{\{ s=-\delta'_{i}\}}$ with the projection to $\R_{\tau} \times \hypersurface$. By the third item above, such maps are diffeomorphisms, since they preserve the $\tau$-level sets. It follows that $\Psi^{-1}_{i}(\{ s = -\delta_{i}\})$ is the graph of a function on $\R_{\tau} \times \hypersurface$ taking values in $(-\delta_{i}', 0)$ inside of the domain of $\Psi_{i}$. The lemma is then established with
\begin{equation*}
\Phi_{i} = \phi_{i-1}\circ \cdots \circ \phi_{0}, \quad s_{i} = \sum_{j \leq i}\delta_{j}'
\end{equation*}
and then possibly passing to another subsequence to ensure that $e^{-s_{i}}\actionBound_{i} < e^{-s_{i+1}}\actionBound_{i+1}$.
\end{proof}

\section{Leaf notation and asymptotics}\label{Sec:LeafNotationAsymptotics}

In this section we set up notation for leaves of the foliations described in \S \ref{Sec:DynamicsAndFoliations} and collect results on ends of holomorphic curves asymptotic to the Reeb orbits, $\orbit_{\zeta}$, of interest. The notation and results obtained will help us to state and prove Lemma \ref{Lemma:CurvesStuckInLeaves} below, which says that (most of) the perturbed holomorphic curves which we'll need to compute contact homology are contained in the leaves of our foliations.

\subsection{Leaf notation}\label{Sec:LeafNotation}

Here we develop notation for the leaves of the foliations $\foliation$ on $\R_{s} \times M$ where either
\be
\item $\foliation = \foliationCDivSet_{f}$ and $M=\C \times \divSet$ (Lemma \ref{Lemma:FoliationRCDivSet}),
\item $\foliation = \foliationRLiouville_{f}$ and $M=\R_{\tau} \times \Liouville_{\sigma_{0}}$ (Lemma \ref{Lemma:FoliationGHProperties}), or
\item $\foliation = \foliationRHypersurface_{f}$ and $M=\R_{\tau} \times \hypersurface$ (Lemma \ref{Lemma:FoliationFGHProperties}).
\ee
We also write $\subsetC$ for
\be
\item $\C$ in the case $M=\C \times \divSet$,
\item $\R_{\tau} \times [-1, \sigma_{0}]_{\sigma}$ in the case $M = \R_{\tau} \times \Liouville_{\sigma_{0}}$, and
\item $\R_{\tau} \times [-1, 1]$ in the case $M=\R_{\tau} \times \hypersurface$.
\ee
Having common notation for these three cases will allow us to reduce redundancies in case-by-case analyses. In any of the above cases we write $J$ for the almost complex structure on $\R_{s} \times M$ which preserves the tangent distribution $T\foliation$ of the foliation $\foliation$. Throughout, we write $s_{0} \in \R_{>0}$ for a constant such that $J$ and the function $f \in \Cinfty(\R_{s} \times \subsetC, \R_{>0})$ are both $s$-translation invariant along the subsets $\{ |s| > s_{0}\} \subset \R_{s} \times M$.

\begin{defn}
We say that a leaf $\leaf$ of such a foliation on $\R_{s} \times M$ is \emph{bounded} if the image of the projection of the leaf to $M$ is disjoint from $\partial M$ and has compact closure. Otherwise we say that $\leaf$ is \emph{unbounded}.
\end{defn}

For any leaf $\leaf$ of $\foliation$, there is some $s_{1} > s_{0}$ such that
\begin{equation*}
\leaf \cap (\R_{s \geq s_{1}} \times M) = \im \phi \times \divSet \subset \R_{s \geq s_{1}} \times \subsetC \times \divSet.
\end{equation*}
where $\phi(s)=(s, z(s)), s \in [s_{1}, \infty)$ is such that $z(s)$ is a flow line of a vector field $F \grad_{f(s_{0}, \ast)}$ on $\subsetC$.

\begin{defn}
If $z(s)$ as above converges to some 
\begin{equation*}
\zeta \in \Crit(f(s_{0}, \ast)) \subset \subsetC
\end{equation*}
as $s \to \infty$, then we say that $\leaf$ is \emph{positively asymptotic to $\zeta$}.
\end{defn}

If $\leaf$ is positively asymptotic to some $\zeta$, then $\leaf$ must be bounded and diffeomorphic to either $\R_{s} \times \divSet$ or a completed Liouville manifold ($\complete{\Liouville}$ for $M = \R_{\tau} \times \Liouville_{\sigma_{0}}$ or one of the $\complete{\Liouville}^{\pm}$ for $M = \R_{\tau} \times \hypersurface$).

\begin{defn}
If $\leaf$ is bounded, diffeomorphic to $\complete{\Liouville}$, and positively asymptotic to $\zeta \in \Crit(f(s_{0}, \ast))$ as above, we say that $\leaf$ is a \emph{leaf of type $(\Liouville, \zeta)$}.
\end{defn}

If a leaf, $\leaf$, is not diffeomorphic to $\complete{\Liouville}$ then there is some $s_{1} > s_{0}$ such that
\begin{equation*}
\leaf \cap (\R_{s \leq - s_{1}} \times M) = \im \phi \times \divSet \subset \R_{s \leq - s_{1}} \times \subsetC \times \divSet.
\end{equation*}
where $\phi(s)=(s, z(s)), s \in (-\infty, -s_{1}]$ is such that $z(s)$ is a flow line of a vector field $F \grad_{f(-s_{0}, \ast)}$ on $\subsetC$.

\begin{defn}
If $z(s)$ as above converges to some 
\begin{equation*}
\zeta \in \Crit(f(-s_{0}, \ast)) \subset \subsetC
\end{equation*}
as $s \to -\infty$, then we say that $\leaf$ is \emph{negatively asymptotic to $\zeta$}.
\end{defn}

\begin{defn}
If $\leaf$ is bounded, diffeomorphic to $\R_{s} \times \divSet$, positively asymptotic to some $\zeta^{+} \in \Crit(f(s_{0}, \ast))$, and negatively asymptotic to some $\zeta^{-} \in \Crit(f(-s_{0}, \ast))$, then we say that $\leaf$ is a \emph{leaf of type $(\zeta^{-}, \zeta^{+})$}.
\end{defn}

\subsection{Assumptions on the function $f$}\label{Sec:fAssumptions}

Here we impose further assumptions on the functions $f\in \Cinfty(\R_{s} \times \subsetC, \R_{>0})$. Throughout, we write $\disk(\zeta, \rho)$ for a disk in $\subsetC$ of radius $\rho > 0$ centered about some $\zeta \in \subsetC$ and
\begin{equation*}
\Crit_{i}(f(\pm s_{0}, \ast)) \subset \Crit(f(\pm s_{0}, \ast))
\end{equation*}
for the critical points of $\ind_{\Morse} = i$. 

\begin{assump}\label{Assump:MorseLocalModel}
We assume that
\begin{equation*}
\begin{gathered}
\tau > 0 \implies \partial_{\tau}f(s, \tau, \sigma) \geq 0,\\
f(s, \tau, \sigma) = f(s, -\tau, \sigma) \implies \Crit(f(\pm s_{0}, \ast)) \subset \{ \tau = 0\}.
\end{gathered}
\end{equation*}
We assume that for each critical point
\begin{equation*}
\zeta = (0, \sigma_{\zeta})\in \Crit(f(\pm s_{0}, \ast))
\end{equation*}
there is a disk $\disk(\zeta, \rho_{\zeta}) \subset \subsetC$ within which the function $f(\pm s_{0}, \ast)$ agrees with one of
\begin{equation*}
\begin{gathered}
f_{h, \epsilon_{\zeta}} =f_{0, \zeta} +  \frac{\epsilon_{\zeta}}{2}(\tau^{2} - (\sigma - \sigma_{\zeta})^{2}) \quad \text{or} \quad f_{e, \epsilon_{\zeta}} = f_{0, \zeta} + \frac{\epsilon_{\zeta}}{2}(\tau^{2} + (\sigma-\sigma_{\zeta})^{2}),\\
f_{0, \zeta} \in \R_{>0}, \quad \epsilon_{\zeta} \in \R_{> 0}.
\end{gathered}
\end{equation*}
Therefore, in such a $\disk(\zeta, \rho_{\zeta})$ neighborhood, we may write
\begin{equation}\label{Eq:DiskXYCoords}
\begin{gathered}
f(\pm s_{0}, \ast)|_{\disk(\zeta, \rho_{\zeta})} = f_{0, \zeta} + \frac{\epsilon_{\zeta}}{2}(x^{2} + \delta_{\zeta} y^{2}),\\
\delta_{\zeta} \in \{ \pm 1\}, \quad x=\tau, \quad y=\sigma-\sigma_{\zeta}.
\end{gathered}
\end{equation}
It follows that $\zeta$ is given by the point $x=y=0$ in such a local coordinate. We'll assume that
\begin{equation}\label{Eq:ConstantAtEllipticAssumption}
\ind_{\Morse}(f(\pm s_{0}, \ast), \zeta) = 0 \implies f_{0, \zeta} = 1.
\end{equation}
We assume that $\epsilon_{\zeta}$ has arbitrarily small absolute value,
\begin{equation*}
0 < | \epsilon_{\zeta} | \ll 1,
\end{equation*}
and will provide a more precise bound in Equation \eqref{Eq:EpsilonZetaBound}. We assume that $d\beta \in \Omega^{2}(\subsetC)$ agrees with the standard symplectic form $d\beta = dx\wedge dy$ in this local model. Consequently, the gradient and Hamiltonian vector fields for $f(\pm s_{0}, \ast)$ are related by
\begin{equation*}
\grad_{f(\pm s_{0}, \ast)} = \epsilon_{\zeta}(x\partial_{x} + \delta_{\zeta} y\partial_{y}), \quad X_{f(\pm s_{0}, \ast)} = \epsilon_{\zeta}(-\delta_{\zeta} y\partial_{x} + x\partial_{y})
\end{equation*}
and the Reeb vector field of $f\alphaDivSet + \beta$ in this neighborhood is computed
\begin{equation}\label{Eq:ReebNearCritPoint}
\begin{gathered}
\Reeb = H^{-1}(\ReebDivSet - X_{f(\pm s_{0}, \ast)})= H^{-1}(\ReebDivSet + \epsilon_{\zeta}(\delta_{y}y\partial_{x} - x\partial_{y})),\\
H = f(\pm s_{0}, \ast) - \beta(X_{ f(\pm s_{0}, \ast)}).
\end{gathered}
\end{equation}
Finally, we assume that certain spaces of $s$-dependent gradient trajectories are empty
\begin{equation}\label{Eq:EmptyTrajectorySpace}
\begin{gathered}
\zeta^{+} \in \Crit_{0}(f(s_{0}, \ast)), \quad \zeta^{-} \in \Crit_{1}(f(-s_{0}, \ast))\\
\implies \left\{ z(s): \R \to \subsetC\ :\  \partial_{s}z(s) = \grad_{f(s, \ast)}z(s), \ \lim_{s \to \pm \infty}z(s)=\zeta^{\pm} \right\} = \emptyset.
\end{gathered}
\end{equation}
\end{assump}

The condition $\epsilon_{\zeta} > 0$ in the above assumptions implies
\begin{equation*}
\Crit(f(\pm s_{0}, \ast)) = \Crit_{0}(f(\pm s_{0}, \ast)) \sqcup \Crit_{1}(f(\pm s_{0}, \ast)).
\end{equation*}
For a critical point $\zeta^{\pm} \in \Crit(f(\pm s_{0}, \ast))$ we write the associated stable and unstable manifolds as
\begin{equation*}
\begin{gathered}
\Stable(f(\pm s_{0}, \ast), \zeta^{\pm}) = \{ z \in \subsetC\ :\ \lim_{t\to \infty}\Flow^{t}_{\grad_{f(\pm s_{0}, \ast)}}z = \zeta^{\pm}\},\\
\Unstable(f(\pm s_{0}, \ast), \zeta^{\pm}) = \{ z \in \subsetC\ :\ \lim_{t\to -\infty}\Flow^{t}_{\grad_{f(\pm s_{0}, \ast)}}z = \zeta^{\pm}\}.
\end{gathered}
\end{equation*}
Recall that $\partial \subsetC$ is non-empty for $M = \R_{\tau} \times \Liouville$ and $M = \R_{\tau} \times \hypersurface$. We write
\begin{equation*}
\begin{gathered}
\partial_{out}\subsetC = \begin{cases}
\emptyset & M = \C \times \divSet,\\
\{ \sigma = \sigma_{0}\} & M=\R_{\tau} \times \Liouville_{\sigma_{0}},\\
\emptyset & M = \R_{\tau} \times \hypersurface,
\end{cases}\quad
\partial_{in}\subsetC = \begin{cases}
\emptyset & M = \C \times \divSet,\\
\{ \sigma = -1\} & M = \R_{\tau} \times \Liouville_{\sigma_{0}},\\
\{ |\sigma| = 1\} & M = \R_{\tau} \times \hypersurface
\end{cases}
\end{gathered}
\end{equation*}
for the components of $\partial \subsetC$ along which $\grad_{f(s, \ast)}$ points out of and into $\subsetC$, respectively. See Figures \ref{Fig:AlphaRLiouvilleAssumptions} and \ref{Fig:AlphaRHypersurfaceAssumptions} for an overview of gradient behavior near $\partial \subsetC$ when it is non-empty. We write
\begin{equation*}
\begin{gathered}
\Stable(f(\pm s_{0}, \ast), \partial_{out}\subsetC) = \{ z \in \subsetC\ :\ \exists t \geq 0,\ \Flow^{t}_{\grad_{f(\pm s_{0}, \ast)}}z \in \partial_{out}\subsetC\},\\
\Unstable(f(\pm s_{0}, \ast), \partial_{in}\subsetC) = \{ z \in \subsetC\ :\ \exists t \leq 0,\ \Flow^{t}_{\grad_{f(\pm s_{0}, \ast)}}z \in \partial_{in}\subsetC\}.
\end{gathered}
\end{equation*}

The proof of the following lemma is immediate from the above assumptions and the behavior of the functions $f$ away from critical points described in Assumptions \ref{Assump:AlphaCDivSet}, \ref{Assump:AlphaRLiouville2}, and \ref{Assump:AlphaRHypersurface} and summarized in Figures \ref{Fig:AlphaCDivSetAssumptions}, \ref{Fig:AlphaRLiouvilleAssumptions}, and \ref{Fig:AlphaRHypersurfaceAssumptions}.

\begin{lemma}\label{Lemma:LeafExitCompact}
For each $\zeta \in \Crit(f(\pm s_{0}, \ast))$ of $\ind_{\Morse}=1$,
\begin{equation*}
\Stable(f(\pm s_{0}, \ast), \zeta) \subset \{ \tau = 0\} \subset \subsetC.
\end{equation*}
For each $z_{0} \in \subsetC$ for which $z_{0} \notin \{ \tau = 0\}$ and any $s_{1} \in \R$, $z_{0}$ is disjoint from the stable manifolds of the $f(\pm s_{0}, \ast)$. Hence the $(t\geq 0)$-parameterized flow lines $z(t)$ of $\grad_{f(\pm s_{0}, \ast)}$ and $\grad_{f(s + t, \ast)}$ with $z(0)=z_{0}$
\be
\item touch $\partial_{out}\subsetC$ in finite time or
\item are defined for all time and eventually exit any subset of the form $\{ |\tau| \leq \tau_{0}\}$, $\tau_{0} > 0$.
\ee
Therefore, provided a compact subset
\begin{equation*}
K \subset \subsetC, \quad K \cap \partial_{out}\subsetC = \emptyset
\end{equation*}
and a point $z_{0} \in \subsetC \setminus \{ \tau = 0\}$, each of the flow lines $z(t)$ described above are such that there is a $t_{0} \in \R_{>0}$ for which $z(t) \notin K$ for $t > t_{0}$.
\end{lemma}

By similar reasoning, we find that the trajectories described in Equation \eqref{Eq:EmptyTrajectorySpace} must be constrained to $\{ \tau = 0\}$. The expected dimension of this space of trajectories is $-1$ and its emptiness is generic in the space of $s$-dependent functions $f|_{\{ \tau 0 \}}$.

The following lemma is clear from the correspondence between leaves of $\foliation$ and gradient flow lines, the fact that $\grad_{f}$ is independent of $s$ for $s \leq -s_{0}$, and the fact that leaves touching $\R_{s} \times \partial_{in}\subsetC \times \divSet$ must be diffeomorphic to a (possibly completed) Liouville domain as described in Lemmas \ref{Lemma:FoliationGHProperties} and \ref{Lemma:FoliationFGHProperties}.

\begin{lemma}\label{Lemma:BoundedUnboundedLeaves}
A leaf, $\leaf$, is negatively asymptotic to $\zeta^{-} \in \Crit(f(-s_{0}, \ast))$ iff
\begin{equation*}
\leaf \cap (\{ s \} \times M) \subset \{s \} \times \Unstable(f(-s_{0}, \ast), \zeta^{-}) \times \divSet\ \forall s \leq -s_{0}.
\end{equation*}
A leaf, $\leaf$, non-trivially intersects $\{-s_{0}\} \times \subsetC \times \divSet$ and is contained in some $\{ s > s_{1} \}$ iff
\begin{equation*}
\leaf \cap (\{ -s_{0} \} \times M) \subset \{-s_{0} \} \times \Unstable(f(-s_{0}, \ast), \partial_{in}\subsetC) \times \divSet.
\end{equation*}
A leaf is unbounded and intersects $\{ s_{0}\} \times \subsetC \times \divSet$ iff
\begin{equation*}
\begin{gathered}
\leaf \cap \left(\{ s_{0}\} \times \subsetC \times \divSet \right) = \{s_{0}\} \times \{z\} \times \divSet,\\
z \notin \Crit_{0}(f(s_{0}, \ast)) \cup \left( \cup_{\zeta \in \Crit_{1}(f(s_{0}, \ast))} \Stable(f(s_{0}, \ast), \zeta)\right)
\end{gathered}
\end{equation*}
\end{lemma}

\subsection{Asymptotics of holomorphic curves}

Here we study asymptotics of holomorphic curves as they approach Reeb orbits of the form $\orbit_{\zeta}$ for $\zeta \in \Crit(f(\pm s_{0}, \ast)) \subset \subsetC$. To do so, we work with a local model as described in Equation \eqref{Eq:DiskXYCoords}.

\subsubsection{Local models near a critical point}\label{Sec:AsymptoticSetup}

We employ the following notation to set up our local models:
\begin{equation}\label{Eq:LocalModelSetup}
\begin{gathered}
X_{\pm} = \R_{\pm s \geq s_{0}} \times \disk_{x,y} \times \divSet,\\
J\partial_{s} = F(\ReebDivSet - X_{f}), \quad J|_{\xiDivSet} = J_{\divSet}, \quad J\partial_{x} = \partial_{y} \quad \implies \quad J\ReebDivSet = -F^{-1}\partial_{s} - \grad_{f},\\
T\foliation = \R(\partial_{s} + F\grad_{f}) \oplus T\divSet = \R(\partial_{s} + F\grad_{f})\oplus \R F\ReebDivSet \oplus \xiDivSet^{J},\\
\alphaDivSet^{J} \in \Omega^{1}(\divSet), \quad \alphaDivSet^{J}(\ReebDivSet) = 1, \quad \alphaDivSet^{J}|_{\xiDivSet^{J}} = 0, \quad \pi_{\xiDivSet^{J}} = \Id_{T\divSet} - \ReebDivSet \otimes \alphaDivSet^{J},\\
\grad_{f} = \partial_{x}f \partial_{x} + \partial_{y}f\partial_{y}, \quad X_{f} = J_{0}\grad_{f} = -\partial_{y}f \partial_{x} + \partial_{x}f\partial_{y}.
\end{gathered}
\end{equation}
Here $\ReebDivSet$ is the Reeb vector field associated to a contact form $\alphaDivSet$ on $\divSet$. We have a distribution $\xiDivSet^{J} \subset T\divSet$ equipped with an almost complex structure $J_{\divSet}$. We have functions $f \in \Cinfty(\disk)$ (as in Equation \eqref{Eq:DiskXYCoords}) and $F \in \Cinfty(\divSet)$ taking values in $\R_{>0}$. By Equation \eqref{Eq:ReebNearCritPoint}, $J\partial_{s}$ is a positive function multiple of the Reeb vector field in the local model of Assumptions \ref{Assump:MorseLocalModel}. For notational simplicity, we ignore the radius, $\rho_{\zeta}$, of the disk. It follows that $T\foliation$ is integrable, and the leaves of the associated foliation $\foliation$ on $X_{\pm}$ are given by the images of maps
\begin{equation*}
(s, m) \mapsto (s, z(s), m), \quad m \in \divSet, \quad \partial_{s}z(s) = F(m)\grad_{f}
\end{equation*}
whose domains are intervals in $\R$ times $\divSet$. As we only care for local behavior near $\{ x=y=0\}$, we will not always specify these intervals.

We write $(\halfcyl_{\pm}, \domainJ)$ for a half-infinite cylinder
\begin{equation}\label{Eq:HalfCylDef}
\halfcyl_{\pm} = \R_{\pm p \geq 0} \times \Circle_{q}, \quad \domainJ \partial_{p} = \partial_{q}.
\end{equation}
We'll assume that $\Circle_{q}$ has length $2\pi$. We are interested in studying $(J, \domainJ)$-holomorphic maps
\begin{equation*}
u: \halfcyl_{\pm} \to X_{\pm}, \quad u = (s, u_{\disk}, u_{\divSet})
\end{equation*}
such that the $(u_{\disk}, u_{\divSet})$ are asymptotic to closed $f_{0, \zeta}^{-1}(\ReebDivSet - X_{f})$ orbits in $\{ 0\} \times \divSet$ as $p \to \pm \infty$. We emphasize here that the $\pm$ signs on the domain and target always match. Such maps will give the restrictions of perturbed holomorphic maps to ends of punctured Riemann surfaces along which the perturbation will vanish. See Assumptions \ref{Assump:PertCompactSupport}. We write $\orbit_{0} = \{ 0\} \times \orbit$ for the closed orbit to which the holomorphic map is asymptotic where $\orbit$ is a closed $\ReebDivSet$ orbit in $\divSet$.

We note that the projection
\begin{equation*}
\pi_{\R \times \divSet}: X_{\pm} \to \R_{\pm s \geq s_{0}} \times \divSet
\end{equation*}
is holomorphic when the target is equipped with the almost complex structure
\begin{equation*}
\JDivSet \partial_{s} = F\ReebDivSet, \quad \JDivSet|_{\xiDivSet^{J}} = J_{\divSet}.
\end{equation*}
Thus $\pi_{\R\times \divSet}\circ u$ is $(\JDivSet, \domainJ)$-holomorphic.

\begin{assump}
We assume that $F = f_{0, \zeta}^{-1}$ over neighborhoods of all closed $\ReebDivSet$ orbits of action at most $\actionBoundDivSet$. Therefore, along over such neighborhoods $F\ReebDivSet$ agrees with the Reeb vector field of $f_{0, \zeta}\alphaDivSet$.
\end{assump}

\subsubsection{Compatibility with the almost complex structures of \cite{BH:ContactDefinition, Pardon:Contact} and asymptotic formulae}

If $\JDivSet$ follows a ``classical'' SFT setup (whence $\xiDivSet^{J}=\xi$ and $F$ is constant over all of $\divSet$) as in \cite{Pardon:Contact}, then $\pi_{\R \times \divSet}u$ must obey the asymptotic formula of \cite{HWZ:Asymptotics}. Specifically, by \cite[Proposition 4.3]{HWZ:Asymptotics} there is a constant $s_{1}$ such that
\begin{equation}\label{Eq:dsDecay}
\begin{gathered}
s = s_{1} + \frac{\action}{2\pi} p + s_{\hot}, \quad s_{1} \in \R, \quad \action = \action(\orbit_{0}) = f_{0, \zeta}\action(\orbit), \quad s_{\hot} \in \bigO\left(e^{\mp|\eigenBound|s}\right)
\end{gathered}
\end{equation}
where $\action$ is the action of $\orbit_{0}$ and $\eigenBound$ is the eigenvalue of smallest absolute value of the asymptotic operator associated to the pair $(\orbit,\JDivSet)$. For further details on asymptotic operators, see \cite{HWZ:Asymptotics}.

Now we describe the compatibility of our $J$ with the simple almost complex structures of \cite{BH:ContactDefinition} which we have reviewed in Definition \ref{Def:SimpleJ}. Suppose that $\orbitDivSet$ is an embedded orbit of $\ReebDivSet$ of action $a$ with a simple neighborhood $\NorbitDivSet$ (Definition \ref{Def:SimpleEnough}) along which $\JDivSet$ is simple. Then we may write
\begin{equation*}
\begin{gathered}
\NorbitDivSet \subset \R_{t} \times \R^{2n_{0}}_{z_{0}} \times \R^{2n_{-}}_{z_{-}}/\sim, \quad (a, z_{0}, z_{-}) \sim (0, z_{0}, -z_{-})
\end{gathered}
\end{equation*}
so that $\R_{\pm s \geq s_{0}} \times \NorbitDivSet \subset \R_{\pm s \geq s_{0}}$ is equipped with the almost complex structure
\begin{equation*}
\JDivSet\partial_{s} = \ReebDivSet = \partial_{t} + X_{0} + X_{-}, \quad \JDivSet \partial_{x_{i}} = \partial_{y_{i}}.
\end{equation*}
Here the $x_{i}, y_{i}$ are standard Euclidean coordinates on $\R^{2(n_{0} + n_{-})}$. It then follows that along $\R_{\pm s \geq s_{0}} \times \disk_{x, y} \times \Norbit$, the complex structure $J$ described in Equation \eqref{Eq:LocalModelSetup} takes the form
\begin{equation*}
\begin{gathered}
\begin{aligned}
J \partial_{s} &= f_{0, \zeta}^{-1}(\partial_{t} + X_{0} - X_{f} + X_{-}),\\
&= \partial_{t'} + X_{0}' + X_{-}'
\end{aligned}\\
t' = f_{0, \zeta}^{-1}t, \quad X_{0}' = f_{0, \zeta}^{-1}(X_{0} - X_{f}), \quad X_{-}' = f_{0, \zeta}^{-1}X_{-},\\
J\partial_{x_{i}} = \partial_{y_{i}}.
\end{gathered}
\end{equation*}
Therefore $J$ is simple along any neighborhood
\begin{equation*}
\Norbit \subset \R_{t'} \times \R^{2 + 2n_{0}}_{z_{0}} \times \R^{2n_{-}}_{z_{-}}/\sim, \quad (f_{0, \zeta}, z'_{0}, z_{-}) \sim (0, z'_{0}, -z_{-}), \quad z_{0}' = (z,z_{0})
\end{equation*}
of $(\Z / f_{0, \zeta}a\Z) \times \{ 0 \}$ which we may identify with a neighborhood of $\{ \zeta \} \times \orbitDivSet$ by swapping orders of factors in the above product and merging $z \in \disk_{x,y}$ into the $z_{0}$ variable.

\begin{lemma}
According to our construction of almost contact forms and almost complex structures, each simple enough neighborhood $\NorbitDivSet \subset \divSet$ of a closed $\ReebDivSet$ orbit $\orbitDivSet$ determines a simple enough neighborhood $\Norbit \subset \disk_{x,y} \times \NorbitDivSet$ of $ \{ \zeta\} \times \orbitDivSet$ in $\subsetC \times \divSet \subset M$. If the almost complex structure $\JDivSet$ is simple along $\NorbitDivSet$, then $J$ is simple along $\Norbit$.
\end{lemma}

When using the almost complex structure $J$ we have constructed along a neighborhood $\Norbit$ as above, we see that the projection of $u^{-1}(\R_{\pm s\geq s_{0}} \times \Norbit)$ to the $s, t'$ coordinates is holomorphic. Therefore possibly after replacing $u: \halfcyl_{\pm} \to X_{\pm}$ with its restriction to some subset of the form $s^{-1}(\{ \pm s \geq C\})$ for $C \gg 0$, we may assume that this projection is free of critical points. Pulling back the coordinate functions on the target to coordinate functions $p,q$ on the domain and rescaling so that $q\in (\Z/2\pi\Z)$, we may express $u$ as a map
\begin{equation*}
u = (s, u_{\disk}, u_{\divSet}), \quad u_{\disk}: \halfcyl_{\pm} \to \disk_{x,y}, \quad u_{\divSet}: \halfcyl_{\pm} \to \NorbitDivSet
\end{equation*}
for which
\begin{equation*}
s = s_{1} + \frac{\action}{2\pi}p.
\end{equation*}
This yields a specialization of Equation \eqref{Eq:dsDecay}, where there is no $s_{\hot}$ term. See \cite[\S 3]{BH:ContactDefinition} for further details.

In either the ``classical SFT'' case or $\actionBound$-simple case we are able to assume
\begin{equation}\label{Eq:FisOne}
F \circ u = f_{0, \zeta}
\end{equation}
possibly after replacing $u(p, q)$ with $u(p \mp p_{0}, q)$ for some sufficiently large $p_{0} \in \R_{>0}$ to guarantee that $\pi_{\divSet}u$ is contained in a small neighborhood of the closed orbit $\orbit$. Up to reparameterization, such a replacement is equivalent to restricting $u$ to a subdomain of the form $\{ |p| > p_{0}\}$.

\subsubsection{Reduction to a Floer-type equation}

Now that we have a good understanding of $\pi_{\R \times \divSet}\circ u$, we seek to understand the
\begin{equation*}
u_{\disk}=\pi_{\disk}\circ u: \halfcyl_{\pm} \to \disk_{x,y}
\end{equation*}
component of the map $u$. We compute $\delbar_{J, \domainJ}u$ component-wise at a tangent vector $v \in T\halfcyl_{\pm}$ as
\begin{equation}\label{Eq:LocalDelBarFactors}
\begin{aligned}
2ds\delbar_{J, \domainJ}u(v) =& \partial_{v}s - F^{-1}\alphaDivSet^{J}(\partial_{\domainJ v}u_{\divSet}),\\
2\pi_{T\disk}\delbar_{J, \domainJ}u(v) =& \partial_{v}u_{\disk} + J\partial_{\domainJ v}u_{\disk} - \alphaDivSet^{J}(\partial_{v}u_{\divSet})\grad_{f} - F(\partial_{\domainJ v}s)X_{f},\\
2\pi_{T\divSet}\delbar_{J, \domainJ}u(v) =& \partial_{v}u_{\divSet} + F(\partial_{\domainJ v}s)\ReebDivSet + J_{\divSet}\pi_{\xiDivSet^{J}}\partial_{\domainJ v}u_{\divSet},\\
\delbar_{\JDivSet, \domainJ}\pi_{\R \times \divSet}u =& ds(\delbar_{J, \domainJ}u)\partial_{s} + \pi_{T\divSet}\delbar_{J, \domainJ}u.
\end{aligned}
\end{equation}
Here $\pi_{T\disk}$ and $\pi_{T\divSet}$ are the projections of
\begin{equation*}
TX_{\pm} = \R\partial_{s} \oplus T\disk \oplus T\divSet
\end{equation*}
onto its direct summand factors. Since we are assuming that $u$ is $(J, \domainJ)$-holomorphic, we have
\begin{equation*}
Fds =u^{\ast}\alphaDivSet^{J}\circ \domainJ.
\end{equation*}

Therefore Equations \eqref{Eq:dsDecay}, \eqref{Eq:FisOne}, and \eqref{Eq:LocalDelBarFactors} combine to provide the Floer-type equation,
\begin{equation}\label{Eq:FloerZ}
u_{\disk} \in \Cinfty(\halfcyl_{\pm}, \disk), \quad \lim_{p\to \pm \infty}u_{\disk}(p,\ast) = 0, \quad \delbar u_{\disk} = (\grad_{f} \otimes ds)^{0,1}.
\end{equation}
The last equation above is $\R$-linear in $u_{\disk} \in \Cinfty(\halfcyl_{\pm}, \disk_{x,y})$ due to the facts that
\be
\item $\grad_{f} = \grad_{f}(u_{\disk})$ is linear in $u_{\disk}$ and
\item $ds \in \Omega^{1}(\halfcyl_{\pm})$ is assumed to be fixed (and thus independent of $u_{\disk}$).
\ee

\subsubsection{Asymptotic analysis for $f_{e,\epsilon_{\zeta}}$}\label{Sec:EllipticAsymptotics}

In the case $f = f_{e, \epsilon_{\zeta}}$ of Assumptions \eqref{Assump:MorseLocalModel}, we can find all possible solutions to Equation \eqref{Eq:FloerZ} as follows: If we write
\begin{equation*}
u_{\disk}=e^{\epsilon_{\zeta} s} u_{\disk,0}, \quad u_{\disk,0}: \halfcyl_{\pm} \to \disk_{x,y}
\end{equation*}
then we find that
\begin{equation*}
\delbar u_{\disk} - (\grad_{f_{e, \epsilon_{\zeta}}}(z)\otimes ds)^{0, 1} = e^{\epsilon_{\zeta} s}\delbar u_{\disk,0} = 0 \iff \delbar u_{\disk,0}=0.
\end{equation*}
Applying Fourier series to each $u_{\disk}(p, \ast): \Circle_{q} \to \disk_{x,y}$ we find that every solution to Equation \eqref{Eq:FloerZ} in this case must take the form
\begin{equation*}
u_{\disk} = e^{\epsilon_{\zeta} s}\sum_{k \in \Z} c_{k}e^{k(p + iq)}, \quad c_{k} \in \C.
\end{equation*}
In order that such a solution, $u_{\disk}$, to Equation \eqref{Eq:FloerZ} converge to $0$ as $p \to \pm \infty$, Equation \eqref{Eq:dsDecay} tells us that
\begin{equation*}
c_{k} \neq 0 \implies \pm\left(\frac{\epsilon_{\zeta} \action}{2\pi}  + k\right) < 0.
\end{equation*}
Since we are assuming that $|\epsilon_{\zeta}|$ is arbitrarily small, we may specifically assume that
\begin{equation}\label{Eq:EpsilonZetaBound}
\epsilon_{\zeta} < 2\pi / \action.
\end{equation}
Then the above Fourier series formula for $u_{\disk}$ may be refined as
\begin{equation}\label{Eq:AsymptoticsElliptic}
u_{\disk} = e^{\epsilon_{\zeta} s}\sum_{k \in I} c_{k}e^{k(p + iq)}, \quad c_{k} \in \C, \quad I = \begin{cases}
\Z_{< 0} & \dom u = \halfcyl_{+},\\
\Z_{\geq 0} & \dom u = \halfcyl_{-}.
\end{cases}
\end{equation}

\begin{rmk}
The above analysis applies just as well to $\ind_{\Morse}=2$ critical points, which we have excluded in Assumptions \ref{Assump:MorseLocalModel}. For such a critical point, we then have $\epsilon_{\zeta} < 0$. Under the assumption that $-\epsilon_{\zeta} < 2\pi / \action$, Equation \eqref{Eq:AsymptoticsElliptic} becomes
\begin{equation*}
u_{\disk} = e^{\epsilon_{\zeta} s}\sum_{k \in I} c_{k}e^{k(p + iq)}, \quad c_{k} \in \C, \quad I = \begin{cases}
\Z_{\leq 0} & \dom u = \halfcyl_{+},\\
\Z_{> 0} & \dom u = \halfcyl_{-}.
\end{cases}
\end{equation*}
\end{rmk}

\subsubsection{Asymptotic analysis for $f_{h, \epsilon_{\zeta}}$}\label{Sec:HyperbolicAsymptotics}

Now we consider the case $f=f_{h, \epsilon_{\zeta}}$ of Assumptions \eqref{Assump:MorseLocalModel}. Writing
\begin{equation*}
u_{\disk} = (e^{\epsilon_{\zeta} s}x, e^{-\epsilon_{\zeta} s}y),\quad x, y \in \Cinfty(\halfcyl_{\pm}, \R),
\end{equation*}
Equation \eqref{Eq:FloerZ} yields
\begin{equation*}
\delbar u_{\disk} - (\grad_{f_{h, \epsilon_{\zeta}}}|_{u_{\disk}}\otimes ds)^{0, 1} = (e^{\epsilon_{\zeta} s}dx, -e^{-\epsilon_{\zeta} s}dy)^{0,1} = 0 \iff e^{\epsilon_{\zeta} s}dx = e^{-\epsilon_{\zeta} s}dy \domainJ.
\end{equation*}
For the purposes of establishing Lemma \ref{Lemma:CurvesStuckInLeaves} below, we will not need to find all solutions, $u_{\disk}$, to this equation. It will suffice to observe that
\begin{equation}\label{Eq:AsymptoticsHyperbolic}
\begin{gathered}
y = 0 \implies \dom u = \halfcyl_{+},\ u_{\disk} = (a e^{\epsilon_{\zeta} s}, 0),\\
x=0 \implies \dom u = \halfcyl_{-},\ u_{\disk}=(0, ae^{\epsilon_{\zeta} s})
\end{gathered}
\end{equation}
for some $a \in \R$. Indeed, if one of the $dx$ or $dy$ vanishes, then the other must vanish as well. The assertions regarding the domains $\dom u$ of $u$ follow from the convergence conditions in Equation \eqref{Eq:FloerZ}.

\section{Perturbed holomorphic curves in foliated cobordisms}\label{Sec:JCurvesInFoliations}

To establish our main result we must understand perturbed holomorphic curves contributing to contact homology differentials and cobordism maps in the almost complex manifolds of \S \ref{Sec:GeometryExamples}. In this section we establish that for a certain class of perturbations defined in \S \ref{Sec:PertAdaptedToFoliation}, the curves of interest must be contained in the leaves of the foliations (Lemma \ref{Lemma:CurvesStuckInLeaves}).

\subsection{Perturbations adapted to a foliation}\label{Sec:PertAdaptedToFoliation}

We temporarily work in a rather general setting. Let $(\complete{X}, J)$ be an almost complex manifold and let $(\Sigma, \domainJ)$ be a Riemann surface. In practice, we may take $\complete{X}$ to be a completed Liouville cobordism. For a $\COne$ map $u: \Sigma \to \complete{X}$ we write
\begin{equation*}
\delbar u = \half\left(Tu + J(Tu)\domainJ\right) \in \Omega^{0, 1}(u^{\ast}T\complete{X}).
\end{equation*}
A \emph{perturbation} for the pair $((\complete{X}, J),(\Sigma, \domainJ))$ is a section
\begin{equation*}
\pert \in \Sec(T\complete{X} \boxtimes T^{\ast}\Sigma) \to \complete{X} \times \Sigma).
\end{equation*}
A \emph{$\pert$-perturbed $J$-holomorphic curve} is a map $u$ as above such that for all $z \in \Sigma$,
\begin{equation*}
\delbar u|_{z} = \pert^{0, 1}|_{(u(z),z)}, \quad \pert^{0, 1} = \half(\pert + J\pert \domainJ).
\end{equation*}
Henceforth, we write this equation simply as
\begin{equation*}
\delbar u = \pert^{0,1}.
\end{equation*}
Whenever $\pert$ is understood from the context at hand, we may simply say that $u$ is a \emph{perturbed $J$-curve}.

\begin{assump}\label{Assump:PertCompactSupport}
We always assume that the support, $\Support(\pert) \subset \complete{X} \times \Sigma$, of a perturbation, $\pert$, as above is such that
\be
\item $\pi_{\Sigma}\Support(\pert)$ has compact closure in $\Sigma$ and
\item $\pi_{\complete{X}}\Support(\pert)$ is disjoint from $\partial \complete{X}$ whenever it is non-empty.
\ee
\end{assump}

For the purpose of defining and computing contact homology using the frameworks of \cite{BH:ContactDefinition, Pardon:Contact}, moduli spaces of perturbed $J$-curves are associated to certain $(X, J)$. In this context $\pert$ will, in general, vary over the space of maps and almost complex structures on the domain, $\Sigma$. Virtual counts of curves in these moduli spaces are then extracted to define chain-level contact homology differentials and cobordism maps.

Provided a perturbation, $\pert$, as above, we may equip $X \times \Sigma$ with the almost complex structure
\begin{equation*}
J_{\pert} = \begin{pmatrix}
J & J\pert-\pert\domainJ\\
0 & \domainJ
\end{pmatrix}.
\end{equation*}
As observed by Gromov in \cite{Gromov:JCurves} (and easily verified by the reader), perturbed $J$-curves $u:\Sigma \to X$ are naturally in one-to-one correspondence with $J_{\pert}$-holomorphic sections
\begin{equation*}
(u, \Id_{\Sigma}): \Sigma \to \complete{X}\times \Sigma
\end{equation*}
of the projection, $\pi_{\Sigma}$. See also \cite[\S 8]{MS:Curves} and \cite[Equation (3.17)]{Pardon:Contact}.

Suppose that $(\complete{X}, J)$ is equipped with a $\codim=2$, integrable, $J$-invariant distribution
\begin{equation*}
T\foliation \subset T\complete{X}, \quad \dim T\foliation|_{x} = 2n, \quad JT\foliation = T\foliation.
\end{equation*}
We have already seen many examples of such $T\foliation$ in \S \ref{Sec:DynamicsAndFoliations}, where the $(X, J)$ are completed Liouville cobordisms equipped with adapted almost complex structures. The leaves
\begin{equation*}
\leaf \subset \complete{X}, \quad T\leaf|_{x} = T\foliation|_{x}\  \forall x \in \leaf
\end{equation*}
of the associated foliation $\foliation$ on $X$ are then $\dim=2n$, $J$-holomorphic submanifolds of $X$.

\begin{defn}\label{Def:PertAdaptedToFoliation}
We say that a perturbation $\pert$ is \emph{adapted to the foliation} $\foliation$ if
\begin{equation*}
\pert \in \Sec(T\foliation \boxtimes T\Sigma).
\end{equation*}
\end{defn}

The following lemma allows us to convert a holomorphic foliation on $(\complete{X}, J)$ and a choice of perturbation $\pert$ into a foliation on $(\complete{X}\times \Sigma, J_{\pert})$. The lemma (whose proof is trivial) will later facilitate the application of intersection positivity arguments.

\begin{lemma}
If $\pert$ is adapted to the foliation $\foliation$, then
\begin{equation*}
T\foliation_{\Sigma} = T\foliation \oplus T\Sigma \subset T(\complete{X} \times \Sigma)
\end{equation*}
is a $\codim=2$, integrable, $J_{\pert}$-invariant distribution. The leaves $\leaf_{\Sigma}$ of the foliation $\foliation_{\Sigma}$ associated to the distribution $T\foliation_{\Sigma}$ all have the form
\begin{equation*}
\leaf_{\Sigma} = \leaf \times \Sigma \subset X \times \Sigma
\end{equation*}
for some leaf $\leaf$ of $\foliation$ and are $J_{\pert}$-holomorphic submanifolds of $(\complete{X} \times \Sigma, J_{\pert})$.
\end{lemma}

We note that $\foliation_{\Sigma}$ and its leaves are independent of the choice of perturbation $\pert$ adapted to $\foliation$. However, the induced almost complex structures on leaves depend both on $J$ and the choice of $\pert$ adapted to $\foliation$.

The following lemmas state foliated versions of the unique continuation \cite[\S 2.3]{MS:Curves} and intersection positivity properties \cite[\S E.2]{MS:Curves} for (perturbed) holomorphic curves. Similar results appear, for example, in \cite{Moreno:Thesis, MorenoSiefring}.

\begin{lemma}[Unique continuation in leaves]\label{Lemma:UniqueContinuation}
Suppose that $u: \Sigma \to \complete{X}$ is a $\pert$-perturbed holomorphic curve, where $\pert$ is adapted to the foliation $\foliation$ and $\Sigma$ is path-connected. If there is an open set in $\Sigma$ which is mapped to a single leaf $\leaf$ of $\foliation$ via $u$, then $\im u$ is entirely contained in $\leaf$.
\end{lemma}

\begin{proof}
By assumption, there is an open subset of $\Sigma$ which is mapped to $\leaf \times \Sigma$ via $(u, \Id_{\Sigma})$. Hence the unique continuation property of holomorphic curves (cf. \cite[\S 2.3]{MS:Curves}) implies that $\im (u, \Id_{\Sigma})$ is entirely contained in $\leaf \times \Sigma$. Therefore $\im u$ is entirely contained in $\leaf$.
\end{proof}

To set up our intersection-positivity result, suppose that $u:\Sigma \to X$ is a $\pert$-perturbed holomorphic curve with $\pert$ adapted to $\foliation$ and $\Sigma$ non-compact. Let $\leaf$ be a leaf of $\foliation$ for which $\im u \cap \leaf$ is a compact set contained in the interior of $\im u$. We define the intersection number
\begin{equation*}
\leaf \bullet u \in \Z
\end{equation*}
to be the signed count of intersections of a $\COne$-small variation of $\im u$ with $\leaf$, supported on a compact subset of $\im u$ disjoint from the image of the boundary of $\Sigma$, ensuring transversality. The well-definedness of $\leaf \bullet u$ follows from the fact that any two such generic variations will give rise to an oriented, $1$-dimensional cobordism between the collections of signed points whose count yields the intersection number.

\begin{lemma}[Intersections of curves with leaves]\label{Lemma:PerturbedIntersections}
In the above setting, the intersection number $\leaf \bullet u$ is non-negative and $\leaf \bullet u =0$ iff $\im u$ is disjoint from $\leaf$. Moreover, if $\leaf_{t}$ is a $1$-parameter family of leaves of $\foliation$ as above and such that $\cup_{t} \leaf_{t}$ is the image of a smooth map from a smooth manifold into $\complete{X}$, then $\leaf_{t} \bullet u$ is independent of $t$.
\end{lemma}

\begin{proof}
The intersection number agrees with the intersection of the $J_{\pert}$-holomorphic submanifolds $\leaf \times \Sigma$ and $\im(u, \Id_{\Sigma})$ in $\complete{X} \times \Sigma$. Using this fact, we can apply the classical intersection positivity arguments, which we sketch here.

For any $\COne$-small variation $u'$ of $u$ as above, the tangent spaces of $\im(u', \Id_{\Sigma})$ will be $\CZero$-close to $J_{\pert}$-complex subspaces of $T(\complete{X} \times \Sigma)$. The $\im(u', \Id_{\Sigma})$ and $\leaf \times \Sigma$ have complementary dimension, and since complex subspaces intersect positively, any local contributions to the intersection number determined by a geometric intersection must be positive. This establishes the first statement in the lemma. The second statement again follows by considering oriented, $1$-dimensional cobordisms between signed counts of points which compute the intersection numbers.
\end{proof}

\subsection{A note on supports of perturbations and a blocking lemma}\label{Sec:PertSupport}

We quickly address the consequences of Assumptions \ref{Assump:PertCompactSupport} in the context of the almost complex manifolds $(\complete{X}, J) = (\R_{s} \times M, J)$ where $M=\C \times \divSet$, $\R_{\tau} \times \Liouville_{\sigma_{0}}$, or $\R_{\tau} \times \hypersurface$. We note that Assumptions \ref{Assump:PertCompactSupport}(1) implies that our perturbations must be disjoint from some (sufficiently small) cylindrical ends of holomorphic curves in $\R_{s} \times M$ which are asymptotic to closed Reeb orbits as $s \to \pm \infty$. Note that this assumption is compatible with both the perturbation schemes of Bao-Honda and Pardon, as described around \cite[Equation (5.1.3)]{BH:ContactDefinition} and \cite[Equation (3.16)]{Pardon:Contact}, respectively. Further discussion of compatibility of perturbations with the frameworks of \cite{BH:ContactDefinition} and \cite{Pardon:Contact} will appear in \S \ref{Sec:BHOverview} and \S \ref{Sec:PardonSpecialization}, respectively.

For the manifolds $\R_{s} \times M$ under consideration, those of the form $M = \R_{\tau} \times \Liouville_{\sigma_{0}}$ are the only ones which $\partial M \neq \emptyset$. Hence these are the only manifolds of interest for which item (2) of Assumptions \ref{Assump:PertCompactSupport} is applicable.

\begin{lemma}\label{Lemma:BlockingRtimesM}
If $u: \Sigma \to \R_{s} \times M$ is a perturbed holomorphic curve without boundary which converges to Reeb orbits at its punctures, then $\im u$ is disjoint from a neighborhood of $\R_{s} \times \partial M$.
\end{lemma}

\begin{proof}
In the case $M = \R_{\tau} \times \Liouville_{\sigma_{0}}$, Lemma \ref{Lemma:BlockingRRLiouville} states that the function $-\sigma$ is $J$-convex along a neighborhood of $\partial M$. Hence the desired result follows from the maximum principle for the function $-\sigma \circ u$ defined along a subset of $\Sigma$ given by the inverse image of a collar neighborhood of $\R_{s} \times \partial M$.
\end{proof}

\subsection{Perturbed $J$-curves are trapped in leaves}\label{Sec:CurvesStuckInLeaves}

\begin{lemma}\label{Lemma:CurvesStuckInLeaves}
Following the notation of \S \ref{Sec:LeafNotation}, suppose that $u: \Sigma \to \R_{s} \times M$ is a $\pert$-perturbed holomorphic curve without boundary, whose punctures are all asymptotic to orbits of the form
\begin{equation*}
\orbit_{\zeta} = \{\zeta\} \times \orbit \subset \subsetC \times \divSet \subset M.
\end{equation*}
If the perturbation $\pert$ is adapted to the foliation, then $\im u$ is contained in some leaf of $\foliation$.
\end{lemma}

The remainder of \S \ref{Sec:CurvesStuckInLeaves} proves this result via case-by-case analysis.

\subsubsection{The case of an elliptic asymptotic at $-\infty$}\label{Sec:CurveInLeafNegativeEllipticAsymptotic}

Suppose that $u: \Sigma \to \R_{s} \times M$ is negatively asymptotic to some orbit of the form
\begin{equation*}
\orbit_{\zeta^{-}} = \{\zeta^{-}\} \times \orbit, \quad \zeta^{-} \in \Crit(f(-s_{0}, \ast)) \subset \subsetC, \quad \ind_{\Morse}(f(-s_{0}, \ast), \zeta^{-})=0
\end{equation*}

Let $\leaf_{\zeta^{-}}$ be the leaf of $\foliation$ uniquely determined by
\begin{equation*}
\leaf_{\zeta^{-}} \cap \{ s \leq -s_{0}\} = \R_{s \leq -s_{0}} \times \{ \zeta^{-}\} \times \divSet.
\end{equation*}
In words, $\leaf_{\zeta^{-}}$ corresponds to a constant gradient trajectory along the negative end $\R_{s \leq -s_{0}}\times M$ of our completed cobordism, $\R_{s} \times M$. The local model for a neighborhood of
\begin{equation*}
\leaf_{\zeta^{-}} \cap \{ s \leq -s_{0}\} \subset \R_{s \leq -s_{0}} \times M
\end{equation*}
is given by
\begin{equation*}
X_{-} = \R_{s \leq -s_{0}} \times \disk(\zeta^{-}, \rho_{\zeta^{-}}) \times \divSet, \quad f(-s_{0}, \ast)=f_{e, \epsilon_{\zeta^{-}}} = 1 + \frac{\epsilon_{\zeta^{-}}}{2}r^{2}.
\end{equation*}
as in Assumptions \eqref{Assump:MorseLocalModel} and Equation \eqref{Eq:LocalModelSetup}. In this local model the disk $\disk(\zeta^{-}, \rho_{\zeta^{-}})$ is a neighborhood of $\zeta^{-}$ in $\subsetC$. Our goal is to show that $\im u$ is contained in a leaf of $\foliation$ by combining
\be
\item the analysis of asymptotics of $u$ from \S \ref{Sec:EllipticAsymptotics} with
\item the unique continuation result of Lemma \ref{Lemma:UniqueContinuation} and the intersection positivity results of Lemma \ref{Lemma:PerturbedIntersections}.
\ee
To this end, we will carefully select a family of leaves $\leaf_{t}$ of the foliation $\foliation$ and track how they intersect $\im u$.

For any $\rho \in (0, \rho_{\zeta^{-}})$, we can find some $s_{\rho} \ll -s_{0}$ such that
\begin{equation*}
\im u \cap \{ s \leq s_{\rho}\} \subset \R_{s \leq s_{\rho}} \times \left(\cup_{\zeta \in \Crit(f(-s_{0}, \ast))} \disk(\zeta, \rho)\right)
\end{equation*}
where each $\disk(\zeta, \rho) \subset \subsetC$ is a disk of radius $\rho$ centered about $\zeta \in \Crit(f(-s_{0}, \ast))$. We assume $\rho$ is sufficiently small so that these disks are pairwise disjoint and contained in the interior of $\subsetC$. The existence of such a $s_{\rho} \in \R_{\leq s_{0}}$ associated to each $\rho$ follows from the negative asymptotic convergence of $u$ to orbits $\orbit_{\zeta}$, $\zeta \in \Crit(f(-s_{0}, \ast))$. For such constants $\rho, s_{\rho}$, we write
\begin{equation}\label{Eq:KRhoDef}
K_{\rho} = \overline{\pi_{\subsetC}(\im u \cap (\R_{s \geq s_{\rho}} \times \subsetC \times  \divSet))} \cup \left(\cup_{\zeta \in \Crit(f(-s_{0}, \ast))} \disk(\zeta, \rho)\right).
\end{equation}
In words, $K_{\rho}$ is obtained by taking the intersection of $\im u$ with $\R_{s \geq s_{\rho}} \times \subsetC \times \divSet \subset R_{s \geq s_{\rho}} \times M$, projecting to $\subsetC$, taking a closure, and finally adjoining the small disks centered about the critical points of $f(-s_{0}, \ast)$. Since $u$ must converge to orbits at the positive end of the cobordism, it follows that $K_{\rho}$ is compact. Moreover, in the case of $M = \R_{\tau} \times \Liouville_{\sigma_{0}}$ (for which $\subsetC = \R_{\tau} \times [-1, \sigma_{0}]$), Lemma \ref{Lemma:BlockingRRLiouville} asserts that $K_{\rho}$ must be disjoint from $\partial_{out}\subsetC$. By construction, we have that
\begin{equation*}
\im u \cap (\R_{s \geq s_{\rho}} \times \subsetC \times \divSet) \subset \R_{s \geq s_{\rho}} \times K_{\rho} \times \divSet.
\end{equation*}

Possibly after shrinking $\rho < \rho_{\zeta^{-}}$ and taking $s_{\rho}$ to be sufficiently negative, we may assume that
\begin{equation*}
u^{-1}(\R_{s \leq s_{\rho}} \times \disk(\zeta^{-}, \rho) \times \divSet) \subset \Sigma
\end{equation*}
is a negative half cylinder, $\halfcyl_{-} = \R_{p\leq 0} \times \Circle_{q}$, and that the restriction of $u$ to $\halfcyl_{-}$ has the form
\begin{equation}\label{Eq:EllipticAsymptoticsNegative}
\begin{gathered}
u|_{\halfcyl_{-}} = (s(p,q), u_{\disk}, u_{\divSet}), \quad u_{\disk}(p,q)= e^{\epsilon_{\zeta^{-}}s(p,q)}\sum_{k \geq 0} c_{k}e^{k(p+iq)}, \quad c_{k} \in \C, \quad dp \neq 0
\end{gathered}
\end{equation}
as described in Equation \eqref{Eq:AsymptoticsElliptic}. Indeed, since the perturbation $\pert$ for which $\delbar_{J, \domainJ}u = \pert^{0, 1}$ is compactly supported (Assumptions \ref{Assump:PertCompactSupport}(1)), we may assume that this support is disjoint from $u^{-1}(\{s\leq s_{\rho}\})$ for $s_{\rho}$ sufficiently negative as described in \S \ref{Sec:PertSupport}. Thus $u|_{\halfcyl_{-}}$ is holomorphic and the results of \S \ref{Sec:EllipticAsymptotics} are applicable. Using the coefficients $c_{k}$ above we define
\begin{equation*}
z_{0} = c_{0}e^{\epsilon_{\zeta^{-}}s_{\rho}} \in \disk(\zeta^{-}, \rho).
\end{equation*}

Note that
\begin{equation*}
\disk(\zeta, |z_{0}|) \subset \Stable(f(-s_{0}, \ast), \zeta^{-})
\end{equation*}
and consider a path
\begin{equation*}
\begin{gathered}
\thicc{z}: \R_{t > 0} \to \Stable(f(-s_{0}, \ast), \zeta^{-}) \setminus \left(\disk(\zeta, |z_{0}|)\cup \{ \tau = 0\}\right),\\
\lim_{t\to 0} \thicc{z}(t) = z_{0}\quad t > 1 \implies \partial_{t}\thicc{z}(t) = \grad_{f(-s_{0}, \ast)}|_{\thicc{z}(t)}.
\end{gathered}
\end{equation*}
These conditions imply that $\thicc{z}(t)$ will eventually exit any compact set in $\subsetC$ disjoint from $\partial_{out}\subsetC$ as in Lemma \ref{Lemma:LeafExitCompact}. Define a $\R_{t > 0}$ family of leaves $\leaf_{t}$ of $\foliation$ by the property that
\begin{equation*}
\leaf_{t} \cap \{ s = s_{\rho}\} = \{ s_{\rho}\} \times \{ \thicc{z}(t)\} \times \divSet.
\end{equation*}
From Lemma \ref{Lemma:BoundedUnboundedLeaves}, we see that all of the $\leaf_{t}$ are unbounded and negatively asymptotic to $\zeta^{-}$. In fact, we can write
\begin{equation}\label{Eq:LeafInDiskNearElliptic}
\thicc{z}(t) \in \disk(\zeta^{-}, \rho_{\zeta^{-}}) \implies \leaf_{t} \cap \{ s \leq s_{\rho}\} = \{ (s, e^{s \epsilon_{\zeta^{-}}}\thicc{z}(t), m)\ : s \leq s_{\rho}, m \in \divSet\}.
\end{equation}
Moreover, the union $\cup_{t} \leaf_{t}$ is the image of a map from a smooth manifold into $\R_{s} \times M$.

We claim that for all $t > 0$, there is some $s_{1} \ll 0$ such that
\begin{equation*}
\leaf_{t} \cap \im u \cap \{ s \leq s_{1}\} = \emptyset.
\end{equation*}
To see this, observe that for $s < s_{\rho} < 0$ we can write
\begin{equation}\label{Eq:NormBoundImU}
\sup_{q}\left( u_{\disk}(p,q)|_{u^{-1}(\{s\} \times \disk(\zeta^{-}, \rho) \times \divSet)} \right) = e^{\epsilon_{\zeta^{-}}s}\left(c_{0} + \bigO(e^{\frac{2\pi}{\action}s})\right)
\end{equation}
as a consequence of the facts that
\be
\item $s \in \Cinfty(\halfcyl_{-})$ is an approximately linear function of $p$ with slope $2\pi/\action$ where $\action$ is the action of $\orbit$ (as described around Equation \eqref{Eq:dsDecay}) and
\item the $c_{k}$ are indexed over non-negative integers in Equation \eqref{Eq:EllipticAsymptoticsNegative}.
\ee
Now let's consider $t \in \R_{>0}$ such that $\thicc{z}(t) \in \disk(\zeta^{-}, \rho_{\zeta^{-}}) \setminus \disk(\zeta^{-}, \rho)$. By comparing Equation \eqref{Eq:LeafInDiskNearElliptic} with Equation \eqref{Eq:NormBoundImU}, we see that the norm of $u_{\disk}$ restricted to the inverse image of $\{ s\} \times \disk(\zeta^{-}) \times \divSet$ decays to zero faster than the norm, $e^{\epsilon_{\zeta^{-}}s}|\thicc{z}(t)|$, of $\pi_{\disk}(\leaf_{t} \cap \{ s\} \times \disk(\zeta^{-}) \times \divSet)$ as $s \to -\infty$. If $\thicc{z}(t)$ is not in $\disk(\zeta^{-}, \rho_{\zeta^{-}})$ the same reasoning applies: The end of the holomorphic curve converges to $0$ in the $\disk(\zeta^{-}, \rho)$ direction faster than $\leaf_{t}$. The claim is thus established.

The properties of the $\leaf_{t}$ so far obtained imply (via Lemma \ref{Lemma:PerturbedIntersections}) that the intersection number $\leaf_{t} \bullet u$ is well-defined and independent of $t \in \R_{>0}$. We claim that this number is zero. Indeed, we see that there is some $t_{0}$ such that for $t > t_{0}$, $\thicc{z}(t)$ is disjoint from $K_{\rho}$ as in Lemma \ref{Lemma:LeafExitCompact}. Since the leaves $\leaf_{t}$ project to $\grad_{f}$ flow lines in $\subsetC$, it follows that
\begin{equation*}
t > t_{0} \implies \leaf_{t} \cap \im u \cap \{ s \geq s_{\rho}\} = \emptyset.
\end{equation*}
Similarly, by looking at the comparative rates of convergence of $\pi_{\subsetC} u$ and $\pi_{\subsetC}\leaf$ as described in the previous paragraph, we find that
\begin{equation*}
t > t_{0} \implies \leaf_{t} \cap \im u \cap \{ s \leq s_{\rho}\} = \emptyset.
\end{equation*}
Therefore $t > t_{0}$ implies that $\leaf_{t} \cap \im u = \emptyset$ and so
\begin{equation*}
\leaf_{t} \bullet u =0\ \forall t > 0.
\end{equation*}

We are ready to complete the proof of Lemma \ref{Lemma:CurvesStuckInLeaves} in the present case. By the preceding equation and Lemma \ref{Lemma:PerturbedIntersections}, the $\leaf_{t}$ and $\im u$ must be disjoint for all $t$. As $t \to 0$, the $\thicc{z}(t) = \pi_{\subsetC}(\leaf_{t} \cap \{s = s_{\rho}\})$ approach the point $z_{0} \in \disk(\zeta^{-}, \rho)$ for which we can write
\begin{equation*}
u_{\disk}|_{\partial \halfcyl_{-}} = u_{\disk}(0, q) = z_{0} + e^{\epsilon_{\zeta^{-}}s_{\rho}}\sum_{k > 0} c_{k}e^{kiq}
\end{equation*}
using Equation \eqref{Eq:EllipticAsymptoticsNegative}. Consequently, $u_{\disk}|_{\partial \halfcyl_{-}}$ parameterizes a smooth map from $\Circle_{q}$ to $\disk(\zeta^{-}, \rho_{\zeta^{-}})$ with mean value $z_{0}$. Hence the image of this circle must touch some $\thicc{z}(t)$ with $t \sim 0$ unless $u_{\disk}|_{\partial \halfcyl_{-}}$ is the constant map with value $z_{0}$. Since $\leaf_{t} \cap \{s = s_{\rho}\} \cap \im u = \emptyset$ for all $t$, $u_{\disk}|_{\partial \halfcyl_{-}}$ must be constant, and so the Fourier coefficients $c_{k}$ of $u_{\disk}$ must vanish for $k > 0$. It follows that
\begin{equation*}
u_{\disk}|_{\halfcyl_{-}}(p,q) = e^{\epsilon_{\zeta^{-}}s(p,q)}c_{0}
\end{equation*}
so that $\im u \cap \{ s \leq s_{\rho}\}$ is contained in the leaf $\leaf$ of $\foliation$ for which 
\begin{equation*}
\leaf \cap \{ s = s_{\rho}\} = \{ s_{\rho} \} \times \{z_{0}\} \times \divSet.
\end{equation*}
Applying Lemma \ref{Lemma:UniqueContinuation}, $\im u$ is completely contained in this leaf, $\leaf$, of the foliation, $\foliation$, and the present case of Lemma \ref{Lemma:CurvesStuckInLeaves} is established.

\subsubsection{The case of a hyperbolic asymptotic at $-\infty$}\label{Sec:CurvesInLeavesHyperbolic}

Suppose that $u: \Sigma \to \R_{s} \times M$ is negatively asymptotic to some $\orbit_{\zeta^{-}} = \{\zeta^{-}\} \times \orbit$ where $\zeta^{-} \in \subsetC$ is a Morse critical point of $f(-s_{0}, \ast) \in \Cinfty(\subsetC)$ of $\ind_{\Morse}(f(-s_{0}, \ast), \zeta^{-})=1$. We assume that $u$ is not negatively asymptotic to any closed Reeb orbit, $\orbit_{\zeta}$, with $\zeta \in \Crit(f(-s_{0}, \ast))$ of $\ind_{\Morse}(f(-s_{0}, \ast), \zeta)=0$; otherwise we could appeal to the argument of \S \ref{Sec:CurveInLeafNegativeEllipticAsymptotic} to see that $u$ is contained in a leaf of the foliation, $\foliation$.

The argument in the present case follows essentially the same setup as that of \S \ref{Sec:CurveInLeafNegativeEllipticAsymptotic}. Again, we let $\leaf_{\zeta^{-}}$ be the leaf of $\foliation$ uniquely determined by
\begin{equation*}
\leaf_{\zeta^{-}} \cap \{ s \leq -s_{0}\} = \R_{s \leq -s_{0}} \times \{ \zeta^{-}\} \times \divSet.
\end{equation*}
The local model for a neighborhood of $\leaf_{\zeta^{-}} \cap \{ s \leq -s_{0}\}$ in $\R_{s} \times M$ is given by
\begin{equation*}
X_{-} = \R_{s \leq -s_{0}} \times \disk(\zeta^{-}, \rho_{\zeta^{-}}) \times \divSet, \quad f(-s_{0}, \ast)=f_{h, \epsilon_{\zeta^{-}}} = f_{0, \zeta^{-}} + \frac{\epsilon_{\zeta^{-}}}{2}(x^{2} - y^{2})
\end{equation*}
as described in Assumptions \eqref{Assump:MorseLocalModel} and Equation \eqref{Eq:LocalModelSetup} for some $\epsilon_{\zeta^{-}} \in \R_{>0}$. We recall that $x,y$ are standard coordinates on the disk $\disk(\zeta^{-}, \rho_{\zeta^{-}}) \subset \subsetC$ for which the point $x=y=0$ is $\zeta^{-}$. We choose a small $\rho \in (0, \rho_{\zeta^{-}})$, an $s_{\rho} \ll 0$, and a compact subset $K_{\rho} \subset \subsetC$ with $K_{\rho} \cap \partial_{out}\subsetC = \emptyset$ as in Equation \eqref{Eq:KRhoDef}. Again, we may assume that $s_{\rho}$ is sufficiently negative so that $u^{-1}(\R_{s \leq s_{\rho}} \times \disk(\zeta^{-}, \rho) \times \divSet) \subset \Sigma$ is a negative half cylinder, $\halfcyl_{-} = \R_{p\leq 0} \times \Circle_{q}$, and that the restriction of $u$ to $\halfcyl_{-}$ is $J$-holomorphic, having the form
\begin{equation*}
\begin{gathered}
u|_{\halfcyl_{-}} = (s(p,q), u_{\disk}, u_{\divSet}), \quad s(0, q) = s_{\rho}.
\end{gathered}
\end{equation*}

In the $x,y$-coordinates on $\disk(\zeta^{-}, \rho_{\zeta^{-}})$, we have
\begin{equation*}
\begin{gathered}
\Stable(f(-s_{0}, \ast), \zeta^{-}) \cap \disk(\zeta^{-}, \rho_{\zeta^{-}}) = \{ y = 0\},\\
\Unstable(f(-s_{0}, \ast), \zeta^{-}) \cap \disk(\zeta^{-}, \rho_{\zeta^{-}}) = \{ x = 0\}.
\end{gathered}
\end{equation*}
Our strategy will be to show that $\pi_{\disk}(\im u \cap \{ s \leq s_{1}\})$ is contained in the union $\{ xy = 0\} \subset \disk(\zeta^{-}, \rho_{\zeta^{-}})$ of these stable and unstable manifolds for some $s_{1} \ll 0$ and then appeal to the results of \S \ref{Sec:HyperbolicAsymptotics} to infer that $\im u$ must be contained in a leaf of $\foliation$.

Let $z_{0}$ be any point in $\disk(\zeta^{-}, \rho_{\zeta^{-}})$ which is disjoint from the $x$- and $y$-axes and let $\thicc{z}(t)$ be the $\grad_{f(-s_{0}, \ast)}$ flow line with $\thicc{z}(0) = z_{0}$. If $\thicc{z}(t)$ is defined for all $t > 0$ (meaning that the flow-line never touches $\partial_{out}\subsetC$), then we define $\leaf_{t}$ to be the $\R_{\geq 0}$-family of leaves of $\foliation$ for which
\begin{equation*}
\leaf_{t} \cap (\{ -s_{0}\}\times M) = \{-s_{0}\} \times \{ \thicc{z}(t)\} \times \divSet.
\end{equation*}
If $\thicc{z}(t_{0}) \in \partial_{out}\subsetC$ for some $t_{0} \geq 0$, then we define $\leaf_{t}$ for all $t \geq 0$ by
\begin{equation*}
\begin{gathered}
t \leq t_{0} \implies \leaf_{t} \cap (\{ -s_{0}\}\times M) = \{-s_{0}\} \times \{ \thicc{z}(t)\} \times \divSet,\\
t > t_{0} \implies \leaf_{t} \cap (\{ -s_{0} + t_{0} - t\}\times M) = \{-s_{0} + t_{0} - t\} \times \{ \thicc{z}(t_{0})\} \times \divSet.
\end{gathered}
\end{equation*}
Intuitively speaking, for $t > t_{0}$, we push the leaves downward in $\R_{s} \times M$. It follows that
\be
\item all $\leaf_{t}$ are unbounded,
\item the projection of $\leaf_{t} \cap (\R_{s} \times \subsetC \times \divSet)$ to $\subsetC$ is a (possibly non-autonomous) $\grad_{f}$ flow line passing through $z_{0}$,
\item all $\leaf_{t}$ have the same diffeomorphism type and negative asymptotic (if any), and
\item the $\cup_{t}\leaf_{t}$ form the image of a map from a smooth manifold into $\R_{s} \times M$.
\ee
By assuming that $z_{0}$ is sufficiently small, we may assume that if the $\leaf_{t}$ have a negative asymptotic, then it is always an $\ind_{\Morse} = 0$ critical point of $f(-s_{0}, \ast)$.
This follows from the facts that $\Crit_{1}(f(-s_{0}, \ast))$ is finite and the unstable manifolds of such critical points are $\grad_{f(-s_{0}, \ast)}$ flow-lines. Since $u$ is not negatively asymptotic to any $\orbit_{\zeta}$ with $\zeta \in \Crit_{0}(f(-s_{0}, \ast))$, we find that for all $t$, there is some $s_{1} \ll s_{\rho}$ such that
\begin{equation*}
\leaf_{t} \cap \im u \cap \{ s \leq s_{1}\} = \emptyset.
\end{equation*}

These observations indicate that Lemma \ref{Lemma:PerturbedIntersections} is applicable to the family $\leaf_{t}$ and the perturbed curve $u$. Thus $\leaf_{t} \bullet u \in \Z_{\geq 0}$ is well-defined and independent of $t$. We claim that for $t$ sufficiently large,
\begin{equation*}
t \gg 0 \implies \im u \cap \leaf_{t} = \emptyset.
\end{equation*}

When the parameterized $\grad_{f(-s_{0}, \ast)}$ flow line $\thicc{z}(t)$ is defined for all time, let $t_{1}$ be such that for $t > t_{1}$
\be
\item $\thicc{z}(t) \notin K_{\rho}$,
\item $\thicc{z}(t)$ is contained in the subset of $\subsetC$ along which $f$ is independent of $s$.
\ee
In this case we find that $t > t_{1}$ implies $\im u \cap \leaf_{t} \cap \{ s \geq s_{\rho}\} = \emptyset$. Moreover for such $t$, the projection of $\leaf_{t} \cap (\R_{s} \times \subsetC \times \divSet)$ to $\subsetC$ is a $\grad_{f(-s_{0}, \ast)}$ flow line passing through $z_{0}$. Since $\thicc{z}(t)$ remains a bounded distance from the $\ind_{\Morse}=1$ critical points of $f(-s_{0}, \ast)$, to which $\pi_{\subsetC}u$ is negatively asymptotic, we find that there is some $t_{2} \geq t_{1}$ such that $t > t_{2}$ implies $\im u \cap \leaf_{t} \cap \{ s \leq s_{\rho}\} = \im u \cap \leaf_{t} = \emptyset$.

When the parameterized $\grad_{f(-s_{0}, \ast)}$ flow line $\thicc(t)$ touches $\partial_{out}\subsetC$ at some time $t_{0}$, then similar reasoning applies. We find that for large $t \gg t_{0}$, the $\leaf_{t} \cap (\R_{s} \times \subsetC \times \divSet)$ all project to the same $\grad_{f(-s_{0}, \ast)}$ flow line passing through $z_{0}$ in $\subsetC$. We then find $t_{2} > t_{1} > t_{0}$ as in the previous paragraph such that $t > t_{2}$ implies $\im u \cap \leaf_{t} = \emptyset$.

Applying Lemma \ref{Lemma:PerturbedIntersections} again, we find that $\leaf_{t} \bullet u = 0$ and so $\leaf_{t} \cap \im u = \emptyset$ for all $t$. Since the points $z_{0}$ which we've used to construct the $\leaf_{t}$ fill up $\disk(\zeta^{-}, \rho) \setminus \{ xy=0\}$ for some small $\rho \in (0, \rho_{\zeta^{-}})$, we find that $\im u |_{\halfcyl_{-}}$ must be contained in $\disk(\zeta^{-}, \rho) \cap \{ xy=0\}$. We may then apply Equation \eqref{Eq:AsymptoticsHyperbolic} which says that $\im u|_{\halfcyl_{-}}$ must be contained in a leaf of $\foliation$. Lemma \ref{Lemma:UniqueContinuation} then informs us that $\im u$ must be contained in a leaf of $\foliation$.

\subsubsection{Curves without negative asymptotics}

Finally, we consider the case when the map $u$ has no negative asymptotics. The logic here is quite similar to the cases considered in \S \ref{Sec:CurveInLeafNegativeEllipticAsymptotic} and \ref{Sec:CurvesInLeavesHyperbolic}. To simplify matters, we assume that the function $f \in \Cinfty(\R_{s} \times \subsetC)$ is independent of $s$, so that we can write $f \in \Cinfty(\subsetC)$. The proof of the general case is conceptually identical to this simplified case but is notationally much more complicated.

Suppose that $u: \C \to \R_{s} \times M$ is positively asymptotic to some orbit of the form 
\begin{equation*}
\orbit_{\zeta^{+}} = \{\zeta^{+}\} \times \orbit, \quad \zeta^{+} \in \Crit(f(s_{0}, \ast)) \subset \subsetC.
\end{equation*}
Define a subset
\begin{equation*}
K = \overline{\pi_{\subsetC}\left(\im u \cap (\R_{s} \times \subsetC \times  \divSet) \right)} \subset \subsetC.
\end{equation*}
Then $K$ is compact (by the positive asymptotic convergence of $u$ to $\orbit_{\zeta^{+}}$) and disjoint from $\R_{s} \times \partial M = \R_{s} \times \partial_{out}\subsetC \times \divSet$ (by Lemma \ref{Lemma:BlockingRtimesM}).

For each $z_{0} \in \subsetC \setminus \{ \tau = 0\}$ let $\thicc{z}(t)$ be the $\grad_{f}$ flow-line with $\thicc{z}(0)=z_{0}$ defined for
\be
\item $t \in [0, t_{0}]$ if $\thicc{z}(t)$ touches $\partial_{out}\subsetC$ at some finite time $t_{0}$
\item $t \in \R_{\geq 0}$ otherwise.
\ee
Define a two-parameter family of leaves $\leaf_{s,t}$ in $\foliation$ by
\begin{equation*}
\leaf_{s,t} \cap (\{ s\} \times M) = \{s\} \times \{ \thicc{z}(t)\} \times \divSet.
\end{equation*}
Here the domain of the parameters $(s,t)$ is either $\R_{s} \times \R_{t \geq 0}$ or $\R_{s} \times [0,t_{0}]$, depending on whether or not $\thicc{z}(t)$ touches $\partial_{out}\subsetC$ in finite time. It follows that
\be
\item the $\leaf_{s,t}$ are all unbounded,
\item $\leaf_{s, t} \cap (\R_{s} \times \subsetC \times \divSet)$ is the same $\grad_{f}$ flow-line passing through $z_{0}$ for all $s, t$, and
\item the $\cup_{s,t} \leaf_{s, t}$ form the image of a map from smooth manifold into $\R_{s} \times M$.
\ee
Hence Lemma \ref{Lemma:PerturbedIntersections} informs us that $\leaf_{s, t} \bullet u \in \Z_{\geq 0}$ is well-defined and independent of the variables $s,t$.

As described in Lemma \ref{Lemma:LeafExitCompact}, for each $z_{0}$, there is some $t_{1} \geq 0$ such that $\thicc{z}(t) \notin K$ for $t \geq t_{1}$. Thus $\leaf_{s, t} \cap \im u = \emptyset$ for all $t \geq t_{1}$ and $s < \min s\circ u$. Therefore $\leaf_{s,t} \bullet u = 0$ and $\leaf_{s,t} \cap \im u = \emptyset$ for all $s, t$. We conclude that
\begin{equation*}
\im u \cap (\R_{s} \times \subsetC \times \divSet) = \im u \cap (\R_{s} \times \{ \tau = 0\} \times \divSet).
\end{equation*}
As in the final arguments of \S \ref{Sec:CurveInLeafNegativeEllipticAsymptotic} and \S \ref{Sec:CurvesInLeavesHyperbolic}, it then follows from Equations \eqref{Eq:AsymptoticsElliptic} and \eqref{Eq:AsymptoticsHyperbolic} that $\im u$ is contained in a leaf of $\foliation$.

\section{Fredholm theory for perturbed curves in foliated cobordisms}\label{Sec:FredholmTheory}

In this section we study linearizations of the $\delbar$ equations of perturbed holomorphic curves in almost complex manifolds mapping to leaves of a $\codim=2$, integrable, $J$-invariant foliation. After working out some generalities, we focus on certain normal linearized operators $\Dlinearized^{\pert}_{u, \domainJ,\foliationNormal}$ associated to curves in the manifolds $\R_{s} \times M$ of \S \ref{Sec:DynamicsAndFoliations}. For many such curves, $u$, we'll be able to apply the automatic transversality results of \cite{Wendl:Automatic} to guarantee that such operators $\Dlinearized^{\pert}_{u, \domainJ,\foliationNormal}$ are surjective in Lemma \ref{Lemma:AppliedAutoTransverse}, which is the main result of this section. This transversality result will be essential in establishing the main result of \S \ref{Sec:ThickenedModuli}, which roughly states that the thickenings of certain moduli spaces of curves in $\R \times M$ used to define contact homology can be identified with thickenings of moduli spaces of curves in certain leaves of $\foliation$.

\subsection{Manifolds of maps}\label{Sec:MfldOfMaps}

We quickly run through a setup of a Banach manifold of maps of Riemann surfaces into a completion $(\complete{X}, \lambda)$ of a Liouville cobordism
\begin{equation*}
(X, \lambda): (\partial_{+}X, \alpha_{+}) \to (\partial_{-}X, \alpha_{-}).
\end{equation*}
This setup will be necessary to state and prove the results of this section as well as those of \S \ref{Sec:ThickenedModuli}. Following the notations of \S \ref{Sec:TameJOverview}, we assume that $\complete{X}$ is equipped with a tame almost complex structure $J$ which is $s$-invariant along its collar ends $\R_{\pm s \geq -\epsilon} \times \partial_{\pm}X$, where the $\{ 0\} \times \partial_{\pm}X$ in these collars are identified with the $\partial_{\pm}X \subset X \subset \complete{X}$.

For the time being we work with general curves of SFT type, meaning that non-zero genus and multiple positive punctures are allowed. We do this for three reasons:
\be
\item The results of Lemma \ref{Lemma:CurvesStuckInLeaves} are applicable to such curves.
\item This general setup is not much more difficult than for the curves of contact homology.
\item Application of this general setup will also allow us to see in \S \ref{Sec:NormalOperatorAnalysis} that the transversality results required to establish our main theorem apply \emph{only} to certain contact homology type curves. See the discussion preceding Lemma \ref{Lemma:AppliedAutoTransverse}.
\ee

\subsubsection{Orbit and homotopy class data}

We fix choices of
\be
\item an ordered tuple $\vec{\orbit}^{+} = (\orbit^{+}_{1}, \dots, \orbit^{+}_{\NpositivePunctures})$ of closed $\Reeb^{+}$ orbits,
\item an ordered tuple $\vec{\orbit}^{-} = (\orbit^{-}_{ 1}, \dots, \orbit^{-}_{ \NnegativePunctures})$ of closed $\Reeb^{-}$ orbits,
\item a homotopy class of maps
\begin{equation*}
\mapHomotopyClass \in \pi_{0}\left[\left(S_{g, \Npunctures}, \partial S_{g, \Npunctures}\right), \left(X, \left(\cup_{1}^{\NpositivePunctures} \orbit^{+}_{i}\right) \cup \left(\cup_{1}^{\NnegativePunctures} \orbit^{-}_{i}\right)\right) \right].
\end{equation*}
\ee

Here the $\Reeb^{\pm}$ are the Reeb vector fields of the $\alpha_{\pm}$. We write the covering multiplicity of each orbit $\orbit^{\pm}_{i}$ as
\begin{equation*}
	\cm^{\pm}_{i} \in \Z_{\geq 1}.
\end{equation*}
Write $\orbit_{\pm,i}'$ for the embedded Reeb orbits which are covered by the $\orbit^{\pm}_{i}$. Following \S \ref{Sec:ReebBasics} we assume that these embedded orbits are equipped with markers $m_{\orbit'_{+}}, m_{\orbit'_{-}}$.

In the last item above, the $S_{g, \Npunctures}$ are compact, connected, orientable surfaces of genus $g = g(\Sigma)$ with
\begin{equation*}
\Npunctures := \NnegativePunctures + \NpositivePunctures
\end{equation*}
boundary components. We require $\NpositivePunctures \in \Z_{\geq 1}$ but allow $\vec{\orbit}^{-}$ to be empty. To pass from homotopy classes of maps to relative second homology classes, we'll write
\begin{equation}\label{Eq:MapHomologyClass}
	[\mapHomotopyClass] \in H_{2}\left(X, \left(\cup_{1}^{\NpositivePunctures} \orbit^{+}_{i}\right) \cup \left(\cup_{1}^{\NnegativePunctures} \orbit^{-}_{i}\right)\right).
\end{equation}
Such classes (with $\Z$ coefficients) are well-defined using the fact that the domain $\Sigma$ of such a class is oriented.

\subsubsection{Moduli spaces of decorated surfaces}

Fix a non-negative integer $r$ and let $(\Sigma_{r}, \domainJ_{0})$ be an oriented Riemann surface with genus $g$, $\Npunctures$ punctures which we denote by
\begin{equation*}
\vec{\thicc{z}} = (\thicc{z}_{1}, \dots, \thicc{z}_{\Npunctures}),
\end{equation*}
and $r$ marked points. Write $(\overline{\Sigma}_{r}, \domainJ_{0})$ for the closed surface with $\Npunctures + r$ marked points determined by filling in the punctures and viewing them as marked points. We assume that 
\begin{equation*}
2g+ \Npunctures + r \geq 3
\end{equation*}
so that the surface is stable (e.g. in the sense of \cite{SFTCompactness}), implying that the group of holomorphic automorphisms of $(\overline{\Sigma}_{r}, \domainJ_{0})$ preserving the union of punctures and marked points (set-wise) is finite.

Provided such a punctured Riemann surface $(\Sigma_{r}, \domainJ)$, we will write $\halfcyl_{+, i} \subset \Sigma_{r}$ for holomorphically embedded half-infinite cylinders centered about the punctures $\thicc{z}_{i}, i=1, \dots, \NpositivePunctures$ and half cylinders $\halfcyl_{-, i}$ about the $\thicc{z}_{i}, i=\NpositivePunctures+1,\dots, \Npunctures$. We recall from Equation \eqref{Eq:HalfCylDef} that $\halfcyl_{\pm} = \R_{\pm p\geq 0} \times \Circle$ and we require that the point at $p \to \pm \infty$ in $\overline{\Sigma}_{r}$ corresponds to $\thicc{z}_{i}$ along each $\halfcyl_{\pm, i}$. A collection of \emph{symptotic markers} for such a is an ordered tuple
\begin{equation}\label{Eq:Markers}
\vec{m} = (m_{\thicc{z}_{1}}, \dots, m_{\thicc{z}_{\Npunctures}}), \quad m_{\thicc{z}_{i}} \in (T_{\thicc{z}_{i}}\overline{\Sigma} \setminus \{0\})/\R_{>0}
\end{equation}
of tangent directions at the punctures. Here the $\R_{>0}$ action is determined by rescaling by positive constants. Such asymptotic markers may be determined by tangent spaces of the half-infinite line segments $\{ q = q_{0}\} \subset \halfcyl_{\pm, i}$. See any of \cite{BH:ContactDefinition, SFTCompactness, BM:Orientations, Pardon:Contact} for further details.

Write $\ModSpace^{m}_{g, \Npunctures, r}$ for the moduli space of equivalence classes of triples $(\Sigma_{r}, \domainJ, \vec{m})$. Two such triples are considered to be equivalent if they are related by a biholomorphism carrying (ordered) sets of punctures, (ordered) sets of marked points, and asymptotic markers on the domain to the associated structures on the target. The elements of $\ModSpace^{m}_{g, \Npunctures, r}$ will be denoted by pairs $(\domainJ, \vec{m})$.

We have natural maps
\begin{equation*}
\ModSpace^{m}_{g, \Npunctures, r} \to \ModSpace_{g, \Npunctures, r} \to \ModSpace_{g, \Npunctures + r}
\end{equation*}
where $\ModSpace_{g, \Npunctures, r}$ forgets asymptotic markers and $\ModSpace_{g, \Npunctures + r}$ forgets the distinction between punctures and marked points. Thus $\ModSpace_{g, \Npunctures + r}$ is the classical Deligne-Mumford space, which is an orbifold in general and a manifold in the case $g=0$. See for example \cite[Appendix D]{MS:Curves}. Maps which forget at an $(r+1)$th marked point give the vertical arrows in a diagram
\begin{equation*}
\begin{tikzcd}
\ModSpace^{m}_{g, \Npunctures, r+1} \arrow[r]\arrow[d] & \ModSpace_{g, \Npunctures, r+1} \arrow[r]\arrow[d] &\ModSpace_{g, \Npunctures + r+1} \arrow[d]\\
\ModSpace^{m}_{g, \Npunctures, r} \arrow[r] & \ModSpace_{g, \Npunctures, r} \arrow[r] &\ModSpace_{g, \Npunctures + r}
\end{tikzcd}
\end{equation*}
The fibers of the first vertical arrow over a pair $(\domainJ, \vec{m})$ may then be identified with the curve $(\Sigma_{r}, \domainJ, \vec{m})$.

\subsubsection{Manifold of maps}

Let $\Map^{\mapHomotopyClass, m}_{g, r}(\vec{\orbit}^{+}, \vec{\orbit}^{-})$ be the space of triples $(u, \vec{m}, \domainJ)$ consisting of elements $(\domainJ, \vec{m}) \in \ModSpace_{g, \Npunctures, r+1}$ and maps $u: \Sigma_{r} \to \complete{X}$ satisfying the following conditions:
\be
\item There is some $C > 0$ such that the subsets $\{ \pm p \geq C\} \subset \halfcyl_{\pm, i}$ are mapped into the half-cylindrical ends $\R_{\pm s \geq -\epsilon} \times \partial_{\pm}X$ of $\complete{X}$ (with $\pm$ signs matching). Along such ends we can write
\begin{equation}\label{Eq:SplitUAlongEnds}
	u = (s, u_{\partial _{\pm}X}).
\end{equation}
\item There exist half-cylindrical ends $\halfcyl_{\pm, i}$ of the domain along which $u$ has the above form with
\begin{equation}\label{Eq:OrbitConvergence}
	\lim_{p \to \pm \infty} u_{\partial_{\pm} X}(p, \ast): \Circle_{q} \to \partial_{\pm}X
\end{equation}
is a constant speed parameterization of $\orbit^{\pm}_{i}$.
\item Along each $\halfcyl_{\pm, i}$, the line segment $\{q=q_{0}\}$ corresponding to the marker $m_{\thicc{z}^{\pm}_{i}}$ is such that $u_{\partial _{\pm}X}(p,q)$ converges to the marker $m_{\orbit^{\pm}_{0,i}}$ on the embedded orbit $\orbit^{\pm}_{0,i}$ underlying $\orbit^{\pm}_{i}$ as $p \to \pm \infty$.
\item Consider the compactification $S_{g, \Npunctures}$ of $\Sigma$ obtained by adding circles at infinity $\{ \pm \infty\} \times \Circle$ to each neighborhood $\R_{\pm p \geq 0} \times \Circle_{q} = \halfcyl_{\pm, i} \subset \Sigma$ about a puncture. Let $\psi$ be a diffeomorphism $\R_{s \geq -\epsilon} \to [-\epsilon, 0)$ with positive derivative and define a map $\Psi: \complete{X} \to X$ taking the form $(s, x) \to (\pm \psi(\pm s), x)$ along the collars. Then we require that the map $\Psi \circ u: \Sigma \to X$ extends over the compactification $S_{g, \Npunctures}$ to a map in the homotopy class of $\mapHomotopyClass$,
\begin{equation*}
	\overline{\Psi \circ u}: (S_{g, \Npunctures}, \partial S_{g, \Npunctures}) \to \left(X, \left(\cup_{1}^{\NpositivePunctures} \orbit^{+}_{i}\right) \cup \left(\cup_{1}^{\NnegativePunctures} \orbit^{-}_{i}\right)\right), \quad [\overline{\Psi \circ u}] = \mapHomotopyClass.
\end{equation*}
\ee
We say that $\Map^{\mapHomotopyClass, m}_{g, r}(\vec{\orbit}^{+}, \vec{\orbit}^{-})$ is the \emph{manifold of maps with $r$-unrestrained marked points}.

\subsubsection{Teichm\"{u}ller slices and coordinate charts in the manifold of maps}

We will use Teichm\"{u}ller slices (cf. \cite{BH:ContactDefinition, Wendl:Automatic}) to give $\Map^{\mapHomotopyClass, m}_{g, r}(\vec{\orbit}^{+}, \vec{\orbit}^{-})$ the structure of a Banach manifold. This will be especially useful for defining and computing linearized $\delbar$ equations.

Fix a punctured Riemann surface with $r$ marked points and asymptotic markers at punctures, $(\Sigma_{r}, \domainJ_{0}, \vec{m})$. Choose a linear subspace and a ball inside of it centered about the origin,
\begin{equation*}
T_{\domainJ_{0}} \subset \Omega^{0, 1}(\Sigma_{r}, \domainJ_{0}),\quad \B_{\domainJ_{0}} \subset T_{\domainJ_{0}}
\end{equation*}
so that
\be
\item $\dim T_{\domainJ_{0}} = 6g + 2(\Npunctures + r) - 6$,
\item the support of each $\dot{\domainJ} \in T_{\domainJ_{0}}$ is compact and disjoint from the punctures and marked points and
\item the map
\begin{equation}\label{Eq:TeichSlice}
\B_{\domainJ_{0}} \to \Teich_{g,\Npunctures + r}, \quad \dot{\domainJ} \mapsto \domainJ = \domainJ_{0} e^{\dot{\domainJ}}
\end{equation}
determines a ($\codim=0$) embedding into the Teichm\"{u}ller space, $\Teich_{g,\Npunctures+r}$, of closed, genus $g$ Riemann surfaces with $\Npunctures + r$ marked points.
\ee
In the language of \cite{BH:ContactDefinition, Wendl:Automatic}, the image
\begin{equation*}
\Teich_{\domainJ_{0}} \subset \Teich_{g, \Npunctures + r}
\end{equation*}
of $\B_{\domainJ_{0}}$ under the embedding of Equation \eqref{Eq:TeichSlice} is a \emph{Teichm\"{u}ller slice}. Observe that this embedding induces a trivialization
\begin{equation*}
T\Teich_{\domainJ_{0}} \simeq T_{\domainJ_{0}} \times \Teich_{\domainJ_{0}} \to \Teich_{\domainJ_{0}}.
\end{equation*}

Let $\Map^{\mapHomotopyClass}_{g, r, \domainJ_{0}}(\vec{\orbit}^{+}, \vec{\orbit}^{-})$ be the space maps $u: \Sigma_{r} \to \complete{X}$ satisfying the following conditions:
\be
\item There is some $C > 0$ such that the subsets $\{ \pm p \geq C\} \subset \halfcyl_{\pm, i}$ are mapped into the half-cylindrical ends $\R_{\pm s \geq -\epsilon} \times \partial_{\pm}X$ of $\complete{X}$ (with $\pm$ signs matching). Along such ends we can write $u = (s, u_{\partial _{\pm}X})$.
\item There exist half-cylindrical ends $\halfcyl_{\pm, i}$ of the domain along which $u$ limits to a constant speed parameterization of $\orbit^{\pm}_{i}$.
\item $u$ is within the homotopy class of $\mapHomotopyClass$ as described above.
\ee
Observe that this space neglects asymptotic markers and we view the domains of such maps as having a fixed complex structure. The standard way to properly equip $\Map^{\mapHomotopyClass}_{g, r, \domainJ_{0}}(\vec{\orbit}^{+}, \vec{\orbit}^{-})$ with the structure of a Banach manifold includes
\be
\item specifying that the maps $u$ satisfy a local Sobolev regularity constraint,
\item choosing Sobolev regularity constants to ensure that the maps $u$ are sufficiently differentiable (typically $\mathcal{C}^{2}$) via Sobolev embedding theorems,
\item specifying the convergence rates of the $u_{\partial_{\pm} X}$ (of Equation \eqref{Eq:OrbitConvergence}) over the $\halfcyl_{\pm, i}$ using weighted Sobolev spaces, and
\item ensuring that these convergence rates are chosen so that $\Map_{g, r, \domainJ_{0}}^{\mapHomotopyClass}(\vec{\orbit}^{+}, \vec{\orbit}^{-})$ includes all (finite energy) punctured holomorphic curves of interest with complex structure $\domainJ_{0}$ on the target (e.g. via application of the results of \cite{HWZ:Asymptotics}).
\ee
Our analysis will be independent of these choices, so we do not specify particular choices of Sobolev constants and weights. We assume the above conditions hold and defer to the citations \cite{BH:ContactDefinition, SFTCompactness, BM:Orientations, Pardon:Contact} for details. With such choices in place, the tangent space of $\Map_{g, r, \domainJ_{0}}^{\mapHomotopyClass}(\vec{\orbit}^{+}, \vec{\orbit}^{-})$ is then identified as
\begin{equation*}
T_{u}\Map_{g, r, \domainJ_{0}}^{\mapHomotopyClass}(\vec{\orbit}^{+}, \vec{\orbit}^{-}) = \Sec(u^{\ast}T\complete{X})
\end{equation*}
where the right-hand side is a space of sections satisfying conditions on Sobolev regularity and decay along neighborhoods of punctures. To identify sections of $u^{\ast}T\complete{X} \to \Sigma$ with germs of paths in the space of maps, we apply such sections to the exponential map of a geodesically-complete Riemannian metric $g_{\complete{X}}$ on $\complete{X}$.

For each $u \in \Map_{g, r}^{\mapHomotopyClass}(\vec{\orbit}^{+}, \vec{\orbit}^{-})$ there are $\cm^{\pm}_{i}$ choices of half-infinite line-segments $\{ q = q_{0}\} \subset \halfcyl_{\pm, i}$ such that the restriction of the map $u_{\partial_{\pm}X}$ (Equation \eqref{Eq:SplitUAlongEnds}) to this line segment converges to the marker on the underlying embedded Reeb orbit. Thus the space of pairs
\begin{equation*}
(u, \vec{m}) \in \Map_{g, r, \domainJ_{0}}^{\mapHomotopyClass, m}(\vec{\orbit}^{+}, \vec{\orbit}^{-})
\end{equation*}
consisting of a $u \in \Map_{g, r}^{\mapHomotopyClass}(\vec{\orbit}^{+}, \vec{\orbit}^{-})$ together with a choice of asymptotic markers on the domain then has the structure of a fiber bundle
\begin{equation*}
G_{\vec{\cm}} \to \Map_{g, r, \domainJ_{0}}^{\mapHomotopyClass, m}(\vec{\orbit}^{+}, \vec{\orbit}^{-}) \to \Map_{g, r, \domainJ_{0}}^{\mapHomotopyClass}(\vec{\orbit}^{+}, \vec{\orbit}^{-}), \quad G_{\vec{\cm}} = \prod_{\orbit^{\pm}_{i}} \left(\Z/\cm^{\pm}_{i}\Z\right)
\end{equation*}
such that the projection corresponds to modding out by the obvious $G_{\vec{\cm}}$ action.

Finally we define
\begin{equation*}
\Map_{g, r, \Teich_{\domainJ_{0}}}^{\mapHomotopyClass, m}(\vec{\orbit}^{+}, \vec{\orbit}^{-}) = \Map_{g, r, \domainJ_{0}}^{\mapHomotopyClass, m}(\vec{\orbit}^{+}, \vec{\orbit}^{-}) \times \Teich_{\domainJ_{0}}
\end{equation*}
to be the space of triples $(u, \vec{m}, \domainJ)$ where
\be
\item $u$ is a map in $\Map_{g, r}^{\mapHomotopyClass, m}(\vec{\orbit}^{+}, \vec{\orbit}^{-})$,
\item $\vec{m}$ is a choice of asymptotic markers, and 
\item $\domainJ$ is a complex structure on $\Sigma$ in our Teichm\"{u}ller slice.
\ee
As the choices of markers are discrete, this space is a Banach manifold whose tangent space is given by
\begin{equation*}
T_{(u,\vec{m},\domainJ)}\Map_{g, r, \Teich_{\domainJ_{0}}}^{\mapHomotopyClass, m}(\vec{\orbit}^{+}, \vec{\orbit}^{-}) = \Sec(u^{\ast}T\complete{X}) \oplus T_{\domainJ_{0}}.
\end{equation*}

We have natural embeddings
\begin{equation*}
\Map_{g, r, \Teich_{\domainJ_{0}}}^{\mapHomotopyClass, m}(\vec{\orbit}^{+}, \vec{\orbit}^{-}) \to \Map_{g, r}^{\mapHomotopyClass, m}(\vec{\orbit}^{+}, \vec{\orbit}^{-})
\end{equation*}
which we view as Banach manifold charts on the manifold of maps with $r$ unrestrained marked points determined by the Teichm\"{u}ller slice $\Teich_{\domainJ_{0}}$ centered about the complex structure $\domainJ_{0}$. These charts clearly cover the manifold of maps and it is not difficult to check that the associated transition maps are $\COne$ differentiable for appropriate choices of Sobolev coefficients and weights.

\subsection{$\delbar$ equations and their linearizations}

Continuing the above discussion, we also have a Banach space bundle
\begin{equation*}
\BanachBundle \to \Map_{g, r}^{\mapHomotopyClass, m}(\vec{\orbit}^{+}, \vec{\orbit}^{-})
\end{equation*}
whose fiber over $(u, \vec{m}, \domainJ)$ is the space of $(\domainJ, J)$ anti-holomorphic one-forms on $\Sigma$ with values in $u^{\ast}T\complete{X}$,
\begin{equation*}
\BanachBundle_{(u, \vec{m}, \domainJ)} = \Omega^{0, 1}_{J, \domainJ}(u^{\ast}T\complete{X}).
\end{equation*}
While the markers $\vec{m}$ so far play no role in the definitions of tangent bundles and of $\BanachBundle$, they will later be necessary for orienting thickened moduli spaces and gluing constructions. Using Teichm\"{u}ller slices it can be seen that $\BanachBundle$ has the structure of a $\COne$-differentiable vector bundle over the manifold of maps with $r$-unrestrained marked points.

Using this bundle, we can set up a ``universally perturbed'' $\delbar$ equation as follows. Consider
\begin{equation}\label{Eq:PertSpaceDef}
\pertSpace \subset \Sec(T\complete{X} \boxtimes T^{\ast}\Sigma)
\end{equation}
to be the space of perturbations as described in \S \ref{Sec:PertSupport} (and in particular, satisfying Assumptions \ref{Assump:PertCompactSupport} which concerns the supports of the perturbations). Define
\begin{equation}\label{Eq:DelbarBig}
\begin{gathered}
\Map_{g, r}^{\mapHomotopyClass, m}(\vec{\orbit}^{+}, \vec{\orbit}^{-}) \times \pertSpace \to \BanachBundle,\\
((u,\vec{m},\domainJ), \pert) \mapsto \delbar_{J, \domainJ}^{\pert}u = \delbar_{J, \domainJ}u - \pert^{0,1} \in \Omega^{0, 1}_{J, \domainJ}(u^{\ast}T\complete{X}).
\end{gathered}
\end{equation}

Let us write $\exp$ for the exponential map and $\nabla$ for the covariant derivative associated to our metric $g_{\complete{X}}$ on $\complete{X}$. To define and compute the linearization
\begin{equation}\label{Eq:LinearizationGeneral}
\Dlinearized: T\Map_{g, r}^{\mapHomotopyClass, m}(\vec{\orbit}^{+}, \vec{\orbit}^{-}) \times \pertSpace \to \BanachBundle
\end{equation}
at a point $((u, \vec{m}, \domainJ), \pert)$ we consider about a Teichm\"{u}ller slice $\Teich_{\domainJ_{0}}$, consider
\be
\item $\dot{u} \in \Sec(u^{\ast}T\complete{X})$ so that $u_{t} = \exp_{u}(t\dot{u})$ is a $(-\epsilon, \epsilon)_{t}$-family of maps $\Sigma \to \complete{X}$,
\item $\dot{\domainJ} \in T_{\domainJ_{0}}$ so that $\domainJ e^{t\dot{\domainJ}} \in \Teich_{\domainJ_{0}}$ is a $(-\epsilon, \epsilon)_{t}$-family of complex structures on $\Sigma$ (for $\epsilon$ small), and
\item $\dot{\pert} \in \pertSpace$ so that $\pert_{t} = \pert + t\dot{\pert}$ is a $(-\epsilon, \epsilon)_{t}$-family of perturbations.
\ee
Then the linearization at such a point is defined as
\begin{equation*}
\begin{gathered}
T\Map_{g, r, \Teich_{\domainJ_{0}}}^{\mapHomotopyClass, m}(\vec{\orbit}^{+}, \vec{\orbit}^{-}) \oplus \pertSpace = \Sec(u^{\ast}T\complete{X})\oplus T_{\domainJ_{0}} \oplus \pertSpace \to \Omega^{0, 1}_{J, \domainJ}(u^{\ast}T\complete{X}),\\
(\dot{u}, \dot{\domainJ}, \dot{\pert}) \mapsto \Dlinearized^{\pert}_{u, \domainJ}(\dot{u}, \dot{\domainJ}, \dot{\pert}) = \partial_{t}|_{t=0} \left(\delbar^{\pert_{t}}_{J, \domainJ e^{t\dot{\domainJ}}} u_{t}\right).
\end{gathered}
\end{equation*}
Via a mild generalization of \cite[Proposition 3.1.1]{MS:Curves}, we compute
\begin{equation}\label{Eq:DlinearizedBig}
\begin{aligned}
\Dlinearized^{\pert}_{u, \domainJ}(\dot{u}, \dot{\domainJ}, \dot{\pert})=& (\nabla \dot{u})^{0,1} - \half \left(J(\nabla_{\dot{u}}J)(\delbar_{J, \domainJ}u)  - J(Tu) \domainJ\dot{\domainJ}\right)\\
&- \dot{\pert}^{0,1} - (\nabla_{\dot{u}}\pert)^{0,1} - \half \left((\nabla_{\dot{u}}J)\pert \domainJ + J\pert \domainJ \dot{\domainJ}\right).
\end{aligned}
\end{equation}

\subsection{Linearized $\delbar$ operators in the presence of holomorphic foliations}

Now we combine the above notation with that of \S \ref{Sec:PertAdaptedToFoliation} to study the linearized operator $\Dlinearized^{\pert}_{u, \domainJ}$ of Equation \eqref{Eq:DlinearizedBig} in the presence of a holomorphic foliation.

We assume that $T\foliation \subset T\complete{X}$ is a $\codim=2$, integrable, $J$-invariant distribution with associated foliation $\foliation$. The leaves of $\foliation$ are denoted by $\leaf$. We do not necessarily assume that $\pert$ is adapted to the foliation $\foliation$ (Definition \ref{Def:PertAdaptedToFoliation}). The subspace of $\pertSpace$ of perturbations adapted to $\foliation$ will be denoted by
\begin{equation*}
\pertSpace_{\foliation} = \pertSpace \cap \Sec(T\foliation \boxtimes T^{\ast}\Sigma).
\end{equation*}
We also write $\foliationNormal \subset T\complete{X}$ for a smooth subbundle for which
\begin{equation}\label{Eq:TXSplitting}
T\foliation \oplus \foliationNormal  = T\complete{X}, \quad J \foliationNormal = \foliationNormal.
\end{equation}
Such a $\foliationNormal$ always exists by taking the orthogonal complement $\foliationNormal = (T\foliation)^{\perp}$ of $T\foliation$ in $T\complete{X}$ with respect to some $J$-invariant Riemann metric on $\complete{X}$. Alternatively, one can (non-uniquely) construct a $J$-invariant metric with respect to which $\foliationNormal = (T\foliation)^{\perp}$ provided a given choice of $\foliationNormal$. Observe that such a $\foliationNormal$ is neither uniquely determined by $T\foliation$ (for $\dim \complete{X} \geq 4$) nor necessarily integrable. We now assume that our Riemannian metric $g_{\complete{X}}$ is $J$-invariant and that $T\foliation$ is totally geodesic with respect to $g_{\complete{X}}$.

Using the splitting of $T\complete{X}$ described in Equation \eqref{Eq:TXSplitting} we may decompose $\Sec(u^{\ast}T\complete{X})$ and the fibers of $\BanachBundle$ as
\begin{equation*}
\Sec(u^{\ast}T\complete{X}) = \Sec(u^{\ast}T\foliation)\oplus\Sec(u^{\ast}\foliationNormal), \quad \Omega^{0,1}_{J, \domainJ}(u^{\ast}T\complete{X}) = \Omega^{0,1}_{J, \domainJ}(u^{\ast}T\foliation) \oplus \Omega^{0,1}_{J, \domainJ}(u^{\ast}\foliationNormal).
\end{equation*}
Therefore the linearized $\delbar$ operator can be expressed in block matrix form as
\begin{equation}\label{Eq:DlinearizedBigBlock}
\begin{gathered}
\Dlinearized^{\pert}_{u, \domainJ} = \begin{pmatrix}
\Dlinearized^{\pert}_{u,\domainJ, T\foliation} && \Dlinearized^{\pert}_{u, \domainJ, UR}\\
\Dlinearized^{\pert}_{u, \domainJ, LL} && \Dlinearized^{\pert}_{u, \domainJ,\foliationNormal}
\end{pmatrix},\\
\begin{aligned}
\Dlinearized^{\pert}_{u, \domainJ, T\foliation} &= \pi_{\Omega^{0,1}_{J, \domainJ}(u^{\ast}T\foliation)}\Dlinearized^{\pert}_{u, \domainJ}|_{\Sec(u^{\ast}T\foliation)\oplus T_{\domainJ_{0}} \oplus \pertSpace},\\
\Dlinearized^{\pert}_{u, \domainJ, UR} &= \pi_{\Omega^{0,1}_{J, \domainJ}(u^{\ast}T\foliation)}\Dlinearized^{\pert}_{u, \domainJ}|_{\Sec(u^{\ast}\foliationNormal)},\\
\Dlinearized^{\pert}_{u, \domainJ, LL} &= \pi_{\Omega^{0,1}_{J, \domainJ}(u^{\ast}\foliationNormal)}\Dlinearized^{\pert}_{u, \domainJ}|_{\Sec(u^{\ast}T\foliation)\oplus T_{\domainJ_{0}} \oplus \pertSpace},\\
\Dlinearized^{\pert}_{u, \domainJ, \foliationNormal} &= \pi_{\Omega^{0,1}_{J, \domainJ}(u^{\ast}\foliationNormal)}\Dlinearized^{\pert}_{u, \domainJ}|_{\Sec(u^{\ast}\foliationNormal)}.
\end{aligned}
\end{gathered}
\end{equation}
Here $UR$ and $LL$ stand for ``upper-right'' and ``lower-left'', respectively.

\begin{lemma}\label{Lemma:LinearizationUpperTriangular}
In the above notation, suppose that
\be
\item the domain of Equation \eqref{Eq:DelbarBig} is restricted to $\Map_{g, r, \domainJ_{0}}^{\mapHomotopyClass, m}(\vec{\orbit}^{+}, \vec{\orbit}^{-}) \times \pertSpace_{\foliation}$ so that the linearized operator at a map $(u, m, \domainJ)$ and a $\pert$ adapted to $\foliation$ becomes
\begin{equation*}
\Dlinearized^{\pert}_{u, \domainJ}: \Sec(u^{\ast}T\complete{X})\oplus T_{\domainJ_{0}} \oplus \pertSpace_{\foliation} \to \Omega^{0,1}(u^{\ast}T\complete{X}),
\end{equation*}
\item $u$ solves the perturbed $\delbar$ equation, $\delbar_{J, \domainJ}u = \pert^{0,1}|_{u}$,
\item and $\im u$ is contained in a leaf $\leaf$ of $\foliation$.
\ee
Then the block matrix form of $\Dlinearized^{\pert}_{u, \domainJ}$ given in Equation \eqref{Eq:DlinearizedBigBlock} can be refined as
\begin{equation*}
\begin{gathered}
\Dlinearized^{\pert}_{u, \domainJ} = \begin{pmatrix}
\Dlinearized^{\pert}_{u, \domainJ, T\foliation} && \Dlinearized^{\pert}_{u, \domainJ, UR}\\
0 && \Dlinearized^{\pert}_{u, \domainJ,\foliationNormal}
\end{pmatrix},\\
\begin{aligned}
\Dlinearized^{\pert}_{u, \domainJ, T\foliation}&: \Sec(u^{\ast}T\leaf)\oplus T_{\domainJ_{0}} \oplus \pertSpace_{\foliation} \to \Omega^{0,1}_{J, \domainJ}(u^{\ast}T\leaf),\\
\Dlinearized^{\pert}_{u, \domainJ, UR} &: \Sec(u^{\ast}\foliationNormal) \to \Omega^{0,1}_{J, \domainJ}(u^{\ast}T\leaf),\\
\Dlinearized^{\pert}_{u, \domainJ,\foliationNormal} &: \Sec(u^{\ast}\foliationNormal) \to \Omega^{0,1}_{J, \domainJ}(u^{\ast}\foliationNormal)
\end{aligned}
\end{gathered}
\end{equation*}
In particular, $\pi_{\foliationNormal}\Dlinearized^{\pert}_{u, \domainJ}$ is independent of the variations
\be
\item $\dot{u}_{\leaf}$ of the map $u$ along vector fields tangent to $\leaf$,
\item $\dot{\pert}$ of the perturbation adapted to the foliation, and
\item $\dot{\domainJ}$ of the complex structures on $\Sigma$.
\ee
The operator $\Dlinearized^{\pert}_{u, \domainJ,\foliationNormal}$ takes the explicit form
\begin{equation*}
\dot{u}_{\foliationNormal} \in \Sec(u^{\ast}\foliationNormal), \quad \Dlinearized^{\pert}_{u, \domainJ,\foliationNormal}(\dot{u}_{\foliationNormal}) = \pi_{\foliationNormal}\left((\nabla \dot{u}_{\foliationNormal})^{0,1} - (\nabla_{\dot{u}_{\foliationNormal}}\pert)^{0,1} - \half (\nabla_{\dot{u}_{\foliationNormal}}J)\pert \domainJ \right)
\end{equation*}
\end{lemma}

\begin{proof}
The proof follows by inspecting the term-by-term contributions of Equation \eqref{Eq:DlinearizedBig} using the fact that $\nabla$ preserves $T\foliation$ and
\begin{equation*}
\pi_{T\foliation}u^{\ast}T\complete{X} = u^{\ast}T\leaf, \quad \delbar_{J, \domainJ}u \in \Omega^{0, 1}(u^{\ast}T\leaf).
\end{equation*}
In particular $\Dlinearized^{\pert}_{u, \domainJ, LL}$ vanishes since each corresponding summand in Equation \eqref{Eq:DlinearizedBig} vanishes under the provided hypotheses. By eliminating summands in $\Dlinearized^{\pert}_{u, \domainJ,\foliationNormal}$ using the aforementioned facts, we obtain the explicit expression appearing in the statement of the lemma.
\end{proof}

\subsection{Automatic transversality of $\Dlinearized^{\pert}_{u, \domainJ,\foliationNormal}$ for curves in $\R_{s} \times M$}\label{Sec:NormalOperatorAnalysis}

Here we study the operator $\Dlinearized^{\pert}_{u, \domainJ,\foliationNormal}$ when
\be
\item $(\complete{X}, J) = (\R_{s} \times M, J)$, following the notation of \S \ref{Sec:LeafNotation} and
\item $u: \Sigma \to \R_{s} \times M$ is a perturbed $J$-holomorphic curve mapped to a leaf $\leaf$ of $\foliation$ as in Lemma \ref{Lemma:CurvesStuckInLeaves}.
\ee
In this setting Lemma \ref{Lemma:LinearizationUpperTriangular} is clearly applicable. We'll see that for $u$ of \emph{contact homology type} (meaning that there is a single positive asymptotic and $g(\Sigma)=0$) results from \cite{Wendl:Automatic} can often be applied to guarantee that $\Dlinearized^{\pert}_{u, \domainJ,\foliationNormal}$ is surjective.

We first describe the geometry of some choice of complex line bundle $\foliationNormal \to \Sigma$ providing a splitting as in Equation \eqref{Eq:TXSplitting} and an associated linearized operator $\Dlinearized^{\pert}_{u, \domainJ,\foliationNormal}$. To do so, we will need to recall notation from \S \ref{Sec:LeafNotation}. Note that for each of $M = \C_{\tau,\sigma} \times \divSet$, $\R_{\tau} \times \Liouville_{\sigma_{0}}$, and $\R_{\tau} \times \hypersurface$, the vector field $\partial_{\tau}$ never lies in the distribution $T\foliation$. See \S \ref{Sec:FoliationCDivSet}, \ref{Sec:FoliationRLiouville}, and \ref{Sec:FoliationRHypersurface}, respectively. Therefore we may take
\begin{equation*}
\foliationNormal = \R \partial_{\tau} \oplus \R J \partial_{\tau} \implies \foliationNormal|_{\R_{s} \times \subsetC \times \divSet} = \R \partial_{\tau} \oplus \R \partial_{\sigma}
\end{equation*}
where $\subsetC \subset \C$ is as described in \S \ref{Sec:LeafNotation}. In particular, $\foliationNormal$ comes equipped with a preferred global trivialization.

With this choice of $\foliationNormal$, we can easily work out $\Dlinearized^{\pert}_{u, \domainJ,\foliationNormal}$ along ends of $\Sigma = \dom u$ in \S \ref{Sec:AsymptoticSetup} using our previously established computations. Following the discussion of \S \ref{Sec:PertSupport}, for each puncture $\thicc{z}_{i} \in \overline{\Sigma}$ of $\Sigma$ which is $\pm$-asymptotic to a closed Reeb orbit $(\orbit_{i})_{\zeta_{i}}, \zeta_{i} \in \Crit(f(\pm s_{0}, \ast))$ there is a neighborhood $\halfcyl_{\pm} \subset \Sigma$ about the puncture along which our perturbation, $\pert$, vanishes. Then we may apply the local model of Equation \eqref{Eq:FloerZ} along which the function $f(\pm s_{0}, \ast)$ takes the form $f_{0, \zeta_{i}} + \frac{\epsilon_{\zeta_{i}}}{2}(x^{2} + \delta_{\zeta_{i}}y^{2})$ (in the notation of Assumptions \ref{Assump:MorseLocalModel}). Since that equation ($\pi_{\disk}\delbar_{J, \domainJ}u = \delbar u_{\disk} - (\grad f(\pm s_{0}, \ast) \otimes ds)^{0,1}$) is $\R$-linear in its arguments (as remarked immediately after the equation), the linearization takes an essentially equivalent form in local coordinates,
\begin{equation*}
(\Dlinearized^{\pert}_{u, \domainJ,\foliationNormal}\dot{u}_{\foliationNormal})|_{\halfcyl_{\pm}} = \delbar \dot{u}_{\foliationNormal} + \AsymptoticOp_{\foliationNormal, i}\dot{u}_{\foliationNormal}, \quad \AsymptoticOp_{\foliationNormal, i}\dot{u}_{\foliationNormal} = - \epsilon_{\zeta_{i}}\left(\begin{pmatrix}
1 & 0\\ 0 & \delta_{\zeta_{i}}
\end{pmatrix}\dot{u}_{\foliationNormal}\otimes ds\right)^{0,1}.
\end{equation*}
Here $ds \in \Omega^{1}(\Sigma)$ is determined by the $s$-component of the map $u=(s, u_{M})$ and $i$ indexes the punctures $\thicc{z}_{i}$ of $\Sigma$.

In the language of \cite[\S 2]{Wendl:Automatic}, $\Dlinearized^{\pert}_{u,\domainJ, \foliationNormal}$ is a ($\R$-linear) Cauchy-Riemann operator on $\foliationNormal \simeq \R^{2}$ with asymptotic operators $\AsymptoticOp_{\foliationNormal, i}$ associated to its punctures. Each such asymptotic operator has an associated Conley-Zehnder index, $\CZ(\AsymptoticOp_{\foliationNormal, i}) \in \Z$, which is easily computed to be
\begin{equation*}
\CZ(\AsymptoticOp_{\foliationNormal, i}) = -\Delta_{\ind_{\Morse}(f(\pm s_{0}, \ast),\zeta_{i}),0} = - \Delta_{\delta_{\zeta_{i}}, 1}= \begin{cases}
-1 & \ind_{\Morse}(f(\pm s_{0}, \ast), \zeta_{i})=0\\
0 & \ind_{\Morse}(f(\pm s_{0}, \ast), \zeta_{i})=1
\end{cases}
\end{equation*}
where $\Delta_{x,y}$ is the Kronecker delta (equal to $1$ if $x=y$ and $0$ otherwise). Since we are working under the assumption that $u$ is contained in a leaf, $\leaf$, of our foliation $\foliation$, it follows that
\be
\item all of the $\zeta_{i}$ agree with some $\zeta^{+} \in \Crit(f(s_{0}, \ast))$ at positive asymptotics and
\item all of the $\zeta_{i}$ agree with some $\zeta^{-} \in \Crit(f(-s_{0}, \ast))$ at negative asymptotics.
\ee
With these observations and notation on hand, we can compute the Fredholm index of $\Dlinearized^{\pert}_{u,\domainJ, \foliationNormal}$ using \cite[Theorem 3.3.11]{Schwarz:Thesis} (see also \cite[Equation (2.1)]{Wendl:Automatic}) to obtain
\begin{equation*}
\begin{aligned}
\ind(\Dlinearized^{\pert}_{u, \domainJ,\foliationNormal}) =& \chi(\Sigma) + \sum_{\text{positive asymptotics}}\CZ(\AsymptoticOp_{\foliationNormal, i}) - \sum_{\text{negative asymptotics}}\CZ(\AsymptoticOp_{\foliationNormal, i})\\
=& \chi(\Sigma) - \Delta_{\ind_{\Morse}(\zeta^{+}),0}\NpositivePunctures + \Delta_{\ind_{\Morse}(\zeta^{-}),0}\NnegativePunctures\\
=& 2 - 2g(\Sigma) - (\Delta_{\ind_{\Morse}(\zeta^{+}),0}+1)\NpositivePunctures + (\Delta_{\ind_{\Morse}(\zeta^{-}),0}-1)\NnegativePunctures.
\end{aligned}
\end{equation*}
Let's make the above formula more explicit by considering the possible asymptotics of the leaf, $\leaf$, of $\foliation$ which contains $\im u$:
\be
\item  If $\leaf$ is positively asymptotic to an $\ind_{\Morse}=0$ critical point and $\leaf$ either has no negative end or is negatively asymptotic to an $\ind_{\Morse}=0$ critical point, then
\begin{equation*}
\ind(\Dlinearized^{\pert}_{u, \domainJ,\foliationNormal}) = 2 - 2g(\Sigma) - 2\NpositivePunctures.
\end{equation*}
In particular for the curves of contact homology, we obtain
\begin{equation*}
\NpositivePunctures=1, g(\Sigma)=0 \implies \ind(\Dlinearized^{\pert}_{u, \domainJ,\foliationNormal}) = 0.
\end{equation*}
\item If $\leaf$ is positively asymptotic to some $\ind_{\Morse}=1$ critical point and either has no negative end or is negatively asymptotic to some $\ind_{\Morse}=0$ critical point, then
\begin{equation*}
\ind(\Dlinearized^{\pert}_{u, \domainJ,\foliationNormal}) = 2 - 2g(\Sigma) - \NpositivePunctures.
\end{equation*}
In particular for the curves of contact homology, we obtain
\begin{equation*}
\NpositivePunctures=1, g(\Sigma)=0 \implies \ind(\Dlinearized^{\pert}_{u, \domainJ,\foliationNormal}) = 1.
\end{equation*}
\item If $\leaf$ is both positively and negatively asymptotic to $\ind_{\Morse}=1$ critical points, then
\begin{equation*}
\ind(\Dlinearized^{\pert}_{u, \domainJ,\foliationNormal}) = 2 - 2g(\Sigma) - \NpositivePunctures- \NnegativePunctures = \chi(\Sigma).
\end{equation*}
In particular for the curves of contact homology, we obtain
\begin{equation*}
\NpositivePunctures=1, g(\Sigma)=0 \implies \ind(\Dlinearized^{\pert}_{u, \domainJ,\foliationNormal}) = \chi(\Sigma) = 1 - \NnegativePunctures.
\end{equation*}
\ee
By Equation \eqref{Eq:EmptyTrajectorySpace} of Assumptions \ref{Assump:MorseLocalModel}, there are no leaves which are positively asymptotic to $\ind_{\Morse}=0$ critical points of $f(s_{0}, \ast)$ and negatively asymptotic to $\ind_{\Morse}=1$ critical points of $f(-s_{0}, \ast)$. Therefore this case is excluded.

For the curves of rational symplectic field theory (allowing any $\NpositivePunctures \geq 1$ with $g(\Sigma)=0$) or ``full'' symplectic field theory (allowing any $\NpositivePunctures \geq 1$ and any $g(\Sigma) \geq 0$) \cite{EGH:SFTIntro}, none of the above cases guarantee non-negativity of $\ind(\Dlinearized^{\pert}_{u, \domainJ,\foliationNormal})$. We also see that $\ind(\Dlinearized^{\pert}_{u, \domainJ,\foliationNormal})$ may be negative for contact-homology type curves described in item (3) above. For contact homology curves in the first two items above, Wendl's automatic transversality criterion for line bundles over punctured curves \cite[Proposition 2.2]{Wendl:Automatic} is applicable to $\Dlinearized^{\pert}_{u, \domainJ,\foliationNormal}$. The first item of the following lemma will be most important for our contact homology computations. For similar results (applying \cite{Wendl:Automatic} to normal linearized operators), see for example \cite[Lemma 4.13]{Moreno:Thesis}.

\begin{lemma}\label{Lemma:AppliedAutoTransverse}
Following the above notation, suppose that $u: \Sigma \to \R_{s} \times M$ is a $\pert$-perturbed curve of contact homology type mapping to a leaf, $\leaf$, of $\foliation$ where $\pert$ is adapted to $\foliation$ and that we are in one of the following situations:
\be
\item $\leaf$ is positively asymptotic to an $\ind_{\Morse}=0$ critical point $\zeta^{+}$ and $\leaf$ is either negatively asymptotic to an $\ind_{\Morse}=0$ critical point $\zeta^{-}$ or has no negative end, so that $\ind(\Dlinearized^{\pert}_{u, \domainJ,\foliationNormal})=0$.
\item $\leaf$ is positively asymptotic to an $\ind_{\Morse}=1$ critical point $\zeta^{+}$ and $\leaf$ is either negatively asymptotic to an $\ind_{\Morse}=0$ critical point $\zeta^{-}$ or has no negative end, so that $\ind(\Dlinearized^{\pert}_{u, \domainJ,\foliationNormal})=1$.
\ee
Then $\Dlinearized^{\pert}_{u, \domainJ,\foliationNormal}$ is surjective.
\end{lemma}

\begin{proof}
Let's write $\Dlinearized = \Dlinearized^{\pert}_{u, \domainJ,\foliationNormal}$ to temporarily simplify our notation. To apply the results of \cite{Wendl:Automatic}, we must recall the following notation: The \emph{adjusted first Chern number}, $c_{1}(\Dlinearized)$ is defined
\begin{equation*}
2c_{1}(\Dlinearized) = \ind(\Dlinearized) - 2 + h(\Dlinearized)
\end{equation*}
where $h(\Dlinearized)$ is the number of punctures of $\Sigma$ whose associated asymptotic operators have even Conley-Zehnder index (which is denoted by $\# \Gamma_{0}$ in the citation). The following results are special cases of \cite[Proposition 2.2]{Wendl:Automatic}:
\be
\item If $\ind(\Dlinearized) \leq 0$ and $c_{1}(\Dlinearized) < 0$, then $\Dlinearized$ is injective.
\item If $\ind(\Dlinearized) \geq 0$ and $\ind(\Dlinearized) > c_{1}(\Dlinearized)$ then $\Dlinearized$ is surjective.
\ee

Under the hypothesis provided,
\begin{equation*}
h(\Dlinearized) = \Delta_{\ind_{\Morse}(\zeta^{+}), 1} + \NnegativePunctures\Delta_{\ind_{\Morse}(\zeta^{-}),1}.
\end{equation*}
Therefore in the case of item (1) in the statement of the lemma we have $\ind(\Dlinearized)=0$ and $c_{1}(\Dlinearized) = -1$, implying that $\Dlinearized$ is injective and hence surjective. In the cases of item (2) in the statement of the lemma we have $\ind(\Dlinearized)=1$ and $c_{1}(\Dlinearized)=0$, so that surjectivity is guaranteed.
\end{proof}

\section{Thickened moduli spaces for foliated cobordisms and curve counts}\label{Sec:ThickenedModuli}

The goal of this section is to establish the following result pertaining to contact homology algebras defined with $\Ring=\Field$, $\Ring=\Field[\ChernVar^{\pm 1}]$, and $\Ring=\Field[H_{2}(M)]$ coefficients for $\Field$ a field containing $\Q$.

\begin{lemma}\label{Lemma:CHMorphism}
Suppose that $M$ is one of the contact manifolds of \S \ref{Sec:DynamicsAndFoliations} and that $\actionBoundDivSet$ and $\actionBound$ are as in Lemmas \ref{Lemma:DynamicsOnDivSetC}, Lemma \ref{Lemma:DynamicsOnRLiouville}, or \ref{Lemma:DynamicsRHypersurface}. Write $\alphaDivSet$ and $\alpha$ for the associated contact forms on $\divSet$ and $M$ and let $\zeta$ be an $\ind_{\Morse}=0$ critical point of the function $f \in \Cinfty(\subsetC)$ so that for every closed Reeb orbit $\orbitDivSet$ of action $\leq \actionBoundDivSet$, we have an associated orbit $\orbit_{\zeta} = \{ \zeta \} \times \orbitDivSet$ in $M$. Then there are choices of perturbative data defining chain-level contact homology algebras so that the $\orbit_{\zeta}$ span a subalgebra
\begin{equation*}
\Algebra_{\zeta} \subset \Algebra(\alpha, \actionBound)
\end{equation*}
preserved by the contact homology differential, $\partial$, and so that the restriction of $\partial$ to the subalgebra counts perturbed holomorphic curves in $\R_{s} \times \{ \zeta \} \times \divSet$. When $\Algebra(\alpha, \actionBound)$ is defined using either $\Ring=\Q$ or $\Ring=\Q[\ChernVar^{\pm 1}]$ coefficients, then the association $\orbitDivSet \mapsto \orbit_{\zeta}$ determines an isomorphism
\begin{equation*}
\Algebra(\alphaDivSet, \actionBoundDivSet) \to \Algebra_{\zeta}
\end{equation*}
where $\Algebra(\alphaDivSet, \actionBoundDivSet)$ is the chain-level contact homology algebra for the pair $(\alphaDivSet, \actionBoundDivSet)$ determined by appropriate choices of perturbation and orientation data.

Moreover, suppose that $\R_{s} \times M$ is a non-$\R_{s}$-invariant cobordism which is a completion of strict contact manifolds $(M, \alpha^{\pm})$ as described in \S \ref{Sec:DynamicsAndFoliations}. Let $\leaf$ be a leaf of the foliation $\foliation$ on $\R_{s} \times M$ which is $\pm$ asymptotic to $\ind_{\Morse}=0$ critical points $\zeta_{\pm}$ of the $f(\pm s_{0}, \ast) \in \Cinfty(\subsetC \times \divSet)$. Then there are choices of perturbative data so that the CDGA inclusion morphisms $\Algebra_{\zeta_{\pm}} \to \Algebra(\alpha_{\pm}, \actionBound_{\pm})$ fit into a commutative diagram
\begin{equation*}
\begin{tikzcd}
\Algebra_{\zeta_{+}} \arrow[d, "\Phi_{\leaf}"] \arrow[r] & \Algebra(\alpha_{+}, \actionBound_{+}) \arrow[d, "\Phi_{\R_{s} \times M}"] \\
\Algebra_{\zeta_{-}} \arrow[r] & \Algebra(\alpha_{-}, \actionBound_{-})
\end{tikzcd}
\end{equation*}
where the map $\Phi_{\leaf}$ count perturbed holomorphic maps in $\leaf$. When using $\Ring=\Field$ or $\Ring=\Field[\ChernVar^{\pm 1}]$ coefficients, $\Phi_{\leaf}$ is the cobordism map $\Algebra(\alphaDivSet_{+}, \actionBoundDivSet_{+}) \to \Algebra(\alphaDivSet_{-}, \actionBoundDivSet_{-})$ between contact homology algebras for appropriate choices of orientation and perturbation data.
\end{lemma}

See \S \ref{Sec:CHBHDefinition} for details on our use of twisted coefficient systems.

We will provide a detailed proof of this result for chain-level contact homology algebras defined within the framework of \cite{BH:ContactDefinition}. Within this framework, the above result is an immediate consequence of Equation \eqref{Eq:ZeroSetHomeo} which concerns counts of perturbed holomorphic curves. We provide some explanation of how Lemma \ref{Lemma:CHMorphism} may be derived in the framework of \cite{Pardon:Contact} in \S \ref{Sec:PardonSpecialization}. Both \cite{BH:ContactDefinition} and \cite{Pardon:Contact} use
\be
\item variations of \cite{FOOO} to construct locally thickened moduli spaces (i.e. associated to Kuranishi charts) of stable maps satisfying perturbed $\delbar$ equations into completed symplectic cobordisms and then
\item diverge strategically regarding how information from local charts is synthesized to produce virtual fundamental cycles and hence counts of perturbed holomorphic curves.
\ee
We try to stay framework-agnostic as long as possible and combine the language of citations \cite{BH:ContactDefinition} and \cite{Pardon:Contact} prior to \S \ref{Sec:BHOverview}. Likewise, we work with curves of general SFT type (arbitrary genus and numbers of positive punctures) for most of the section.

\subsection{Transverse subbundles in the abstract}\label{Sec:TransverseSubbundles}

We momentarily work in an abstract functional analytic setting which will be applied to the manifold of maps construction from \S \ref{Sec:MfldOfMaps} shortly, in \S \ref{Sec:TransverseSubbundlesOverMfldOfMaps}.

Let $k \in \Z_{\geq 1}$, let $\pi: \BanachBundle\to \Banach$ be a $\mathcal{C}^{k}$-smooth Banach space bundle over a connected Banach manifold $\Banach$, and let $\phi: \Banach \to \BanachBundle$ be a $\mathcal{C}^{k}$-smooth section. By considering germs of paths in the total space of $\BanachBundle$ which are tangent to the fibers, we see that $\pi^{\ast}\BanachBundle$ is naturally a subbundle of the tangent bundle $T\BanachBundle$ of (the total space of) $\BanachBundle$. Provided any $\mathcal{C}^{k}$-smooth splitting
\begin{equation}\label{Eq:BanachBundleSplitting}
T\BanachBundle \simeq \pi^{\ast}\BanachBundle \oplus \pi^{\ast}T\Banach
\end{equation}
which identifies the subbundle $\pi^{\ast}\BanachBundle \subset T\BanachBundle$ with $\pi^{\ast}\BanachBundle \oplus 0$ we can define the linearization
\begin{equation*}
\Dlinearized_{x}: T\Banach|_{x} \to \BanachBundle|_{x}, \quad \Dlinearized_{x} = \pi_{\BanachBundle|_{x}}\circ T\phi|_{x}.
\end{equation*}

\begin{defn}
The section $\phi$ is \emph{Fredholm} if $\Dlinearized_{x}$ is a Fredholm operator for all $x \in \Banach$.
\end{defn}

The Fredholm property for $\phi$ and the Fredholm indices
\begin{equation*}
\ind \Dlinearized_{x} \in \Z
\end{equation*}
are then independent of the choice of splitting and the point $x \in \Banach$.

\begin{defn}\label{Def:TransverseSubbundle}
Provided such a Fredholm section, $\phi$, as above, we say that a finite rank subbundle $E^{0,1} \subset \BanachBundle$ is \emph{$\phi$-transverse} if for each $x \in \phi^{-1}(E^{0,1})$,
\begin{equation*}
\im \Dlinearized_{x} + E^{0,1} = \BanachBundle|_{x}.
\end{equation*}
\end{defn}

\begin{rmk}
The $0,1$ superscript on $E^{0,1}$ is cosmetic in the present context. It is included in anticipation of the fact that we will soon consider $E^{0,1}$ which are contained in bundles $\BanachBundle$ of complex anti-linear forms over manifolds of maps of Riemann surfaces into almost complex manifolds.
\end{rmk}

When $E^{0,1}$ is $\phi$-transverse, the implicit function theorem (cf. \cite[\S A.3]{MS:SymplecticIntro}) guarantees that $\phi^{-1}(E^{0,1}) \subset \Banach$ is a $\mathcal{C}^{k}$-smooth submanifold of
\begin{equation*}
\dim \phi^{-1}(E^{0,1}) = \ind \Dlinearized_{x} + \rank E^{0,1}.
\end{equation*}

\subsection{Transverse subbundles over manifolds of maps and holomorphic curve moduli spaces}\label{Sec:TransverseSubbundlesOverMfldOfMaps}

Now we apply the above abstract framework to the manifold of maps described in \S \ref{Sec:MfldOfMaps}. The result will be a variation of the construction of a single Kuranishi chart from \cite{FOOO}. We therefore assume the following notation:
\be
\item $\Banach \subset \Map_{g, r}^{\mapHomotopyClass}(\vec{\orbit}^{+}, \vec{\orbit}^{-})$ is some open set in the manifold of maps of Riemann surfaces with $r$ unrestrained marked points mapping into a completed Liouville cobordism with adapted almost complex structure $(\complete{X}, J)$ which are $\pm$ asymptotic to Reeb orbits in the ends of $\complete{X}$. We recall that the number, $r$, of marked points is chosen so that the topological type, $\Sigma$, of the domains of maps are stable.
\item Points in $\Banach$ are written as tuples $(u, \vec{m}, \domainJ)$ where $u$ is a map, $\vec{m}$ is a choice of markers, and $\domainJ$ is a complex structure on the domain $\Sigma$ of $u$. We require that $u$ lies in the homotopy class $\mapHomotopyClass$.
\item $\BanachBundle \to \Banach$ is the bundle of anti-holomorphic $1$-forms on $\Sigma$ with values in $u^{\ast}T\complete{X}$.
\item $\pertSpace$ is the space of perturbations as described in Equation \eqref{Eq:PertSpaceDef}.
\ee

With these choices in place the map
\begin{equation*}
\Banach \to \BanachBundle, \quad (u, \vec{m}, \domainJ) \mapsto \delbar_{J, \domainJ}u
\end{equation*}
is a Fredholm section. This is the map of Equation \eqref{Eq:DelbarBig} without perturbations. Indeed, by working within a subset $\Map_{g, r, \Teich_{\domainJ_{0}}}^{\mapHomotopyClass, m}(\vec{\orbit}^{+}, \vec{\orbit}^{-})$ associated to a Teichm\"{u}ller slice $\Teich_{\domainJ_{0}}$, we can split the tangent space of $\Banach$ as the direct sum of
\be
\item the tangent space of the manifolds of maps $\Map_{g, r, \domainJ_{0}}^{\mapHomotopyClass, m}(\vec{\orbit}^{+}, \vec{\orbit}^{-})$ with fixed complex structure on the domain, $\Sigma$, and 
\item the $\dim = 6g + 2(\Npunctures + r) - 6$ vector space $T_{\domainJ_{0}}$ which we identify with the tangent space of Teichm\"{u}ller space at $\domainJ_{0}$.
\ee
A choice of connection, $\nabla$, associated to a Riemannian metric on $\complete{X}$ then induces a splitting of $T\BanachBundle$ as in Equation \eqref{Eq:BanachBundleSplitting}. The associated linearized operator, $\Dlinearized_{u, \domainJ}$ is as described in Equation \eqref{Eq:DlinearizedBig} with $\pert=0$. The restriction of this operator to the first summand of $T\Banach$ is Fredholm (cf. \cite{Schwarz:Thesis}) and since the second summand is finite dimensional, $\Dlinearized_{u, \domainJ}$ is then Fredholm. The Fredholm index is then the Fredholm index of the restriction to the first summand plus the dimension of $T_{\domainJ_{0}}$. The index of $\Dlinearized_{u, \domainJ}$ is computed via a slight variation of \cite{Schwarz:Thesis}:
\begin{equation*}
\ind \Dlinearized_{u, \domainJ}|_{\Map_{g, r, \domainJ_{0}}^{\mapHomotopyClass}(\vec{\orbit}^{+}, \vec{\orbit}^{-})} = n\chi(\Sigma) + \sum \CZ_{\framing}(\orbit^{+}_{i}) - \sum \CZ_{\framing}(\orbit^{-}_{ i}) + 2c_{1}^{rel}(\framing) + \Npunctures.
\end{equation*}
In this formula
\be
\item $n = \half \dim \complete{X}$,
\item $\framing$ is a choice of framings of the Reeb orbits with associated relative first Chern class $c_{1}^{rel}(\framing) \in \Z$, and
\item $|\orbit^{\pm}_{i}|_{\framing} = n-3 + \CZ_{\framing}(\orbit^{\pm}_{i})$ are the contact homology gradings of the orbits with respect to $\framing$.
\ee
As described in \cite{BH:ContactDefinition, BM:Orientations, Wendl:Automatic}, the extra $\Npunctures$ summand (which is not present in the Floer case of \cite{Schwarz:Thesis}) comes from the fact that vectors in $T\Map_{g, r, \domainJ_{0}}^{\mapHomotopyClass, m}(\vec{\orbit}^{+}, \vec{\orbit}^{-})$ may have at most linear coefficients (plus exponential decaying terms) in their $\partial_{s}$ summands along ends of the domain. Taking into account the dimension of the Deligne-Mumford space of curves and applying some basic arithmetic, we compute
\begin{equation*}
\begin{aligned}
\ind \Dlinearized_{u, \domainJ} &= \ind \Dlinearized_{u, \domainJ}|_{\Map_{g, r, \domainJ_{0}}^{\mapHomotopyClass}(\vec{\orbit}^{+}, \vec{\orbit}^{-})} + \dim T_{\domainJ_{0}}\\
& = \sum |\orbit^{+}_{i}|_{\framing} - \sum |\orbit^{-}_{ i}|_{\framing} + 2c_{1}^{rel}(\framing) + 2(n-3)(1 - \NpositivePunctures - g) + 2r\\
&= \ind_{SFT}u + 2r.
\end{aligned}
\end{equation*}
Here $\ind_{SFT}u$ is the \emph{SFT index} of the curve with its marked points forgotten. For readers familiar with algebraic formalisms of (rational) $SFT$ \cite{CL:SFTStringTop, EGH:SFTIntro}, the summand $2(n-3)(1 - \NpositivePunctures - g)$ is accounted for by the $\hslash$ variable. For contact homology curves, this summand is zero.

Now let $E$ be a finite rank subbundle of the trivial bundle
\begin{equation*}
E \subset \Banach \times \pertSpace \to \Banach.
\end{equation*}
We write
\begin{equation*}
E^{0,1} \subset \BanachBundle
\end{equation*}
for the image of $E$ under the bundle morphism
\begin{equation*}
\Banach \times \pertSpace \to \BanachBundle, \quad ((u, \vec{m}, \domainJ), \pert) \mapsto \pert^{0,1}|_{(u,\vec{m},\domainJ)} = \half(\pert + J\pert \domainJ)|_{(u, \vec{m}, \domainJ)}.
\end{equation*}
Note that in general, $E^{0,1}$ will not be a vector bundle as the ranks of fiber-wise kernels of the above map restricted to $E$ may vary over $\Banach$.

\begin{defn}
The \emph{moduli space of parameterized maps with $r$ unrestrained marked points thickened by $E$} to be
\begin{equation*}
P\ModSpace(\Banach, E) = \{ ((u,\vec{m},\domainJ), \pert): \ \delbar_{J, \domainJ}u = \pert^{0,1}\} \subset E.
\end{equation*}
\end{defn}

The ``$P$'' (for ``parameterized'') decoration emphasizes that we are not considering maps to be equivalent if they differ by a holomorphic automorphism of the domain. Such equivalence relations will be dealt with shortly.

We observe that a nested pair of subbundles
\begin{equation*}
E' \subset E_{\pertSpace} \subset \Banach \times \pertSpace
\end{equation*}
naturally gives rise to an inclusion
\begin{equation}\label{Eq:ModSpaceInclusion}
P\ModSpace(\Banach, E') \subset P\ModSpace(\Banach, E).
\end{equation}
In particular, for any choice of $E$, the inclusion
\begin{equation*}
\Banach \simeq \Banach \times \{ 0\} \subset \Banach \times \pertSpace
\end{equation*}
induces a natural inclusion of the space of unperturbed holomorphic maps into the space of perturbed holomorphic maps,
\begin{equation*}
P\ModSpace(\Banach, 0) \hookrightarrow P\ModSpace(\Banach, E), \quad P\ModSpace(\Banach, 0) = \{(u, m, \domainJ) \ : \delbar_{J, \domainJ}u=0\} \subset \Banach.
\end{equation*}

Let us consider the case when
\be
\item $E^{0,1}$ is a vector bundle over the open set $\Banach$,
\item $E \to E^{0,1}$ is a bundle isomorphism, and
\item $E^{0,1}$ is $\delbar_{J, \domainJ}$-transverse (Definition \ref{Def:TransverseSubbundle}).
\ee
According to the application of the implicit function theorem in \S \ref{Sec:TransverseSubbundles}, $P\ModSpace(\Banach, E)$ is a smooth manifold of
\begin{equation}\label{Eq:ThiccModSpaceDimension}
\dim P\ModSpace(\Banach, E) = \ind \Dlinearized_{u, \domainJ} + \rank E = \ind_{SFT}u + 2r + \rank E.
\end{equation}
Provided our assumptions, we can alternatively write
\begin{equation*}
P\ModSpace(\Banach, E) = (\delbar_{J, \domainJ}^{\pert})^{-1}(0) \cap E = \left(\delbar_{J, \domainJ}^{\pert}|_{E}\right)^{-1}(0)
\end{equation*}
where $\delbar_{J, \domainJ}^{\pert}$ is the ``universally perturbed'' $\delbar$ equation of Equation \eqref{Eq:DelbarBig} whose linearization is described in Equation \eqref{Eq:LinearizationGeneral}.

Provided collections of orbits $\vec{\orbit}^{\pm}$ as above, a genus $g$, and a relative second homology class $[\mapHomotopyClass]$ (Equation \eqref{Eq:MapHomologyClass}) we write
\begin{equation*}
\ModSpace_{g}^{[\mapHomotopyClass], m}(\vec{\orbit}^{+}, \vec{\orbit}^{-}) = \left\{ (u, \vec{m}, \domainJ)\ :\ u: \Sigma \to \complete{X},\ \domainJ \in \ModSpace_{g, \Npunctures}, \ \delbar_{J, \domainJ}u = 0, \ [\im u] = [\mapHomotopyClass]\right\}/\sim
\end{equation*}
for the associated moduli space of holomorphic curves. Here $\sim$ is the equivalence relation defined by $(u, \vec{m}, \domainJ) \sim (u', \vec{m}', \domainJ')$ iff there is a holomorphic isomorphism $\phi: (\Sigma, \domainJ) \to (\Sigma, \domainJ')$ preserving asymptotic markers (Equation \eqref{Eq:Markers}) and for which $u = u'\circ \phi$. We emphasize that the domains $\Sigma$ are without marked points and are therefore not required to be stable. For any choice of open set $\Banach$ as above, we have a natural map
\begin{equation}\label{Eq:ModuliForgetfulMap}
P\ModSpace(\Banach, 0) \to \ModSpace_{g}^{[\mapHomotopyClass], m}(\vec{\orbit}^{+}, \vec{\orbit}^{-})
\end{equation}
obtained by forgetting the $r$ marked points and modding out by the relation $\sim$.

The geometric strategy underlying both \cite{BH:ContactDefinition} and \cite{Pardon:Contact} is to cover $\ModSpace_{g}^{[\mapHomotopyClass], m}(\vec{\orbit}^{+}, \vec{\orbit}^{-})$ with sets of the form $P\ModSpace(\Banach, 0)$ (determined by adding marked points) contained in thickened moduli spaces $P\ModSpace(\Banach, E)$ the latter of which are guaranteed to be smooth manifolds (by the transverse subbundle criteria). In \cite{BH:ContactDefinition} it is required that the $E^{0,1}$ are always transverse subbundles and in \cite{Pardon:Contact} the requirement appears as the ``covering axiom''.

Gromov compactifications of $\ModSpace_{0}^{[\mapHomotopyClass], m}(\vec{\orbit}^{+}, \vec{\orbit}^{-})$ \cite{SFTCompactness} must also be considered for applications to contact homology. The compactified moduli spaces of contact homology curves are much simpler than those of ``full'' SFT curves. Strata of these compactified moduli spaces may be organized using combinatorial trees described as follows:
\be
\item Let $\orbit$ be a $\Reeb^{+}$ orbit and let $\vec{\orbit}^{-}$ be an ordered tuple of $\Reeb^{-}$ orbits.
\item Let $\tree$ be a connected, directed tree with finitely many vertices, $\vertex$, and directed edges, $\edge$.
\item We associate to each vertex a manifold $X_{\vertex}$ which is one of either $\R_{s} \times \partial_{-}X$, $\complete{X}$, or $\R_{s} \times \partial_{+}X$.
\item For each vertex $\vertex$ of $\tree$, we also choose tuples $(\orbit_{\vertex}, \vec{\orbit}_{\vertex}^{-}, \mapHomotopyClass_{\vertex})$ of Reeb orbits, ordered tuples of Reeb orbits, and homotopy classes of maps.
\item When $X_{\vertex} = \R_{s} \times \partial_{\pm}X$, all orbits are $\Reeb^{\pm}$ orbits and the target of $\mapHomotopyClass$ is $[0,1]\times \partial_{\pm}X$. Otherwise $\orbit_{\vertex}$ is a $\partial_{+}X$ orbit, the $\orbit^{-}_{\vertex, i} \in \vec{\orbit}_{\vertex}^{-}$ are $\Reeb^{-}$ orbits, and the target of $\mapHomotopyClass$ is $X$.
\item The directed edges $\edge$ of $\tree$ connecting a vertex, $\vertex$, to another vertex, $\vertex'$, are labeled with a $\orbit_{\edge} = \orbit_{\vertex, i}^{-} \in \vec{\orbit}_{\vertex}^{-}$ which is required to be equal to $\orbit_{\vertex}$.
\item The edge data provides instructions for ``pregluing'' maps $(u_{\vertex}, \domainJ_{\vertex}, \vec{m}_{\vertex}) \in \Map_{0}^{\mapHomotopyClass_{\vertex}, m}(\orbit_{\vertex}, \vec{\orbit}_{\vertex}^{-})$ along half-cylindrical ends to produce a map which we require to live in $\Map_{0}^{\mapHomotopyClass, m}(\orbit, \vec{\orbit}^{-})$. This imposes combinatorial restrictions on the labeling data for edges and vertices when $\mapHomotopyClass$, $\orbit$, and $\vec{\orbit}$ have been specified.
\ee

\begin{defn}\label{Def:CHTree}
We say that $\tree$ together with the labeling data for vertices and edges is a \emph{$CH$ tree}, which we will denote simply by $\tree$.
\end{defn}

We remark that the above definition covers the definition of $CH$ trees from \cite{BH:ContactDefinition} (for both symplectizations and cobordisms with fixed almost complex structures) but only covers the definitions of $CH$ trees of ``type $I$ and $II$'' from \cite{Pardon:Contact}. A full treatment of the definition and proof of invariance of contact homology in \cite{Pardon:Contact} also requires definitions of $CH$ trees of ``types $III$ and $IV$'' which treat time-dependent almost complex structures on cobordisms. See also \cite{MorenoZhou:RSFTI} for a treatment of compactified moduli spaces for the curves of rational SFT (which also admit ``tree-like'' compactifications). In this section, we work primarily with the $\ModSpace_{g}^{[\mapHomotopyClass], m}(\vec{\orbit}^{+}, \vec{\orbit}^{-})$ rather than their compactifications.

\subsection{A quick review of \cite{BH:ContactDefinition}}\label{Sec:BHOverview}

In this subsection we summarize the construction of virtual fundamental cycles in \cite{BH:ContactDefinition}. We then describe how such virtual fundamental cycles are used to define contact homology over various coefficient systems in \S \ref{Sec:CHBHDefinition}. Up until that section, we assume that an action bound $\actionBound \in \R_{>0}$ has been fixed and that the Reeb vector fields $\Reeb^{\pm}$ are $\actionBound$-non-degenerate and that all Reeb orbits mentioned have action $\leq \actionBound$. Readers who are focused on using the techniques of \cite{Pardon:Contact} may skip to \S \ref{Sec:SubbundleExtension}, the beginning of which is framework-independent.

\subsubsection{Selection of open sets}\label{Sec:InteriorChartSubsets}

In \cite{BH:ContactDefinition}, the open subsets $\Banach$ of the manifold of maps are chosen as follows: When the target $(\complete{X}, J)$ of our maps is a symplectization $(\R_{s} \times M, J)$ (with $J$ being $\R_{s}$-translation invariant), one fixes $\R_{s}$-invariant subsets
\begin{equation*}
K \subset \ModSpace^{[\mapHomotopyClass]}(\orbit^{+}, \vec{\orbit}^{-})
\end{equation*}
such that the subset $K/\R_{s} \subset \ModSpace^{[\mapHomotopyClass]}(\orbit^{+}, \vec{\orbit}^{-})/\R_{s}$ of the moduli space reduced by $\R_{s}$-translation is compact. If this reduced moduli space is compact, we require that $K/\R_{s}$ is the entire reduced moduli space and otherwise require that the $K/\R_{s}$ cover a complement of a neighborhood of its boundary in the associated compactified moduli space $\overline{\ModSpace}^{[\mapHomotopyClass]}(\orbit^{+}, \vec{\orbit}^{-})$. In the non-symplectization case, we choose compact subsets $K$ which cover a complement of a neighborhood of the boundary of $\ModSpace^{[\mapHomotopyClass]}(\orbit^{+}, \vec{\orbit}^{-})$ inside of its compactification. If this moduli space is compact, we require that $K$ is the entire moduli space. The spaces $\Banach$ are indexed by the subsets $K$. When the target is a symplectization, we require that $\Banach$ is invariant under $\R_{s}$-translation.

Note that we are only presently working over compact subsets of the interiors of moduli spaces and will use this data to construct ``interior Kuranishi charts'' through \S \ref{Sec:ThickingOverInteriorCharts}. The construction of such charts along holomorphic curves which are close to breaking into multi-level SFT buildings \cite{SFTCompactness} will be dealt with in \S \ref{Sec:CloseToBreaking}.

For each subset $K$ of the moduli space as described above, we require that the associated subset $P\ModSpace(\Banach, 0)$ covers $K$ by the forgetful map of Equation \eqref{Eq:ModuliForgetfulMap}. Over each such $\Banach$, the number of marked points is given $r=2\cm(\orbit^{+})$ (twice the covering multiplicity of the positive asymptotic) when $\NnegativePunctures \leq 1$ and is zero otherwise (in which case the domains are maps are already stable without the addition of marked points).

\subsubsection{Construction of perturbations}\label{Sec:BHPerturbations}

We now provide a summary of the construction of perturbations over the open sets $\Banach$ following \cite[\S 5]{BH:ContactDefinition} and describe how such perturbations can be expressed in the language of this article (this is as elements of the space $\pertSpace \subset \Sec(T\complete{X}\boxtimes T^{\ast}\Sigma)$ of perturbations described around Equation \eqref{Eq:PertSpaceDef}). Getting into some details here will set us up to study the foliated cobordisms of \S \ref{Sec:DynamicsAndFoliations} in \S \ref{Sec:SubbundleExtension}.

Suppose that we have a simple enough neighborhood $\Norbit$ of an embedded Reeb orbit, $\orbit$, of action $a \in \R_{>0}$ in one of the $\partial_{\pm}X$. We assume that $J$ is simple (Definition \ref{Def:SimpleJ}) along a subset $\R_{\pm s \geq s_{0}} \times \Norbit$ contained within one of the ends of the completed cobordism $\complete{X}$. Following the notation of \S \ref{Sec:SimpleModels}, we can describe $\Norbit$ and the restriction of $J$ to $\R_{\pm s \geq s_{0}} \times \Norbit$ as
\begin{equation*}
\begin{gathered}
\Norbit \subset (\R_{t} \times \R^{2n_{0}}_{z_{0}} \times \R^{2n_{-}}_{z_{-}}) / \sim, \quad (a, z_{0}, z_{-}) \sim (0, z_{0}, -z_{-}),\\
J\partial_{s} = \partial_{t} + X_{0} + X_{-}, \quad J\partial_{x_{i}} = \partial_{y_{i}}
\end{gathered}
\end{equation*}
where the $\partial_{x_{i}}$ and $\partial_{y_{i}}$ are standard coordinates on $\R^{2n_{0}}_{z_{0}} \times \R^{2n_{-}}_{z_{-}} \simeq \C^{n_{0}} \times \C^{n_{-}}$.

Let $\twistedFunc_{n_{0},n_{-}}$ be the space of \emph{twisted periodic functions},
\begin{equation*}
(h_{0},h_{-}): [0, 2\pi] \to \R^{2n_{0}} \times \R^{2n_{-}}, \quad h_{0}(2\pi) = h_{0}(0), \quad h_{-}(2\pi) = -h_{-}(0).
\end{equation*}
Write $\twistedFunc_{n_{0},n_{-}}^{1,2}$ and $\twistedFunc_{n_{0},n_{-}}^{2}$ for the associated $\Sobolev^{1,2}$ and $\Ltwo$ spaces. Associated to our almost complex structure $J$ and a choice of covering multiplicity, $\cm$, we may define the \emph{normal asymptotic operator},
\begin{equation*}
\AsymptoticOp: \twistedFunc_{n_{0},n_{-}}^{1,2} \to \twistedFunc_{n_{0},n_{-}}^{2}, \quad \AsymptoticOp(h_{0},h_{-}) = -\frac{\cm a}{2\pi}J_{0}(\partial_{t}(h_{0},h_{-}) - (X_{0}h_{0}, X_{-}h_{-})).
\end{equation*}
We also have an associated \emph{(full) asymptotic operator},
\begin{equation*}
\widetilde{\AsymptoticOp}: \Sobolev^{1,2}(\Circle, \C) \oplus \twistedFunc_{n_{0},n_{-}}^{1,2} \to \Ltwo(\Circle, \C) \oplus \twistedFunc_{n_{0},n_{-}}^{2}, \quad \widetilde{\AsymptoticOp} = -J_{0}\partial_{t} \oplus \AsymptoticOp.
\end{equation*}
This operator is self-adjoint and hence determines a $\Ltwo$-orthogonal direct sum decomposition of the domain into eigenspaces. As may be seen by direct calculation of examples (cf. \cite[\S 4.1]{BH:Cylindrical}), the eigenspaces of $\AsymptoticOp$ are real (rather than complex) subspaces in general.

Since the full asymptotic operator splits as a direct sum (presentable as a diagonal block matrix), we can write the associated eigenvalue-eigenfunction pairs as
\begin{equation*}
(k, (e^{ikt}, 0)), \quad (k, (J_{0}e^{ikt}, 0)), \quad (\lambda_{\ell}, (0, \eta_{\ell}))
\end{equation*}
where $k \in \Z$ and the $(\lambda_{\ell}, \eta_{\ell})$ are eigenvalue-eigenfunction pairs associated to the operator $\AsymptoticOp$. Associated to these pairs we have vector fields on $\R_{\pm s \geq s_{0}} \times \Norbit$,
\begin{equation}\label{Eq:PertVectorFields}
e^{Jkt}\partial_{s}, \quad Je^{Jkt}\partial_{s} = e^{Jkt}(\partial_{t} + X_{0} + X_{-}), \quad \eta_{\ell}
\end{equation}
where for $\eta_{\ell}$, we identify $\R^{2n_{0}} \times \R^{2n_{-}}$ with the subspace $\ker(ds, dt) \subset T(\R_{\pm s \geq s_{0}} \times \Norbit)$.

Now suppose that $(u, \vec{m}, \domainJ) \in \Banach$ and that $\halfcyl_{\pm} = \R_{\pm p \geq 0} \times \Circle_{q} \subset \Sigma$ is a half-cylindrical subset centered about some puncture $\thicc{z}_{j}$ of the domain mapping to $\R_{\pm s \geq s_{0}} \times \Norbit$ which is $\pm$ asymptotic to the $\cm$-fold covering of $\orbit$. Then along $\halfcyl_{\pm}$, we can write $u = (s, t, u_{\disk})$ with $u_{\disk}$ taking values in the fibers of $\Norbit$ (Equation \eqref{Eq:NorbitFibration}). Let $b: \R_{\pm p \geq 0} \to [0, 1]$ be a cutoff function such that $b = 0$ for $|p| < 1$ and $b = 1$ for $|p| > 2$. The above vector fields determine elements of $\Omega^{1}(u^{\ast}T\complete{X})$ which are compactly supported in $\Sigma$ given by
\begin{equation*}
\mu_{s, k} = \partial_{p}b\ e^{Jkt}\partial_{s} \otimes ds, \quad \mu_{\Reeb, k} = \partial_{p}b\ Je^{Jkt}\partial_{s} \otimes ds, \quad \mu_{\ell}= \partial_{p}b\ \eta_{\ell} \otimes ds
\end{equation*}
along $\halfcyl_{\pm}$ (and $0$ outside of $\halfcyl_{\pm}$). For any $\eigenBound > 0$ we write
\begin{equation*}
F_{\eigenBound, \thicc{z}_{j}} = \left(\bigoplus_{|k| \leq \eigenBound}\left(\R \mu_{s,k} \oplus \R \mu_{\Reeb, k} \right)\right) \oplus \left(\bigoplus_{|\lambda_{\ell}| \leq \eigenBound}\mu_{\ell} \right) \subset \Sec(u^{\ast}T\complete{X}), \quad F_{\eigenBound} = \oplus_{\thicc{z_{j}}}F_{\eigenBound, \thicc{z}_{j}}.
\end{equation*}
The proof of \cite[Theorem 5.1.2]{BH:ContactDefinition} states that for $\eigenBound > 0$ sufficiently large and $u$ such that $u|_{\halfcyl_{\pm}}$ sufficiently $\COne$ close to a holomorphic map, the subspace
\begin{equation*}
F_{\eigenBound, \thicc{z}_{j}}^{0,1} = \left(\bigoplus_{|k| \leq \eigenBound}\left(\R \mu_{s,k}^{0,1} \oplus \R \mu_{\Reeb, k}^{0, 1}\right)\right) \oplus \left(\bigoplus_{|\lambda_{\ell}| \leq \eigenBound}\mu_{\ell}^{0, 1}\right) \subset \Omega^{0,1}_{J, \domainJ}(u^{\ast}T\complete{X})
\end{equation*}
determined by $F_{\eigenBound, \thicc{z}_{j}}$ and hence the subspace
\begin{equation*}
F_{\eigenBound}^{0,1} = \oplus_{\thicc{z_{j}}}F_{\eigenBound, \thicc{z}_{j}}^{0,1}
\end{equation*}
satisfies the transversality conditions
\begin{equation*}
\im \Dlinearized_{u} + F_{\eigenBound, \thicc{z}_{j}}^{0,1} = \im \Dlinearized_{u} + F_{\eigenBound}^{0,1} =\Omega^{0, 1}(u^{\ast}T\complete{X}).
\end{equation*}
To extend this single subspace to a transverse bundle over all of $\Banach$, it suffices to determine the half-infinite cylindrical neighborhoods $\halfcyl_{\pm}$ about the punctures $\thicc{z}_{j}$. To this end, we require that they are hyperbolic cusps living in the $\delta$-thin part of a $\delta$-thick-thin decomposition of $\Sigma$ associated to the unique hyperbolic metric for $(\Sigma, j)$ with $\delta$ sufficiently small. For an overview of thick-thin decomposition, see any of \cite{BH:ContactDefinition, SFTCompactness, Hummel}.

To make the above construction adhere to the language of this article, we seek to view the $F_{\eigenBound, \thicc{z}_{i}}$ as being determined by subspaces of $\pertSpace \subset \Sec(T\complete{X} \boxtimes T^{\ast}\Sigma)$ (Equation \eqref{Eq:PertSpaceDef}). To this end let $b_{\orbit}$ be a bump function on $\partial_{\pm}X$ equal to $1$ on a neighborhood $\Norbit$ of radius $\rho$ and equal to $0$ on the complement of a neighborhood of radius $\rho'>\rho$. In the $\R_{s}$-invariant case of a symplectization the sections
\begin{equation}\label{Eq:BumpedSections}
\mu^{b}_{s, k} = b_{\orbit}\partial_{p}b\ e^{Jkt}\partial_{s} \otimes ds, \quad \mu_{\Reeb, k}^{b} = b_{\orbit}\partial_{p}b\ Je^{Jkt}\partial_{s} \otimes ds, \quad \mu_{\ell}^{b}= b_{\orbit}\partial_{p}b\ \eta_{\ell} \otimes ds
\end{equation}
of $T\complete{X} \boxtimes T^{\ast}\Sigma$ determine subspaces
\begin{equation*}
E_{\eigenBound, \thicc{z}_{j}} = \left(\bigoplus_{|k| \leq \eigenBound}\left(\R \mu_{s,k}^{b} \oplus \R \mu_{\Reeb, k}^{b}\right)\right) \oplus \left(\bigoplus_{|\lambda_{\ell}| \leq \eigenBound}\mu_{\ell}^{b}\right) \subset \Sec\left(T\complete{X} \boxtimes T^{\ast}\Sigma\right)
\end{equation*}
of sections which determine the $\mu_{s, k}, \mu_{\Reeb, k}, \mu_{\ell}$ for a given map $u$ and are supported on a single cylindrical end of the domain centered about some puncture $\thicc{z}_{j}$ which is asymptotic to the $\cm$-fold covering of $\orbit$. In the case of a non-$\R_{s}$-invariant cobordism, we may replace the sections of Equation \eqref{Eq:BumpedSections} with
\begin{equation}\label{Eq:BumpedSectionsCob}
\mu^{b}_{s, k} = b(|s|)b_{\orbit}\partial_{p}b\ e^{Jkt}\partial_{s} \otimes ds, \quad \mu_{\Reeb, k}^{b} = b(|s|)b_{\orbit}\partial_{p}b\ Je^{Jkt}\partial_{s} \otimes ds, \quad \mu_{\ell}^{b}= b(|s|)b_{\orbit}\partial_{p}b\ \eta_{\ell} \otimes ds
\end{equation}
to ensure that their supports project to subsets of the ends $\R_{\pm s \geq 0} \times \partial_{\pm}X$ of the cobordism.

For $\eigenBound$ sufficiently large, the subspaces
\begin{equation*}
E_{\eigenBound} = \oplus_{\thicc{z}_{j}} E_{\eigenBound, \thicc{z}_{j}}
\end{equation*}
then form a subbundle of $\Banach \times \pertSpace$ so for $\rho \in \R_{>0}$ the associated disk bundle 
\begin{equation*}
E_{\eigenBound, \rho} \subset E_{\eigenBound} \subset \BanachBundle
\end{equation*}
consisting of sections of norm $\leq \rho$ exactly agrees with the disk bundle
\begin{equation*}
F_{\eigenBound, \rho} \subset F_{\eigenBound}
\end{equation*}
likewise consisting of section of norm $\leq \rho$. Indeed for $\rho$ small, the factors $b_{\orbit}$ in the above explicit descriptions of perturbations always evaluate to $1$.

To define our thickened moduli spaces associated to $\Banach$, we can choose any subbundle
\begin{equation*}
E \subset E_{\eigenBound}
\end{equation*}
for which the associated space
\begin{equation*}
E^{0,1} \subset E_{\eigenBound}^{0,1} \subset \BanachBundle
\end{equation*}
is $\delbar$ transverse along a neighborhood of the zero section. The bundle morphism $E \to E^{0,1}$ is an isomorphism since $F_{\eigenBound} \to F_{\eigenBound}^{0,1}$ is an isomorphism. Since the bundles $E^{0,1}$ are always assumed to be $\delbar_{J, \domainJ}$-transverse, the thickened moduli space of maps with unrestrained marked points, $P\ModSpace(\Banach, E)$, is guaranteed to be a smooth manifold.

In practice, we can replace $E_{\eigenBound}$ with any finite-rank subbundle of $\BanachBundle \to \Banach$ whose sections are sufficiently $\COne$-close to the sections spanning $E_{\eigenBound}$. Such a slight modification will appear in the construction of \S \ref{Sec:BHAdaptedExtension}.

\subsubsection{Marked points and constraints}

The locations of marked points on the domains, $\Sigma$, or maps (called removable punctures in \cite{BH:ContactDefinition}) may be constrained by the marker $m_{\orbit^{+}}$ on the orbit $\orbit^{+}$ and a small constant $\epsilon \in \R_{>0}$. The constraint requires that the marked points lie in the inverse image of $\R_{s \geq s_{1}} \times \partial_{+}X$ for some $s_{1} \gg 0$ along which $u$ is an immersion and we can write $u = (s, u_{\partial_{+}X})$. Specifically, for $\NnegativePunctures \leq 1$, marked points are chosen to lie on a circle $s^{-1}(s_{1})$ so that the integral of $\alpha_{+}$ along this circle is $\action(\orbit_{+}) - \epsilon$ and the location of the $r=2\cm(\orbit^{+})$ marked points are determined by $m_{\orbit^{+}}$. See \cite[\S 4.2.4]{BH:ContactDefinition} for details.

Imposing restraints on the locations of marked points using the data $\epsilon, m_{\orbit^{+}}$ yields a smooth manifold
\begin{equation*}
P\ModSpace_{\epsilon, m_{\orbit^{+}}}(\Banach, E) \subset P\ModSpace(\Banach, E)
\end{equation*}
which we call the \emph{moduli space of parameterized maps with constrained marked points thickened by $E_{\pertSpace}$}. Since the image of $u$ is an immersion when restricted to a neighborhood of the marked points, we find that the above inclusion gives $P\ModSpace_{\epsilon, m_{\orbit^{+}}}(\Banach, E_{\pertSpace})$ the structure of a $\codim=2r$ submanifold of $P\ModSpace(\Banach, E_{\pertSpace})$. In particular, we obtain
\begin{equation*}
\dim P\ModSpace_{\epsilon, m_{\orbit^{+}}}(\Banach, E) = |\orbit^{+}|_{\framing} - \sum |\orbit^{-}_{i}|_{\framing} + 2c_{1}^{rel}(\framing) + \rank E = \ind_{SFT}u  + \rank E.
\end{equation*}
following the notation around Equation \eqref{Eq:ThiccModSpaceDimension}.

\subsubsection{Orbifold structure on thickened moduli spaces over interior charts}\label{Sec:ThickingOverInteriorCharts}

Each open set $\Banach$ as above is acted on by the finite group $G$ generated by
\be
\item rotating markers at punctures, and
\item swapping indices of punctures on the domain which are negatively asymptotic to the same $\Reeb^{-}$ orbit.
\ee
Note that $G$ acts naturally on $E_{\eigenBound}$ and we require that $E \subset E_{\eigenBound}$ is chosen so that it is preserved by the $G$ action. The $G$ action then induces an orbifold structure on the bundles
\begin{equation*}
\begin{tikzcd}
E \arrow[dr] \arrow[d, "\simeq"] \\
E^{0,1} \arrow[r] & P\ModSpace_{\epsilon, m_{\orbit^{+}}}(\Banach, E)
\end{tikzcd}
\end{equation*}
over the moduli space of parameterized maps with constrained marked points thickened by $E$. We denote the associated orbibundles by
\begin{equation*}
\begin{tikzcd}
\thicc{E} \arrow[dr] \arrow[d, "\simeq"] \\
\thicc{E}^{0,1} \arrow[r] & \ModSpace_{\epsilon, m_{\orbit^{+}}}(\Banach, E)
\end{tikzcd}
\end{equation*}
and say that the base $\ModSpace_{\epsilon, m_{\orbit^{+}}}(\Banach, E)$ is the \textit{moduli space of maps with constrained marked points thickened by $E$}. For a review of orbifolds, orbibundles, and their multisections, see \cite[\S 2]{BH:ContactDefinition}.

\subsubsection{Close to breaking curves}\label{Sec:CloseToBreaking}

Now we deal with open sets of the manifold of maps which contain holomorphic curves which are close to breaking into multi-level SFT buildings. Let $\tree$ be a $CH$ tree with vertex labeling data $(\orbit_{\vertex}, \vec{\orbit}_{\vertex}^{-}, \mapHomotopyClass_{\vertex})$ and edge labeling data $\orbit_{\edge}$ such that the pregluings of maps yields elements of some fixed $\Map_{0}^{\mapHomotopyClass, m}(\orbit, \vec{\orbit}^{-})$. See Definition \ref{Def:CHTree}.

Maps $(u, \domainJ, \vec{m}) \in \Map_{0,r}^{\mapHomotopyClass, m}(\orbit, \vec{\orbit}^{-})$ which are \emph{close to breaking (into a tree of type $\tree$)} have the following properties:
\be
\item For each edge, $\edge$, of $\tree$ connecting distinct vertices, there is an annulus $\annulus_{\edge} \subset (\Sigma, \domainJ)$ which is biholomorphic to some $[-C_{\vertex}, C_{\vertex}]_{p} \times \Circle_{q}$, $C_{\vertex} \in \R_{>0}$. We require that such $\annulus_{\edge}$ are connected components of the thin part of a $\delta$-thick-thin decomposition of $(\Sigma_{r}, \domainJ)$ for a specified $\delta \ll 1$ which we omit from our notation.
\item The restriction of $u$ to each $\annulus_{\edge}$ maps into some $[s_{1}, s_{2}] \times \Norbit \subset \complete{X}$ contained in one of the half-cylindrical ends of $\complete{X}$. Here $\Norbit$ is a simple neighborhood of the embedded orbit underlying $\orbit_{\vertex}$. The map $u|_{\annulus_{\edge}}$ is required to be sufficiently $\COne$ close to $(p,q) \mapsto (\cm/2\pi)(p, q, 0)$ where $\cm$ is the covering multiplicity of $\orbit_{\vertex}$.
\item We require that the $\annulus_{\edge}$ are disjoint so that the connected components $\Sigma_{\vertex}^{\circ}$ of $\Sigma \setminus \cup_{\edge} \annulus_{\edge}$ are indexed by the vertices, $\vertex$, of $\tree$. Each $(\Sigma_{\vertex}^{\circ}, \domainJ)$ is then a Riemann surface with boundary. We can extend $u|_{\Sigma_{\vertex}}$ to a map $(u_{\vertex}, \domainJ_{\vertex}, \vec{m}_{\vertex}) \in \Map^{\mapHomotopyClass_{\vertex}, m}(\orbit_{\vertex}, \vec{\orbit}^{-}_{\vertex})$ by appending half cylindrical ends of the form $\R_{\pm p \geq C_{\vertex}} \times \Circle$ to $\Sigma_{\vertex}^{\circ}$ and requiring that $u_{\vertex}$ maps these half cylindrical ends into simple neighborhoods into which the boundary components are mapped. We may assume that the restriction of $u_{\vertex}$ to the newly-added cylindrical ends are determined by some fixed cutoff function. Write $\Sigma_{\vertex}$ for the domain of each $u_{\vertex}$.
\item We require that the marked points are disjoint from the annuli $\annulus_{\edge}$ and there are exactly \be
\item $2\cm(\orbit_{\vertex})$ marked points on $\Sigma_{\vertex}^{\circ}$ when $\chi(\Sigma_{\vertex}^{\circ}) \geq 0$ and
\item $0$ marked points on $\Sigma_{\vertex}^{\circ}$ otherwise.
\ee
\ee
Note that the last condition is open in the manifold of maps and determines the number $r$ of marked points. We think of $(u, \domainJ, \vec{m})$ as being a pregluing of the $(u_{\vertex}, \domainJ_{\vertex}, \vec{m}_{\vertex})$. Alternatively, we can think of the $(u_{\vertex}, \domainJ_{\vertex}, \vec{m}_{\vertex})$ as being obtained by ``ungluing'' $(u, \domainJ, \vec{m})$ along the annuli $\annulus_{\edge}$. For a full definition of maps which are close to breaking, see \cite[\S 6.2]{BH:ContactDefinition}.

For a $CH$ tree, $\tree$, as above, suppose that we have selected subsets $K_{\vertex} \subset \ModSpace^{[\mapHomotopyClass], m}(\orbit_{\vertex}, \vec{\orbit}^{-}_{\vertex})$ and associated open sets $\Banach_{\vertex} \in \Map_{r_{\vertex}}^{\mapHomotopyClass_{\vertex}, m}(\orbit_{\vertex}, \vec{\orbit}_{\vertex}^{-})$ as described in \S \ref{Sec:InteriorChartSubsets}. Let
\begin{equation*}
\Banach_{\tree} \subset \Map_{r}^{\mapHomotopyClass}(\orbit, \vec{\orbit}^{-}), \quad r = \sum_{\vertex} r_{\vertex}
\end{equation*}
be the subset for which $(u_{\vertex}, \domainJ_{\vertex}, \vec{m}_{\vertex}) \in \Banach_{\tree}$ for all $\vertex$. Suppose moreover that we have chosen subbundles
\begin{equation*}
E_{\vertex} \subset E_{\vertex, \eigenBound_{\vertex}} \to \Banach_{\vertex}
\end{equation*}
so that the associated spaces $E_{\vertex}^{0,1}$ are $\delbar$-transverse for sufficiently large constants $\eigenBound_{\vertex} \in \R_{\geq 0}$. We will use these sections to construct a subbundle of $\Sec(T\complete{X}\boxtimes T^{\ast}\Sigma) \to \Banach_{\tree}$, where $\Sigma$ is the domain of a map in $\Banach_{\tree}$. 

The sections spanning $E_{\vertex, \eigenBound_{\vertex}} \subset \Sec(T\complete{X}\boxtimes T^{\ast}\Sigma_{\vertex})$ of Equation \eqref{Eq:BumpedSections} can be naturally converted to sections of $\Sec(T\complete{X}\boxtimes T^{\ast}\Sigma)$, where $\Sigma$ is the domain of a map in $\Banach_{\tree}$: We simply convert the bump function $b(p)$ supported on a half-cylindrical end $\halfcyl_{\pm} \subset \Sigma_{\vertex}$ into a bump function supported on either
\be
\item the corresponding half-cylindrical end of $\Sigma$ if $\halfcyl_{\pm} \subset \Sigma_{\vertex}$ corresponds to a half-cylindrical end in $\Sigma$ after gluing or
\item the subset $\mp p \leq C$ corresponding to the annulus $\annulus_{\edge} \subset \Sigma$ otherwise.
\ee
This conversion of bump functions thus determines subspaces
\begin{equation*}
E_{\vertex, \tree} \subset E_{\vertex, \eigenBound_{\vertex}, \tree} \subset \Sec(T\complete{X}\boxtimes T^{\ast}\Sigma)
\end{equation*}
of the trivial bundle of perturbations over $\Banach_{\tree}$. Since the supports of the sections spanning these subspaces project to pairwise disjoint subsets of $\Sigma$, we see that
\begin{equation*}
\vertex \neq \vertex' \implies E_{\vertex, \eigenBound_{\vertex}} \cap E_{\vertex', \eigenBound_{\vertex'}, \tree} = 0.
\end{equation*}
We then define
\begin{equation*}
E_{\tree} = \oplus_{\vertex} E_{\vertex, \tree} \subset \oplus_{\vertex} E_{\vertex, \eigenBound_{\vertex}, \tree}.
\end{equation*}
A gluing argument can then be applied to show that if $(u, \domainJ, m)$ is sufficiently close to breaking and all of the $E_{\vertex, \tree}^{0,1} \to \Banach_{\vertex}$ associated to the $E_{\vertex, \tree}$ are transverse subbundles, then so is the bundle $E_{\tree}^{0,1} \to \Banach_{\tree}$ associated to $E_{\tree}$.

Associated to these bundles we have associated thickened moduli spaces of parameterized and unparameterized maps and bundles over them
\begin{equation*}
\begin{tikzcd}
E_{\tree} \arrow[dr] \arrow[d, "\simeq"] \\
E^{0,1}_{\tree} \arrow[r] & P\ModSpace_{\epsilon, m_{\orbit^{+}}}(\Banach_{\tree}, E_{\tree}),
\end{tikzcd} \quad
\begin{tikzcd}
\thicc{E}_{\tree} \arrow[dr] \arrow[d, "\simeq"] \\
\thicc{E}^{0,1}_{\tree} \arrow[r] & \ModSpace_{\epsilon, m_{\orbit^{+}}}(\Banach_{\tree}, E_{\tree}).
\end{tikzcd}
\end{equation*}
See \cite[\S 8.2]{BH:ContactDefinition} for details around automorphism groups $G_{\tree}$ and orbifold structures for trees having more than one vertex. In particular, if $\tree$ and $\tree'$ are two $CH$-trees such that maps associated to these trees yield the same asymptotic data ($\orbit, \vec{\orbit}^{-}$) and homotopy class ($\mapHomotopyClass$) after pregluing, we obtain a diagram
\begin{equation}\label{Eq:OverlappingCharts}
\begin{tikzcd}
\thicc{E}_{\tree} \arrow[r] \arrow[d] & \thicc{E}_{\tree} + \thicc{E}_{\tree'} \arrow[d] & \arrow[l]\arrow[d] \thicc{E}_{\tree'}\\
\ModSpace_{\epsilon, m_{\orbit^{+}}}(\Banach_{\tree} \cap \Banach_{\tree}', E_{\tree}) \arrow[r] & \ModSpace_{\epsilon, m_{\orbit^{+}}}(\Banach_{\tree} \cap \Banach_{\tree}', E_{\tree} + E_{\tree'}) & \ModSpace_{\epsilon, m_{\orbit^{+}}}(\Banach_{\tree}\cap \Banach_{\tree}', E_{\tree'}) \arrow[l]
\end{tikzcd}
\end{equation}
where the horizontal arrows are embeddings of smooth manifolds and the vertical arrows are bundle projections. Here we have ignored details of adding, removing, and constraining the locations of marked points as $\tree$ varies. See \cite[\S 6.2]{BH:ContactDefinition}. Note that the thickened moduli spaces associated to different trees will in general be of different dimensions as these dimensions vary with the ranks of their transverse subbundles.

To make the notation of the previous subsections coherent with the notation of the present subsection, we may view thickened moduli spaces associated to a single choice of a $K \subset \ModSpace^{[\mapHomotopyClass]}(\orbit, \vec{\orbit}^{-})$ as being associated to a single vertex tree. In order to cover $\ModSpace^{[\mapHomotopyClass]}(\orbit, \vec{\orbit}^{-})$ with subsets $K_{\tree}$ which are covered by the subsets
\begin{equation*}
\ModSpace_{\epsilon, m_{\orbit^{+}}}(\Banach_{\tree}, 0) \subset \ModSpace_{\epsilon, m_{\orbit^{+}}}(\Banach_{\tree}, E_{\tree})
\end{equation*}
one chooses the $K$ in an inductive fashion as described by the discussion of ``complexity'' in \cite{BH:ContactDefinition}.

\subsubsection{Multisections and zero loci}

Each bundle $E_{\tree}^{0,1} \to P\ModSpace_{\epsilon, m_{\orbit^{+}}}(\Banach_{\tree}, E_{\tree})$ has a canonical section, $\multisec_{\Id}$, sending each point $((u, m, \domainJ), \pert)$ in the base to the element
\begin{equation*}
\pert^{0, 1} = \delbar_{J, \domainJ}u
\end{equation*}
in the fiber. This equality here is provided by the fact that $((u, m, \domainJ), \pert) \in E_{\tree}$ according to the definition of the parameterized moduli space thickened by $E_{\tree}$. This section is $G_{\tree}$ equivariant and that its zero locus is exactly $P\ModSpace_{\epsilon, m_{\orbit^{+}}}(\Banach, E_{\tree})$. Hence $\multisec_{\Id}$ descends to a multisection of $\thicc{E}_{\tree}$, which we continue to denote by $\multisec_{\Id}$.

We say that a multisection $\multisec_{\tree}$ of $\thicc{E}_{\tree}$ is \emph{admissible} if its support is contained in
\begin{equation*}
\ModSpace_{\epsilon, m_{\orbit^{+}}}(\Banach, E_{\tree, \rho}) = \{ [(u, m, \domainJ), \pert]\ :\ \norm{ \pert} < \rho\} \subset \ModSpace_{\epsilon, m_{\orbit^{+}}}(\Banach, E_{\tree})
\end{equation*}
for some small $\rho \in \R_{>0}$ (that is, a small neighborhood of $\ModSpace_{\epsilon, m_{\orbit^{+}}}(\Banach, 0)$ in $\ModSpace_{\epsilon, m_{\orbit^{+}}}(\Banach, E_{\tree})$) and is $\mathcal{C}^{1}$ close to the zero section. Here the norm $\norm{\pert}$ may be computed as the Euclidean norm associated to the provided basis of $E_{\tree}$. Provided such a multisection, we may define the associated zero locus to be
\begin{equation*}
\ZeroSet_{\multisec_{\tree}} = \{ \multisec_{\Id} - \multisec_{\tree} = 0\} \subset \ModSpace_{\epsilon, m_{\orbit^{+}}}(\Banach, E_{\tree, \rho}).
\end{equation*}
Assuming that $\multisec - \multisec_{\Id}$ is transverse to the zero section, $\ZeroSet_{\multisec_{\tree}}$ carries the structure of an oriented, smooth, branched weighted orbifold (by the implicit function theorem). In the case of a $\R_{s}$-invariant symplectization, we require that the $\multisec_{\tree}$ are $\R_{s}$-invariant, implying that the $\ZeroSet_{\multisec_{\tree}}$ come equipped with free $\R_{s}$ actions. Since the $\ZeroSet_{\multisec_{\tree}}$ are obtained from taking sections of $\rank E_{\tree}$ bundles over orbifolds of dimension $\ind u + \rank E_{\tree}$, we have
\begin{equation*}
\dim \ZeroSet_{\multisec_{\tree}} = \ind u = |\orbit|_{\framing} - \sum |\orbit^{-}_{i}|_{\framing} + 2c_{1}^{rel}(\framing)
\end{equation*}
where the $\framing$ indicate choices of framings of orbits and $|\ast|_{\framing}$ are associated contact homology gradings.

In \cite{BH:ContactDefinition} compatibility conditions on such multisections are described (as consequences of a ``melding construction''). Compatible choices of $\multisec_{\tree}$ yield identifications of the zero locus of the multisection for $E_{\tree} + E_{\tree'}$ over $\ModSpace_{\epsilon, m_{\orbit^{+}}}(\Banach_{\tree} \cap \Banach_{\tree}', E_{\tree} + E_{\tree'})$ with $\ZeroSet_{\multisec_{\tree}} \cap \ZeroSet(\multisec_{\tree'})$. Thus coherent choices $\Multisec$ of $\multisec_{\tree}$ over all $\tree$ yield an oriented, branched, weighted orbifold
\begin{equation*}
\ZeroSet_{\Multisec} = \ZeroSet_{\Multisec}^{[\mapHomotopyClass]}(\orbit, \vec{\orbit}^{-})
\end{equation*}
determined by the identification of overlaps
\begin{equation}\label{Eq:OverlappingZeroSets}
\begin{tikzcd}
\ModSpace_{\epsilon, m_{\orbit^{+}}}(\Banach_{\tree} \cap \Banach_{\tree}', E_{\tree}) \arrow[r] \arrow[d] & \ModSpace_{\epsilon, m_{\orbit^{+}}}(\Banach_{\tree}, E_{\tree} + E_{\tree'}) & \ModSpace_{\epsilon, m_{\orbit^{+}}}(\Banach_{\tree} \cap \Banach_{\tree}', E_{\tree'}) \arrow[l] \arrow[d]\\
\ModSpace_{\epsilon, m_{\orbit^{+}}}(\Banach_{\tree}, E_{\tree}) & & \ModSpace_{\epsilon, m_{\orbit^{+}}}(\Banach_{\tree'}, E_{\tree'})\\
\ZeroSet_{\multisec_{\tree}} \arrow[u] \arrow[dr] & \arrow[l] \ZeroSet_{\multisec_{\tree}} \cap \ZeroSet_{\multisec_{\tree'}} \arrow[uu] \arrow[d] \arrow[r] & \ZeroSet_{\multisec_{\tree'}} \arrow[dl] \arrow[u] \\
& \ZeroSet_{\Multisec}
\end{tikzcd}
\end{equation}
of zero loci for each diagram as in Equation \eqref{Eq:OverlappingCharts}. The fundamental classes of these zero loci are the fundamental cycles using which chain-level contact homology and cobordism maps will be defined.

\begin{rmk}
To be completely rigorous, one must vary the parameter $\epsilon$ as well as the $\delta$ (appearing in the relevant thick-thin decompositions) with $\tree$. As described in \cite[\S 5.4]{BH:ContactDefinition}, induce diffeomorphisms between thickened moduli spaces and associated zero loci. We ignore these details for simplicity.
\end{rmk}

Assume that all orbits under considerations are good. These zero loci may be coherently oriented by choices of orientation data associated to each orbit of action $\leq \actionBound$. With such choices fixed, it follows that
\begin{equation*}
\ZeroSet_{\Multisec}^{[\mapHomotopyClass]}(\orbit, (\orbit^{-}_{1},\dots, \orbit^{-}_{\NnegativePunctures})) = (-1)^{|\orbit^{-}_{k}|\ |\orbit^{-}_{k+1}|}\ZeroSet_{\Multisec}^{[\mapHomotopyClass]}(\orbit, (\orbit^{-}_{1},\dots,\orbit^{-}_{k+1}, \orbit^{-}_{k},\dots, \orbit^{-}_{\NnegativePunctures}))
\end{equation*}
as branched, weighted, \emph{oriented} orbifolds. Further discussion of coherent orientations appears after Equation \eqref{Eq:ZeroSetHomeo}.

We assume that such choices of orientation have been made. When we are working in a $\R_{s}$-invariant symplectization and $\ind u = 1$, then the $\ZeroSet_{\Multisec}/\R_{s}$ is a rationally-weighted collection of points and so has a well-defined count
\begin{equation*}
\# (\ZeroSet_{\Multisec}/\R_{s}) \in \Q.
\end{equation*}
Likewise when $\ind u = 0$, $\ZeroSet_{\Multisec}$ is a rationally-weighted collection of points and so has a well-defined count
\begin{equation*}
\# \ZeroSet_{\Multisec} \in \Q.
\end{equation*}

\subsubsection{Contact homology algebras and cobordism maps}\label{Sec:CHBHDefinition}

Now we use the above counts of perturbed holomorphic curves to define contact homology differentials and cobordism maps over various rings. We assume that an action bound $\actionBound > 0$ is fixed and that $\alpha$ is a contact form on $M$ of $\dim=2n-1$ which is $\actionBound$-nondegenerate.

We first define contact homology with coefficients in $\Ring=\Field \supset \Q$ and $\Z/2\Z$ grading. The algebra underlying the chain-level $CH$ algebra, $\Algebra(\alpha, \actionBound, \Field)$ is the graded-commutative tensor algebra $\tensorAlgGraded(\orbitVS)$ over
\begin{equation*}
\orbitVS = \oplus_{\action(\orbit) \leq \actionBound} \Field \orbit, \quad |\orbit| = \CZ_{2}(\orbit) + n-3 \in \Z/2\Z.
\end{equation*}
In the case $n=1$, whence $M$ is a disjoint union of circles, we take $\CZ_{2}(\orbit)=0$ as definition. Assume that we have fixed choices of perturbation data and thus zero cycles $\ZeroSet_{\Multisec}^{[\mapHomotopyClass]}(\orbit, \vec{\orbit}^{-})$. For each closed Reeb orbit of action $\leq \actionBound$, we define the contact homology differential
\begin{equation}\label{Eq:CHDiffUnTwisted}
\partial \orbit = \sum \cm(\orbit, \vec{\orbit}^{-})^{-1}\#(\ZeroSet_{\Multisec}^{[\mapHomotopyClass]}(\orbit, \vec{\orbit}^{-})/\R_{s})\orbit^{-}_{1}\cdots \orbit^{-}_{\NnegativePunctures}, \quad \cm(\orbit, \vec{\orbit}^{-}) = \cm(\orbit)\prod_{1}^{\NnegativePunctures} \cm(\orbit^{-}_{i}).
\end{equation}
Here the sum runs over choices of ordered tuples $\vec{\orbit}^{-}$ associated to monomials $\orbit^{-}_{1}\cdots \orbit^{-}_{\NnegativePunctures}$ which additively generate $\tensorAlgGraded(\orbitVS)$ and classes $[\mapHomotopyClass]$ so that the expected dimension of $\ZeroSet_{\Multisec}^{[\mapHomotopyClass]}$ is $1$. Following the main results of \cite{BH:ContactDefinition}, the contact homology differentials $\partial$ give
\begin{equation*}
\Algebra(\alpha, \actionBound) = \Algebra(\alpha, \actionBound, \Field) = (\tensorAlgGraded(\orbitVS), \partial)
\end{equation*}
the structure of a semi-free, finitely filtered, CDGA over $\Field$. Note here that we are suppressing choices of perturbation data and orientations from our notation: A priori the isomorphism class of $\Algebra(\alpha, \actionBound, \Ring)$ depends on this data.

Now we define cobordism maps between chain-level $CH$ CDGAs. We work on a completion $\complete{X}$ of a strict Liouville cobordism $(X, \lambda)$ between strict contact manifolds $(\partial_{\pm}X, \alpha^{\pm})$ which are $\actionBound^{\pm}$ non-degenerate with $\actionBound_{-} \geq \actionBound_{+}$. We assume that perturbation data has been associated to the symplectizations of the $\partial_{\pm}X$ as well as to $\complete{X}$. We write $\orbitVS^{\pm}$ for the free $\Field$-module generated by good $\Reeb^{\pm}$ orbits of action $\leq \actionBound^{\pm}$. Define cobordism maps
\begin{equation*}
\Phi_{X}: \Algebra(\alpha_{+}, \actionBound_{+}) \to \Algebra(\alpha_{-}, \actionBound_{-})
\end{equation*}
by declaring that for each generator $\orbit \in \orbitVS_{+}$,
\begin{equation}\label{Eq:CHCobUntwisted}
\Phi_{X} \orbit = \sum \cm(\orbit, \vec{\orbit}^{-})^{-1}\#\ZeroSet_{\Multisec}^{[\mapHomotopyClass]}(\orbit, \vec{\orbit}^{-})\orbit^{-}_{1}\cdots \orbit^{-}_{\NnegativePunctures}.
\end{equation}
Here the sum runs over choices of ordered tuples $\vec{\orbit}^{-}$ associated to monomials $\orbit^{-}_{1}\cdots \orbit^{-}_{\NnegativePunctures}$ which additively generate $\tensorAlgGraded(\orbitVS)$ and classes $[\mapHomotopyClass]$ so that the expected dimension of $\ZeroSet_{\Multisec}^{[\mapHomotopyClass]}$ is $0$. As is the case with differentials, the map $\Phi_{X}$ depends on choices of perturbation data which we exclude from our notation for simplicity. According to the main results of \cite{BH:ContactDefinition}, the map $\Phi_{X}$ is a CDGA morphism and the induced map on the homology
\begin{equation*}
H\Phi_{X}: H\Algebra(\alpha_{+}, \actionBound_{+}) \to H\Algebra(\alpha_{-}, \actionBound_{-})
\end{equation*}
is independent of the choices of perturbative data underlying the definition of $\Phi_{X}$.

To define the contact homology of $\Mxi$ with coefficients in $\Ring$, we consider $CH$-admissible sequences of contact forms $\alpha_{i} = h_{i}\alpha$ as described in the introduction. In this case there exist cobordism maps
\begin{equation*}
\Phi_{i}: \Algebra(\alpha_{i}, \actionBound_{i}) \to \Algebra(\alpha_{i+1}, \actionBound_{i+1}).
\end{equation*}
Using the induced maps $H\Phi_{i}$ on homology, we may then define the contact homology algebra of $\Mxi$ as the colimit
\begin{equation*}
CH\Mxi = CH(M, \xi, \Field)= \varinjlim H\Algebra(\alpha_{i}, \actionBound_{i})
\end{equation*}
of the associated directed system, which is a graded ring. See \cite[\S 9]{BH:ContactDefinition} for details.

\subsubsection{Alternate gradings and twisted coefficient systems}\label{Sec:TwistedCoeffsOverview}

Now we modify the definitions of \S \ref{Sec:CHBHDefinition} to define $CH$ algebras with enhanced gradings and coefficient systems.

To define $CH$ with $\Ring=\Field[\ChernVar^{\pm 1}]$ coefficients ($\Field \supset \Q$) and $\Z$ grading, we may choose any framings $\framing$ for the orbits $\orbit$ of action $\leq \actionBound$ and declare that
\begin{equation*}
|\orbit| = \CZ_{\framing}(\orbit) + n-3.
\end{equation*}
Let $\ChernVar$ be a variable of homological grading $2$. We write elements of $\Ring$ as finite sums of the form $\sum a_{i}\ChernVar^{k_{i}}$ with $a_{i} \in \Field$ and $k_{i} \in \Z$. We take the algebra underlying chain-level contact homology to be $\tensorAlgGraded(\orbitVS)$ where
\begin{equation*}
\orbitVS = \oplus_{\action(\orbit) \leq \actionBound}\Ring \orbit.
\end{equation*}
To define $CH$ differentials in this context, we modify Equation \eqref{Eq:CHDiffUnTwisted} as
\begin{equation*}
\partial \orbit = \sum \cm(\orbit, \vec{\orbit}^{-})^{-1}\#(\ZeroSet_{\Multisec}^{[\mapHomotopyClass]}(\orbit, \vec{\orbit}^{-})/\R_{s})\ChernVar^{-c_{1}^{rel}(\framing, \mapHomotopyClass)}\orbit^{-}_{1}\cdots \orbit^{-}_{\NnegativePunctures}.
\end{equation*}
In this setting the $SFT$ index of a map $u$ in the homotopy class $\mapHomotopyClass$ reads
\begin{equation*}
\ind_{SFT}u = |\orbit| - |\vec{\orbit}^{-}| + 2c_{1}^{rel}(\framing, \mapHomotopyClass)
\end{equation*}
so that $\partial$ has $\deg=-1$ as desired. Likewise, to define cobordism maps and contact homology as a colimit, we modify Equation \eqref{Eq:CHCobUntwisted} as
\begin{equation}
\Phi_{X} \orbit = \sum \cm(\orbit, \vec{\orbit}^{-})^{-1}\#\ZeroSet_{\Multisec}^{[\mapHomotopyClass]}\ChernVar^{-c_{1}^{rel}(\framing, \mapHomotopyClass)}(\orbit, \vec{\orbit}^{-})\orbit^{-}_{1}\cdots \orbit^{-}_{\NnegativePunctures}.
\end{equation}

Finally we define contact homology with coefficients in $\Ring = \Field[H_{2}(M)]$ with 
\be
\item $\Z$-grading when $\Tor H_{1}(M) = 0$ and
\item $\Field$-grading otherwise
\ee
as originally proposed in \cite{EGH:SFTIntro}. We follow the exposition of \cite[\S 9.1.2]{BH:ContactDefinition} and refer also to \cite{Bourgeois:ContactIntro}. We express elements of $\Ring$ as
\begin{equation*}
\sum c_{h}e^{h} \in \Ring
\end{equation*}
over finite sums of elements $h \in \HomologyCoeffs$ with $c_{h} \in \Field$. Choose a decomposition
\begin{equation}\label{Eq:H1FreeTorsionSplitting}
H_{1}(M, \Z) = F \oplus T
\end{equation}
where $F$ is free and $T$ is torsion. Choose $1$-cycles $f_{i}$ and $t_{i}$ in $M$ whose homology classes provide a minimal set of generators of $F$ and $T$ respectively. Choose framings $\framing_{f_{i}}$ and $\framing_{t_{i}}$ of $\xi$ restricted to these cycles. Such framings induce framings $\framing_{kf_{i}}$ and $\framing_{kt_{i}}$ on $k$-fold iterates is the $f_{i}$ and $t_{i}$. Write $\ord_{i} \in \Z_{\geq 2}$ for the order of each $t_{i}$. For each closed Reeb orbit $\orbit$, choose oriented surfaces $S_{\orbit}$ and $S_{t_{i}}$ mapped into $M$ such that
\begin{equation*}
\partial S_{\orbit} = \orbit - \sum a_{i}f_{i} - \sum b_{i}t_{i}, \quad a_{i} \in \Z, \quad b_{i} \in \{ 0,\dots, \ord_{i}-1\}, \quad \partial S_{t_{i}} = \ord_{i}t_{i}.
\end{equation*}
We say that the $S_{\orbit}$ and $S_{t_{i}}$ are \emph{spanning surfaces}. Extend the framings $\framing_{f_{i}}$ and $\framing_{t_{i}}$ to framings $\framing_{S_{\orbit}}$ and $\framing_{S_{t_{i}}}$ of $\xi$ over the spanning surfaces. The restriction of these extensions over the Reeb orbits determines framings $\framing_{\orbit}$ of $\xi$ along the orbits. Such framings are independent of the choices of extension as a consequences of the facts that (1) homotopy classes of framings of $\xi$ along the $S_{\orbit}$ are in one-to-one correspondence with elements of
\begin{equation*}
\pi_{0}[S_{\orbit}, \Symp_{2n-2}] \simeq \pi_{0}[S_{\orbit}, \Circle] \simeq H^{1}(S_{\orbit}) \simeq \Hom(H_{1}(S_{\orbit}), \Z)
\end{equation*}
and (2) the ``incoming'' and ``outgoing'' boundaries of the $S_{\orbit}$ are homologous in $S_{\orbit}$. Likewise the $S_{t_{i}}$ induce framings over the $\ord_{i}t_{i}$ and we write $\wind_{i} \in \Z$ for the rotation number of this framing with respect to $\framing_{\ord_{i}t_{i}}$. Using this framing data we define the contact homology grading to be
\begin{equation*}
|\orbit| = CZ_{\framing_{\orbit}}(\orbit) + 2\sum b_{i}\frac{\wind_{i}}{\ord_{i}}+ n-3.
\end{equation*}
These gradings are $\Z$-valued when $\Tor H_{1}(M) = 0$. We also define degrees of classes $h \in H_{2}(M)$ by
\begin{equation*}
|e^{h}| = -2\langle c_{1}(\xi), h \rangle.
\end{equation*}
For orbit $\orbit$, tuple of orbits $\vec{\orbit}$, and each $[\mapHomotopyClass] \in H_{2}(M, \orbit \cup \vec{\orbit})$ with $\partial[\mapHomotopyClass] = [\orbit - \vec{\orbit}]$ we can define a class \begin{equation*}
h = [\pi_{M}\mapHomotopyClass - S_{\orbit} + \sum S_{\orbit^{-}_{i}}] \in H_{2}(M).
\end{equation*}
We take the algebra underlying chain-level contact homology to be $\tensorAlgGraded(\orbitVS)$ where
\begin{equation*}
\orbitVS = \oplus_{\action(\orbit) \leq \actionBound}\Ring \orbit.
\end{equation*}
To define differentials, we modify Equation \eqref{Eq:CHDiffUnTwisted} as
\begin{equation}\label{Eq:CHDiffTwisted}
\partial \orbit = \sum \cm(\orbit, \vec{\orbit}^{-})^{-1}\#(\ZeroSet_{\Multisec}^{[\mapHomotopyClass]}(\orbit, \vec{\orbit}^{-})/\R_{s})e^{h}\orbit^{-}_{1}\cdots \orbit^{-}_{\NnegativePunctures}.
\end{equation}
Following the main results of \cite{BH:ContactDefinition}, the contact homology differentials $\partial$ give
\begin{equation*}
\Algebra(\alpha, \actionBound, \Ring) = (\tensorAlgGraded(\orbitVS), \partial)
\end{equation*}
the structure of a semi-free, finitely filtered, graded-commutative DGA. Note here that we are suppressing choices of perturbation data, orientations, framing data, and spanning surface data from our notation. To define cobordism maps for topologically trivial cobordisms (as are required to define contact homology as a direct limit) we choose spanning surfaces and compatible framing data over the orbits on the ends of the cobordism and then modify Equation \eqref{Eq:CHCobUntwisted} in the obvious way.

\subsection{Thickening parameterized moduli spaces in $\R_{s} \times M$ by subbundle extension}\label{Sec:SubbundleExtension}

In this subsection we work on the almost complex manifolds $(\complete{X}, J) = (\R_{s} \times M, J)$ of \S \ref{Sec:DynamicsAndFoliations} and \S \ref{Sec:LeafNotation}. Let $\Reeb^{+}$ (resp. $\Reeb^{-}$) be the Reeb vector field associated to the positive (resp. negative) end of the cobordism as described in Lemmas \ref{Lemma:DynamicsOnDivSetC}, \ref{Lemma:DynamicsOnRLiouville}, and \ref{Lemma:DynamicsRHypersurface}. Fix an action bound $\actionBound \in \R_{>0}$ so that all closed $\Reeb^{\pm}$ orbits of action $\leq \actionBound$ have the form $\orbit_{\zeta^{\pm}}$ where
\be
\item $\orbit$ is a non-degenerate $\ReebDivSet$ orbit in $\divSet$ of action $\leq \actionBoundDivSet \leq \actionBound$ and
\item $\zeta^{\pm} \in \Crit(f(\pm s_{0}, \ast))$ are Morse critical points
\ee
as described in the aforementioned lemmas. We momentarily work with curves of arbitrary genus and numbers of positive and negative punctures.

Suppose that $\leaf$ is a leaf of $\foliation$ which is $\pm$ asymptotic to some $\zeta^{\pm} \in \Crit(f(\pm s_{0}, \ast))$. Let
\begin{equation*}
\vec{\orbitDivSet}^{\pm} = (\orbitDivSet_{i}^{\pm})
\end{equation*}
be collections of closed $\ReebDivSet$ orbits and let
\begin{equation*}
\vec{\orbit}^{\pm} = (\{ \zeta^{\pm}\} \times \orbitDivSet_{i}^{\pm})
\end{equation*}
be the associated ordered collection of $\Reeb^{\pm}$ orbits in $\R_{s} \times M$. Let
\begin{equation*}
\Banach^{\leaf} \subset \Map_{g, r}^{\mapHomotopyClass, m}(\vec{\orbitDivSet}^{+}, \vec{\orbitDivSet}^{-})
\end{equation*}
be an open set in the manifold of maps into $\leaf$ with $r$ unrestrained marked points and let
\begin{equation*}
E^{\leaf} \subset \Banach^{\leaf} \times \Sec(T\leaf \boxtimes T^{\ast}\Sigma)
\end{equation*}
be a subbundle of perturbations. Let
\begin{equation*}
\Banach^{\R_{s} \times M} \subset \Map_{g, r}^{\mapHomotopyClass', m}(\vec{\orbit}^{+}, \vec{\orbit}^{-})
\end{equation*}
be an open subset containing the image of the inclusion
\begin{equation*}
\begin{tikzcd}
\Banach^{\leaf} \arrow[rr, bend right=20] \arrow[r] & \Map_{g, r}^{\mapHomotopyClass, m}(\vec{\orbitDivSet}^{+}, \vec{\orbitDivSet}^{-}) \arrow[r] & \Map_{g, r}^{\mapHomotopyClass', m}(\vec{\orbit}^{+}, \vec{\orbit}^{-}) 
\end{tikzcd}
\end{equation*}
given by composing maps with the inclusion
\begin{equation*}
I_{\leaf}: \leaf \to \R_{s} \times M.
\end{equation*}
Here $\mapHomotopyClass' = (I_{\leaf})_{\ast}\mapHomotopyClass$ and we assume that if $(u, \vec{m}, \domainJ) \in \Banach^{\R_{s} \times M}$ is such that $\im u \subset \leaf$, then $(u, \vec{m}, \domainJ) \in \Banach^{\leaf}$.

\begin{defn}\label{Def:SubbundleExtension}
In the above setting, we say $\Banach^{\R_{s} \times M}$ \emph{covers} $\Banach^{\leaf}$ and that a subbundle
\begin{equation*}
E^{\R_{s} \times M} \subset \Banach^{\R_{s} \times M} \times \Sec(T(\R_{s} \times M) \boxtimes T^{\ast}\Sigma)
\end{equation*}
is a \emph{$\foliation$-adapted extension of $E^{\leaf}$} if
\be
\item each $\pert \in E^{\R_{s} \times M}_{\pertSpace}$ is adapted to $\foliation$ (Definition \ref{Def:PertAdaptedToFoliation}) and
\item $E^{\R_{s} \times M}|_{\leaf \times \Sigma} = E^{\leaf}$.
\ee
\end{defn}

\begin{lemma}\label{Lemma:SubbundleExtension}
Following the notation of Definition \ref{Def:SubbundleExtension}, assume $\ind_{\Morse}(\zeta^{+}) = 0$ and $E^{\R_{s} \times M}$ is a $\foliation$-adapted extension of $E^{\leaf}$. Then the inclusion $I_{\leaf}$ induces an equivalence of moduli spaces of parameterized maps with $r$ unrestrained marked points,
\begin{equation*}
P\ModSpace(\Banach^{\leaf}, E^{\leaf}) \simeq P\ModSpace(\Banach^{\R_{s} \times M}, E^{\R_{s} \times M}).
\end{equation*}
Moreover, provided a commutative diagram
\begin{equation*}
\begin{tikzcd}
E^{\leaf} \arrow[r]\arrow[d, hook] & E^{\R_{s} \times M} \arrow[d, hook]\\
(E^{\leaf})' \arrow[r] & (E^{\R_{s} \times M})'
\end{tikzcd}
\end{equation*}
for which the horizontal arrows are $\foliation$-adapted extensions over $\Banach^{\leaf} \subset \Banach^{\R_{s} \times M}$, we have an induced commutative diagram
\begin{equation}\label{Eq:ThiccModSpaceEquiv}
\begin{tikzcd}
P\ModSpace(\Banach^{\leaf}, E^{\leaf})\arrow[d, hook] \arrow[r, "\simeq"] & P\ModSpace(\Banach^{\R_{s} \times M}, E^{\R_{s} \times M}) \arrow[d, hook]\\
P\ModSpace(\Banach^{\leaf}, (E^{\leaf})') \arrow[r, "\simeq"] & P\ModSpace(\Banach^{\R_{s} \times M}, (E^{\R_{s} \times M})').
\end{tikzcd}
\end{equation}
Here the vertical arrows are as in Equation \eqref{Eq:ModSpaceInclusion}.
\end{lemma}

\begin{proof}
Since each $\pert \in E^{\R_{s} \times M}$ is adapted to $\foliation$, any $((u, \vec{m}, \domainJ), \pert) \in P\ModSpace(\Banach^{\R_{s} \times M}, E^{\R_{s} \times M})$ must be such that $\im u$ is contained in some leaf of $\foliation$ by Lemma \ref{Lemma:CurvesStuckInLeaves}. Since $\ind_{\Morse}(\zeta^{+}) = 0$, the positive end of the leaf $\leaf$ must have the form $\R_{s \geq s_{0}} \times \{ \zeta^{+}\} \times \divSet \subset \R_{s \geq s_{0}} \times \subsetC \times \divSet$ and $\leaf$ is the only leaf of $\foliation$ which is positively asymptotic to $\zeta^{+}$. Thus any $((u, \vec{m}, \domainJ), \pert) \in P\ModSpace(\Banach^{\R_{s} \times M}, E^{\R_{s} \times M})$ must be such that $\im u \subset \leaf$. Therefore the fact that the horizontal arrows of Equation \eqref{Eq:ThiccModSpaceEquiv} are equivalences follows from the conditions characterizing $\Banach^{\R_{s} \times M}$.
\end{proof}

\begin{lemma}\label{Lemma:CHSubbundleExtension}
Continuing to use the notation of Lemma \ref{Lemma:SubbundleExtension}, suppose that
\be
\item $\vec{\orbitDivSet}^{+}$ is a single orbit and $g(\Sigma)=0$ (so that we are working with contact homology curves) and
\item the bundle $(E^{\leaf})^{0,1} \to \Banach^{\leaf}$ of anti-holomorphic $1$-forms associated to $E^{\leaf}$ is a transverse subbundle.
\ee
Then the bundle $(E^{\R_{s} \times M})^{0,1} \to \Banach^{\R_{s} \times M}$ of anti-holomorphic $1$-forms associated to $E^{\R_{s} \times M}$ is a transverse subbundle.
\end{lemma}

\begin{proof}
The proof follows immediately from Lemmas \ref{Lemma:LinearizationUpperTriangular} and \ref{Lemma:AppliedAutoTransverse}. In the present context Lemma \ref{Lemma:LinearizationUpperTriangular} is applicable as a consequence of the relevant curves being contained in $\leaf$ and the fact that the perturbations under consideration are tangent to $\foliation$.

According to Lemma \ref{Lemma:LinearizationUpperTriangular}, the linearization $\Dlinearized_{u, \domainJ}^{\pert}$ of the equation defining $P\ModSpace(\Banach^{\R_{s} \times M}, E^{\R_{s} \times M}_{\pertSpace})$ at a solution $((u, \vec{m}, \domainJ), \pert)$ can be expressed as an upper-triangular block matrix. Hence it is surjective iff its upper-left and lower-right blocks are surjective.

The upper-left block is the linearization of the equation defining $P\ModSpace(\Banach^{\leaf}, E^{\leaf}_{\pertSpace})$ and is surjective by the assumption that $E^{\leaf}$ is a transverse subbundle. The lower-right block is the operator $\Dlinearized_{u, \domainJ, \foliationNormal}^{\pert}$ studied in \S \ref{Sec:NormalOperatorAnalysis}. Under the provided hypotheses, this operator is guaranteed to be surjective by item (1) of Lemma \ref{Lemma:AppliedAutoTransverse}. Hence $\Dlinearized_{u, \domainJ}^{\pert}$ is surjective which is equivalent to $E^{\R_{s} \times M}$ being a transverse subbundle.
\end{proof}

While Lemma \ref{Lemma:SubbundleExtension} is applicable to perturbed holomorphic curves of any topological type, the restriction to contact homology type curves in the statement Lemma \ref{Lemma:CHSubbundleExtension} is essential. This is due to the fact that the normal linearized operator $\Dlinearized_{u, \domainJ, \foliationNormal}^{\pert}$ fails to be surjective (or even have non-negative index) for curves with positive genus or multiple positive punctures. Therefore we impose the following assumptions for the remainder of the article.

\begin{assump}
We henceforth assume that perturbed holomorphic curves $u: \Sigma \to \complete{X}$ are of contact homology type. Thus $g(\Sigma)=0$, $\NpositivePunctures=1$, and $\NnegativePunctures \in \Z_{\geq 0}$ is the number of negative asymptotics of a map $u$. We also write $\orbit^{+} = \vec{\orbit}^{+}$ for the positive asymptotic of curves and exclude the subscript $g$ from moduli space notation. We will only briefly remark on the curves of RSFT and ``full'' SFT again in \S \ref{Sec:GeometryExamples}.
\end{assump}

The preceding two lemmas are also independent of the framework used to define virtual fundamental cycles for contact homology. In the remained of this subsection, we will work within the framework of \cite{BH:ContactDefinition} as described in \S \ref{Sec:BHOverview}.

\subsubsection{$\foliation$-adapted extensions in the \cite{BH:ContactDefinition} framework}\label{Sec:BHAdaptedExtension}

We now show that there are natural choices of $\foliation$-adapted extensions associated to subbundles $E^{\leaf}$ over open sets $\Banach^{\leaf}$ of maps into $\leaf$ when working within the framework of \cite{BH:ContactDefinition}. Such extensions then induce equivalences between the virtual fundamental cycles associated to $\leaf$ and those associated to $\R_{s} \times M$ when $\leaf$ is positively asymptotic to an $\ind_{\Morse}(\zeta^{+}, f(s_{0}, \ast))=0$ critical point as described in the statement of Lemma \ref{Lemma:SubbundleExtension}.

Again let $\zeta^{+} \in \Crit_{0}(f(s_{0}, \ast))$ as in Lemmas \ref{Lemma:SubbundleExtension} and \ref{Lemma:CHSubbundleExtension}. As explained in the proof of Lemma \ref{Lemma:SubbundleExtension} here is exactly one leaf, $\leaf$, of the foliation $\foliation$ on $(\R_{s} \times M, J)$ which is positively asymptotic to $\zeta^{+}$. Note that if we are working on a symplectization (i.e. $J$ is $\R_{s}$-invariant on $\R_{s} \times M$), then $(\leaf, J|_{T\leaf})$ is a symplectization as well having the form
\begin{equation*}
\leaf = \R_{s} \times \{ \zeta^{+}\} \times \divSet \subset \R_{s} \times \subsetC \times \divSet \subset \R_{s} \times M.
\end{equation*}
The moduli spaces $\ModSpace^{[\mapHomotopyClass]}(\orbit_{\zeta^{+}}, \vec{\orbit}^{-})$ can only be non-empty when the negative orbits $\orbit^{-}_{i} \in \vec{\orbit}^{-}$ are of the form $\{ \zeta^{-}\} \times \orbit$ where $\zeta^{-}$ is the negative asymptotic of $\leaf$. This is a consequence of Lemma \ref{Lemma:CurvesStuckInLeaves} applied to unperturbed holomorphic curves. Therefore, we assume that all negative orbits have this form when analyzing moduli spaces of curves in $\R_{s} \times M$.

Let $\tree^{\leaf}$ be any $CH$-tree for $\leaf$. Let $\Banach^{\leaf}_{\tree^{\leaf}}$ be an open subset of the manifold of maps into $\leaf$ as described in \S \ref{Sec:CloseToBreaking} and let $\Banach^{\R_{s} \times M}_{\tree^{\R_{s} \times M}}$ be an open subset of the manifold of maps into $\R_{s} \times M$ which covers $\Banach^{\leaf}_{\tree^{\R_{s} \times M}}$ (Definition \ref{Def:SubbundleExtension}). Here $\tree^{\R_{s} \times M}$ is obtained by replacing all orbit labels $\orbitDivSet$ with $\orbit = \orbitDivSet_{\zeta^{\pm}}$ for an appropriate choice of $\pm$-sign. Let $E^{\leaf}_{\eigenBound, \tree} \to \Banach^{\leaf}_{\tree^{\leaf}}$ be as in \S \ref{Sec:BHPerturbations} and \ref{Sec:CloseToBreaking} with $E^{\leaf}_{\tree} \subset E^{\leaf}_{\eigenBound, \tree} \to \Banach^{\leaf}_{\leaf^{\tree}}$ a transverse subbundle over $\Banach^{\leaf}_{\tree^{\leaf}}$. We will now describe a $\foliation$-adapted extension of $E^{\leaf}_{\tree^{\leaf}}$ over $\Banach^{\R_{s} \times M}_{\tree^{\R_{s} \times M}}$.

Let's say that $\NorbitDivSet$ is a simple enough neighborhood of a closed $\ReebDivSet$ orbit $\orbitDivSet$ in $\divSet$ along which $\JDivSet$ is simple. Thus it takes the form
\begin{equation*}
\JDivSet \partial_{s} = \partial_{t} + X_{0} + X_{-}, \quad \JDivSet \partial_{x_{i}} = \partial_{y_{i}}
\end{equation*}
using the coordinate system around Definition \ref{Def:SimpleJ}. This matches the setting of Equation \eqref{Eq:LocalModelSetup} for simple $\JDivSet$. The associated neighborhood $\Norbit = \disk^{2}_{x, y} \times \NorbitDivSet$ of the orbit $\orbit$ in $M$ is such that along this neighborhood
\begin{equation*}
\begin{gathered}
J\partial_{s} = \partial_{t} + X_{0} + X_{-} - X_{f} = \JDivSet\partial_{s} - X_{f}, \quad J\partial_{x} = \partial_{y}, \quad J\partial_{x_{i}} = \partial_{y_{i}},\\
f = \frac{\epsilon}{2}(x^{2} + y^{2}) \quad \implies \quad X_{f} = \epsilon(x\partial_{x} + y\partial_{y})
\end{gathered}
\end{equation*}
so that $-X_{f}$ is tangent to the $\disk_{x, y}^{2}$ factor of $\Norbit$ and vanishes along $\NorbitDivSet = \{ x=y=0\}$. Hence the full asymptotic operator splits into a direct sum
\begin{equation*}
\widetilde{\AsymptoticOp} = -J_{0}\partial_{t} \oplus \AsymptoticOp^{\disk} \oplus \AsymptoticOp^{\divSet}
\end{equation*}
so that $\AsymptoticOp^{\divSet}$ is the normal asymptotic operator for $\NorbitDivSet$ and 
\begin{equation*}
\AsymptoticOp^{\disk} = -\frac{\cm a}{2\pi}J_{0}(\partial_{t} + X_{f})
\end{equation*}
preserves the $T\disk^{2}_{x, y} \simeq \C$ summand of $\R^{2(1 + n_{0} + n_{-})}$. Following slight variations of the calculations of \S \ref{Sec:EllipticAsymptotics}, the eigenspaces of $\AsymptoticOp^{\disk}$ are spanned (over $\R$) by eigenfunctions
\begin{equation*}
\eta^{x}_{k} = e^{k t J_{0}}\partial_{x}, \quad \eta^{y}_{k} = e^{k t J_{0}}\partial_{y}
\end{equation*}
for $k \in \Z$. Therefore in the present setting, the vector fields described in Equation \eqref{Eq:PertVectorFields}  on $\R_{\pm s \geq s_{0}} \times \Norbit$ associated to the eigenfunctions of $\widetilde{\AsymptoticOp}$ can then be enumerated as
\begin{equation*}
e^{Jkt}\partial_{s}, \quad Je^{Jkt}\partial_{s} = e^{Jkt}(\partial_{t} + X_{0} + X_{-} - X_{f}), \quad \eta_{\ell}^{\disk}, \quad \eta^{x}_{k}, \quad \eta^{y}_{k}.
\end{equation*}
Here the $\eta_{\ell}^{\disk}$ are determined by the eigenfunctions of $\AsymptoticOp^{\disk}$ and are tangent to the $\disk^{2}_{x,y}$ summand of $\Norbit$.

Let's suppose for the moment that $\R_{s}\times M$ and hence $\leaf$ is a symplectization. Let $b_{p}$ be a cutoff function such that $\partial_{p}b_{p}$ is supported on a half-cylindrical end of the domain $\Sigma$ of a map centered about some puncture $z_{j}$ as in \S \ref{Sec:BHPerturbations}. Thus the perturbations
\begin{equation}\label{Eq:PertOnLeaf}
\mu^{b, \divSet}_{s, k} = b_{\orbitDivSet}\partial_{p}b\ e^{\JDivSet kt}\partial_{s} \otimes ds, \quad \mu_{\ReebDivSet, k}^{b} = b_{\orbitDivSet}\partial_{p}b\ \JDivSet e^{\JDivSet kt}\partial_{s} \otimes ds, \quad \mu_{\ell}^{b, \divSet}= b_{\orbitDivSet}\partial_{p}b\ \eta_{\ell}^{\divSet} \otimes ds
\end{equation}
of Equation \eqref{Eq:BumpedSections} span the subspace $E^{\leaf}_{\eigenBound, z_{j}} \subset E^{\leaf}_{\eigenBound}$. We will assume that $\eigenBound$ is sufficiently large to achieve transversality and that the $E^{\leaf}_{\eigenBound}$ are used to construct the thickened moduli spaces of maps into $\leaf$ from which virtual fundamental cycles for the triple $(\orbitDivSet, \vec{\orbitDivSet}^{-}, [\mapHomotopyClass])$ will be constructed.

Let $b_{\disk}: \disk^{2}_{x,y} \to [0,1]$ be a bump function agreeing with the constant $1$ near $x = y = 0$ and vanishing outside of a neighborhood of this point. We extend $d_{\disk}$ to be $0$ outside of a neighborhood of $\leaf$. Associated to each of the above perturbations (with corresponding eigenvalues of $-J_{0}\partial_{t} \oplus \AsymptoticOp$ are at most $\eigenBound$ in absolute value) we have associated perturbations
\begin{equation}\label{Eq:PertFromLeaf}
\mu^{b}_{s, k} = b_{\orbit}\partial_{p}b\ e^{Jkt}(\partial_{s}-X_{f}) \otimes ds, \quad \mu_{\Reeb, k}^{b} = b_{\orbit}\partial_{p}b\ Je^{Jkt}(\partial_{s}-X_{f}) \otimes ds, \quad \mu_{\ell}^{b}= b_{\orbit}\partial_{p}b\ \eta_{\ell}^{\divSet} \otimes ds
\end{equation}
in the space $\pertSpace^{\R_{s} \times M} \subset \Sec(T(\R_{s} \times M)\boxtimes T^{\ast}\Sigma)$ of perturbations satisfying Assumptions \ref{Assump:PertCompactSupport}. Here the bump function $b_{\orbit}$ is given by $b_{\orbitDivSet} \cdot b_{\disk}$. These sections span a subbundle
\begin{equation*}
E_{\eigenBound, \thicc{z}_{j}} \to \Banach_{\tree^{\R_{s} \times M}}^{\R_{s} \times M}
\end{equation*}
This bundle agrees exactly with $E^{\leaf}_{\eigenBound, \thicc{z}_{j}}$ along $\Banach^{\leaf}_{\tree^{\leaf}}$ since
\begin{equation*}
J|_{\leaf} = \JDivSet, \quad b_{\disk}|_{\leaf} = 1, \quad X_{f}|_{\leaf} = 0 \implies (\partial_{s} - X_{f})|_{\leaf} = \partial_{s}|_{\leaf} \in T\leaf.
\end{equation*}
Moreover, we observe that
\be
\item $J$ preserves $T\foliation$,
\item the vector field $\partial_{s} - X_{f}$ is tangent to $T\foliation \subset T(\R_{s} \times M)$, and
\item the vector fields $b_{\orbit}\eta_{\ell}^{\divSet}$ are contained in $T\divSet \subset T\foliation$.
\ee
Hence the $\mu^{b}_{s,k}$, $\mu^{b}_{\Reeb, k}$ and $\mu^{b}_{\ell}$ live in $\Sec(T\foliation \boxtimes T^{\ast}\Sigma)$. Therefore $E_{\eigenBound, \thicc{z}_{j}}$ is a $\foliation$-adapted extension of $E^{\leaf}_{\eigenBound_{z_{j}}}$. When $\leaf$ is a non-$\R_{s}$-invariant cobordism an analogous construction of $\foliation$-adapted extensions $E_{\eigenBound, z_{j}}$ of the $E_{\eigenBound, z_{j}}^{\leaf}$ holds via the addition of bump functions to the above formulas as in Equation \eqref{Eq:BumpedSectionsCob}. For perturbations associated to annular subsets of $\Sigma$ as in \S \ref{Sec:CloseToBreaking}, the above formulas likewise yield $\foliation$-adapted extensions
\begin{equation*}
E_{\eigenBound, \vertex} \subset\pertSpace^{\R_{s} \times M}, \quad E_{\eigenBound, \vertex}|_{\Banach^{\leaf}_{\tree^{\leaf}}} = E^{\leaf}_{\eigenBound, \vertex}.
\end{equation*}
By taking direct sums over all vertices of $\tree$, we obtain $\foliation$-adapted extensions
\begin{equation}
E_{\eigenBound} = \oplus_{\vertex} E_{\eigenBound, \vertex} \subset \pertSpace^{\R_{s} \times M} \quad E_{\eigenBound}|_{\Banach^{\leaf}_{\tree}} = E^{\leaf}_{\eigenBound, \tree^{\leaf}}.
\end{equation}

It then follows from Lemma \ref{Lemma:CHSubbundleExtension} that $E_{\eigenBound}^{0,1} \to \Banach_{\tree^{\R_{s} \times M}}^{\R_{s} \times M}$ is a transverse subbundle. By Lemma \ref{Lemma:SubbundleExtension}, the inclusion map of $\leaf$ into $\R_{s} \times M$ induces a homeomorphism of thickened moduli spaces of parameterized maps with unrestrained marked points
\begin{equation*}
P\ModSpace(\Banach^{\leaf}_{\tree^{\leaf}}, E^{\leaf}_{\eigenBound, \tree^{\leaf}}) \simeq P\ModSpace(\Banach^{\R_{s} \times M}_{\tree^{\R_{s} \times M}}, E_{\eigenBound, \tree^{\R_{s} \times M}}).
\end{equation*}
Now suppose that we restrain marked points on maps into $\leaf$ by requiring that they lie on circles in the domain over which the integrals of the contact forms $\alpha^{\divSet}_{\pm}$ on the $\pm$-ends of $\leaf$ are specified.

Now we deal with moduli spaces with restrained marked points. By Equation \eqref{Eq:ConstantAtEllipticAssumption}, the $1$-forms $I_{\leaf}^{\ast}\alpha^{\R_{s} \times M}_{\pm}$ and $\alpha^{\divSet}_{\pm}$ agree on $\leaf$. Therefore
\be
\item constraining the locations of marked points on a map in $P\ModSpace(\Banach^{\leaf}_{\tree^{\leaf}}, E^{\leaf}_{\eigenBound, \tree^{\leaf}})$ then view the map as having target $\R_{s} \times M$ and
\item  constraining the locations of marked points in $P\ModSpace(\Banach^{\R_{s} \times M}_{\tree^{\R_{s} \times M}}, E_{\eigenBound, \tree^{\R_{s} \times M}})$
\ee
yield equivalent thickened moduli spaces of parameterized maps with restrained marked points. Therefore the inclusion map $I_{\leaf}$ induces an equivalence
\begin{equation*}
P\ModSpace_{\epsilon, m_{\orbit^{+}}}(\Banach^{\leaf}_{\tree^{\leaf}}, E^{\leaf}_{\eigenBound, \tree^{\leaf}}) \simeq P\ModSpace_{\epsilon, m_{\orbit^{+}}}(\Banach^{\R_{s} \times M}_{\tree^{\R_{s} \times M}}, E_{\eigenBound, \tree^{\R_{s} \times M}})
\end{equation*}
assuming that the markers of orbits in $M$ of the form $\orbit_{\zeta}$ are determined by choices of markers on the orbits in $\divSet$. The groups $G_{\tree^{\leaf}}$ and $G_{\tree^{\R_{s} \times M}}$ are clearly isomorphic via the above equivalence and the correspondence between perturbations (e.g. moving from Equation \eqref{Eq:PertOnLeaf} to Equation \eqref{Eq:PertFromLeaf}) induces equivalences between the orbibundles
\begin{equation*}
\begin{tikzcd}
\thicc{E}^{\leaf}_{\eigenBound, \tree^{\leaf}} \arrow[dr] \arrow[d, "\simeq"] \\
(\thicc{E}^{\leaf}_{\eigenBound, \tree^{\leaf}})^{0,1} \arrow[r] & \ModSpace_{\epsilon, m_{\orbit^{+}}}(\Banach_{\tree^{\leaf}}^{\leaf}, E_{\eigenBound, \tree^{\leaf}}^{\leaf}),
\end{tikzcd} \quad
\begin{tikzcd}
\thicc{E}_{\eigenBound, \tree^{\R_{s} \times M} }\arrow[dr] \arrow[d, "\simeq"] \\
(\thicc{E}_{\eigenBound, \tree^{\R_{s} \times M}})^{0,1} \arrow[r] & \ModSpace_{\epsilon, m_{\orbit^{+}}}(\Banach_{\tree^{\R_{s} \times M}}^{\R_{s} \times M}, E_{\eigenBound, \tree^{\R_{s} \times M}}).
\end{tikzcd}
\end{equation*}

Now we address multi-sections and their zero loci. Let $\multisec_{\tree^{\leaf}}$ be a multisection of $(\thicc{E}^{\leaf}_{\eigenBound, \tree^{\leaf}})^{0,1}$. Using the above isomorphism of orbibundles, we can push forward $\multisec_{\tree^{\leaf}}$ to obtain a multisection $\multisec_{\tree^{\R_{s} \times M}}$ of $(\thicc{E}_{\eigenBound, \tree^{\R_{s} \times M}})^{0,1}$ such that we have a commutative diagram
\begin{equation*}
\begin{tikzcd}
\ZeroSet_{\multisec_{\tree^{\leaf}}} \arrow[d, hook] \arrow[r, "\simeq"] & \ZeroSet_{\multisec_{\tree^{\R_{s} \times M}}} \arrow[d, hook]\\
\ModSpace_{\epsilon, m_{\orbit^{+}}}(\Banach^{\leaf}_{\tree^{\leaf}}, E^{\leaf}_{\eigenBound, \tree^{\leaf}}) \arrow[r, "\simeq"] & \ModSpace_{\epsilon, m_{\orbit^{+}}}(\Banach^{\R_{s} \times M}_{\tree^{\R_{s} \times M}}, E_{\eigenBound, \tree^{\R_{s} \times M}})
\end{tikzcd}
\end{equation*}
where the horizontal arrows are induced by $I_{\leaf}$. Taking the union over all trees, we have correspondences between collections $\Multisec^{\leaf}$ and $\Multisec^{\R_{s} \times M}$ of multisections indexed by trees. Considering isomorphisms of diagrams as described in Equation \eqref{Eq:OverlappingZeroSets}, we obtain homeomorphisms
\begin{equation}\label{Eq:ZeroSetHomeo}
\ZeroSet^{[\mapHomotopyClass]}_{\Multisec^{\leaf}}(\orbitDivSet, \vec{\orbitDivSet}^{-}) \simeq \ZeroSet^{[I_{\leaf}\mapHomotopyClass]}_{\Multisec^{\R_{s} \times M}}(\orbit, \vec{\orbit}^{-}).
\end{equation}

Now we explain how choices of orientation data can be made so that the homeomorphism of Equation \eqref{Eq:ZeroSetHomeo} is orientation preserving, beginning with a review of some generalities. For each parameterized map $(u, \domainJ, \vec{m})$ in the manifold of maps (into either $M$ or $\leaf$), we write $\Det_{(u, \domainJ, \vec{m})}$ for the \emph{determinant line},
\begin{equation*}
\Det_{(u, \domainJ, \vec{m})} = \wedge^{top}(\ker \Dlinearized) \otimes \wedge^{top}(\coker \Dlinearized)^{\vee}.
\end{equation*}
Here $\Dlinearized$ is the linearized $\delbar$-operator at $(u, \domainJ, \vec{m})$ and $\vee$ indicates taking a dual. Following the discussion of orientations in \cite[9.1.3]{BH:ContactDefinition}, in order to coherently orient the zero sets relevant to contact homology, it suffices to coherently orient the lines $\Det_{(u, \domainJ, \vec{m})}$ over maps appearing in the thickened moduli spaces of parameterized maps. We defer to \cite{Bao:Orientations, BH:ContactDefinition, BM:Orientations} for definitions of the coherency conditions. According to \cite{BM:Orientations} such coherent orientations of the lines $\Det_{(u, \domainJ, \vec{m})}$ indeed exist, and consequently our zero sets may be oriented whenever all of the orbits involved are good.

Now suppose that $(u, \domainJ, \vec{m})$ is a map into a leaf $\leaf$ of our foliation on $\R_{s}\times M$ as described in the statement of Lemma \ref{Lemma:CHMorphism}. As a consequence of Lemma \ref{Lemma:AppliedAutoTransverse}, the kernels and cokernels of the linearized operators $\Dlinearized^{\leaf}$ for $(u, \domainJ, \vec{m})$ and the associated map $(I_{\leaf}u, \domainJ, \vec{m})$ into $M$ are naturally identified as equivalent. This natural equivalence equates the associated determinant line bundles for $(u, \domainJ, \vec{m})$ and $(I_{\leaf}u, \domainJ, \vec{m})$. Applying the aforementioned discussion of  \cite[9.1.3]{BH:ContactDefinition}, this natural identification coupled with choices of orientation data for maps into $\R_{s} \times M$ determines choices of coherent orientation data for maps into $\leaf$. Thus any such choices of orientation data simultaneously orient the zero sets of Equation \eqref{Eq:ZeroSetHomeo} so that the identification is orientation preserving.

Regarding the case of $\Ring=\Field \supset \Q$ coefficients, the results of Lemma \ref{Lemma:CHMorphism} are immediate. In the case of $\Ring=\Field[\ChernVar^{\pm}]$ coefficients, the desired results follow when the $\orbit_{\zeta}$ are equipped with framings of the form $\partial_{\tau} \oplus \framing_{\orbit_{\divSet}}$ for choices of framings $\framing_{\orbit_{\divSet}}$ of the closed Reeb orbits $\orbitDivSet$ in $\divSet$.

\subsection{Specialization to the framework of \cite{Pardon:Contact}}\label{Sec:PardonSpecialization}

Here we provide a rough explanation of how Lemma \ref{Lemma:CHMorphism} may be worked out in the framework of \cite{Pardon:Contact}, aimed at readers who are familiar with the citation.

We recall that the VFCs of \cite{Pardon:Contact} associated to symplectizations and Liouville cobordisms are constructed by first specifying thickening data (\cite[Definitions 3.9 and 3.10]{Pardon:Contact}) indexing Kuranishi charts of thickened-moduli spaces ((\cite[Definitions 3.17 and 3.18]{Pardon:Contact})). Such charts form an implicit atlas (with cell-like stratification) \cite[Definition 3.2]{Pardon:Contact}. Then the VFCs are defined via relative \v{C}ech cohomology cycles associated to the implicit atlas.

Recall that such Kuranishi charts associated to completed Liouville cobordisms, $\complete{X}$, equipped with almost complex structures, $J$, a positive orbit $\orbit$, a tuple of negative orbits $\vec{\orbit}^{-}$, and a choice of $[\mapHomotopyClass]$ are determined by the following \emph{thickening data}:
\be
\item Choose $r \in \Z_{\geq 0}$ specifying numbers of marked points added to domains of maps to ensure their stability.
\item Choose a local divisor $D$ in the target $\complete{X}$ of the maps, which is a submanifold with boundary of $\codim=2$. In the case of a symplectization (with $\R_{s}$-invariant almost complex structure $J$), we require that $D$ is as well and that $D/\R_{s}$ is compact. Otherwise we require that $D$ is compact.
\item Choose a linear map $\pert$ from $\R^{k}$ into the space of perturbations. We require that the perturbations are $\R_{s}$-invariant in the case of a symplectization. In any case we specify that $\pert(\R^{k})$ is invariant under reordering of punctures negatively asymptotic to the same orbit, rotation of markers, and reordering of marked points.
\ee
In contrast with the setup described in previous sections (which we've applied to the framework of \cite{BH:ContactDefinition}), the maps of \cite{Pardon:Contact} may have possibly nodal domains. To deal with this issue, we replace the perturbation spaces $\pertSpace \subset \Sec(T\complete{X} \boxtimes T^{\ast}\Sigma)$ with subspaces of $\Sec(T\complete{X} \boxtimes HT\overline{\ModSpace}_{r + \NnegativePunctures + 1})$ where $HT\overline{\ModSpace}_{r + \NnegativePunctures + 1} \subset T\overline{\ModSpace}_{r + \NnegativePunctures + 1}$ is the kernel of the forgetful map (AKA universal curve) $\overline{\ModSpace}_{r + \NnegativePunctures + 1} \to \overline{\ModSpace}_{r+\NnegativePunctures}$ forgetting a point. These forgotten points parameterize the domains of maps, which are the fibers of the universal curve. To define linearized $\delbar$ operators in this context, we consider direct sums of linearized $\delbar$ operators over nodes of trees. See \cite[Definition 3.22]{Pardon:Contact} for details.

Provided a choice of thickening data $(r, D, \pert)$ as above the associated thickened moduli space, which we denote by $\ModSpace_{(r, D, \pert)}^{[\mapHomotopyClass]}(\orbit, \vec{\orbit}^{-})$, consists of $CH$ trees equipped with $r$ marked points in the domain which are mapped to $D$ and satisfy the $\pert$-perturbed $\delbar_{J, \domainJ}$-equation. A point in the thickened moduli space is \emph{regular} if the associated linearization of the $\delbar$ operator is surjective.

We assume that $\leaf$ is a leaf of $\foliation$ as in the statement of Lemma \ref{Lemma:CHMorphism}. Applying Lemma \ref{Lemma:CurvesStuckInLeaves} to unperturbed curves (and associated buildings of $CH$ trees), we already have homeomorphisms of (compactified) moduli spaces
\begin{equation*}
\overline{\ModSpace}^{[\mapHomotopyClass]}(\orbitDivSet, \vec{\orbitDivSet}^{-}) \to \overline{\ModSpace}^{[I_{\leaf}\mapHomotopyClass]}(\orbit^{+}, \vec{\orbit}^{-})
\end{equation*}
induced by the inclusion map $I_{\leaf}: \leaf \to \R_{s} \times M$ where
\begin{equation*}
\orbitDivSet = \{ \zeta_{+}\} \times \orbitDivSet, \quad \vec{\orbit}^{-} = (\orbit_{1}^{-},\dots,\orbit_{\NnegativePunctures}^{-}), \quad \orbit_{k}^{-} = \{\zeta_{-}\} \times \orbitDivSet_{k}.
\end{equation*}
Moreover, moduli spaces of curves in $\R_{s} \times M$ which are positively asymptotic to orbits of the form $\{ \zeta_{+}\} \times \orbitDivSet$ are empty unless the negative asymptotics take the above form.

In order to establish the desired result it suffices to show that there exists an implicit atlases for $\overline{\ModSpace}^{[\mapHomotopyClass]}(\orbitDivSet, \vec{\orbitDivSet}^{-})$ and $\overline{\ModSpace}^{[I_{\leaf}\mapHomotopyClass]}(\orbit^{+}, \vec{\orbit}^{-})$ as above whose Kuranishi charts are related by a bijective correspondence which
\be
\item induce homeomorphisms between thickened moduli spaces,
\item respects structure maps (i.e. Kuranishi maps and footprint maps),
\item sends regular loci to regular loci via such homeomorphisms.
\ee
Write choices of thickening data for the $\leaf$ and $\R_{s} \times M$ moduli spaces respectively as 
\begin{equation*}
(r^{\leaf}, D^{\leaf}, \pert^{\leaf}), \quad (r^{\R_{s} \times M}, D^{\R_{s} \times M}, \pert^{\R_{s} \times M}).
\end{equation*}

We assume that each $D^{\leaf}$ maps into the subset $\R_{s} \times \subsetC \times \divSet$ via the leaf inclusion map $I_{\leaf}$. This automatically holds when $\R_{s} \times M$ is a symplectization, whence $\leaf$ must have the form $\R_{s} \times \{ \zeta_{+}\} \times \divSet$. In the non-symplectization case with $\leaf$ diffeomorphic to some complete Liouville cobordism $\complete{\Liouville}$, this assumption demands that $D^{\leaf}$ be contained in positive end of $\complete{\Liouville}$. Since in this case holomorphic curves in $\leaf$ must be non-trivially intersect the positive end of $\complete{\Liouville}$, this restriction does not obstruct the satisfaction of the implicit atlas axioms \cite[Definition 3.2]{Pardon:Contact}.

For each thickening data $(r^{\leaf}, D^{\leaf}, \pert^{\leaf})$ choose $(r^{\R_{s} \times M}, D^{\R_{s} \times M}, \pert^{\R_{s} \times M})$ so that
\be
\item $r^{\R_{s} \times M} = r^{\leaf}$,
\item $D^{\R_{s} \times \leaf}$ is a product of a disk $\disk \subset \subsetC$ containing the $\zeta_{\pm}$ with $I_{\leaf}D^{\leaf}$, and
\item $\pert^{\R_{s} \times M}$ is a $\foliation$-adapted extension of $\pert^{\leaf}$ as in Definition \ref{Def:SubbundleExtension}.
\ee
Then a rephrasing of Lemma \ref{Lemma:CHSubbundleExtension} informs us that for such choices the inclusion map $I_{\leaf}$ induces a homeomorphism of associated thickened moduli spaces
\begin{equation*}
\ModSpace_{(r^{\leaf}, D^{\leaf}, \pert^{\leaf})}^{[\mapHomotopyClass]}(\orbitDivSet, \vec{\orbitDivSet}^{-}) \to \ModSpace_{(r^{\R_{s}\times M}, D^{\R_{s} \times M}, \pert^{\R_{s} \times M})}^{[I_{\leaf}\mapHomotopyClass]}(\orbit, \vec{\orbit}^{-}).
\end{equation*}
It is clear that this homeomorphism respects structure maps. The same lemma informs us that $I_{\leaf}$ induces a homeomorphism between the loci of regular points in each thickened moduli space. Unpacking the definitions involved, this means that a map $u^{\leaf}$ to $\leaf$ satisfying a perturbed $\delbar$ equation is transversely cut out iff the associated map $u = I_{\leaf}u^{\leaf}$ satisfying the associated perturbed $\delbar$ equation is transversely cut out as a map to $\R_{s} \times M$. We have established this fact long ago in Lemma \ref{Lemma:AppliedAutoTransverse}. Therefore our choices of thickening data for $\R_{s} \times M$ from those for $\leaf$ determine an implicit atlas for $\overline{\ModSpace}^{[I_{\leaf}\mapHomotopyClass]}(\orbit^{+}, \vec{\orbit}^{-})$ as desired.

\section{Definitions of example geometries}\label{Sec:GeometryExamples}

In this section we build on the results of \S \ref{Sec:DynamicsAndFoliations} to define particular examples of contact manifolds and cobordisms between them which will be used to establish Theorem \ref{Thm:MainChainLevel}. Using these examples (with appropriate contact forms, action bounds, and choices of perturbative data in place), we will establish that
\be
\item the chain-level contact homology of $\C \times \divSet$ can realize a (small resolution) cylinder $\Algebra^{c}$ on the chain-level contact homology, $\Algebra$, of $\divSet$,
\item the chain-level contact homologies of $\R_{\tau} \times \Liouville$ realize the $\Algebra$-module $\Ring_{\aug}$ as well as a semi-free resolution of $\Ring_{\aug}$ associated to the augmentation of $\Algebra$ induced by $\Liouville$, and
\item the chain-level contact homology of a neighborhood $\R_{s} \times \hypersurface$ of a convex hypersurface $\hypersurface$ constructed by gluing Liouville domains $(\Liouville^{\pm}, \beta^{\pm})$ along their boundaries is quasi-isomorphic to $\Ring_{\aug^{+}} \otimes_{\Algebra} \Algebra^{c} \otimes_{\Algebra} \Ring_{\aug^{-}}$, where the $\aug^{\pm}$ are the augmentations determined by the $\Liouville^{\pm}$.
\ee

Throughout this section, we work with a fixed contact form $\alphaDivSet$ on a $\dim=2n-1$ manifold $\divSet$ and a constant $\actionBoundDivSet > 0$ so that all closed $\ReebDivSet$ orbits of action $\leq \actionBoundDivSet$ are non-degenerate.

When discussing contact homology algebras, we assume that $\Ring=\Field$ or $\Ring=\Field[\ChernVar^{\pm 1}]$ are in use until \S \ref{Sec:FinalTwistedCoeffs}. There we will explain how to define the chain-level contact homology $\Algebra(\alphaDivSet, \actionBoundDivSet)$ with coefficients in $\Field[H_{2}(\hypersurface)]$ so that the $\Liouville^{\pm} \subset \hypersurface$ induce $\Field[H_{2}(\hypersurface)]$-valued augmentations.

\subsection{$\C \times \divSet$ and $\R \times \C \times \divSet$}

\subsubsection{Contact forms $\alphaCDivSet_{1}$ and foliations on symplectizations}\label{Sec:GeometryF1}

Here we describe a contact form $\alphaCDivSet_{1}$ on $\C \times \divSet$ determined by a $f_{1} \in \Cinfty(\C)$ and the corresponding foliation on its symplectization determined by Lemma \ref{Lemma:FoliationRCDivSet}.

Throughout we use both Euclidean coordinates $\tau, \sigma$ and polar coordinates $r, \theta$ on $\C$. Write $\disk(\rho) \subset \C$ for the disk of radius $\rho$ with $\disk = \disk(1)$. Let $f_{1} \in \Cinfty(\C)$ and $\beta\in \Omega^{1}(\C)$ satisfy the following properties:
\be
\item $f_{1} = 1 + \frac{\epsilon}{2}r^{2}$ for some $\epsilon \in \R_{>0}$.
\item $d\beta = d\tau\wedge d\sigma$, with $\beta|_{\disk(\half)} = -\sigma d\tau$ and $\beta|_{\C \setminus \disk} = \half r^{2}d\theta$.
\ee
Therefore $f_{1}$ has an $\ind_{\Morse}=0$ critical point at the origin, $0\in \C$, with $f_{1}(0)=1$. Applying the constructions of \S \ref{Sec:ContactOnDivSetC}, we write
\begin{equation*}
\alphaCDivSet_{1} = f_{1}\alphaDivSet + \beta.
\end{equation*}
It follows that if $\epsilon$ is sufficiently small, the conclusions of Lemmas \ref{Lemma:ContactOnDivSetC} and \ref{Lemma:DynamicsOnDivSetC} are applicable. Consequently the following conditions are satisfied:
\be
\item The assignment $\orbitDivSet \to \orbit = \{ 0 \} \times \orbitDivSet$ gives a one-to-one correspondence between
\be
\item closed $\ReebDivSet$ orbits of action $\leq \actionBoundDivSet$ and
\item closed $\ReebCDivSet_{1}$ orbits of action $\leq \actionBound = (1 + \pi\epsilon)\actionBoundDivSet$.
\ee
\item The actions, $\Z$-valued Conley-Zehnder indices, and $\Z$-valued contact homology gradings are
\begin{equation*}
\action(\orbit) = \action(\orbitDivSet), \quad \CZ(\orbit) = \CZ(\orbitDivSet) -1, \quad |\orbit| = |\orbitDivSet|.
\end{equation*}
\ee

\begin{figure}[h]
\begin{overpic}[scale=.6]{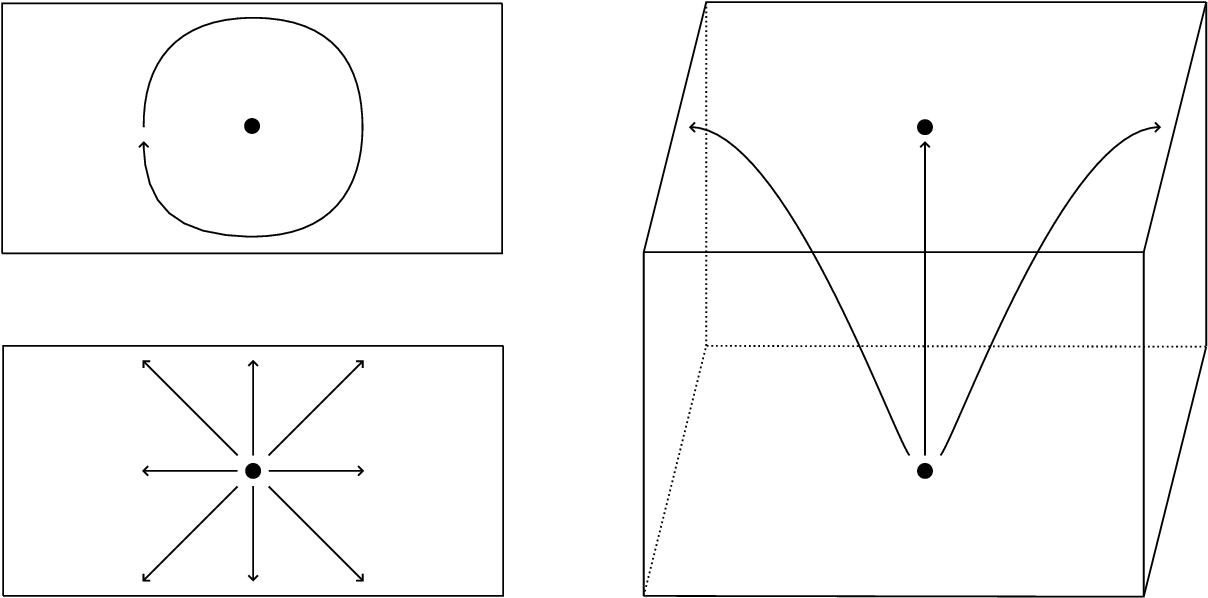}
\end{overpic}
\caption{The upper-left shows the projection of $\ReebCDivSet_{1}$ to $\C$. The lower-left shows $\grad_{f_{1}}$ flow lines in $\C$. Associated $(\partial_{s}, \grad_{f_{1}})$ trajectories in $\R_{s} \times \C$ are shown on the right, which are the projections of leaves of $\foliationCDivSet_{1}$.}
\label{Fig:AlphaCDivSet1}
\end{figure} 

Equip $\R_{s} \times \C \times \divSet$ with the symplectization-type Liouville $1$-form, $\lambdaCDivSet_{1}$, and an $s$-invariant almost complex structure $\JCDivSet_{1}$ so that Lemmas \ref{Lemma:JCompatibleOnRDivSetC} and \ref{Lemma:FoliationRCDivSet} hold, yielding a foliation $\foliationCDivSet_{1}$. The gradient
\begin{equation*}
\grad_{f_{1}} = \epsilon r \partial_{r}
\end{equation*}
is such that the only bounded trajectory is given by the constant map, $t\mapsto 0$. Hence the foliation $\foliationCDivSet_{1}$ of Lemma \ref{Lemma:FoliationRCDivSet} has only a single bounded leaf
\begin{equation*}
\leaf_{f_{1}}^{(0,0)} = \R_{s} \times \{ 0\} \times \divSet \subset \R_{s} \times \C \times \divSet.
\end{equation*}
In the notation of \S \ref{Sec:LeafNotation}, this is a leaf of type $(0,0)$. The projections of the leaves of $\foliationCDivSet_{1}$ to $\C$ and $\R_{s} \times \C$ are shown in Figure \ref{Fig:AlphaCDivSet1}. In the right subfigure, $\R_{s} \times \{ 0\}$ corresponds to the straight, vertical line.

\begin{lemma}
There exist choices of perturbative data so that the association $\orbitDivSet \to \orbit$ defines an isomorphism between chain-level contact homology algebras
\begin{equation*}
\Algebra(\alphaDivSet, \actionBoundDivSet) \to \Algebra(\alphaCDivSet_{1}, \actionBound).
\end{equation*}
\end{lemma}

\begin{proof}
This is immediate from the one-to-one correspondence between (good) orbits in question and Lemma \ref{Lemma:CHMorphism} applied to the leaf $\leaf^{(0,0)}_{f_{1}}$ of the foliation $\foliationCDivSet_{1}$.
\end{proof}

\subsubsection{Contact forms $\alphaCDivSet_{3}$ and foliations on symplectizations}\label{Sec:GeometryF3}

Here we describe a contact form $\alphaCDivSet_{3}$ on $\C \times \divSet$ associated to a $f_{3} \in \Cinfty(\C)$ and the corresponding foliation on its symplectization determined by Lemma \ref{Lemma:FoliationRCDivSet}. We will later identify $\Algebra(\alphaCDivSet_{3}, \actionBound)$ with a cylinder (cf. \S \ref{Sec:CylinderDGAs}) on $\Algebra(\alphaDivSet,  \actionBoundDivSet)$ for an appropriate $\actionBound$.

Let $f_{3} \in \Cinfty(\C, \R_{>0})$ and $\beta \in \Omega^{1}(\C)$ satisfy the following conditions.
\be
\item $\beta$ is exactly as in \S \ref{Sec:GeometryF1} and $f_{3} = 1 + \frac{\epsilon}{2}r^{2}$ outside of $\disk \subset \C$ for some $\epsilon \in \R_{>0}$.
\item $f_{3}$ is Morse with two $\ind_{\Morse}=0$ critical points at $\zeta^{-}, \zeta^{+} \in \disk(\half)$ with $f_{3}(\zeta^{\pm}) = 1$ and a single $\ind_{\Morse}=1$ critical point at $0 \in \C$.
\item $\grad_{f_{3}}$ has exactly two bounded, non-constant flow-lines $\phi_{\pm}$ (up to reparameterization) which are negatively asymptotic to the $\zeta^{\pm}$ and positively asymptotic to $0$.
\ee
Applying the constructions of \S \ref{Sec:ContactOnDivSetC}, we write
\begin{equation*}
\alphaCDivSet_{3} = f_{3}\alphaDivSet + \beta.
\end{equation*}
We assume that $f_{3}|_{\disk}$ is sufficiently $\COne$-close to $1$ and $\epsilon$ is sufficiently small so that the conclusions of Lemmas \ref{Lemma:ContactOnDivSetC} and \ref{Lemma:DynamicsOnDivSetC} are applicable. Consequently, the following conditions are satisfied:
\be
\item The assignment $(\orbitDivSet, \zeta) \mapsto \orbit_{\zeta} = \{ \zeta \} \times \orbitDivSet$ gives a one-to-one correspondence between
\be
\item pairs of closed $\ReebDivSet$ orbits of action $\leq \actionBoundDivSet$ and critical points $\zeta \in \{ \zeta^{-}, \zeta^{+}, 0\}$ of $f_{3}$ and
\item closed $\ReebCDivSet_{3}$ orbits of action $\leq \actionBound = \max(f_{3}|_{\disk}, 1+\epsilon\pi)\actionBoundDivSet$.
\ee
\item The actions, $\Z$-valued Conley-Zehnder indices, and $\Z$-valued contact homology gradings are
\begin{equation*}
\begin{gathered}
\action(\orbit_{\zeta^{\pm}}) = \action(\orbitDivSet), \quad \CZ(\orbit_{\zeta^{\pm}}) = \CZ(\orbitDivSet) -1, \quad |\orbit_{\zeta^{\pm}}| = |\orbitDivSet|,\\
\action(\orbit_{0}) = f_{3}(0)\action(\orbitDivSet), \quad \CZ(\orbit_{0}) = \CZ(\orbitDivSet), \quad |\orbit_{0}| = |\orbitDivSet| + 1.
\end{gathered}
\end{equation*}
\ee

\begin{figure}[h]
\begin{overpic}[scale=.6]{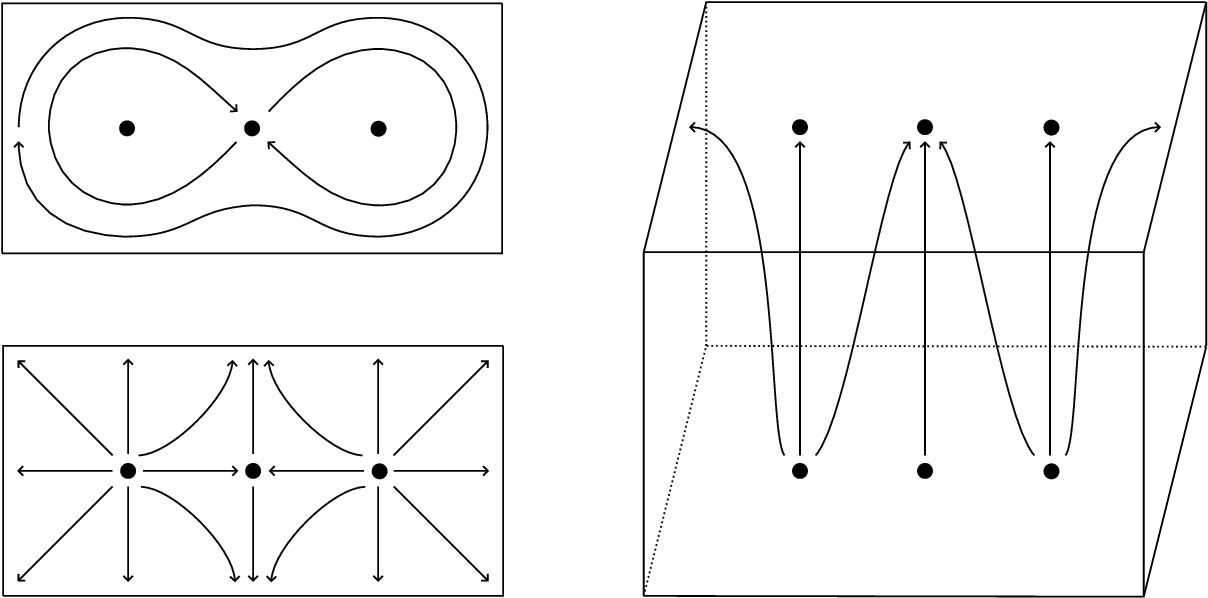}
\put(6, 38){$\zeta^{-}$}
\put(16, 38){$0$}
\put(27, 38){$\zeta^{+}$}
\end{overpic}
\caption{The upper-left shows the projection of $\ReebCDivSet_{3}$ to $\C$. The lower-left shows $\grad_{f_{3}}$ flow lines in $\C$. Associated flow lines in $\R_{s} \times \C$ are shown on the right, which are the projections of leaves of $\foliationCDivSet_{3}$.}
\label{Fig:AlphaCDivSet3}
\end{figure} 

Equip the symplectization $\R_{s} \times \C \times \divSet$ with an almost complex structure $\JCDivSet_{3}$ so that Lemmas \ref{Lemma:JCompatibleOnRDivSetC} and \ref{Lemma:FoliationRCDivSet} hold, yielding a foliation $\foliationCDivSet_{3}$. There are three leaves of $\foliationCDivSet_{3}$ corresponding to constant $\grad_{f_{3}}$ trajectories. They are the symplectizations over the critical points of $f_{3}$:
\begin{equation*}
\begin{gathered}
\leaf_{f_{3}}^{(\zeta^{-},\zeta^{-})} = \R_{s} \times \{ \zeta^{-} \} \times \divSet,\\
\leaf_{f_{3}}^{(0,0)} = \R_{s} \times \{ 0 \} \times \divSet,\\
\leaf_{f_{3}}^{(\zeta^{+},\zeta^{+})} = \R_{s} \times \{ \zeta^{+} \} \times \divSet.
\end{gathered}
\end{equation*}
On the right-hand side of Figure \ref{Fig:AlphaCDivSet3}, these leaves correspond to the straight, vertical lines. In the figure the dots indicate the points $\zeta^{-}, 0, \zeta^{+}$ (appearing left-to-right in each subfigure). In the language of \S \ref{Sec:LeafNotation}, these are leaves of type $(\zeta^{-}, \zeta^{-})$, $(0, 0)$, and $(\zeta^{+}, \zeta^{+})$. The remaining bounded leaves have the form
\begin{equation*}
\leaf_{f_{3}}^{(\zeta^{\pm},0)} = \{ (s, x, \phi_{\pm}(s))\ :\ s \in \R, \ x \in \divSet\}, \quad \pi_{\C}\leaf_{f_{3}}^{(\zeta^{\pm},0)} = \im \phi_{\pm}
\end{equation*}
and come in $\R$-families (determined by the $\R$-family of parameterizations of $\phi_{\pm}$). The projections of these leaves to $\C$ and $\R_{s} \times \C$ are shown as trajectories connecting distinct critical points in Figure \ref{Fig:AlphaCDivSet3}. These leaves are of type $(\zeta^{-}, 0)$ and $(\zeta^{+}, 0)$.

\begin{lemma}\label{Lemma:f3Injections}
There exist choices of perturbative data so that the associations $\orbitDivSet \to \orbit_{\zeta^{\pm}}$ define injective morphisms between chain-level contact homology algebras
\begin{equation*}
I^{\pm}: \Algebra(\alphaDivSet, \actionBoundDivSet) \to \Algebra(\alphaCDivSet_{3}, \actionBound).
\end{equation*}
\end{lemma}

\begin{proof}
This is immediate from the one-to-one correspondences $\orbitDivSet \mapsto \orbit_{\zeta^{\pm}}$ between (good) orbits in question and Lemma \ref{Lemma:CHMorphism} applied to the leaves $\leaf^{(\zeta^{\pm},\zeta^{\pm})}_{f_{3}}$ of the foliation $\foliationCDivSet_{3}$.
\end{proof}

\subsubsection{Cobordisms from $\alphaCDivSet_{3}$ to $\alphaCDivSet_{1}$}

Now we describe complete Liouville cobordisms interpolating between the contact forms $\alphaCDivSet_{3}$ and $\alphaCDivSet_{1}$ on $\C \times \divSet$ described above. In establishing that $\Algebra(\alphaCDivSet_{3}, \actionBound)$ is a cylinder on $\Algebra(\alphaDivSet,  \actionBoundDivSet)$, the DGA morphism associated to this cobordism will be used to realize the surjective quasi-isomorphism $\pi$ of Equation \eqref{Eq:CylinderStructure}. The maps $I^{\pm}$ (of the same equation) are as described in Lemma \ref{Lemma:f3Injections}.

Let $s_{0} \in \R_{>0}$ and let $f_{3,1} \in \Cinfty(\R_{s} \times \C_{\tau, \sigma})$ be such that
\be
\item $f_{3, 1}(s, \ast) = f_{3}$ for $s > s_{0}$ and
\item $f_{3, 1}(s, \ast) = f_{1}$ for $s < -s_{0}$.
\ee
Following the construction of \S \ref{Sec:LiouvilleRCDivSet}, this data determines a $1$-form
\begin{equation*}
\lambdaCDivSet_{f_{3,1}} = e^{s}(f_{3, 1}\alphaDivSet + \beta).
\end{equation*}
We assume that $f_{3,1}$ is sufficiently $\COne$-close to the constant function, $1$, on $\disk$ so that the conclusions of Lemma \ref{Lemma:LiouvilleRDivSetC} hold. This can be arranged by replacing a given $f_{3, 1}$ with $1 + c(f_{3,1}-1)$ as in the proof of Lemma \ref{Lemma:DynamicsOnDivSetC}. In particular, $\lambdaCDivSet_{f_{3,1}}$ gives $\R_{s} \times \C \times \divSet$ the structure of the completion of a strict Liouville cobordism
\begin{equation*}
([-s_{0}, s_{0}] \times \C \times \divSet, \lambdaCDivSet_{f_{3,1}}): (\C \times \divSet, e^{s_{0}}\alphaCDivSet_{3}) \to (\C \times \divSet, e^{-s_{0}}\alphaCDivSet_{1}).
\end{equation*}

\begin{figure}[h]
\begin{overpic}[scale=.6]{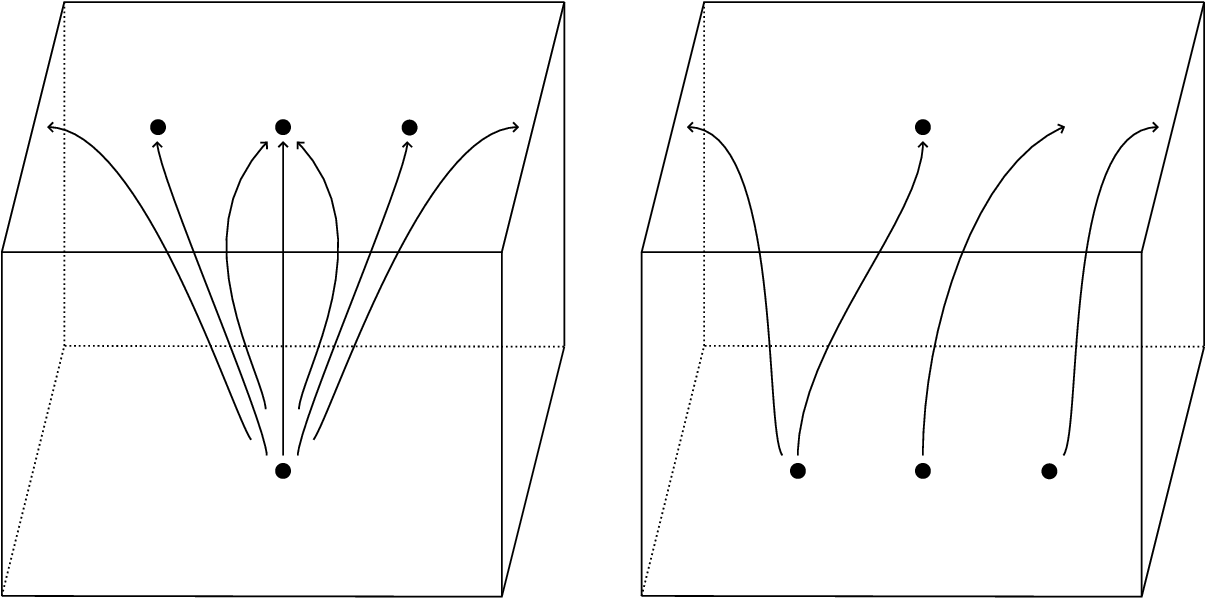}
\put(8, 39){$\zeta^{-}$}
\put(20, 39){$0$}
\put(29, 39){$\zeta^{+}$}
\put(61, 9){$\zeta^{-}$}
\put(73, 9){$0$}
\put(82, 9){$\zeta^{+}$}
\end{overpic}
\caption{Flow lines associated to the foliations $\foliationCDivSet_{3,1}$ (left) and $\foliationCDivSet_{1,3}$ (right) in $\R_{s} \times \C$.}
\label{Fig:AlphaCDivSetCobordism}
\end{figure} 

Equip $\R_{s} \times \C \times \divSet$ with an almost complex structure $\JCDivSet_{3, 1}$ so that Lemmas \ref{Lemma:JCompatibleOnRDivSetC} and \ref{Lemma:FoliationRCDivSet} hold, yielding a foliation $\foliationCDivSet_{3, 1}$. The leaves of $\foliationCDivSet_{3,1}$ project to $(s, z(s))$ where the $z(s)$ are non-autonomous gradient flow line solutions to the equation $\partial_{s}z(s) = \grad_{f_{3, 1}(s, \ast)}$ as depicted in the left-hand side of Figure \ref{Fig:AlphaCDivSetCobordism}. The flow lines $z(s)$ which are convergent to $0 \in \C$ as $s \to -\infty$ flow outward from the origin of $\C$ for $s < -s_{0}$ (as in the left-hand side of Figure \ref{Fig:AlphaCDivSet1}) and then follow the trajectories of Figure \ref{Fig:AlphaCDivSet3} for $s > s_{0}$. By inspection we see that, modulo reparameterization, there is
\be
\item exactly one such flow line with $\lim_{s \to \infty}z(s) = \zeta^{+}$ and
\item exactly one such flow line with $\lim_{s \to \infty}z(s) = \zeta^{-}$.
\ee
Indeed, these are exactly the flow-lines for which $z(s_{0}) = \zeta^{+}$ and $z(s_{0}) = \zeta^{-}$ respectively, since these flow lines follow $\grad_{f_{3}}$ for $s > s_{0}$. The corresponding leaves of $\foliationCDivSet_{3, 1}$ are thus of types $(0, \zeta^{+})$ and $(0, \zeta^{-})$, respectively. Since the leaves of $\foliation_{f_{3, 1}}$ are in one-to-one correspondence with the trajectories $(s, z(s))$, we conclude that there is exactly one leaf of type $(0, \zeta^{+})$ and exactly one leaf of type $(0, \zeta^{-})$. We denote these leaves by
\begin{equation*}
\leaf_{f_{3,1}}^{(0,\zeta^{+})}, \leaf_{f_{3, 1}}^{(0,\zeta^{-})} \subset \R_{s} \times \C \times \divSet.
\end{equation*}

The remaining bounded leaves correspond to trajectories $(s, z(s))$ for which $\lim_{s\to \pm \infty}z(s) = 0$. This is equivalent to $z(s_{0})$ being contained in the stable manifold of $0\in \C$ associated to $\grad_{f_{3}}$ and $z(-s_{0})$ being contained in the unstable manifold of $0\in \C$ associated to $\grad_{f_{1}}$. Such leaves, which we denote by $\leaf_{f_{3,1}}^{(0,0)}$, come in $1$-parameter families. Three flow-lines associated to such leaves, which fit into a $1$-parameter family, are depicted in the left-hand side of Figure \ref{Fig:AlphaCDivSetCobordism}.

\subsubsection{Cobordisms from $\alphaCDivSet_{1}$ to $\alphaCDivSet_{3}$}\label{Sec:f1f3Cobordism}

Let $s_{0} \in \R_{>0}$ and let $f_{1,3} \in \Cinfty(\R_{s} \times \C_{\tau, \sigma}, \R_{>0})$ be such that
\be
\item $f_{1, 3}(s, \ast) = f_{1}$ for $s > s_{0}$ and
\item $f_{1,3}(s, \ast) = f_{3}$ for $s < -s_{0}$.
\ee
Following the construction of \S \ref{Sec:LiouvilleRCDivSet}, this data determines a $1$-form
\begin{equation*}
\lambdaCDivSet_{f_{1,3}} = e^{s}(f_{1,3}\alphaDivSet + \beta).
\end{equation*}
We assume that $f_{1,3}$ is sufficiently $\COne$-close to the constant function, $1$, so that the conclusions of Lemma \ref{Lemma:LiouvilleRDivSetC} hold. In particular, $\lambdaCDivSet_{f_{1,3}}$ gives $\R_{s} \times \C \times \divSet$ the structure of the completion of a strict Liouville cobordism
\begin{equation*}
([-s_{0}, s_{0}] \times \C \times \divSet, \lambdaCDivSet_{f_{1,3}}): (\C \times \divSet, e^{s_{0}}\alphaCDivSet_{1}) \to (\C \times \divSet, e^{-s_{0}}\alphaCDivSet_{3}).
\end{equation*}

Equip $\R_{s} \times \C \times \divSet$ with an almost complex structure $\JCDivSet_{1, 3}$ so that Lemmas \ref{Lemma:JCompatibleOnRDivSetC} and \ref{Lemma:FoliationRCDivSet} hold, yielding a foliation $\foliationCDivSet_{1, 3}$. As in the case of $\foliationCDivSet_{3,1}$ we may enumerate bounded leaves of $\foliationCDivSet_{1,3}$ by studying trajectories $(s, z(s))$ where $z(s): \R \to \C$ solves an $s$-dependent gradient flow equation, with asymptotics determined by the values $z(\pm s_{0}) \in \C$. In order that such a $z(s)$ converge as $s \to \infty$, we must have $z(s_{0}) = 0$, since $0$ is the only critical point of $f_{1}$ and is a global minimum. Hence there is exactly one such convergent trajectory. Such a trajectory will converge to a $\zeta \in \Crit(f_{3})$ iff $z(-s_{0})$ is contained in the unstable manifold of $\zeta$ associated to $\grad_{f_{3}}$. The unstable manifold for $0$ is $1$-dimensional, so for generic choices of $f_{1,3}$, the negative asymptotic of the trajectory $z(s)$ will be one of the $\zeta^{\pm}$. To be concrete, let's assume that $\lim_{s \to -\infty}z(s) = \zeta^{-}$. Then the $(s,z(s))$ trajectories are as they appear in the right subfigure of Figure \ref{Fig:AlphaCDivSetCobordism} and we will denote the unique bounded leaf of $\foliationCDivSet_{1,3}$ by
\begin{equation*}
\leaf_{f_{1,3}}^{(\zeta^{-},0)} \subset \R_{s} \times \C \times \divSet.
\end{equation*}

\begin{lemma}\label{Lemma:CHCylinder}
There exist choices of perturbative data so that we have a diagram of chain-level contact homology algebras
\begin{equation*}
\begin{tikzcd}
\Algebra(\alphaDivSet, \actionBoundDivSet) \arrow[r, "I^{+}"] \arrow[dr, "\Id"]& \Algebra(\alphaCDivSet_{3}, \actionBound) \arrow[d, "\Phi_{f_{3,1}}"]& \Algebra(\alphaDivSet, \actionBoundDivSet) \arrow[l, "I^{-}", swap] \arrow[dl, "\Id", swap] \\
& \Algebra(\alphaDivSet, \actionBoundDivSet)
\end{tikzcd}
\end{equation*}
where $\Phi_{f_{3, 1}}$ is the cobordism map associated to the completed Liouville cobordism $(\R_{s} \times \C \times \divSet, \lambdaCDivSet_{f_{3,1}})$. The morphism $\Phi_{f_{3, 1}}$ is a surjective quasi-isomorphism. Hence $\Algebra(\alphaCDivSet_{3}, \actionBound)$ is a cylinder object for $\Algebra(\alphaDivSet, \actionBoundDivSet)$ in the sense of \S \ref{Sec:CylinderDGAs}.
\end{lemma}

\begin{proof}
We apply Lemma \ref{Lemma:CHMorphism} to the map $\Phi_{f_{3, 1}}$ and each of the leaves $\leaf^{(0,\zeta^{\pm})}_{f_{3, 1}}$. This yields commutative diagrams
\begin{equation*}
\begin{tikzcd}
\Algebra(\alphaDivSet, \actionBoundDivSet) \arrow[d, "\Phi_{\leaf^{(0,\zeta^{\pm})}}"]\arrow[r, "I^{\pm}"] & \Algebra(\alphaCDivSet_{3}, \actionBound) \arrow[d, "\Phi_{f_{3, 1}}"] \\
\Algebra(\alphaDivSet, \actionBoundDivSet) \arrow[r, "\simeq"] & \Algebra(\alphaCDivSet_{1}, \actionBound).
\end{tikzcd}
\end{equation*}
The leaf $\leaf^{(0,\zeta^{\pm})}_{f_{3, 1}}$ is itself a symplectization of $\divSet$, as described in the last statement of Lemma \ref{Lemma:FoliationRCDivSet}. Therefore according to the proofs of \cite[Proposition 9.2.2]{BH:ContactDefinition} and \cite[Lemma 1.2]{Pardon:Contact}, $\Phi_{\leaf^{(\zeta^{\pm}, 0)}}$ is exactly the identity morphism. Since the lower horizontal arrow is an equality, we may collapse it to obtain the triangular sides of the diagram in the statement of the lemma.

To complete the proof, we must show that $\Phi_{f_{3,1}}$ is a quasi-isomorphism, since this map is clearly surjective. To this end, observe that we can deform (appropriately ordered) concatenations of our cobordisms to either a finite symplectization of $(\C \times \divSet, \alphaCDivSet_{1})$ or a finite symplectization of $(\C \times \divSet, \alphaCDivSet_{3})$. Thus according to the functoriality of contact homology, it follows that the maps
\begin{equation*}
CH(\alphaCDivSet_{1}) \xrightarrow{(\Phi_{f_{3,1}})_{\ast}\circ (\Phi_{f_{1,3}})_{\ast}} CH(\alphaCDivSet_{1}), \quad CH(\alphaCDivSet_{3}) \xrightarrow{(\Phi_{f_{1,3}})_{\ast}\circ (\Phi_{f_{3, 1}})_{\ast}} CH(\alphaCDivSet_{3})
\end{equation*}
are the identity maps. We therefore conclude that
\begin{equation*}
\Phi_{f_{3,1}}: \Algebra(\alphaCDivSet_{3}, \actionBound) \to \Algebra(\alphaDivSet, \actionBoundDivSet)
\end{equation*}
is a quasi-isomorphism as desired.
\end{proof}

\subsection{$\R_{\tau} \times \Liouville_{\sigma_{0}}$ and $\R_{s} \times \R_{\tau} \times \Liouville_{\sigma_{0}}$}

Throughout this subsection $(\Liouville, \beta_{\Liouville})$ is a Liouville domain with strict contact boundary $(\divSet, \alphaDivSet)$ and $\sigma_{0} \in (-\half, 0)$. Here we will describe specific examples of the geometric structures of \S \ref{Sec:GeometryRLiouville} on $\R_{\tau} \times \Liouville_{\sigma_{0}}$ and $\R_{s} \times \R_{\tau} \times\Liouville_{\sigma_{0}}$.

\subsubsection{Contact forms $\alphaRLiouville_{0}$ and foliations on symplectizations}

Here we describe a contact form $\alphaRLiouville_{0}$ on $\R_{\tau} \times \Liouville_{\sigma_{0}}$ whose associated Reeb vector field $\ReebRLiouville_{0}$ will have no closed orbits. Consequently, the chain-level contact homology $\Algebra(\alphaRLiouville_{0}, \actionBound)$ associated to this contact form for any $\actionBound \in \R_{>0}$ must be the underlying coefficient ring, $\Ring$.

\begin{figure}[h]
\begin{overpic}[scale=.6]{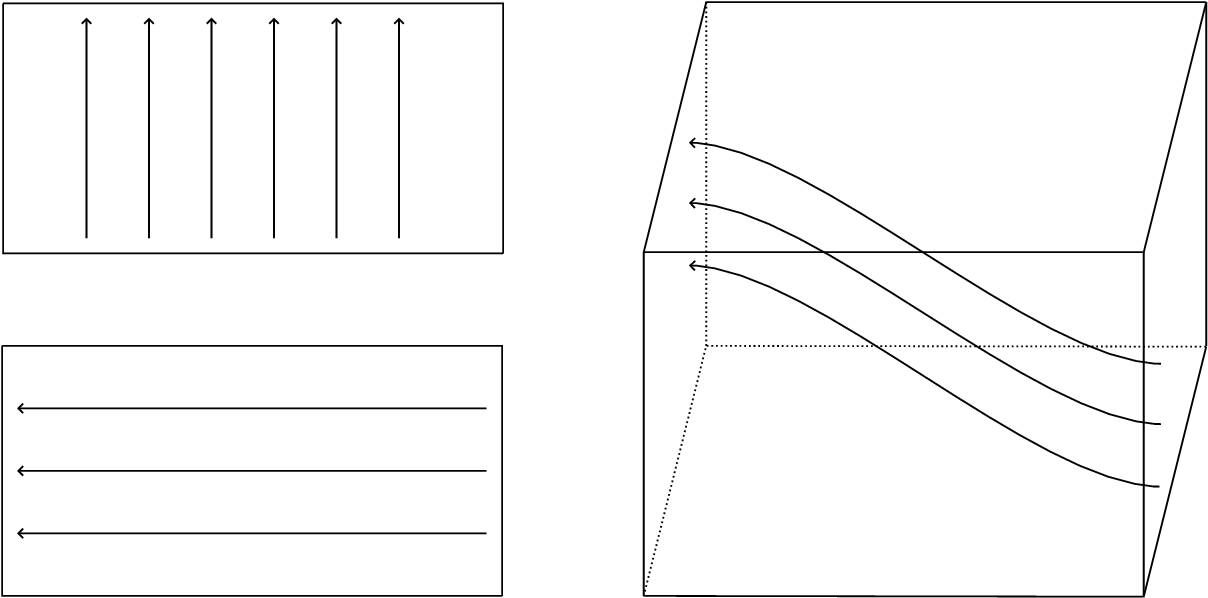}
\end{overpic}
\caption{The upper-left shows the projection of $\ReebRLiouville_{0}$ to $\R_{\tau} \times [-1, \sigma_{0}]$. Flow lines associated to leaves of the foliation $\foliationRLiouville_{0}$ are shown in $\R_{\tau} \times [-1, \sigma_{0}]$ (lower-left) and in $\R_{s} \times \R_{\tau} \times [-1, \sigma_{0}]$ (right). In the left-hand subfigures, $\partial_{\tau}$ points upwards and $\partial_{\sigma}$ points to the right. The flow line entering the right subfigure from the right corresponds to a leaf of the foliation which is ``capped off'' with a copy of $\Liouville_{-1}$.}
\label{Fig:AlphaRLiouville0}
\end{figure} 

Let $f_{0}, g \in \Cinfty([-3, \sigma_{0}]_{\sigma})$ be as in \S \ref{Sec:AlphaContactization} such that
\begin{equation*}
\partial_{\sigma}f_{0} > 0
\end{equation*}
point-wise. Note that $f_{0}$ is independent of $\tau$. Using such functions, define a contact form
\begin{equation*}
\alphaRLiouville_{0} = \begin{cases}
3d\tau + \beta_{\Liouville} & \text{along } \R_{\tau} \times \Liouville_{-3},\\
gd\tau + f_{0}\alphaDivSet & \text{along } \R_{\tau} \times [-3, \sigma_{0}] \times \divSet.
\end{cases}
\end{equation*}
on $\R_{\tau} \times \Liouville_{\sigma_{0}}$ as in Lemma \ref{Lemma:AlphaContactization}. It follows that $d\tau(\ReebRLiouville_{0}) > 0$ everywhere so that there are no closed $\ReebRLiouville_{0}$ orbits, as desired. See Figure \ref{Fig:AlphaRLiouville0}.

Let $\JRLiouville_{0}$ be an almost complex structure on the symplectization $(\R_{s} \times \R_{\tau} \times \Liouville_{\sigma_{0}}, \lambdaRLiouville_{0})$ as described in Lemma \ref{Lemma:FoliationGHProperties}, with associated foliation $\foliationRLiouville_{0}$. Each leaf, $\leaf \subset \R_{s} \times \R_{\tau} \times \Liouville_{\sigma_{0}}$, of $\foliationRLiouville_{0}$ is then as described in item (4) of Lemma \ref{Lemma:FoliationGHProperties}. In particular, the projection of each leaf to $\Liouville_{\sigma_{0}}$ is a diffeomorphism.

\subsubsection{Contact forms $\alphaRLiouville_{2}$ and foliations on symplectizations}\label{Sec:GH2}

Here we describe a contact form $\alphaRLiouville_{2}$ on $\R_{\tau} \times \Liouville_{\sigma_{0}}$ such that $\Algebra(\alphaRLiouville_{2}, \actionBound)$ contains $\Algebra(\alphaDivSet, \actionBoundDivSet)$ as a sub-algebra for $\actionBound = \max(f_{2})\actionBoundDivSet$. Hence we'll be able to view $\Algebra(\alphaRLiouville_{2}, \actionBound)$ as a semi-free module over $\Algebra(\alphaDivSet, \actionBoundDivSet)$.

Let $g$ be as above and let $f_{2} \in \Cinfty(\R_{\tau} \times [-3, \sigma_{0}])$ be as in \S \ref{Sec:AlphaContactization} satisfying the following properties:
\be
\item $f_{2}$ agrees with $f_{0}$ outside of a closed neighborhood of $\{ \tau = 0, \sigma > -1\}$. Thus $f_{2}$ is independent of $\tau$ outside of this neighborhood.
\item $f_{2}$ has exactly two critical points
\begin{equation*}
\begin{gathered}
\zeta_{e}, \zeta_{h} \in \{ \tau = 0, \sigma \in (-1, \sigma_{0}) \}, \quad \sigma(\zeta_{e}) > \sigma(\zeta_{h}),\\
f_{2}(\zeta_{e}) = 1, \quad \ind_{\Morse}(f_{2}, \zeta_{e})=0, \quad \ind_{\Morse}(f_{2}, \zeta_{h}) = 1.
\end{gathered}
\end{equation*}
\item There is a single $\grad_{f_{2}}$ flow line in $\R_{\tau} \times [-1, \sigma_{0}]$ which is negatively asymptotic to $\zeta_{e}$ and positively asymptotic to $\zeta_{h}$.
\ee
Here the subscripts for $\zeta_{e}$ and $\zeta_{h}$ stand for ``elliptic'' and ``hyperbolic'', respectively.

Using such functions $g$ and $f_{2}$, define a contact form
\begin{equation*}
\alphaRLiouville_{2} = \begin{cases}
3d\tau + \beta_{\Liouville} & \text{along } \R_{\tau} \times \Liouville_{-3},\\
gd\tau + f_{2}\alphaDivSet & \text{along } \R_{\tau} \times [-3, \sigma_{0}] \times \divSet.
\end{cases}
\end{equation*}
on $\R_{\tau} \times \Liouville_{\sigma_{0}}$ as in Lemma \ref{Lemma:AlphaContactization}. We assume that $f_{2}$ is sufficiently $\COne$-close to $1$ along $\{ \sigma \in [-1, \sigma_{0}]\}$ so that the conclusions of Lemma \ref{Lemma:DynamicsOnRLiouville} are applicable. In particular, the following conditions are satisfied:
\be
\item The assignment $(\orbitDivSet, \zeta) \mapsto \orbit_{\zeta}$ for $\zeta \in \{ \zeta_{e}, \zeta_{h}\}$ gives a one-to-one correspondence between
\be
\item pairs of closed $\ReebDivSet$ orbits $\orbitDivSet$ of action $\leq \actionBoundDivSet$ and critical points $\zeta \in \{ \zeta_{e}, \zeta_{h}\}$ of $f_{2}$ and
\item closed $\ReebRLiouville_{2}$ orbits of action $\leq \actionBound = \max(f_{2})\actionBoundDivSet$.
\ee
\item The actions, $\Z$-valued Conley-Zehnder indices, and $\Z$-valued contact homology gradings are
\begin{equation*}
\begin{gathered}
\action(\orbit_{\zeta_{e}}) = \action(\orbitDivSet), \quad \CZ(\orbit_{\zeta_{e}}) = \CZ(\orbitDivSet) -1, \quad |\orbit_{\zeta_{e}}| = |\orbitDivSet|,\\
\action(\orbit_{\zeta_{h}}) = f_{2}(\zeta_{h})\action(\orbitDivSet), \quad \CZ(\orbit_{\zeta_{h}}) = \CZ(\orbitDivSet), \quad |\orbit_{\zeta_{h}}| = |\orbitDivSet| + 1.
\end{gathered}
\end{equation*}
\ee

\begin{figure}[h]
\begin{overpic}[scale=.6]{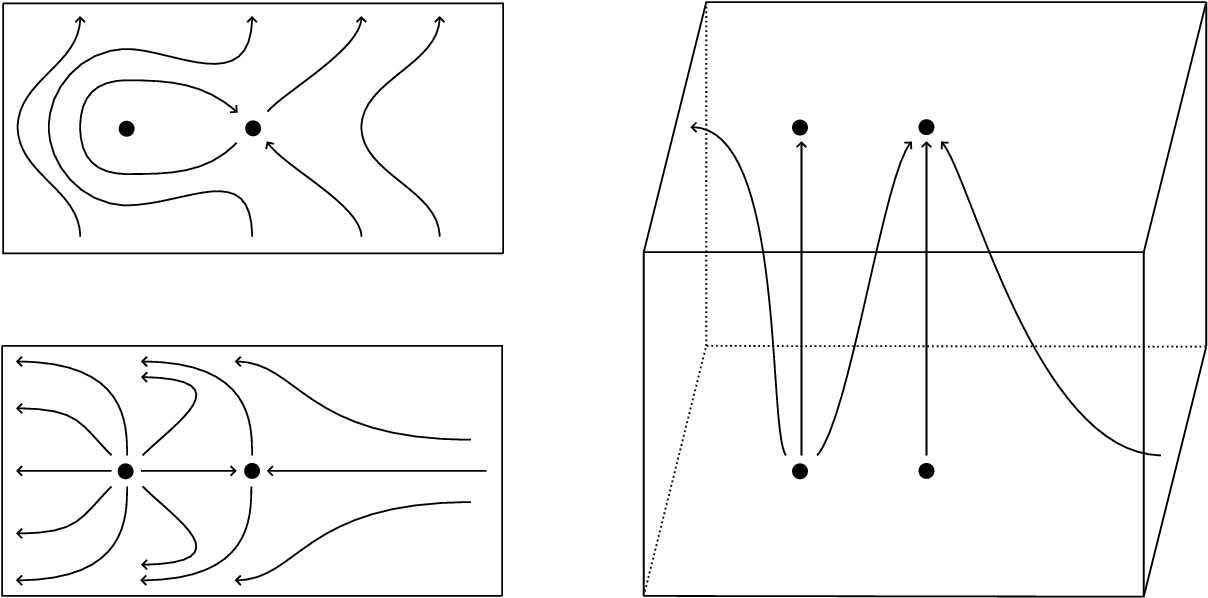}
\put(23, 38){$\zeta_{h}$}
\put(12, 38){$\zeta_{e}$}
\end{overpic}
\caption{The upper-left shows the projection of $\ReebRLiouville_{2}$ to $\R_{\tau} \times [-1, \sigma_{0}]$. Flow lines associated to the foliation $\foliationRLiouville_{2}$ are shown in $\R_{\tau} \times [-1, \sigma_{0}]$ (lower-left) and in $\R_{s} \times \R_{\tau} \times [-1, \sigma_{0}]$ (right). In the left-hand subfigures, $\partial_{\tau}$ points upwards and $\partial_{\sigma}$ points to the right. The flow line entering the right subfigure from the right corresponds to a leaf of the foliation which is ``capped off'' with a copy of $\Liouville_{-1}$.}
\label{Fig:AlphaRLiouville2}
\end{figure} 

Let $\JRLiouville_{2}$ be an almost complex structure on the symplectization $(\R_{s} \times \R_{\tau} \times \Liouville_{\sigma_{0}}, \lambdaRLiouville_{2})$ as described in Lemma \ref{Lemma:FoliationGHProperties}, with associated foliation $\foliationRLiouville_{2}$. Each leaf, $\leaf \subset \R_{s} \times \R_{\tau} \times \Liouville_{\sigma_{0}}$ of $\foliationRLiouville_{2}$ is then as described in Lemma \ref{Lemma:FoliationGHProperties}. The flow lines in $\R_{\tau} \times [-1, \sigma_{0}]$ and $\R_{s} \times \R_{\tau} \times [-1, \sigma_{0}]$ associated to $\foliationRLiouville_{2}$ are depicted in Figure \ref{Fig:AlphaRLiouville2}. Such flow lines are obtained by taking the intersections of leaves with the region $\{ \sigma \geq -1\}$ and then projecting the result to the $\tau, \sigma$ or $s, \tau, \sigma$ coordinates.
 
The only leaves of types $(\zeta_{e}, \zeta_{e})$ and $(\zeta_{h}, \zeta_{h})$ are the symplectizations
\begin{equation}\label{Eq:RigidLeavesGH2}
\leaf_{f_{2}}^{(\zeta_{e},\zeta_{e})} = \R_{s} \times \{ \zeta_{e}\} \times \divSet, \quad \leaf_{f_{2}}^{(\zeta_{h},\zeta_{h})} = \R_{s} \times \{ \zeta_{h}\} \times \divSet,
\end{equation}
which are fully contained in $\{ \sigma \in [-1, \sigma_{0}]\} \subset \R_{s} \times \R_{\tau} \times \Liouville_{\sigma_{0}}$. These leaves are isolated in the sense that they are invariant under the $\R_{s}$-translation of the symplectization.

We also have the following bounded leaves, which come in $1$-parameter families:
\be
\item Leaves of type $(\zeta_{e}, \zeta_{h})$ project to $\R_{\tau} \times [-1, \sigma_{0}]_{\sigma}$ as the unique, non-constant $\grad_{f_{2}}$ flow line connecting $\zeta_{e}$ to $\zeta_{h}$.
\item Leaves of type $(\Liouville, \zeta_{h})$ correspond to $\grad_{f_{2}}$ flow lines in $\R_{\tau} \times [-1, \sigma_{0}]_{\sigma}$. Taking the intersection of such a leaf with $\R_{s} \times \R_{\tau} \times \Liouville_{-1}$ and projecting the result to $\Liouville_{-1}$ yields a diffeomorphism.
\ee
Such leaves will respectively be denoted by
\begin{equation}\label{Eq:NonRigidLeavesGH2}
\leaf^{(\zeta_{e},\zeta_{h})}_{f_{2}}, \leaf^{(\Liouville, \zeta_{h})}_{f_{2}} \subset \R_{s} \times \R_{\tau} \times \Liouville_{\sigma_{0}}.
\end{equation}
Since the leaves of Equations \eqref{Eq:RigidLeavesGH2} and \eqref{Eq:NonRigidLeavesGH2} account for all $\grad_{f_{2}}$ trajectories in $\R_{\tau} \times [-1, \sigma_{0}]$ which are positively asymptotic to $\zeta_{e}$ and $\zeta_{h}$, these leaves account for all bounded leaves of $\foliationRLiouville_{2}$.

\subsubsection{Cobordisms from $\alphaRLiouville_{2}$ to $\alphaRLiouville_{0}$}

Now we describe complete Liouville cobordisms interpolating between the contact forms $\alphaRLiouville_{2}$ and $\alphaRLiouville_{0}$ on $\R_{\tau} \times \Liouville_{\sigma_{0}}$ described above. These cobordisms will be needed to establish that $\Algebra(\alphaRLiouville_{0}, \actionBound)$ is quasi-isomorphic to $\Ring_{\aug}$ as a $\Algebra(\alphaDivSet, \actionBoundDivSet)$-module, where $\aug: \Algebra(\alphaDivSet, \actionBoundDivSet) \to (\Ring, 0)$ is an augmentation induced by the filling $(\Liouville, \betaLiouville)$ of $(\divSet, \alphaDivSet)$.

Let $s_{0} \in \R_{>0}$, let $g$ be as above, and let $f_{2,0} \in \Cinfty(\R_{s} \times \R_{\tau} \times [-3, \sigma_{0}])$ be such that
\be
\item $f_{2, 0}(s, \ast) = f_{2}$ for $s > s_{0}$ and
\item $f_{2, 0}(s, \ast) = f_{0}$ for $s < -s_{0}$.
\ee
Following \S \ref{Sec:GeometryRLiouville}, this data determines a $1$-form $\lambdaRLiouville_{f_{2,0}}$ on $\R_{s} \times \R_{\tau} \times \Liouville_{\sigma_{0}}$, giving it the structure of the completion of a strict Liouville cobordism
\begin{equation*}
([-s_{0}, s_{0}] \times \R_{\tau} \times \Liouville_{\sigma_{0}}, \lambdaRLiouville_{f_{2,0}}): (\R_{\tau} \times \Liouville_{\sigma_{0}}, e^{s_{0}}\alphaRLiouville_{2}) \to (\R_{\tau} \times \Liouville_{\sigma_{0}}, e^{-s_{0}}\alphaRLiouville_{0}).
\end{equation*}

\begin{figure}[h]
\begin{overpic}[scale=.6]{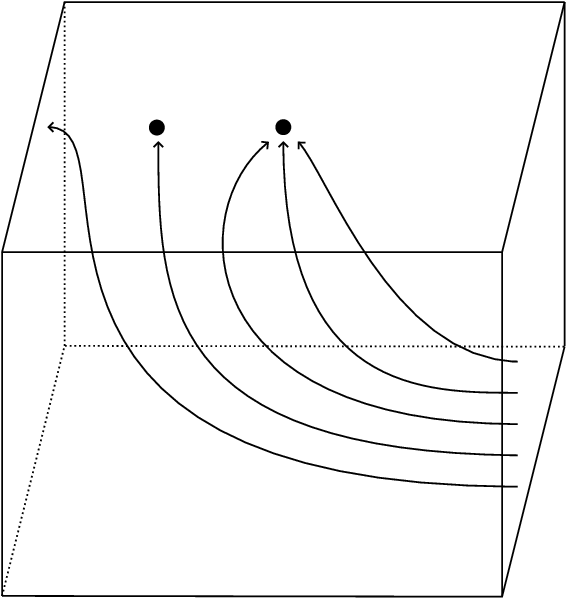}
\put(51, 78){$\zeta_{h}$}
\put(30, 78){$\zeta_{e}$}
\end{overpic}
\caption{Trajectories associated to the foliation $\foliationRLiouville_{f_{2,0}}$ are shown in $\R_{s} \times \R_{\tau} \times [-1, \sigma_{0}]$.}
\label{Fig:AlphaRLiouville20}
\end{figure} 

Equip $\R_{s} \times \R_{\tau} \times \Liouville_{\sigma_{0}}$ with an almost complex structure $\JRLiouville_{f_{2,0}}$ so that Lemmas \ref{Lemma:JGHCompatible} and \ref{Lemma:FoliationGHProperties} hold, yielding a foliation $\foliationRLiouville_{f_{2,0}}$. If we take the intersection of a leaf, $\leaf$, of $\foliationRLiouville_{f_{2,0}}$, and project $\leaf \cap \{ \sigma \geq -1\}$ to $\R_{s} \times \R_{\tau} \times [-1, \sigma_{0}]$ we obtain a parameterized path of the form $(s, z(s))$ with $\partial_{s}z = \grad_{f_{2,0}(s, \ast)}$. Such paths are depicted in Figure \ref{Fig:AlphaRLiouville20}. Since $f_{0}$ has no critical points, all bounded leaves of $\foliationRLiouville_{f_{2,0}}$ must either be of type $(\Liouville, \zeta_{h})$ or $(\Liouville, \zeta_{e})$. Since for $s > s_{0}$, $\grad_{f_{2,0}}$ agrees with $\grad_{f_{2}}$ and points outward from $\zeta_{e}$, only one such path $(s, z(s))$ can have $\lim_{s \to \infty}z(s) = \zeta_{e}$. Therefore there is exactly one leaf of type $(\Liouville, \zeta_{e})$, which will be denoted by
\begin{equation*}
\leaf_{f_{2,0}}^{(\Liouville, \zeta_{e})} \subset \R_{s} \times \R_{\tau} \times \Liouville_{\sigma_{0}}.
\end{equation*}

\begin{lemma}\label{Lemma:CHAugmentationModule}
There exist choices of perturbative data so that we have a diagram of chain-level contact homology algebras
\begin{equation*}
\begin{tikzcd}
\Algebra(\alphaDivSet, \actionBoundDivSet) \arrow[r, "\orbitDivSet \mapsto \orbit_{\zeta_{e}}"] \arrow[dr, "\aug"]& \Algebra(\alphaRLiouville_{2}, \actionBound) \arrow[d, "\Phi_{f_{2,0}}"]\\
& \Ring.
\end{tikzcd}
\end{equation*}
Here $\aug$ is the augmentation defined by the leaf $\leaf^{(\Liouville, \zeta_{e})}_{f_{2,0}} \simeq \complete{\Liouville}$ with some choices of perturbative data and $\Phi_{f_{2, 0}}$ is the cobordism map associated to the completed Liouville cobordism $(\R_{s} \times \R_{\tau} \times \Liouville_{\sigma_{0}}, \lambdaRLiouville_{f_{2,0}})$. The morphism $\Phi_{f_{2, 0}}$ is a surjective quasi-isomorphism. Hence the morphism $\orbitDivSet \mapsto \orbit_{\zeta_{e}}$ gives $\Algebra(\alphaRLiouville_{2}, \actionBound)$ the structure of a semi-free $\Algebra(\alphaDivSet, \actionBoundDivSet)$-module which is quasi-isomorphic to $\Ring_{\aug}$ as a $\Algebra(\alphaDivSet, \actionBoundDivSet)$-module.
\end{lemma}

\begin{proof}
The proof is very similar to that of Lemma \ref{Lemma:CHCylinder}. Applying Lemma \ref{Lemma:CHMorphism} to the leaves $\leaf^{(\zeta_{e}, \zeta_{e})}_{f_{2}}$, we see that the association $\orbitDivSet \mapsto \orbit_{\zeta_{e}}$ induces a CDGA morphism which is clearly injective. The composition of this morphism with $\Phi_{f_{2, 0}}$ counts perturbed holomorphic curves in the leaf $\leaf^{(\Liouville, \zeta_{e})}_{f_{2,0}}$ and so exactly gives $\aug$. The fact that $\Phi_{f_{2, 0}}$ is a quasi-isomorphism follows from the functoriality of contact homology and our ability to construct an ``inverse cobordism'' whose concatenation with $(\R_{s} \times \R_{\tau} \times \Liouville_{\sigma_{0}}, \lambdaRLiouville_{f_{2,0}})$ is deformation equivalent to a symplectization (as we did in \S \ref{Sec:f1f3Cobordism}). The morphism $\Phi_{f_{2, 0}}$ is clearly surjective.
\end{proof}

\subsection{$\R_{\tau} \times \hypersurface$ and $\R_{s} \times \R_{\tau} \times \hypersurface$}

Finally we describe specific examples of geometric structures on $\R_{\tau} \times \hypersurface$ and $\R_{s} \times \R_{\tau} \times \hypersurface$ as described in \S \ref{Sec:GeometryRHypersurface}. Following the notation of that section, let $(\Liouville^{\pm}, \beta_{\Liouville^{\pm}})$ be a pair of strict Liouville fillings of $(\divSet, \alphaDivSet)$, determining a neighborhood $\R_{\tau} \times \hypersurface$ of a convex hypersurface. In particular, we will describe a contact form $\alphaRHypersurface_{5}$ on $\R_{\tau} \times \hypersurface$ so that $\Algebra(\alphaRHypersurface_{5}, \actionBound)$ can be identified with a tensor product of semi-free bimodules over $\Algebra = \Algebra(\alphaDivSet, \actionBoundDivSet)$,
\begin{equation}\label{Eq:TensorProductSummary}
\Algebra(\alphaRHypersurface_{5}, \actionBound) = \Algebra(\alpha^{\R \times \Liouville^{+}}_{2}, \actionBound) \otimes_{\Algebra} \Algebra(\alphaCDivSet_{3}, \actionBound) \otimes_{\Algebra}  \Algebra(\alpha^{\R \times \Liouville^{-}}_{2}, \actionBound).
\end{equation}
Here $\alphaCDivSet_{3}$ is as described in \S \ref{Sec:GeometryF3}. Likewise, the $\alpha^{\R \times \Liouville^{\pm}}_{2}$ are contact forms on the $\R_{\tau} \times \Liouville_{\sigma_{0}}^{\pm}$ as in \S \ref{Sec:GH2}. As described at the start of \S \ref{Sec:GH2}, these algebras will later be shown to be quasi-isomorphic to $\Ring_{\aug^{\pm}}$ where $\aug^{\pm}$ are the augmentations associated to the $(\Liouville^{\pm}, \beta_{\Liouville^{\pm}})$. In light of the results of \S \ref{Sec:CylinderDGAs} (specifically Equation \eqref{Eq:DerivedTensorFromCylinder}), Lemma \ref{Lemma:CHCylinder}, and Equation \eqref{Eq:TensorProductSummary}, we will then be able to identify
\begin{equation*}
H\Algebra(\alphaRHypersurface_{5}, \actionBound) = \Ring_{\aug^{+}} \otimesDerived_{\Algebra} \Ring_{\aug^{-}},
\end{equation*}
establishing Theorem \ref{Thm:MainChainLevel}.

\subsubsection{Contact forms $\alphaRHypersurface_{5}$ and foliations on symplectizations}

Let $\sigma_{0} \in (-\half, 0)$ be of arbitrarily small absolute value. Choose functions $f^{\pm}_{2}, g^{\pm} \in \Cinfty(\R_{\tau} \times [-3, \sigma_{0}])$ as in \S \ref{Sec:GH2} which, in particular, satisfy the following properties:
\be
\item $g^{\pm} = -\sigma$ along $\{ \sigma > -1\}$.
\item The $f^{\pm}_{2}$ each have exactly one $\ind_{\Morse}=0$ critical point at some $\zeta_{e}$ and exactly one $\ind_{\Morse}=1$ critical point at some $\zeta_{h}$, with $\zeta_{e}, \zeta_{h} \in \{ \tau=0, \sigma \in (-1, \sigma_{0})\}$. We assume that these critical points are the same for $f^{+}_{2}$ and $f^{-}_{2}$ and take the same critical values.
\item Outside of a closed neighborhood of the interval $\{ \tau=0, \sigma \in (-1, \sigma_{0})\}$, the $f^{\pm}_{2}$ are $\COne$ close to $1$, independent of $\tau$, and satisfy $\partial_{\sigma}f_{2}^{\pm} > 0$.
\ee
Write $\sigma_{1}, \sigma_{2} \in \R_{<0}$ for constants such that
\begin{equation*}
-1 < \sigma_{2} < \sigma(\zeta_{h}) < \sigma_{1} < \sigma(\zeta_{e}) < \sigma_{0}.
\end{equation*} 

Let $f_{5} \in \Cinfty(\R_{\tau} \times [-1, 1])$ be a function satisfying the following properties:
\be
\item $f_{5}(\tau, \sigma) = f^{\pm}_{2}(\pm \tau, \pm \sigma)$ for $\pm \sigma \in [-1, \sigma_{0}]$ (as in Equation \eqref{Eq:FHOverlap}).
\item $f_{5}$ has a unique critical point of $\ind_{\Morse}=1$ within $\{ |\sigma| < -\sigma_{0}\}$. We assume that this critical point is $0$ and that $f_{5}(0) = f_{5}(\zeta_{h})$.
\item Outside of a closed neighborhood of $\{ \tau = 0, |\sigma| < -\sigma_{2}\}$, we assume that $f_{5}$ is independent of $\tau$.
\ee
Thus $f_{5}$ has exactly five critical points, which are all contained in $\{ \tau = 0, |\sigma| < -\sigma_{2}\}$. Ordering the critical points by their $\sigma$ values, they may be written
\begin{equation*}
\zeta_{h^{+}} := \zeta_{h}, \quad \zeta_{e^{+}}: = \zeta_{e}, \quad 0,\quad  \zeta_{e^{-}}:=-\zeta_{e}, \quad \zeta_{h^{-}}:=-\zeta_{h}.
\end{equation*}
The flow lines of the Hamiltonian vector field $X_{f_{5}}$ on $\R_{\tau} \times [-1, 1]_{\sigma}$ with respect to the symplectic form $d\tau \wedge d\sigma$ are depicted in the top-left subfigure of Figure \ref{Fig:AlphaF3GH2}.

\begin{figure}[h]
\begin{overpic}[scale=.6]{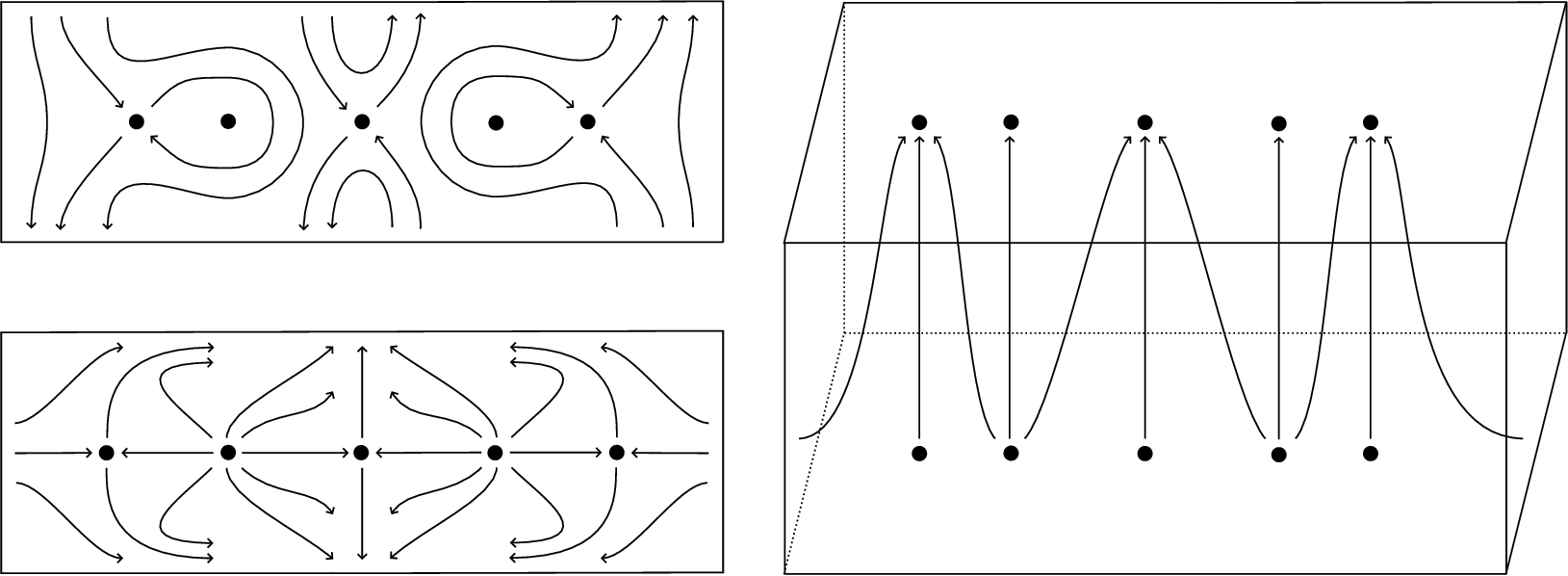}
\put(57, 31){$\zeta_{h^{-}}$}
\put(63.5, 31){$\zeta_{e^{-}}$}
\put(72.5, 30.5){$0$}
\put(80.5, 31){$\zeta_{e^{+}}$}
\put(86.5, 31){$\zeta_{h^{+}}$}
\end{overpic}
\caption{The upper-left shows the projection of $\ReebRHypersurface_{5}$ to $\R_{\tau} \times [-1,1]_{\sigma}$. The lower-left shows associated gradient flow lines in $\R_{\tau} \times [-1,1]_{\sigma}$. In these subfigures, $\partial_{\tau}$ points upwards and $\partial_{\sigma}$ points to the left. Associated trajectories in $\R_{s} \times \R_{\tau} \times [-1,1]_{\sigma}$ are shown in the right subfigure.}
\label{Fig:AlphaF3GH2}
\end{figure}

As in \S \ref{Sec:GeometryRHypersurface}, we construct a contact form on
\begin{equation*}
\R_{\tau} \times \hypersurface = (\R_{\tau} \times \Liouville_{\sigma_{0}}^{+}) \cup (\R_{\tau} \times [-1, 1] \times \divSet) \cup (\R_{\tau} \times \Liouville_{\sigma_{0}}^{-})
\end{equation*}
such that
\be
\item the restriction to $\R_{\tau} \times \Liouville_{\sigma_{0}}^{\pm}$ is $\alpha^{\R \times \Liouville^{\pm}}_{2}$ and
\item the restriction to $\R_{\tau} \times [-1,1] \times \divSet$ is $\alpha_{f_{5}} = f_{5}\alphaDivSet -\sigma d\tau$.
\ee
This contact form and its associated Reeb vector field will respectively be denoted by
\begin{equation*}
\alphaRHypersurface_{5} \in \Omega^{1}(\R_{\tau} \times \hypersurface), \quad \ReebRHypersurface_{5} \in \mathfrak{X}(\R_{\tau} \times \hypersurface).
\end{equation*}
Following Lemma \ref{Lemma:DynamicsRHypersurface}, the following conditions are satisfied:
\be
\item The assignment $(\orbitDivSet, \zeta) \mapsto \{ \zeta\} \times \orbitDivSet = \orbit_{\zeta}$ gives a one-to-one correspondence between
\be
\item pairs of closed $\ReebDivSet$ orbits $\orbitDivSet$ of action $\leq \actionBoundDivSet$ and critical points $\zeta \in \{ \zeta_{e^{+}}, \zeta_{h^{+}}, 0, \zeta_{h^{-}}, \zeta_{e^{-}}\}$ of $f_{5}$ and
\item closed $\ReebRHypersurface_{5}$ orbits of action $\leq \actionBound = \max(f_{5})\actionBoundDivSet$.
\ee
\item The actions, $\Z$-valued Conley-Zehnder indices, and $\Z$-valued contact homology gradings are
\begin{equation*}
\begin{gathered}
\action(\orbit_{\zeta_{e^{\pm}}}) = \action(\orbitDivSet), \quad \CZ(\orbit_{\zeta_{e^{\pm}}}) = \CZ(\orbitDivSet) -1, \quad |\orbit_{\zeta_{e^{\pm}}}| = |\orbitDivSet|,\\
\action(\orbit_{0}) = \action(\orbit_{\zeta_{h^{\pm}}})= f_{5}(0)\action(\orbitDivSet), \quad \CZ(\orbit_{0})= \CZ(\orbit_{\zeta_{h^{\pm}}}) = \CZ(\orbitDivSet), \quad |\orbit_{0}| = |\orbit_{\zeta_{h^{\pm}}}| = |\orbitDivSet| + 1.
\end{gathered}
\end{equation*}
\ee

Now we describe geometric structures on $\R_{s} \times \R_{\tau} \times \hypersurface$ associated to the contact form $\alphaRHypersurface_{5}$. We write
\begin{equation*}
\lambdaRHypersurface_{5} = e^{s}\alphaRHypersurface_{5}
\end{equation*}
for the symplectization $1$-form and $\JRHypersurface_{5}$ for an associated, $s$-invariant almost complex structure on $\R_{s} \times \R_{\tau} \times \hypersurface$ as described in \S \ref{Sec:JRHypersurface}. The foliation associated to $\JRHypersurface_{5}$ (as described in Lemma \ref{Lemma:FoliationFGHProperties}) will be denoted by $\foliationRHypersurface_{5}$.

Following the narrative pattern of the previous subsections, we find that there are exactly five leaves of $\foliationRHypersurface_{5}$ which are invariant under $s$-translation. Such leaves correspond to the vertical trajectories in Figure \ref{Fig:AlphaF3GH2}. They are
\begin{equation*}
\begin{gathered}
\leaf_{5}^{(\zeta_{h^{+}},\zeta_{h^{+}})} = \R_{s} \times \{ \zeta_{h^{+}}\} \times \divSet, \quad \leaf_{5}^{(0,0)} = \R_{s} \times \{ 0\} \times \divSet, \quad \leaf_{5}^{(\zeta_{h^{-}},\zeta_{h^{-}})} = \R_{s} \times \{ \zeta_{h^{-}}\} \times \divSet,\\
\leaf_{5}^{(\zeta_{e^{+}},\zeta_{e^{+}})} = \R_{s} \times \{ \zeta_{e^{+}}\} \times \divSet, \quad \leaf_{5}^{(\zeta_{e^{-}},\zeta_{e^{-}})} = \R_{s} \times \{ \zeta_{e^{-}}\} \times \divSet.
\end{gathered}
\end{equation*}

\subsubsection{Proof of the main theorems for $\Field$ and $\Field[\ChernVar^{\pm 1}]$ coefficients}\label{Sec:MainProofsUntwisted}

We now use the contact forms described above to establish the main results of this article. The statement of the following lemma is graphically summarized in Figure \ref{Fig:Alpha5Decomposition}.

\begin{figure}[h]
\begin{overpic}[scale=.55]{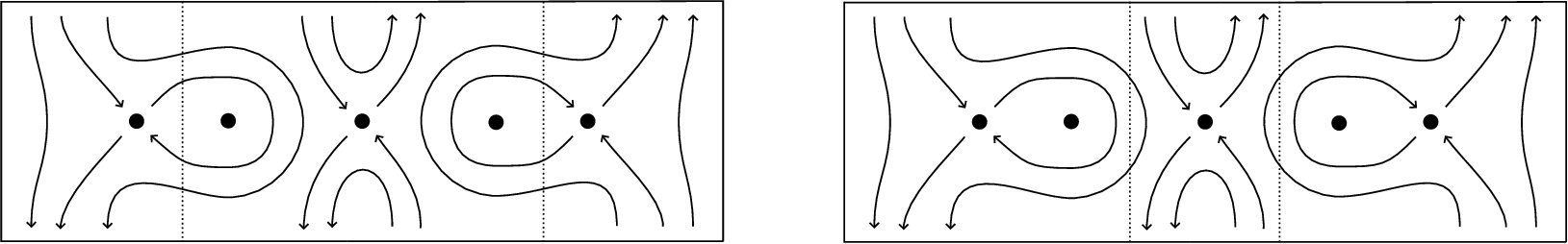}
\put(11.5, -.5){$\underbrace{\hspace{35.3mm}}_{\text{\normalsize{$\Algebra(\alphaCDivSet_{3}, \actionBound)$}}}$}
\put(53.8,-.5){$\underbrace{\hspace{28mm}}_{\text{\normalsize{$\Algebra(\alpha^{\R\times \Liouville^{-}}_{2}, \actionBound)$}}}$}
\put(81.7,-.5){$\underbrace{\hspace{28mm}}_{\text{\normalsize{$\Algebra(\alpha^{\R \times \Liouville^{+}}_{2}, \actionBound)$}}}$}
\end{overpic}
\vspace{7mm}
\caption{On the left, the dashed lines, $\sigma = \pm \sigma_{1}$, indicate the boundary of a region $\{ |\sigma| < -\sigma_{1}\} \subset \R_{\tau} \times \hypersurface$ whose chain-level contact homology will recover a cylinder over $\Algebra(\alphaDivSet, \actionBoundDivSet)$. On the right the  dashed lines, $\sigma = \pm \sigma_{0}$, indicate the boundaries of the $\R_{\tau} \times \Liouville_{\sigma_{0}}^{\pm} \subset \R_{\tau} \times \hypersurface$ whose chain-level contact homologies contain $\Algebra(\alphaDivSet, \actionBoundDivSet)$ as subalgebras and are quasi-isomorphic to the $\Ring_{\aug^{\pm}}$.}
\label{Fig:Alpha5Decomposition}
\end{figure}

\begin{lemma}\label{Lemma:GeometricDerivedTensor}
Here we assume that all chain-level contact homology algebras are defined with $\Ring=\Field$ or $\Ring=\Field[\ChernVar^{\pm 1}]$ coefficients. There exist choices of perturbative data so that the following hold:
\be
\item The good $\ReebRHypersurface_{5}$ orbits of the form $\orbit_{\zeta_{e^{+}}}, \orbit_{\zeta_{h}^{+}}$ span a subalgebra of $\Algebra(\alphaRHypersurface_{5}, \actionBound)$ which is canonically isomorphic to the algebra $\Algebra(\alpha^{\R \times \Liouville^{+}}_{2}, \actionBound)$ of Lemma \ref{Lemma:CHAugmentationModule} (in the case $(\Liouville, \beta) = (\Liouville^{+}, \beta^{+})$).
\item The good $\ReebRHypersurface_{5}$ orbits of the form $\orbit_{\zeta_{e^{-}}}, \orbit_{\zeta_{h}^{-}}$ span a subalgebra of $\Algebra(\alphaRHypersurface_{5}, \actionBound)$ which is canonically isomorphic to the algebra $\Algebra(\alpha^{\R \times \Liouville^{-}}_{2}, \actionBound)$ of Lemma \ref{Lemma:CHAugmentationModule} (in the case $(\Liouville, \beta) = (\Liouville^{-}, \beta^{-})$).
\item The good $\ReebRHypersurface_{5}$ orbits of the form $\orbit_{\zeta_{e^{+}}}, \orbit_{0}, \orbit_{\zeta_{e^{-}}}$ span a subalgebra of $\Algebra(\alphaRHypersurface_{5}, \actionBound)$. This subalgebra is canonically isomorphic to the algebra $\Algebra(\alphaCDivSet_{3}, \actionBound)$ of Lemma \ref{Lemma:CHCylinder}. Therefore this subalgebra of $\Algebra(\alphaRHypersurface_{5}, \actionBound)$ is a cylinder on $\Algebra(\alphaDivSet, \actionBoundDivSet)$ whose structure maps $I^{\pm}$ are given by the morphisms $\orbitDivSet \mapsto \orbit_{\zeta_{e^{\pm}}}$.
\ee
When combined, these identifications provide an identification
\begin{equation*}
\Algebra(\alphaRHypersurface_{5}, \actionBound) = \Algebra(\alpha^{\R \times \Liouville^{+}}_{2}, \actionBound) \otimes_{\Algebra} \Algebra(\alphaCDivSet_{3}, \actionBound) \otimes_{\Algebra}  \Algebra(\alpha^{\R \times \Liouville^{-}}_{2}, \actionBound), \quad \Algebra = \Algebra(\alphaDivSet, \actionBoundDivSet).
\end{equation*}
\end{lemma}

\begin{proof}
For item (1), the restriction of $\alphaRHypersurface_{5}$ to the submanifold $\R_{\tau} \times \Liouville^{+}_{\sigma_{0}}$ is exactly $\alphaRLiouville_{2}$ as described in Lemma \ref{Lemma:CHAugmentationModule} (in the case $\Liouville = \Liouville^{+}_{\sigma_{0}}$). The ``blocking'' result of Lemma \ref{Lemma:BlockingRRLiouville} tells us that perturbed holomorphic curves which are positively asymptotic to orbits of the form $\orbit_{\zeta_{e^{+}}}, \orbit_{\zeta_{h}^{+}}$ must be contained in $\R_{s} \times \R_{\tau} \times \Liouville^{+}_{\sigma_{0}}$, so long as our perturbations vanish along a neighborhood of the subset $\{ \sigma = \sigma_{0}\}$. Therefore the good orbits of the form $\orbit_{\zeta_{e^{+}}}, \orbit_{\zeta_{h}^{+}}$ generate a subalgebra of the chain-level contact homology algebra of $\R_{\tau} \times \hypersurface$ and this subalgebra is computed in Lemma \ref{Lemma:CHAugmentationModule}. The proof of item (2) is identical.

For item (3), consider a small, open neighborhood $N$ of the subset $\{ \tau =0, \sigma \in [\sigma(\zeta_{e^{-}}), \sigma(\zeta_{e^{+}})]\} \times \divSet$ in $\R_{\tau} \times \hypersurface$. We can identify $(N, \alphaRHypersurface_{5})$ with a subset of $\C \times \divSet$ containing $\{ \tau =0, \sigma \in [\sigma(\zeta_{e^{-}}), \sigma(\zeta_{e^{+}})]\} \times \divSet$, equipped with the contact form $\alphaCDivSet_{3}$. This follows from the fact that $f_{5}$ agrees with $f_{3}$ along $N$. Likewise, we can identify the almost complex structure $\JRHypersurface_{5}$ on $\R_{s} \times N$ with the restriction of $\JCDivSet_{3}$ to the corresponding subset of $\R_{s} \times \C \times \divSet$. All holomorphic curves in $\R_{s} \times \R_{\tau} \times \hypersurface$ asymptotic to the orbits of the form $\orbit_{\zeta_{e^{\pm}}}, \orbit_{0}$ are contained in $\R_{s} \times N$ (since they are contained in the leaves of the foliation $\foliationRHypersurface_{5}$ by Lemma \ref{Lemma:CurvesStuckInLeaves}). Hence all perturbed holomorphic curves contributing to contact homology differentials may be assumed to live in $\R_{s}\times N$. To see this consider perturbations which are sufficiently $\COne$ close to $0$, so that the perturbed curves counted are $\COne$ close to (unperturbed) holomorphic curves. Thus the identification of $N$ with a subset of $\C \times \divSet$ allows us to identify the subalgebra generated by orbits of the form $\orbit_{\zeta_{e^{\pm}}}, \orbit_{0}$ with $\Algebra(\alphaCDivSet_{3}, \actionBound)$. Our identification of this subalgebra with a cylinder on $\Algebra$ (with prescribed structure maps) is described in Lemma \ref{Lemma:CHCylinder}.

Finally, we note that the intersection of the subalgebras
\begin{equation*}
\begin{gathered}
\tensorAlgGraded\left((\oplus \Ring\orbit_{\zeta_{e^{\pm}}}) \oplus (\oplus \Ring \orbit_{\zeta_{h}^{\pm}}), \partial \right) \simeq \Algebra(\alpha^{\R \times \Liouville^{\pm}}_{2}, \actionBound),\\
\tensorAlgGraded\left((\oplus \Ring \orbit_{\zeta_{e^{+}}}) \oplus (\oplus \Ring \orbit_{0})\oplus (\oplus \Ring\orbit_{\zeta_{e^{-}}}), \partial \right) \simeq \Algebra(\alphaCDivSet_{3}, \actionBound)
\end{gathered}
\end{equation*}
of $\Algebra(\alphaRHypersurface_{5}, \actionBound)$ is the subalgebra $\tensorAlgGraded(\oplus \Ring \orbit_{\zeta_{e^{\pm}}}, \partial)$ which we can identify with $\Algebra$ itself following Lemma \ref{Lemma:CHMorphism} (applied to the leaf $\R_{s} \times \{ \zeta_{e^{\pm}}\} \times \divSet$). From these observations, the last statement of the lemma follows.
\end{proof}

Now we complete the proofs of Theorems \ref{Thm:MainChainLevel} and \ref{Thm:MainHomology} for $\Ring=\Field$ and $\Ring=\Field[\ChernVar^{\pm 1}]$ coefficients.

\begin{proof}[Proof of Theorem \ref{Thm:MainChainLevel}]
We take the contact form $\alpha$ in the statement of the theorem to be $\alphaRHypersurface_{5}$. The computation of $\Algebra(\alphaRHypersurface_{5}, \actionBound)$ and its homology is immediate from Lemma \ref{Lemma:GeometricDerivedTensor} when combined with
\be
\item \S \ref{Sec:CylinderDGAs} which identifies derived tensor products $\Ring_{\aug^{+}}\otimesDerived \Ring_{\aug^{-}}$ with the homologies of $\Module_{+} \otimes_{\Algebra} \Algebra^{c} \otimes_{\Algebra} \Module_{-}$ where $\Algebra^{c}$ is a cylinder for $\Algebra$ and the $\Module_{\pm}$ are semi-free $\Algebra$-modules quasi-isomorphic to the $\Ring_{\aug^{\pm}}$,
\item Lemma \ref{Lemma:CHCylinder} which identifies $\Algebra(\alphaCDivSet_{3}, \actionBound)$ with a cylinder over $\Algebra = \Algebra(\alphaDivSet, \actionBoundDivSet)$, and
\item Lemma \ref{Lemma:CHAugmentationModule} which identifies the $\Algebra(\alpha^{\R \times \Liouville^{\pm}}_{2}, \actionBound)$ as semi-free $\Algebra$-modules which are quasi-isomorphic to the $\Ring_{\aug^{\pm}}$.
\ee
The statement equating DG homotopy of the augmentations and non-vanishing of
\begin{equation*}
H\Algebra(\alphaRHypersurface_{5}, \actionBound) = \Ring_{\aug^{+}} \otimesDerived_{\Algebra} \Ring_{\aug^{-}}
\end{equation*}
in the case $\Ring=\C$ with $\Z/2\Z$ grading is then immediate from Lemma \ref{Lemma:HomotopicIffNonzeroDTP}. The statement regarding $CH$-admissible sequences follows from Lemma \ref{Lemma:CHSequence}.
\end{proof}

\begin{proof}[Proof of Theorem \ref{Thm:MainHomology}]
Following our computation of $\Algebra(\alphaRHypersurface_{5}, \actionBound)$, the first statement in Theorem \ref{Thm:MainHomology} is immediate from Lemma \ref{Lemma:SublinearOtimes}. To deal with $CH(\R_{\tau} \times \hypersurface, \xi)$ (without action bounds) we apply a colimit construction as in \cite[\S 9.3]{BH:ContactDefinition}.

Let $(\alphaDivSet_{i}, \actionBoundDivSet_{i})_{i \in \Z_{\geq 0}}, \alphaDivSet_{i} = h_{i}\alphaDivSet$ be a $CH$-admissible sequence of contact forms which allow us to compute contact homology as a colimit. We write $\Algebra_{k} = \Algebra(\alphaDivSet_{k}, \actionBoundDivSet_{k})$ for the associated contact homology algebras (for some choices of almost complex structures and perturbative data) which we assume to be sublinear. We also write
\begin{equation*}
\phi_{k}: \Algebra_{k} \to \Algebra_{k+1}, \quad H\phi_{k}: H\Algebra_{k} \to H\Algebra_{k+1}
\end{equation*}
for the (chain-level and homological) cobordism maps. The underlying cobordisms $(\Liouville_{i}, \beta_{i})$ here are taken to be the space between the graphs of $h_{i}$ and $h_{i+1}$ in the symplectization of $(\divSet, \alphaDivSet)$. By iteratively removing the $(\Liouville_{i}, \beta_{i})$ from the collar ends of the strict fillings $(\Liouville^{\pm}, \beta^{\pm})$ of $(\divSet, \alphaDivSet)$, we obtain strict fillings $(\Liouville^{\pm}_{i}, \beta_{i}^{\pm})$ of the $(\divSet, \alphaDivSet_{i})$ whose associated augmentations we denote by $\aug^{\pm}_{i}$. Appealing to the functoriality of contact homology, we obtain a commutative diagram
\begin{equation*}
\begin{tikzcd}[sep=large]
& & \Ring\\
\cdots \arrow[r]\arrow[drr]\arrow[urr] & H\Algebra_{i}\arrow[dr,"H\aug^{-}_{i}"] \arrow[ur, "H\aug^{+}_{i}", swap] \arrow[rr, "H\phi_{i}"] && H\Algebra_{i+1} \arrow[dl] \arrow[ul] \arrow[r] & \cdots \arrow[ull] \arrow[dll]\\
& & \Ring
\end{tikzcd}
\end{equation*}
Thus we may compute the maps $H\aug^{\pm}: CH\MxiDivSet \to \Ring$ as colimits of the $H\aug^{\pm}_{i}$.

Likewise we can compute the contact homology $CH(\R_{\tau} \times \hypersurface, \xi)$ using a colimit of contact homologies associated to contact forms $\alphaRHypersurface_{5, i}$ and action bounds $\actionBound_{i}$, where each $\alphaRHypersurface_{5, i}$ is determined by $\alphaDivSet_{i}$ and each $\actionBound_{i}$ is determined by $\actionBoundDivSet_{i}$ as in Lemma \ref{Lemma:DynamicsRHypersurface}. The existence of Liouville cobordisms $(\R_{\tau} \times \hypersurface, \alphaRHypersurface_{5, i}) \to (\R_{\tau} \times \hypersurface, \alphaRHypersurface_{5, i+1})$ follows from Lemma \ref{Lemma:CHSequence} with $f=f_{5}$.

We find that the following are equivalent:
\be
\item $H\aug^{+} \neq H\aug^{-}$.
\item $1 \in \im(H\aug^{+} - H\aug^{-})$, since we are working over a field.
\item There is some $i$ such that $1 \in \im(H\aug^{+}_{i} - H\aug^{-}_{i})$, since
\begin{equation*}
\im (H\aug^{+} - H\aug^{-}) = \cup_{i}\im (H\aug^{+}_{i} - H\aug^{-}_{i}) \subset \Ring = \Field,
\end{equation*}
and hence $H\aug^{+}_{i} \neq H\aug^{-}_{i}$.
\item There is some $i$ such that $H\Algebra(\alphaRHypersurface_{5, i}, \actionBound_{i}) = 0$ (by Theorem \ref{Thm:MainChainLevel} and Lemma \ref{Lemma:SublinearOtimes}).
\item $CH(\R_{\tau} \times \hypersurface, \xi) = 0$.
\ee
\end{proof}

\subsubsection{$\Field[H_{2}(\hypersurface)]$ coefficient systems}\label{Sec:FinalTwistedCoeffs}

In order to prove Theorems \ref{Thm:MainChainLevel} for $\Ring = \Field[H_{2}(\hypersurface)]$ coefficients we must define (chain-level) contact homology algebras for the pairs $(\alphaDivSet, \actionBoundDivSet)$ with coefficients in $\Ring$. Following the definitions of contact homology algebras from \S \ref{Sec:TwistedCoeffsOverview}, we must associate to relative class $[\mapHomotopyClass] \in H_{2}(\divSet, \orbitDivSet \cup \vec{\orbitDivSet})$ a class $h \in H_{2}(\hypersurface)$ whence $CH$ differentials may be defined using Equation \eqref{Eq:CHDiffTwisted}.

To this end, we choose a decomposition
\begin{equation*}
H_{1}(\hypersurface) = F \oplus T
\end{equation*}
whose direct summands have generators $f_{i}$ and $t_{i}$ as described in Equation \eqref{Eq:H1FreeTorsionSplitting}. Choose spanning surfaces $S_{\orbit}$ for the orbits of $\alphaRHypersurface_{5}$ as described in that equation.

\begin{assump}\label{Assump:SpanningSurface}
Let $\orbitDivSet$ be an orbit of $\ReebDivSet$ with associated orbits $\orbit_{\zeta_{e^{\pm}}}, \orbit_{\zeta_{h}^{\pm}}, \orbit_{0}$ of $\ReebRHypersurface_{5}$. We assume that each spanning surface $S_{\orbit_{\zeta_{e^{\pm}}}}$ and $S_{\orbit_{\zeta_{h}^{\pm}}}$ has the form $S_{\orbit_{0}} \cup (I \times \orbitDivSet)$ where $I \subset \{ |\sigma| < - \sigma_{2}\}$ is an interval.
\end{assump}

With the above assumptions in place, it then follows that if $[\mapHomotopyClass]$ is a $H_{2}(\divSet, \orbitDivSet \cup \vec{\orbitDivSet})$ class, then the classes
\begin{equation*}
h_{\zeta} = [\{ \zeta\} \times \{\mapHomotopyClass\} - S_{\orbit_{\zeta}} + \sum_{i} S_{\orbit_{i, \zeta}^{-}}] \in H_{2}(\hypersurface), \zeta \in \{ 0, \zeta_{e^{\pm}}, \zeta_{h}^{\pm}\}
\end{equation*}
are independent of $\zeta$. If $[\mapHomotopyClass^{\pm}]$ is a $H_{2}(\Liouville^{\pm}, \orbitDivSet)$ class, then we also define classes
\begin{equation*}
h^{\pm} = [\pi_{M}I_{\leaf^{\pm}}\mapHomotopyClass^{\pm} - S_{\orbit_{\zeta_{e}^{\pm}}}] \in H_{2}(\hypersurface).
\end{equation*}
Here $\leaf^{\pm}$ are the leaves of $\foliation^{\R \times \Liouville^{\pm}}_{2,0}$ which we view as being contained in $\R_{s} \times \R_{\tau} \times \hypersurface$.

\begin{defn}
We define $\Algebra(\alphaDivSet, \actionBoundDivSet, \Field[H_{2}(\hypersurface)])$ to be the algebra generated by the $\ReebDivSet$ orbits of action $\leq \actionBoundDivSet$ with differential given by
\begin{equation}
\partial \orbitDivSet = \sum \cm(\orbit, \vec{\orbit}^{-})^{-1}\#(\ZeroSet_{\Multisec}^{[\mapHomotopyClass]}(\orbitDivSet, \vec{\orbit}^{-})/\R_{s})e^{h_{\zeta}}\orbitDivSet^{-}_{1}\cdots \orbitDivSet^{-}_{\NnegativePunctures},
\end{equation}
modifying Equation \eqref{Eq:CHDiffTwisted}. We define augmentations
\begin{equation*}
\aug^{\pm}: \Algebra(\alphaDivSet, \actionBoundDivSet, \Field[H_{2}(\hypersurface)]) \to \Field[H_{2}(\hypersurface)]
\end{equation*}
by modifying Equation \eqref{Eq:CHCobUntwisted} as
\begin{equation*}
\aug^{\pm}\orbitDivSet = \sum \cm(\orbit)^{-1}\#\ZeroSet_{\Multisec}^{[\mapHomotopyClass]}(\orbitDivSet)e^{h^{\pm}}.
\end{equation*}
\end{defn}

With this definition in place, the proof of Theorem \ref{Thm:MainChainLevel} in the case $\Ring = \Field[H_{2}(\hypersurface)]$ follows the exact same strategy as in the $\Ring=\Field$ case. For simplicity, we adopt the temporary notation
\begin{equation*}
\Algebra = \Algebra(\alphaDivSet, \actionBoundDivSet, \Field[H_{2}(\hypersurface)]).
\end{equation*}
The counts of holomorphic curves are the same here as in the previous proof: It follows from Lemma \ref{Lemma:CHMorphism} that the associations $\orbitDivSet \mapsto \orbit_{\zeta_{e^{\pm}}} = \{ \zeta_{e^{\pm}}\} \times \orbitDivSet$ induce injective CDGA morphisms
\begin{equation*}
I^{\pm} : \Algebra \to \Algebra(\alphaRHypersurface_{5}, \actionBound, \Field[H_{2}(\hypersurface)]).
\end{equation*}
We also find that the subalgebra of $\Algebra(\alphaRHypersurface_{5}, \actionBound, \Field[H_{2}(\hypersurface)])$ generated by orbits $\orbit_{\zeta_{e^{\pm}}}$ and $\orbit_{0}$ is a cylinder object for  $\Algebra$. The subalgebra generated by the orbits of the form $\orbit_{\zeta_{e^{+}}}$ and $\orbit_{\zeta_{h}^{+}}$ is a semi-free $\Algebra$-module which is quasi-isomorphic to the module $\Field[H_{2}(\hypersurface)]_{\aug^{+}}$ determined by the augmentation $\aug^{+}$ as described in Lemma \ref{Lemma:CHAugmentationModule}. A similar statement applies to the subalgebra of $\Algebra(\alphaRHypersurface_{5}, \actionBound, \Field[H_{2}(\hypersurface)])$ generated by the $\orbit_{\zeta_{e^{-}}}$ and $\orbit_{\zeta_{h}^{-}}$. Hence Lemma \ref{Lemma:GeometricDerivedTensor} holds in the case $\Ring = \Field[H_{2}(\hypersurface)]$, so that the chain-level contact homology of the pair $(\alphaRHypersurface_{5}, \actionBound)$ is the derived tensor product of the $\Algebra$-modules $\Field[H_{2}(\hypersurface)]_{\aug^{\pm}}$.

\subsubsection{Contact forms $\alphaRHypersurface_{1}$ and foliations on symplectizations}\label{Sec:VaugonForms}

Now we describe contact forms $\alphaRHypersurface_{1}$ on $\R_{\tau} \times \hypersurface$. In the case $\dim \hypersurface = 2$, the $\alphaRHypersurface_{1}$ are very similar to contact forms constructed in \cite[\S 4]{Vaugon:Bypass}. The geometric structures described are clearly not necessary to prove our main theorems (Theorems \ref{Thm:MainHomology} and \ref{Thm:MainChainLevel}). We will use them to compute the contact homologies of neighborhoods of $\dim =2$ convex hypersurfaces in \S \ref{Sec:Computations} and then compare the results against those provided by our main theorems.

Let $\sigma_{0} \in (-\half, 0)$ be of arbitrarily small absolute value. Choose functions $f^{\pm}_{0}, g^{\pm} \in \Cinfty(\R_{\tau} \times [-3, \sigma_{0}])$ as in \S \ref{Sec:GH2} which, in particular, satisfy the following properties:
\be
\item $g^{\pm} = -\sigma$ along $\{ \sigma > -1\}$.
\item The $f^{\pm}_{0}$ have no critical points.
\item Outside of a closed neighborhood of the interval $\{ \tau=0, \sigma \in (-1, \sigma_{0})\}$, the $f^{\pm}_{0}$ are $\COne$ close to $1$, independent of $\tau$, and satisfy $\partial_{\sigma}f_{0}^{\pm} > 0$.
\ee

Let $f_{1} \in \Cinfty(\R_{\tau} \times [-1, 1])$ be a function satisfying the following properties:
\be
\item $f_{1}(\tau, \sigma) = f^{\pm}_{0}(\pm \tau, \pm \sigma)$ for $\pm \sigma \in [-1, \sigma_{0}]$ (as in Equation \eqref{Eq:FHOverlap}).
\item $f_{1}$ has a unique critical point of $\ind_{\Morse}=1$ within $\{ |\sigma| < -\sigma_{0}\}$. We assume that this critical point is $0$ and that $f_{1}(0) = 1$.
\item Outside of a closed neighborhood of $\{ \tau = 0, |\sigma| < -\sigma_{0}\}$, we assume that $f_{1}$ is independent of $\tau$.
\ee
Thus $0$ is the only critical point of $f_{1}$. The flow lines of the Hamiltonian vector field $X_{f_{1}}$ on $\R_{\tau} \times [-1, 1]_{\sigma}$ with respect to the symplectic form $d\tau \wedge d\sigma$ are depicted in the top-left subfigure of Figure \ref{Fig:AlphaF1}.

\begin{figure}[h]
\begin{overpic}[scale=.6]{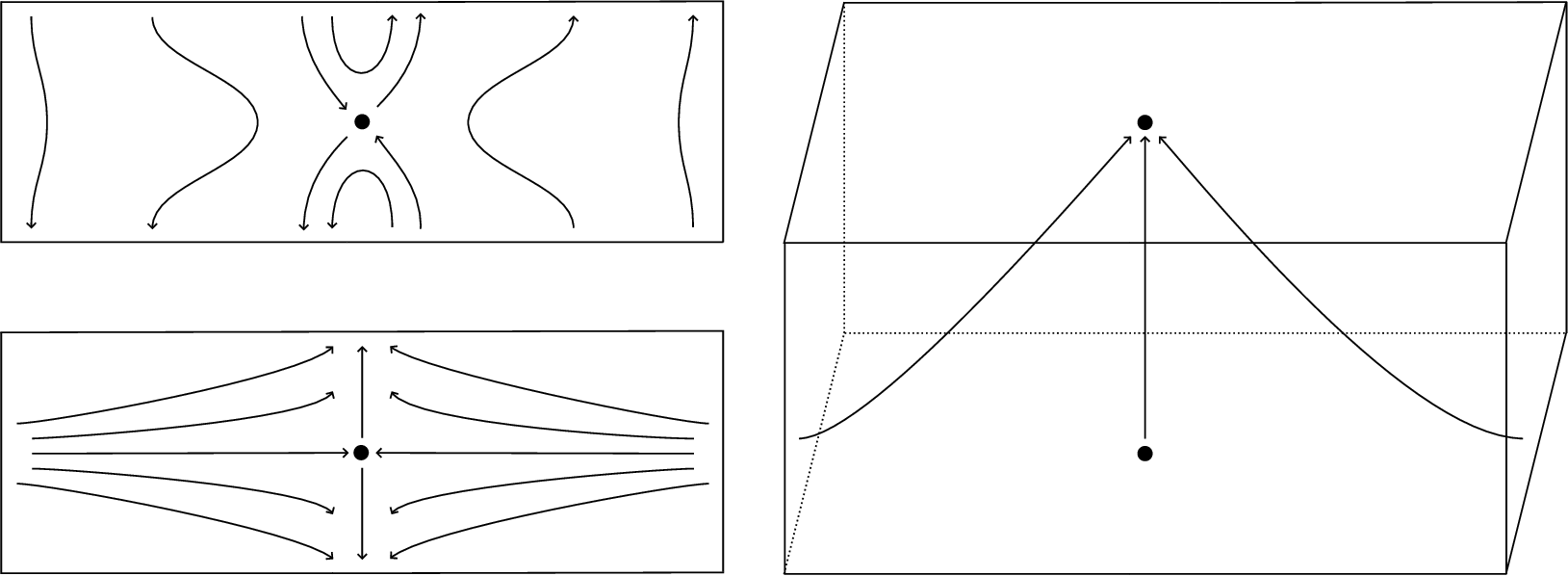}
\end{overpic}
\caption{The upper-left shows the projection of $\ReebRHypersurface_{1}$ to $\R_{\tau} \times [-1,1]_{\sigma}$. The lower-left shows associated gradient flow lines projected to $\R_{\tau} \times [-1,1]_{\sigma}$. In these subfigures, $\partial_{\tau}$ points upwards and $\partial_{\sigma}$ points to the left. Associated trajectories in $\R_{s} \times \R_{\tau} \times [-1,1]_{\sigma}$ are shown in the right subfigure.}
\label{Fig:AlphaF1}
\end{figure}

As in \S \ref{Sec:GeometryRHypersurface}, we construct a contact form on $\R_{\tau} \times \hypersurface$ such that
\be
\item the restriction to $\R_{\tau} \times \Liouville_{\sigma_{0}}^{\pm}$ is $\alpha^{\R \times \Liouville^{\pm}}_{0}$ and
\item the restriction to $\R_{\tau} \times [-1,1] \times \divSet$ is $\alpha_{f_{1}} = f_{1}\alphaDivSet -\sigma d\tau$.
\ee
This contact form and its associated Reeb vector field will respectively be denoted by
\begin{equation*}
\alphaRHypersurface_{1} \in \Omega^{1}(\R_{\tau} \times \hypersurface), \quad \ReebRHypersurface_{1} \in \mathfrak{X}(\R_{\tau} \times \hypersurface).
\end{equation*}
Following Lemma \ref{Lemma:DynamicsRHypersurface}, the following conditions are satisfied:
\be
\item The assignment $\orbitDivSet \mapsto \{ 0 \} \times \orbitDivSet = \orbit$ gives a one-to-one correspondence between
\be
\item closed $\ReebDivSet$ orbits $\orbitDivSet$ of action $\leq \actionBoundDivSet$ and
\item closed $\ReebRHypersurface_{1}$ orbits $\orbit$ of action $\leq \actionBound = \max(f_{1})\actionBoundDivSet$.
\ee
\item The actions, $\Z$-valued Conley-Zehnder indices, and $\Z$-valued contact homology gradings are
\begin{equation*}
\action(\orbit) = \action(\orbitDivSet), \quad \CZ(\orbit)= \CZ(\orbitDivSet), \quad |\orbit| = |\orbitDivSet| + 1.
\end{equation*}
\ee

The foliation $\foliationRHypersurface_{1}$ has exactly one leaf, $\R_{s} \times \{ 0 \}\times \divSet$, which is invariant under $\R_{s}$ translation. There are also families of leaves which are positively asymptotic to the Morse critical point $0$. The leaves in one family are diffeomorphic to $\complete{\Liouville}^{+}$ and the leaves in the other are diffeomorphic to $\complete{\Liouville}^{-}$.

\subsection{Generality of our constructions}

Here we establish that the contact forms $\alphaRHypersurface = \alphaRHypersurface_{5}$ and $\alphaRHypersurface = \alphaRHypersurface_{1}$ we've constructed can be used to describe neighborhoods of any convex hypersurface.

\begin{lemma}\label{Lemma:ConstructionGenerality}
Any $\tau$-invariant contact structure, $\xi$, on $\R_{\tau} \times \hypersurface$ is contact diffeomorphic to the kernel of some $\alphaRHypersurface$ as described above.
\end{lemma}

To prove the result, we need the following easy lemma.

\begin{lemma}
Following the above notation, there is a compactly supported isotopy $\Phi_{t}, t \in [0, 1], \Phi_{0}=\Id$ of $\R_{\tau} \times \hypersurface$ for which $\ker \Phi_{1}^{\ast}\alphaRHypersurface$ is $\tau$-invariant.
\end{lemma}

\begin{proof}[Sketch of the proof]
The functions $f$ and $g$ (with various sub- and superscript decorations) involved in the constructions of the contact forms $\alphaRHypersurface$ are assumed to be $\tau$-invariant outside of a neighborhood of $\{ \tau = 0\}$, and $\COne$-close to $\tau$-invariant functions. Thus they may be deformed to $\tau$-invariant functions yielding $1$-parameter families of contact structures. Since these contact structures differ from the $\alphaRHypersurface$ along a compact set, the usual Darboux-Moser-Weinstein argument \cite[p.112]{MS:SymplecticIntro} is applicable to find the isotopy $\Phi_{t}$.
\end{proof}

As a consequence of \cite[Proposition 6.4, Theorem 6.5]{DG:CircleBundles}, the contact diffeomorphism type of a $\tau$-invariant contact structure on a $\R_{\tau} \times \hypersurface$ depends only on the induced contact structure on $\divSet$ and the homotopy classes of Liouville structures on the positive and negative regions. Thus any such $\tau$-invariant contact structure is contact-diffeomorphic to one of the $(\R_{\tau} \times \hypersurface, \ker \alphaRHypersurface)$ described above. Applying such a contact diffeomorphism with the inverse $\Phi_{1}^{-1}$ of the map provided by the above lemma, we have established Lemma \ref{Lemma:ConstructionGenerality}.

\section{Computations in the case $\dim \hypersurface = 2$}\label{Sec:Computations}

In this section we compute contact homology algebras of neighborhoods $(\R_{\tau} \times \hypersurface, \xi)$ of $\dim = 2$ convex hypersurfaces. We start by explaining how Corollary \ref{Cor:dim2Computation} (equating tightness of such $\xi$ with non-vanishing of contact homology) can be established using previously established results in the cases of $\Ring=\Field$ and $\Ring=\Field[\ChernVar^{\pm 1}]$ coefficients. The fact that Corollary \ref{Cor:dim2Computation} can be proved using both Theorem \ref{Thm:MainHomology} and known results provides a basic sanity check on the validity of our main theorems. We furthermore carry out full computations of contact homologies for such $(\R_{\tau} \times \hypersurface, \xi)$ in Theorems \ref{Thm:ComputeGneq0} and \ref{Thm:ComputeG0}. The latter theorem includes a computation in the $\hypersurface \simeq \sphere^{2}$ case using $\Q[H_{2}(\sphere^{2})]$ coefficients. To the best of our understanding, Theorems \ref{Thm:ComputeGneq0} and \ref{Thm:ComputeG0} do not appear in the literature, though we suspect they are known to experts.

\subsection{Computations using known techniques}

\subsubsection{The case $\hypersurface \nsimeq \sphere^{2}$}

If $\hypersurface \nsimeq \sphere^{2}$, then Giroux's criterion (Theorem \ref{Thm:GirouxCriterion}) informs us that a $\tau$-invariant contact structure $\xi$ on $\R_{\tau} \times \hypersurface$ is tight iff none of the connected components of the $\Liouville^{\pm}$ are contractible.

In this case, \cite[\S 4]{Vaugon:Bypass} describes non-degenerate contact forms for $(\R_{\tau} \times \hypersurface, \xi)$ with exactly one embedded Reeb orbit for each connected component of the dividing set $\divSet$. Such contact forms are essentially identical to the $\alphaRHypersurface_{1}$ of \S \ref{Sec:VaugonForms}.

If $\xi$ is tight, then none of these connected components are contractible in $\hypersurface$. Hence none of the closed orbits can be the positive ends of (perturbed) holomorphic planes in $\R_{s} \times \R_{\tau} \times \hypersurface$. Therefore the contact homology differential cannot have constant terms and so has an augmentation (for any choice of $\Ring$). Thus tightness implies the non-vanishing of contact homology.

A proof of Corollary \ref{Cor:dim2Computation} in this case then follows from the fact that non-vanishing of contact homology implies tightness, as we've described in the introduction.

\subsubsection{The case $\hypersurface \simeq \sphere^{2}$}

Giroux's criterion states that $(\R_{\tau} \times \sphere^{2}, \xi)$ is tight iff the dividing set is connected.

In the tight case, a neighborhood of $\{\tau=0\}$ in $(\R_{\tau} \times \sphere^{2}, \xi)$ can be identified with a neighborhood of the boundary of a $\dim = 3$ Darboux disk, $(\disk^{3}_{x,y,z}, \xi_{std})$. Here
\begin{equation*}
\xi_{std} = \ker \alpha_{std}, \quad \alpha_{std} = dz + xdy - ydx
\end{equation*}
and the boundary of the disk can be seen to be convex by consideration of the radial vector field $z\partial_{z} + x\partial_{x} + y\partial_{y}$ which preserves $\xi_{std}$. When using $\Ring=\Field$ or $\Ring=\Field[\ChernVar^{\pm 1}]$ coefficients, the results of \cite[Theorem 1.9]{CGHH:Sutures} provide a morphism of unital algebras from the contact homology of $(\R_{\tau} \times \sphere^{2}, \xi)$ into that of $(\disk^{3}, \xi_{std})$. But $CH(\disk^{3}, \xi_{std}) = \Ring$ since the Darboux disk has an adapted contact form with no Reeb orbits, namely, $\alpha_{std}$. The existence of a morphism of unital, graded algebras $CH(\R_{\tau} \times \sphere^{2}, \xi) \to \Ring$ implies the non-vanishing of contact homology.

If $(\R_{\tau} \times \sphere^{2}, \xi)$ is not tight, then we can apply the results of \cite{Yau:VanishingCH} to conclude that the contact homology is zero. Thus tightness and non-vanishing of contact homology are equivalent in this case.

\subsection{Computations using the $\alphaRHypersurface_{1}$}\label{Sec:dim2Alpha1Computations}

Now we compute the contact homologies of the $(\R_{\tau} \times \hypersurface, \xi)$ using the contact forms $\alphaRHypersurface_{1}$ and the associated almost complex structures $\JRHypersurface_{1}$ and foliations $\foliationRHypersurface_{1}$ on the symplectization $\R_{s} \times \R_{\tau} \times \hypersurface$. In these cases $H_{1}(\R_{\tau} \times \hypersurface)$ is torsion free, so that $\Z$-valued gradings are available for $\Ring = \Field[H_{2}(\hypersurface)]$ coefficients. To simplify matters, we use $\Z/2\Z$-valued gradings for any choice of coefficient ring.

\subsubsection{The case $\hypersurface \nsimeq \sphere^{2}$}

In this case, suppose first that $\xi$ is overtwisted, so that there is some connected component $\divSet_{i} \subset \divSet$ of the dividing set bounding a connected component $\Liouville^{+}_{i}$ of the positive region for which $\Liouville^{+}_{i}$ is a disk.

Let $\orbit$ be the embedded Reeb orbit of $\ReebRHypersurface_{1}$ corresponding to $\divSet_{i}$. We may assume without loss of generality that $\orbit$ is the closed Reeb orbit of minimal action. According to our description of $\foliationRHypersurface_{1}$, there are leaves of this foliation which are holomorphic planes whose projection to $\hypersurface$ covers the interior of the disk $\Liouville^{+}_{i}$. According to Lemma \ref{Lemma:CurvesStuckInLeaves} all holomorphic planes asymptotic to $\orbit$ must simply cover such leaves. After adding marked points to stabilize domains of such maps and applying our automatic transversality results, we see that the count of holomorphic planes corresponding to these leaves (modulo $\R_{s}$-translation) must be $\pm 1$. Hence we conclude that $\partial \orbit = \pm 1$ with $\Ring=\Field \supset \Q$ coefficients and so $CH(\R_{\tau} \times \hypersurface, \xi, \Field) = 0$. For other coefficient systems, $\partial \orbit$ is a unit in $\Ring$, guaranteeing the vanishing of contact homology.

Now suppose that $\xi$ is tight so that the positive and negative regions have no disk components. According to Lemma \ref{Lemma:CurvesStuckInLeaves}, there are no holomorphic planes contributing to contact homology differentials. The same result coupled with inspection of the leaves of the foliation $\foliationRHypersurface_{1}$ tells us that any holomorphic curves contributing to contact homology differentials must be contained in $\R_{s} \times \{ 0\} \times \divSet$. Thus the contact homology differential is action preserving. To count such differentials, one must count (perturbed) double branched covers of the cylinders $\R_{s} \times \{0\} \times \divSet$. This counting problem is non-trivial and is worked out using obstruction bundle techniques in \cite{Fabert:Pants}, where it is shown that such counts are always zero. In conclusion, we find that the contact homology differentials vanish.

\begin{rmk}
The article \cite{Fabert:Pants} predates the formal constructions of contact homology in \cite{BH:ContactDefinition, Pardon:Contact}. The proofs of the main theorems of \cite{Fabert:Pants}, as they pertain to contact homology, can be translated to the Kuranishi setting via isomorphisms of the obstruction bundles involved with choices of transverse subbundles. However, this translation has not formally appeared in the literature.
\end{rmk}

As the orbits are all positive hyperbolic, they have even Conley-Zehnder indices. All such orbits are hence good and have odd contact homology gradings. In summary we have established the following result which mildly generalizes the cylindrical contact homology computations of \cite{Vaugon:Bypass}.

\begin{thm}\label{Thm:ComputeGneq0}
Suppose that $(\R_{\tau} \times \hypersurface, \xi)$ is a neighborhood of a $\dim = 2$ convex hypersurface with $\hypersurface \nsimeq \sphere^{2}$. If this neighborhood is overtwisted, then $CH(\R_{\tau} \times \hypersurface, \xi) = 0$. Otherwise, the contact homology algebra is freely generated by orbits $\orbit_{i, \cm}, \cm \in \Z_{\geq 1}$, each of odd grading, corresponding to the $\cm$-fold coverings of the connected components $\divSet_{i} \subset \divSet$ of the dividing set.
\end{thm}

\subsubsection{The case $\hypersurface \simeq \sphere^{2}$}

When $\hypersurface \simeq \sphere^{2}$ and has an overtwisted neighborhood, we can establish that $CH(\R_{\tau} \times \hypersurface, \xi) = 0$, by a curve counting argument identical to the $\hypersurface \nsimeq \sphere^{2}$ case: By inspecting the topology of the dividing set, we find an embedded Reeb orbit of minimal action whose contact homology differential is a unit in $\Ring$.

Now consider the case when the contact structure on $\R_{\tau} \times \sphere^{2}$ is tight with $\divSet$ connected. In this case there is one embedded $\ReebRHypersurface_{1}$ orbit $\orbit_{\cm}$ for each $\cm$-fold covering of $\divSet$ with $\cm \in \Z_{\geq 1}$. Again appealing to the fact that holomorphic curves must be stuck in the leaves of our foliation $\foliationRHypersurface_{1}$ on $\R_{s} \times \R_{\tau} \times \sphere^{2}$ and the results of \cite{Fabert:Pants}, we find that the only contributions to the contact homology differential $\partial$ are holomorphic planes which multiply cover the leaves of $\foliationRHypersurface_{1}$ which are diffeomorphic to $\C$. An index computation reveals that the only such planes which can contribute to $\partial$ are those which simply cover the leaves in question. Thus $\partial \orbit_{\cm} = 0$ for $\cm > 1$ and $\partial \orbit_{1} \in \Ring$.

In the case $\Ring = \Field$, we must have $\partial \orbit_{1} = 0$. Otherwise the contact homology would vanish, which we know to be impossible. On the other hand, there are two holomorphic planes contributing to $\partial \orbit_{1}$, one for the positive region and another for the negative region. In the case $\Ring = \Field[H_{2}(\sphere^{2})]$, we may take a spanning surface for $\orbit_{1}$ to be the positive region. The two holomorphic planes when capped off with the spanning surface yield distinct homology classes $0$ and $[\sphere^{2}]$ in $H_{2}(\sphere^{2})$. Because their signs must cancel when the $H_{2}(\sphere^{2})$ classes are ignored, it must be the case that
\begin{equation}\label{Eq:SphereDiff}
\partial \orbit_{1} = \pm (1 - e^{[\sphere^{2}]}).
\end{equation}
Thus this differential equates the generator $e^{[\sphere^{2}]} \in \Ring$ with $1 \in \Ring$ when passing to homology. Provided this explicit description of the CDGA, the computation of homology is straightforward. Our computation of contact homology can be summarized as follows.

\begin{thm}\label{Thm:ComputeG0}
Suppose that $(\R_{\tau} \times \hypersurface, \xi)$ is a neighborhood of a $\dim = 2$ convex hypersurface with $\hypersurface \simeq \sphere^{2}$. If this neighborhood is overtwisted, then $CH(\R_{\tau} \times \hypersurface, \xi) = 0$. Otherwise, when using $\Ring = \Field$ coefficients, the contact homology algebra is freely generated by orbits $\orbit_{\cm}, \cm \in \Z_{\geq 1}$, each of odd grading, corresponding to the $\cm$-fold coverings of the dividing set. When $\xi$ is tight and when using $\Ring=\Field[H_{2}(\sphere^{2})]$ coefficients, the contact homology algebra is isomorphic to the graded-commutative algebra over $\Field$ generated by the $\orbit_{\cm}$ for $\cm >1$.
\end{thm}

As a proof-of-concept demonstrating how our contact homology computations could be used in applications, we show how Theorem \ref{Thm:ComputeG0} can recover a classical result of Eliashberg-Gromov \cite{Eliash:DiskFillings, Gromov:JCurves}. 

\begin{cor}\label{Cor:ContractibleSphere}
Let $(X, \omega)$ be a symplectically-aspherical symplectic filling of a contact $3$-manifold, $\Mxi$. Then any embedded $2$-sphere in $M$ is contractible in $X$.
\end{cor}

\begin{proof}
Since we are establishing a known result, we will move rather quickly and leave some details to the reader. We first work out a homological version of the desired result, showing that the fundamental class of any $\sphere^{2} \subset M$ is in the kernel of the natural map $H_{2}(M) \to H_{2}(X)$.

To see this, first consider $CH(M, \xi, \Ring)$ with coefficients in $\Ring = \Field[H_{2}(M)]$. We can define $CH\Mxi$ with coefficients in $\Field[H_{2}(X)]$ by following the construction of \S \ref{Sec:TwistedCoeffsOverview} and using spanning surfaces in $X$ rather than spanning surfaces in $M$. By counting holomorphic planes in the completion of the filling which are positively asymptotic to Reeb orbits we may define an algebra morphism
\begin{equation*}
\Phi_{X}: CH(M, \xi, \Field[H_{2}(X)]) \to \Field[H_{2}(X)].
\end{equation*}
Since the filling $(X, \omega)$ is symplectically aspherical, the boundary strata of compactified moduli spaces that arise in the definition of this map are the same as in the Liouville setting. Hence $\Phi_{X}$ may be formally defined using the perturbation frameworks of either \cite{BH:ContactDefinition} or \cite{Pardon:Contact} without modification.

Now let $\sphere^{2} \subset M$ be an embedded $2$-sphere, which we assume to be convex without loss of generality by \cite{Giroux:Convexity}. The results of \cite{CGHH:Sutures} give us an arrow $I$ (induced by the inclusion $\sphere^{2} \to M$) in the diagram
\begin{equation*}
\begin{gathered}
\Field[H_{2}(\sphere^{2})] \to CH(\R_{\tau} \times \sphere^{2}, \xi, \Field[H_{2}(\sphere^{2})]) \xrightarrow{I}\\
CH(M, \xi, \Field[H_{2}(M)]) \to CH(M, \xi, \Field[H_{2}(X)]) \xrightarrow{\Phi_{X}} \Field[H_{2}(X)].
\end{gathered}
\end{equation*}
Since $\Mxi$ is symplectically fillable, it is tight and the neighborhood of the embedded sphere must be tight as well. Hence we may apply Theorem \ref{Thm:ComputeG0}, which says that $e^{[\sphere^{2}]} = 1 \in CH(\R_{\tau} \times \sphere^{2}, \xi, \Field[H_{2}(\sphere^{2})])$. Since the map $I$ preserves the unit $e^{[\sphere^{2}]} = 1 \in CH(M, \xi, \Field[H_{2}(M)])$ and $e^{[\sphere^{2}]} - 1$ is in the kernel of the above composition of morphisms. But this composition is induced by the maps $H_{2}(\sphere^{2}) \to H_{2}(M) \to H_{2}(X)$. Therefore $[\sphere^{2}] = 0$ in $H_{2}(X)$. We conclude that every embedded $2$-sphere in $M$ is homologically trivial in $X$.

By consideration of closed Reeb orbits in $M$ which are contractible in $X$ and using spanning disks in $X$ (rather than the more general spanning surfaces of \S \ref{Sec:TwistedCoeffsOverview}), we can define a version of \emph{contractible contact homology} (cf. \cite[\S 1.8]{Pardon:Contact}) of $\Mxi$ with coefficients in $\Field[\pi_{2}(X)]$. Running the above argument with this modified version of contact homology and replacing $H_{2}$ with $\pi_{2}$ throughout then informs us that $[\sphere^{2}]$ is in the kernel of the natural map $\pi_{2}(M) \to \pi_{2}(X)$.
\end{proof}

\subsection{Computations using the $\alphaRHypersurface_{5}$ and derived tensor products}

Now we compute contact homologies of neighborhoods of $\dim=2$ convex hypersurfaces using Theorem \ref{Thm:MainChainLevel}. The purpose of this computation is to sanity check that a computation using derived tensor products yields the same results as Theorems \ref{Thm:ComputeGneq0} and \ref{Thm:ComputeG0}. To be rigorous, we work with action bounds $\actionBoundDivSet, \actionBound > 0$ which we assume to be arbitrarily large.

To carry out such computations, we must first compute the chain-level contact homology algebra $\Algebra = \Algebra(d\theta, \actionBoundDivSet)$ of a disjoint union of circles which constitute the relevant $\MxiDivSet$. Here $d\theta$ is a volume form on $\divSet$ (making it a contact $1$-form). In this case we have one closed Reeb orbit $\orbitDivSet_{i, \cm}$ for each $\cm$-fold covering of the $i$th component $\divSet_{i}$ of $\MxiDivSet$. All orbits have even grading and so the contact homology differential must vanish. Since we are using an action bound, only finitely many such orbits are considered. If $(\Liouville, \beta)$ is a Liouville filling of $\MxiDivSet$ -- which topologically is an oriented Riemann surface with $\# \divSet$ boundary components -- then the induced augmentation $\aug_{\Liouville}$ counts the disk components of $\Liouville$ (up to signs). Thus after choosing orientations, we have that when using $\Ring=\Field$ coefficients
\be
\item $\aug_{\Liouville}\orbitDivSet_{i, \cm} = \pm 1$ when $\cm=1$ and $\divSet_{i}$ bounds a disk component of $\Liouville$ and
\item $\aug_{\Liouville}\orbitDivSet_{i, \cm} = 0$ otherwise.
\ee

Now we take a convex hypersurface of $\dim = 2$ and write $\aug^{\pm}$ for the augmentations determined by the positive and negative regions. Applying Theorem \ref{Thm:MainChainLevel} we can compute the contact homology of $(\R_{\tau} \times \hypersurface, \xi)$ (up to an action threshold $\actionBound$) as $\Ring_{\aug^{+}} \otimesDerived_{\Algebra} \Ring_{\aug^{-}}$. Since the differential on $\Algebra$ is zero, it is a sublinear CDGA as described in \S \ref{Sec:Sublinear}. Therefore, as described in Equation \eqref{Eq:SublinearDerivedTensor}, we can easily compute the derived tensor product using a ``small resolution''. This yields
\be
\item a graded-commutative tensor algebra over generators $\orbit_{i, \cm}$ for each $\orbitDivSet_{i, \cm}$ with $|\orbit_{i, \cm}| = |\orbitDivSet_{i, \cm}|+1$ and
\item differential $\partial \orbit_{i, \cm} = \aug^{-}\orbitDivSet_{i, \cm} - \aug^{+}\orbitDivSet_{i, \cm}$ which is zero unless $\cm=1$.
\ee

These chain-level models of the $\Ring_{\aug^{+}} \otimesDerived_{\Algebra} \Ring_{\aug^{-}}$ give us exactly the chain-level models used to compute contact homology in Theorems \ref{Thm:ComputeGneq0} and \ref{Thm:ComputeG0} (up to a prescribed action bound) as desired. To obtain $\Ring=\Field[H_{2}(\sphere^{2})]$ coefficients in the case of $\hypersurface \simeq \sphere^{2}$ with a tight neighborhood, following Assumptions \ref{Assump:SpanningSurface} yields Equation \eqref{Eq:SphereDiff} from which an action-bounded version of Theorem \ref{Thm:ComputeG0} follows.

\bigskip

\textsc{Sorbonne Université, Université Paris Cité, CNRS, IMJ-PRG, F-75005 Paris, France}\par\nopagebreak
\textit{Email:} \href{mailto:russell@imj-prg.fr}{russell@imj-prg.fr}\par\nopagebreak
\textit{URL:} \href{https://www.russellavdek.com/}{russellavdek.com}

\end{document}